\documentclass[10pt,twoside,openright]{report} 
\title{Dynamic alpha--invariants of del Pezzo surfaces with boundary}
\author{Jes\'us Mart\'inez Garc\'ia}
\date{2013}
\newcommand{\bigcell}[2]{\begin{tabular}{@{}#1@{}}#2\end{tabular}}

\usepackage[phd]{edmaths}
\usepackage{jmgstandard}
\usepackage[all]{xy}
\usepackage{epigraph}
\usepackage{tikz}
\usetikzlibrary{calc,intersections}
\usetikzlibrary{decorations.markings}

\begin{document}

\maketitle

\begin{abstract}
The global log canonical threshold, algebraic counterpart to Tian's alpha--invariant, plays an important role when studying the geometry of Fano varieties. In particular, Tian showed that Fano manifolds with big alpha--invariant can be equipped with a K\"ahler--Einstein metric. In recent years Donaldson drafted a programme to precisely determine when a smooth Fano variety $X$ admits a K\"ahler--Einstein metric. It was conjectured that the existence of such a metric is equivalent to $X$ being K-stable, an algebraic--geometric property. A crucial step in Donaldson's programme consists on finding a K\"ahler--Einstein metric with edge singularities of small angle along a smooth anticanonical boundary. Jeffres, Mazzeo and Rubinstein showed that a dynamic version of the alpha--invariant could be used to find such metrics.

The global log canonical threshold measures how anticanonical pairs fail to be log canonical. In this thesis we compute the global log canonical threshold of del Pezzo surfaces in various settings. First we extend Cheltsov's computation of the global log canonical threshold of complex del Pezzo surfaces to non-singular del Pezzo surfaces over a ground field which is algebraically closed and has arbitrary characteristic. Then we study which anticanonical pairs fail to be log canonical. In particular, we give a very explicit classification of very singular anticanonical pairs for del Pezzo surfaces of degree smaller or equal than $3$. We conjecture under which circumstances such a classification is plausible for an arbitrary Fano variety and derive several consequences. As an application, we compute the dynamic alpha--invariant on smooth del Pezzo surfaces of small degree, where the boundary is any smooth elliptic curve $C$.

Our main result is a computation of the dynamic alpha--invariant on all smooth del Pezzo surfaces with boundary any smooth elliptic curve $C$. The values of the alpha--invariant depend on the choice of $C$. We apply our computation to find K\"ahler--Einstein metrics with edge singularities of angle $\beta$ along $C$.
\end{abstract}

\declaration

\dedication{To Sanja, who kept me sane and happy while working on this thesis. To my parents, who encouraged me to study Mathematics without pushing me into them.}

\acknowledgements{I would like to thank my supervisor, Ivan Cheltsov, who took me as a student when I had one year less to carry out my PhD Thesis than a regular student. He generouslyshared his time and ideas with me and introduced me to challenging and very popular research topics, providing me with the background I needed. His motivation and enthusiasm have made my task both easier and enjoyable.

Another important source of encouragement and valuable career advice has been my second supervisor, Michael Wemyss. I am very grateful for his support.

My examiners, Arend Bayer and Damiano Testa, have given me very detailed comments and corrections, which have improved the shape and readability of this Thesis, as well as providing me with new ideas and different points of view to my arguments and results which are very useful.

I would also like to thank Yuji Odaka, Yuji Sano, Chi Li, Cristiano Spotti, Ruadhai Dervan, Hendrik \Suss, Andrea Fanelli, Costya Shramov, Yuri G. Prokhorov and Sean Paul for useful conversations.

I have been invited to speak on the work carried out in this Thesis in a few locations in the last year. I would like to thank Fedor Bogomolov, Yuri Prokhorov and Constantin Shramov and Hendrik \Suss at Steklov Institute of Mathematics and the Higher School of Economics (Moscow); Paolo Cascini, Andrea \emph{Fano} Fanelli, Yoshinori Gongyo and Yuji Odaka at Imperial College London; Jungkai Chen at the National Taiwan University (Taipei); and Xiuxiong Chen and Sean Paul at the Simons Centre for Geometry and Physics (Stony Brook). In all these places I was given the opportunity to present my work and discuss with others in a friendly and relaxied environment. The support was crucial in the last stages of writing up this thesis when I visited Taipei and Stony Brook and I am very grateful for the invitations.

Studying a Phd is a full-time learning process which would be very difficult, if not impossible without sufficient financial support. My studies have been supported by competitive funded by Fundaci\'on Caja Madrid, EPSRC and the School of Mathematics of the University of Edinburgh. I am grateful for their support. In times of imposed austerity, it is my hope they choose to keep funding new generations of mathematicians.

During my studies I have been lucky to travel to many conferences, workshops and graduate schools, in which I have enlarged the scope my research. Most of the financial support necessary for this trips has come from the Edinburgh Fund and the London Mathematical Society (LMS). I would like to thank them both and very especially the LMS, since it provides a very necessary breath of fresh air for Mathematics in the United Kingdom.

Many of the complicated bureacratic problems that any international student encounters, including meeting tight deadlines have been simplified by Mrs Gill Law, Graduate School Administrator of the School of Mathematics at the University of Edinburgh and I thank her for that.

I would also like to thank my colleagues in the School, who are not just colleagues but friends. Special thanks to my office mates Spiros Adams-Florou, Ciaran Meachan Patrick Orson and Becca Tramel, and to Marina, George, Hari, Elena, Pamela, Amy, Eric, Noel, Tim, Rachel, Joe, Noah, Noel, Chris and Barry.

Finally I am forever in debt to my wife, Sanja Marjanovi\'c, who has made my stay in Edinburgh a truly memorable one and who has always stayed by my side; to my parents Jes\'us and Carmen, who introduced me to Mathematics and supported me all the way through; and to my brother, Alberto, whose visits and phone calls always make me laugh.
}
\onehalfspacing
\tableofcontents

\chapter{Introduction}
\section{The Calabi Problem}
In 1954, Calabi asked a question in the International Congress of Mathematicians in Amsterdam, which in its most general form can be stated as follows:
\begin{que}
\label{que:Calabi}
When does a compact K\"ahler manifold admit a K\"ahler--Einstein metric?
\end{que}
Let $X$ be a complex compact variety of dimension $n$ with mild singularities. The existence of a K\"ahler-Einstein metric on $X$ makes sense only if we assume that the first Chern class of $M$ is either positive, zero, or negative. This problem is natural in algebraic geometry, when we consider $X$ to be projective, since all projective varieties are K\"ahler with a natural metric provided by the Fubini-Study metric of $\bbP^n$ restricted to $X$. The first Chern class condition in Question \ref{que:Calabi} also has a natural interpretation within birational geometry, corresponding to the canonical divisor $K_X$ being negative ($X$ is Fano), zero ($X$ is Calabi--Yau) or positive ($X$ is of general type), respectively. Therefore it has long been suspected that this problem should have a natural algebro--geometric interpretation. Question \ref{que:Calabi} was answered positively for $K_X\equiv 0$ by S.T. Yau in 1978 (see \cite{Yau-Calabi-announcement} and \cite{Yau-Calabi-conjecture-paper}), and by T. Aubin in 1976 for $K_X>0$ \cite{Aubin-CalabiProblem}. In the Fano case, little progress had been made until very recently, when Donaldson and his school revived the subject under a new approach. In the Fano case, pioneering work was carried out by G. Tian for smooth surfaces (see \cite{TianAlphaInvariant}, \cite{Tian-del-Pezzo-2}).

The first step towards answering Question \ref{que:Calabi} taken by Tian in \cite{TianAlphaInvariant} was to introduce a numerical invariant, $\alpha(X)$, known as Tian's $\alpha$--invariant, to give sufficient conditions for the existence of a K\"ahler--Einstein metric on a Fano manifold.
\begin{thm}
\label{thm:Tian-alpha}
Let X be a Fano variety with quotient singularities. Suppose
$$\alpha(X)>\frac{\dim X}{\dim X+1}$$
holds. Then X has an orbifold K\"ahler–Einstein metric.
\end{thm}
This Theorem was proved in the smooth case by Tian in \cite{TianAlphaInvariant}. The generalisation to singular cases are due to \cite{Demailly-Kollar} using methods in \cite{Nadel-KE-metrics}.

Tian's $\alpha$--invariant coincides with the global log canonical threshold, $\glct(X)$, as shown in \cite{Demailly-Kollar} and \cite[App.~A]{CheltsovShramovLct3folds}. This algebraic invariant, which is defined in Section \ref{sec:singularities-pairs} of this Thesis, is algebraic and has a birational nature. Therefore it is possible to compute it in many examples, providing that we have enough information on the variety $X$. Unfortunately Theorem \ref{thm:Tian-alpha} gives a sufficient but not necessary condition for the existence of a K\"ahler--Einstein metric. Indeed, $\glct(\bbP^2)=\frac{1}{3}$ (see Theorem \ref{thm:del-Pezzo-glct-charp}) but the Fubini-Study metric in $\bbP^2$ is clearly K\"ahler-Einstein.

Almost 20 years ago Yau and Tian suggested that $X$ being K\"ahler--Einstein should be equivalent to some type of algebro--geometric stability, known as K--stability. The definitions regarding K--stability are very technical and the details change depending on the author. We follow \cite{Chen-Donaldson-Sun-Kstability}, which we believe to be the most standard algebraic terminology.
\begin{dfn}
\label{dfn:test-configuration}
Let $X$ be a Fano variety with log terminal singularities. A \textbf{test-configuration} for $X$ is a flat family $\pi\colon \calX \rightarrow \bbA^1$ embedded in $\bbP^N\times \bbA^1$ for some $N$, invariant under a $\bbG_m$ action on $\bbP^{N}\times \bC$ covering the standard multiplicative action on $\bbA^1$ such that 
\begin{itemize}
\item $\pi^{-1}(1)\cong X$ and the embedding $X\subset \bbP^N$ is defined by the complete linear system $\vert -r K_{X}\vert $ for some $r\geq 1$;
\item The central fibre $\calX_0=\pi^{-1}(0) $ is a normal variety with klt singularities.

Let $L=\calO_{\bbP^N}(1)\vert_{\calX_0}$ be the hyperplane bundle, $L\rightarrow \calX_{0}$. The vector space $H^0(\calX_0, L^k)$ where $k\geq 0$ has a $\bbG_m$-action. Let $d_k=h^0(\calX_0, L^k)$ and $w_k$ be the total weight of the action. For large $k$, $d_k$ and $w_k$ are polynomials on the variable $k$ of degrees $n, n+1$ respectively. Thus 
$$\frac{w_k}{k d_k} = F_0 + F_1 k^{-1} + O(k^{-2})$$
and the \textbf{Donaldson-Futaki invariant of $\calX$} is $\DF(\calX)=F_1$.
\end{itemize}
\end{dfn}

\begin{dfn}
\label{dfn:K-stability}
A Fano variety $X$ is K-stable if for all test configurations $\calX$ such that $\calX_0\not\cong X$ we have $\DF(\calX)>0$.
\end{dfn}

\emph{A priori} further test configurations can be considered when defining K-stability, by relaxing the conditions on the singularities of $\calX_0$. However C. Li and C. Xu have proved in \cite{Li-Xu-K-stability} that the above definition of test configuration is enough to test K--stability.

\begin{conj}[{Yau-Tian-Donaldson, \cite{Yau-YTD-conjecture}, \cite{TianKEimpliesAnalyticKstability}}]
\label{conj:YTD}
Let $X$ be a complex Fano variety with klt singularities. Then
$$X \text{ is K-stable} \iff X \text{admits a K\"ahler--Einstein metric}.$$
\end{conj}

A lot of progress was recently achieved in solving this long standing conjecture. In particular R. Berman showed the following
\begin{thm}[Berman, \cite{Berman-ke-implies-kstability}]
\label{thm:berman}
Let $X$ be a complex Fano variety with klt singularities which admits a K\"ahler--Einstein metric, then $X$ is K--stable.
\end{thm}

A partial converse has been recently proved by X-X. Chen, S.K. Donaldson and S. Sun in \cite{Chen-Donaldson-Sun-Kstability}, \cite{Chen-Donaldson-Sun-Kstability1}, \cite{Chen-Donaldson-Sun-Kstability2} and \cite{Chen-Donaldson-Sun-Kstability3} and independently by Tian in \cite{Tian-Kstability-solution}:
\begin{thm}
\label{thm:Chen-Donaldson-Sun}
Conjecture \ref{conj:YTD} holds whenever $X$ is a smooth Fano complex manifold.
\end{thm}

Unfortunately Theorem \ref{thm:Chen-Donaldson-Sun} does not provide a very effective method to decide when a given Fano manifold admits a K\"ahler--Einstein metric, since the number of test configurations to check is huge. From that point of view Theorem \ref{thm:Tian-alpha} is often a better method.

Donaldson's programme to solve Theorem \ref{thm:Chen-Donaldson-Sun} was drafted in \cite{DonaldsonStabilityKEsurvey} and it uses a \emph{dynamic approach}. 

\begin{dfn}
\label{dfn:Kahler-Edge-metric}
Let $X$ be a compact K\"ahler manifold of dimension $n$ and $0<\beta\leq 1$. Let $D\sim -K_X$ be a smooth effective divisor and $\beta\in (0,1]$ such that $-(K_X+(1-\beta)D)$ is ample for $0<\beta\ll 1$. A \textbf{K\"ahler edge metric} on $X$ with singularities of angle $2\pi\beta$ along $D$ is a K\"ahler metric $g_\beta$ on $X\setminus D$ such that it is asymptotically equivalent at $D$ to the model edge metric
$$g_\beta := \vert z_1\vert^{2\beta-2} \vert dz_1\vert^2 +\sum_{j=2}^n  \vert dz_j\vert^2, $$
where $z_1, z_2, \ldots, z_n$ are holomorphic coordinates on $X$ such that $D =\{z_1 = 0\}$ locally. When $\beta=1$ they coincide with usual K\"ahler metrics.

If $g_\beta$ is also an Einstein metric, we say that  $g_\beta$ is a \textbf{K\"ahler--Einstein metric with edge singularities} of angle $2\pi\beta$ along $D$.
\end{dfn}

All the above concepts can be generalised to this \emph{dynamic} setting. In particular, we have dynamic versions of Tian's $\alpha$--invariant and Theorem \ref{thm:Tian-alpha} by T. Jeffres, R. Mazzeo and Y. Rubinstein and work of R. Berman:
\begin{thm}[\cite{JeffresMazzeoRubinstein}, see also \cite{Berman-dynamic-alpha}]
\label{thm:Jeffres-Mazzeo-Rubinstein}
Let $(X,D)$ be a smooth log pair where $X$ is a smooth Fano variety. Suppose that $-(K_X+(1-\beta)D)$ is ample for some $\beta\in (0,1]$. If
$$\alpha(S,\sum (1-\beta)D)>\frac{\dim X}{\dim X+1},$$
then there is a K\"ahler--Einstein metric with edge singularities of angle $2\pi\beta$ along $D$.
\end{thm}

The approach described in \cite{Chen-Donaldson-Sun-Kstability} to prove Theorem \ref{thm:Chen-Donaldson-Sun} is the following. First the authors construct a dynamic Donaldson-Futaki invariant $\DF_\beta(\calX)$, which is inear on $\beta$ and show that there is always a K\"ahler-Einstein edge metric for small $0<\beta_0\ll 1$ using Theorem \ref{thm:Jeffres-Mazzeo-Rubinstein}. Then a dynamic version of Theorem \ref{thm:berman} is used to show that $\DF_\beta(\calX)\geq 0$ for all test configurations $\calX$ of $X$ and small $0<\beta\ll 1$. If $X$ does not admit a K\"ahler--Einstein metric, then a test configuration $\calX$ is constructed for $X$ such that $\DF_{\beta'}(\calX)\leq 0$ for some $0<\beta_0<\beta'<1$. Since $\DF_\beta$ is linear on $\beta$, then $\DF_1(\calX)=\DF(\calX)<0$, so $X$ is not K--stable.

\section{Main results}
The above setting and the study by Tian and others of K\"ahler-Einstein metrics on non-singular del Pezzo surfaces suggests the following question.
\begin{que}
\label{que:KEE-metrics}
Let $S$ be a non-singular del Pezzo surface and $C$ be a smooth curve such that $-(K_S + (1-\beta)C)$ is ample for $0<\beta\ll 1$. For which $\beta$ can we find a K\"ahler--Einstein metric with edge singularities along $C$?
\end{que}
An obvious choice for $C$ are elliptic curves $C\in \vert-K_S\vert$. In this thesis we partially answer this question by computing $\alpha(S, (1-\beta)C)$ for all smooth complex del Pezzo surfaces and smooth curves $C\in\vert-K_S\vert$. This is the content of section $4$.

In fact, the only other choice of smooth irreducible curve $C$ is some rational $C\cong \bbP^1$ with certain numerical properties. They are classified in Theorem \ref{thm:del-Pezzo-dynamic-rational-classification}.

Another question comes up naturally:
\begin{que}
\label{que:KEE-metrics-boundary-condition}
Let $\beta<1$ be fixed, and $(X,\Delta)$ be a log pair, where $\Delta$ is a smooth effective hypersurface in a given fixed rational class. Does the existence of a K\"ahler--Einstein edge metric with edge singularities of angle $2\pi\beta$ along $\Delta$ depend on $\Delta$ or does it only depend on the class of $\Delta$?
\end{que}
Our results give indications that the existence of these metrics depends on the class of $\Delta$, since the value of $\alpha(S, (1-\beta)C)$ for $C\in \vert-K_S\vert$ depends on $C$, its value being higher for very general members. On the other hand, for the existence of such a metric, a requirement is that $-(K_S+(1-\beta)\Delta)$ is ample. This condition only depends on the rational class of $\Delta$.

Computing the global log canonical threshold of a Fano variety $X$ or a pair $(X,\Delta)$ requires certain level of understanding of effective $\bbQ$-divisors $D\simq (-K_X-\Delta)$ with singularities worse than log canonical. While this understanding does not need to be thorough for $\alpha(X)$, when computing $\alpha(X,(1-\beta)\Delta)$ when $\Delta$ varies, understanding which pairs $(X,D)$ have bad singularities for $\Delta\simq D$ is crucial. For this reason, trying to answer Question \ref{que:KEE-metrics}, we asked ourselves the following question:
\begin{que}
\label{que:Cats}
Given a Fano variety $X$, can we classify all effective $\bbQ$-divisors $D\simq-K_X$ such that $(X,D)$ is not log canonical?
\end{que}
In this thesis we develop a setting in which Question \ref{que:Cats} can be answered in a rather precise way according to $\Supp(D)$, where $(X,D)$ is not log canonical only if $\Supp(D)$ contains some very singular effective divisor $T\in \vert-mK_X\vert$ for some small $m$. In the literature $D$ is usually known as a \emph{tiger}. For this reason we coined the term \emph{cat} for such $T$. Cats live in a finite number of rational classes $\vert -mK_X\vert$.

K--stability can be understood also when $\charac(k)>0$. In that case the algebraic counterpart of Tian's $\alpha$--invariant, the global log canonical threshold $\glct(X)$, is considered. For instance, Y. Odaka and Y. Sano proved
\begin{thm}[\cite{OdakaSanoAlphaInvariant}]
\label{thm:OdakaSanoResult}
Let $X$ be a $\mathbb{Q}$-Fano variety of dimension $n$ and  suppose that $\glct(X)> \frac{n}{n+1}$ (resp. $\glct(X)\geq \frac{n}{n+1}$). 
Then, $X$ is K-stable $($resp.\ K-semistable$)$. 
\end{thm}

Their proof uses resolution of singularities for dimension $n$, so it is valid in finite characteristic when $\dim(X)\leq 3$.

Although we have introduced K--stability in the context of K\"ahler-Einstein metrics, it is interesting in birational geometry on its own right. For instance, in \cite{OdakaAnnals}, Odaka shows that, given certain conditions, if $(X,L)$ is $K$-stable where $L$ is an ample line bundle then $X$ has only semi-log canonical singularities (the proof assumes $\charac(k)=0$).

This thesis started with a \emph{pet project}: to compute $\glct(S)$ for $S$ a non-singular del Pezzo surface over an algebraically closed field $k$. The computations of the global log canonical thresholds of del Pezzo surfaces were carried out by I. Cheltsov in \cite{CheltsovLCTdP} when $k=\bbC$. 
\begin{thm}[{see \cite{CheltsovLCTdP} and \cite{JMGlctCharP}}]
\label{thm:del-Pezzo-glct-charp}
Let $S$ be a non-singular del Pezzo surface over an algebraically closed field $k$. Then:
\begin{equation*}
\glct(S)=\left\{
	\begin{array}{rl}
		1		&\text{when } K_S^2=1 \text{ and } \vert -K_S\vert \text{ has no cuspidal curves}\\
		5/6 &\text{when } K_S^2=1 \text{ and } \vert -K_S\vert \text{ has a cuspidal curve}\\
		5/6 &\text{when } K_S^2=2 \text{ and } \vert -K_S\vert \text{ has no tacnodal curves}\\
		3/4 &\text{when } K_S^2=2 \text{ and } \vert -K_S\vert \text{ has a tacnodal curve}\\
		3/4 &\text{when } K_S^2=3 \text{ and } \forall C \in \vert -K_S\vert,\ C \text{ has no Eckardt points}\\
		2/3 &\text{when } K_S^2=3 \text{ and } \exists C \in \vert -K_S\vert \text{ with an Eckardt point}\\
		2/3 &\text{when } K_S^2=4\\
		1/2 &\text{when } K_S^2=5,6 \text{ or } S\cong\bbP^1\times\bbP^1\ (K_S^2=8) \\
		1/3 &\text{when } K_S^2=7,9 \text{ or } S\cong\bbF_1\ (K_S^2=8)
\end{array}
\right.
\end{equation*}
\end{thm}
The cases $K_S^2=2,3,4$ in Theorem \ref{thm:del-Pezzo-glct-charp} had been proved in \cite{CheltsovLCTdP} using results not known to be true in finite characteristic. In this thesis we give a new proof for these cases, using the answer to Question \ref{que:Cats} when $K_S^2=2,3$ (section \ref{sec:delpezzo-cats}) and providing a different proof when $K_S^2=4$ (section \ref{sec:delPezzo-4-curves}). These new proofs are independent of the characteristic of $k$.

In Chapter \ref{chap:singularities} we provide basic definitions and results of the theory of singularities of pairs. Then we construct our theory of tigers and cats, which is needed in order to answer Question \ref{que:Cats}. We finish the chapter with a study of singularities in the case of surfaces, recalling well-known results which use intersection theory. Furthermore, we provide new results which are crucial tools for chapters 3 and 4.

Chapter \ref{chap:del-Pezzo} starts with a survey of well-known properties of del Pezzo surfaces. While we believe most of these results to be well known we did not find them in a shape and notation adequate to our needs, which is why we prove several Lemmas of del Pezzo surfaces in Section \ref{sec:delPezzo-properties}. The brave reader may want to skip to the second section of Chapter \ref{chap:del-Pezzo}, where we answer Question \ref{que:Cats} for non-singular del Pezzo surfaces. The last section focuses in the non-singular del Pezzo surface of degree $4$, computing the log canonical threshold, in order to complete the proof of Theorem \ref{thm:del-Pezzo-glct-charp}.

Chapter \ref{chap:dynamic} is the main part of this thesis, where we compute the dynamic $\alpha$--invariant of all non-singular del Pezzo surfaces with a boundary elliptic curve. Each of the degrees has its own section.

Finally, Appendix \ref{app:lcts} contains glorious computations of log canonical thresholds of pairs that we have outsourced from the main text in order to facilitate readability.

\chapter{Singularities of pairs, a theory of cats and tigers}
\label{chap:singularities}
\epigraph{\textit{``I'm off to check my tiger trap!"}}{Calvin, in Calvin \& Hobbes}

\section{Singularities of pairs}
\label{sec:singularities-pairs}
The study of singularities of algebraic varieties and hypersurfaces in them is essential in Birational Geometry. Historically, singularities were studied in terms of local equations, or as quotients of affine space by a group action. This approach allowed P. Du Val to classify del Pezzo surfaces with mild singularities (see \cite{DuVal1}, \cite{DuVal2} and \cite{DuVal3}) in the early 20th century.

In the early 80s, S. Mori pioneered new methods to understand birational contractions (see \cite{Mori-Hartshorne-Conjecture} and \cite{Mori-Cone-Theorem}) with the Cone and Contraction Theorem. Mori results were a generalisation to higher dimensions of Castelnuovo's Contractibility Criterion for non-singular surfaces, establishing intersection theory as the standard tool to understand birational contractions. However, while in dimension $2$ we can contract curves to a non-singular surface, in higher dimensions the contraction theorem may contract a smooth locus into a singular point. A new theory of singularities was needed, developed in terms of the resolution of singular points and the multiplicities of the exceptional divisors. All these ideas form what is nowadays known as \emph{Minimal Model Theory, Mori Theory} or the \emph{Minimal Model Programme}.

In this section we introduce basic notions in the theory of singularities of pairs, following \cite{KollarMori} and \cite{Corti-book-Flips-3folds-4folds}.  Our goal is to define the global log canonical threshold, which is equivalent to Tian's $\alpha$--invariant. The standard reference for the theory of singularities in the Minimal Model Programme is \cite{Kollar-Singularities-of-pairs}.

\begin{dfn}
\label{dfn:log-pair}
Let $X$ be an algebraic variety of dimension $n$ over an algebraically closed field $k$ of characteristic $p\geq 0$. Let $D=\sum b_i D_i$ be an effective $\bbQ$-divisor in $X$ where $D_i$ are prime Weil divisors and $b_i\in \bbQ$. Suppose $K_X+D$ is $\bbQ$-Cartier. We say that $(X, D)$ is a \textbf{log pair}, or simply a \textbf{pair}.
\end{dfn}

\begin{dfn}
Given any birational morphism $f\colon Y\ra X$ and an exceptional divisor $E\subset Y$ such that $f(E)=p$ where $p\in X$ is a point we will say that $E$ is \textbf{exceptional} over $X$. We will say that $F$ is a divisor over $p$ if either $F$ is exceptional over $p$ or $p\in F\subset X$.
\end{dfn}

\begin{dfn}
\label{dfn:discrepancy}
The \textbf{discrepancy} of $(X,D)$ is given by
$$\disc(X,D)=\inf_E\{a(E,X,D)\setsep E \text{ is exceptional}\}$$
and the \textbf{total discrepancy} of $(X,D)$ is given by
$$\totdisc(X,D)=\inf_E\{a(E,X,D)\setsep E \text{ is a divisor}\}.$$

Given $p\in X$, we define the \textbf{discrepancy} of $(X,D)$ at $p$ by
$$\disc_p(X,D)=\inf_E\{a(E,X,D)\setsep E \text{ is an exceptional divisor over } p\}$$
and the \textbf{total discrepancy} of $(X,D)$ at $p$ is
$$\totdisc_p(X,D)=\inf_E\{a(E,X,D)\setsep E \text{ is a divisor over } p\}.$$

\end{dfn}
In the above definitions the infimum is taken over all the irreducible exceptional divisors $E_i$ and all birational morphisms $f\colon Y \ra X$ and over all irreducible divisors of $X$.

\begin{nota}
When $D=0$, we write
$$a(E_i,X)=a(E_i,X,0),\qquad \disc(X)=\disc(X,0)=\totdisc(X,0).$$
\end{nota}

\begin{lem}[{\cite[Lem. 2.2.7]{KollarMori}}]
\label{lem:disc-monotonous}
If $(X,D)$ is a log pair and $D'$ is an effective $\bbQ$-Cartier $\bbQ$-divisor, then
$$a(E,X,D)\geq a(E,X,D+D')$$
for all divisors $E$ over $X$.
\end{lem}

\begin{lem}[{\cite[Cor. 2.3.1]{KollarMori}}]
\label{lem:disc-range}
Let $D$ be a $\bbQ$-divisor such that $K_X+D$ is $\bbQ$-Cartier.
\begin{itemize}
	\item[(i)] Either $\disc(X,D)=-\infty$ or
				$$-1\leq \totdisc(X,D)\leq\disc(X,D)\leq 1.$$
	\item[(ii)] If $X$ is smooth then $\disc(X)=1$.
\end{itemize}
\end{lem}

\begin{dfn}
\label{dfn:MMP-singularities}
Let $(X,D)$ be a pair for $X$ a normal projective variety over an algebraically closed field $k$. Assume $D$ is effective. Then $(X, D)$ is
\begin{equation*}
\left.
\begin{array}{r}
\text{terminal}\\
\text{canonical}\\
\text{Kawamata log terminal (klt)}\\
\text{log canonical}
\end{array} \right\} \text{ at } p \text{ if } \totdisc_p(X,D) \left\{
\begin{array}{l}
>0\\
\geq0\\
>-1\\
\geq -1,\\
\end{array} \right. 
\end{equation*}

When $(X,D)$ is terminal (respectively canonical, klt, log canonical) $\forall p\in X$, we simply say that $(X,D)$ is terminal (respectively canonical, klt, log canonical).

We say $X$ is terminal (respectively canonical, klt, log canonical) if $(X,0)$ is terminal (respectively canonical, klt, log canonical).
\end{dfn}

By Lemma \ref{lem:disc-range} (ii) smooth varieties are terminal. By Lemma \ref{lem:disc-range} (i), log canonical is the widest possible class of singularities for which the concept of discrepancy still makes sense. In this thesis we will be mostly focused on the log canonical class, which is important when defining the log canonical threshold. Kawamata log terminal singularities behave well in cohomology. For instance, the following result can be generalised to log pairs with klt singularities:
\begin{thm}[{Kawamata-Viehweg vanishing theorem, see \cite{Lazarsfeld1}[Thm. 4.3.1]}]
\label{thm:KawamataViehweg}
Let $X$ be a smooth complex projective variety of dimension $n$ and $D$ a nef and big divisor on $X$. Then
$$H^i(X, \calO_X(K_X+D))=0 \text{ for }i>0.$$
\end{thm}
There are other classes of singularities that we will not use such as \emph{purely log terminal} and \emph{divisorial log terminal}. The essay \emph{What is log terminal?} by O. Fujino in \cite{Corti-book-Flips-3folds-4folds} discusses all these concepts.

\begin{dfn}
\label{dfn:log-resolution}
Let $(X,D)$ be a log pair, where $D=\sum b_i D_i$. A \textbf{resolution} of $X$ is a proper birational morphism $f\colon Y \ra X$ such that $Y$ is smooth. A \textbf{log resolution} of $(X,D)$ is a resolution $f\colon Y \ra X$ of $X$ such that $(D_i)_Y$ is smooth for all $i$, the set $\Supp(D_Y)\cup \Exc(f)$ has simple normal crossings and the exceptional locus of $f$, $\Exc(f)$ has pure codimension $1$ or is empty.
\end{dfn}
If $\dim X=2$, there is a \textit{canonical} choic of log resolution, consisting on blowing-up the singular locus of $X$ (which consists of isolated points) and the singular locus in $\Supp(D)$ until obtaining simple normal crossings.

The condition on the codimension of $\Exc(f)$ in the above definition is not always included in the literature (for instance in \cite{KollarMori}). However more recent works (see \cite{Corti-book-Flips-3folds-4folds}) include it in order to simplify arguments. We follow the convention in \cite{Corti-book-Flips-3folds-4folds} in Definition \ref{dfn:log-resolution}. In any case, if $\codim(\Exc(f))\geq 2$ we can blow-up the smooth locus of higher codimension to obtain a log resolution in the sense of Definition \ref{dfn:log-resolution}.

When $\charac(k)=0$, log resolutions exist in all dimensions by \cite{HironakaSingResol}. An accessible source to the proof is \cite{HauserSingResol} or the more recent \cite{Kollar-Resolution-Singularities}. If $\charac(k)=p>0$, log resolutions exists when $\dim X=n\leq 3$ by the classical work of S.S. Abhyankar in high characteristics and the very recent proof of V. Cossart and O. Piltant \cite{CossartPiltant1}, \cite{CossartPiltant2}.

When resolution of singularities is available, it is enough to test the discrepancy on a log resolution, as the following result illustrates.
\begin{lem}[{\cite[Cor. 2.32]{KollarMori}}]
\label{lem:disc-log-resolution}
Given $X$, let $f\colon Y \ra X$ be any resolution of singularities of $X$ where $\Exc(f)$ has pure codimension $1$. Let $D=\sum b_j D_j$ where $b_j\leq 1$. Then there is a log resolution for $(X,D)$ such that if $a(E_i,X,D)\geq -1$ for all exceptional $E_i$, then
$$\disc(X,D)=\min\{\min_i\{a(E_i,X,D)\},\min_j\{1-b_j\},1\}.$$
\end{lem}
In particular, this Lemma tells us that the infimum in Definition \ref{dfn:discrepancy} is actually a minimum.

\begin{rmk}
In Lemma \ref{lem:disc-log-resolution} the existence of the resolution $f$ does not seem to play any role. Indeed it does. So far our setting only assumes that the ground field $k$ of our varieties is algebraically closed. When the characteristic of $k$ is finite, it is an open problem whether resolutions of singularities exist. This is only known in dimensions $1$--$3$. The proof of Lemma \ref{lem:disc-log-resolution} requires the existence of a resolution of singularities for $X$.

\end{rmk}

\begin{dfn}\label{dfn:lcts}
\mbox{}
\begin{itemize}
	\item[(i)] The \textbf{log canonical threshold} of a pair $(X,D)$ is
$$\lct(X,D)=\max\{\lambda \setsep (X,\lambda D) \text{ is log canonical}\}.$$
	\item[(ii)] The \textbf{local log canonical threshold} of the pair $(X,D)$ at $p\in X$ is
$$\lct_p(X,D)=\max\{\lambda \setsep (X,\lambda D) \text{ is log canonical at } p\}.$$
\end{itemize}
\end{dfn}

Observe that
$$\lct(X,D)=\min_{p\in X}\lct_p(X,D).$$

\begin{exa}
\label{exa:P2cuspidal}
Let $C=\{y^2-x^3=0\}\subset S=\bbA^2$ and $p=(0,0)$. Then $\lct(S,C)=\lct_p(S,C)=\frac{5}{6}$. Indeed, $S$ is smooth and $C$ is an irreducible and reduced Cartier divisor whose only singularity is at $p$. The minimal log resolution $f\colon Y\ra S$ consists of $3$ blow-ups over $p$ with exceptional divisors $E_1,E_2,E_3$ (see Figure \ref{fig:cusp}). The log pullback of $(S,C)$ is
$$f^*(K_S+C)\sim K_Y+C_Y + (2-1)E_1 + (3-2)E_2 + (6-4)E_3.$$
In particular $(S,D)$ is not log canonical since $a(E_3,S, C)=-2$. Consider the log pullback of $(S,\lambda C)$, for $\lambda \in \bbQ$:
$$f^*(K_S+\lambda C)\sim K_Y+\lambda C_Y + (2\lambda-1)E_1 + (3\lambda-2)E_2 + (6\lambda-4)E_3.$$
Then $a(E_1,S, \lambda C)\geq -1$ and $a(E_2,S, \lambda C)\geq -1$ for $\lambda\leq 1$ and $a(E_3,S,\lambda C)=4-6\lambda\geq -1$ for $\lambda\leq \frac{5}{6}$. Therefore $\lct_p(S,C)=\frac{5}{6}.$
\begin{figure}[ht]%
\begin{center}
\includegraphics[scale=0.5]{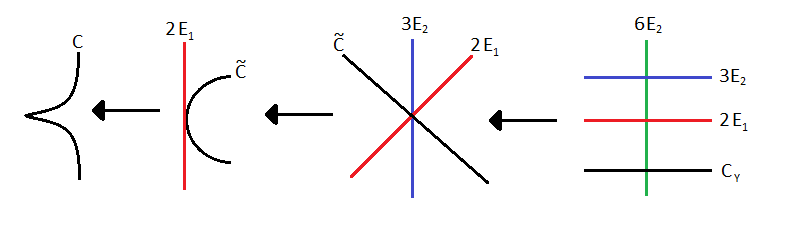}
\end{center}
\caption{Resolution of a cuspidal cubic curve.}%
\label{fig:cusp}%
\end{figure}
\end{exa}

This example is a particular case of the following result.
\begin{lem}[{see \cite{Igusa}}]
\label{lem:Igusa}
Let $f(x,y)=x^m+y^{n} + g(x,y)\in k[x,y]$ where $m\leq n$ and $g$ is a polynomial of degree $d\geq n+1$. Let $p=(0,0)\in \bbA^2$ and $C=\{f=0\}$. Then
$$\lct_p(\bbA^2, C)=\frac{1}{m}+\frac{1}{n}.$$
\end{lem}

\begin{dfn}\label{dfn:LCS}
The \textbf{locus of log canonical singularities} of a log pair $(X,D=\sum d_i D_i)$ is the closed set:
$$
\LCS(X,D)=\left(\bigcup_{d_i\geq 1} D_i\right)\cup \left(\bigcup_{a_j\leq -1}\sigma(F_j)\right)\subsetneq X.
$$
where $F_j$ are the reduced exceptional divisors over $X$ for any birational morphism $\sigma \colon Y\ra X$. This is called the \textit{non-klt locus} in \cite{Kollar-Singularities-of-pairs}.
\end{dfn}

Suppose $(X,D)$ is not log canonical at some point $p$. If $\codim(\LCS(X,D))\geq 2$ we will say that $(X,D)$ is \textbf{log canonical in codimension $1$} or \textbf{log canonical near $p$ but not at $p$}. Moreover, if $\dim(\LCS(X,D))=0$ we will say that $p$ is an \textbf{isolated locus of non-log canonical singularities} or simply that it is \textbf{isolated}, when no confusion is likely. When $\dim X=2$ these two concepts coincide.

\begin{dfn}
\label{dfn:glct}
Let $X$ be a Fano variety. Let $\Delta$ be an effective $\bbQ$-divisor such that $-(K_X-\Delta)$ is ample and $(X,\Delta)$ has klt singularities. The \textbf{global log canonical threshold} of $(X,\Delta)$ is
$$\glct(X,\Delta)=\sup\{\lambda \setsep (X,\Delta + \lambda D) \text{ is log canonical } \forall D\simq -(K_X-\Delta) \text{ effective }\bbQ\text{-divisor}\}.$$
When $\Delta=0$, we simply write $\glct(X)=\glct(X,0)$. When $\Delta=(1-\beta)D$, we call the function $\glct(X,(1-\beta)D)\colon (0,1] \ra \bbR$ with variable $\beta$ the \textbf{dynamic global log canonical threshold} of $(X,D)$.This function can be shown to be continuous when varying $\beta$.
\end{dfn}

\begin{thm}[{see \cite{Demailly-Kollar} or \cite[App. A]{CheltsovShramovLct3folds}}]
Let $(X,D)$ be a pair with klt singularities such that $-(K_X+D)$ is ample. The global log canonical threshold of $(X,D)$coincides with Tian's $\alpha$--invariant:
$$\glct(X,D)=\alpha(X,D).$$
\end{thm}

When the ground field $k$ of $X$ is $k=\bbC$, we will use the notation $\alpha(X,D)$, and when $k=\bar k$ is simply algebraically closed, we will use $\glct(X, D)$.

It is not yet known whether $\glct(X)$ is a rational number. The following conjecture, due to Tian for complex Fano manifolds (see \cite{Tian-conjecture}), generalises the one in \cite[Conj 1.4]{Cheltsov-Shramov-Park-WPS} for complex Fano varieties:
\begin{conj}
\label{conj:Tian-glct-stabilisation}
Let $X$ be a projective Fano variety over an algebraically closed field. Suppose $X$ is $\bbQ$-factorial and has at wost log terminal singularities. Then 
\mbox{}
\begin{itemize}
	\item[(i)] $\exists D\simq -K_X$, an effective $\bbQ$-divisor on $X$ such that $\glct(X)=\lct(X,D)$.
	\item[(ii)] $\glct(X)$ is a rational number.
\end{itemize}
\end{conj}
Note that (i) implies (ii). The conjecture is not known to be true even for complex del Pezzo surfaces with quotient singularities. However, there is strong evidence to support it. In fact, in the case of complex del Pezzo surfaces with Du Val singularities, $D$ can be found in $\vert -mK_X\vert$ for $m\leq 6$ (see \cite{Kosta-thesis}, \cite{Park-Won-lctDuVal}).

When computing global log canonical thresholds the following results and its corollaries are essential. They allow us to remove some irreducible divisor $B_i$ from any other effective $\bbQ$-divisor $D$ for which we only know that $(X,D)$ is not log canonical, where $B_i\subseteq\Supp(B)$ for $B$ an effective divisor that we understand. Then we can use intersection theory on $B_i$ and $D$. Together with Lemma \ref{lem:adjunction} in section \ref{sec:surfaces-intersection-theory}, the following results are an essential tool for computations.
\begin{lem}
\label{lem:higher-dim-convexity-forward}
Let $X$ be a $\bbQ$-factorial variety, $D$ and $B$ effective $\bbQ$-divisors on $X$. If $(X,D)$ and $(X,B)$ are log canonical then, for all $\alpha\in [0,1]\cap\bbQ$, the pair
$$(X,\alpha D+(1-\alpha)B)$$
is log canonical.
\end{lem}
\begin{proof}
Let $f\colon Y\ra X$ be any proper birational morphism with exceptional divisor $\bigcup E_i$. Then we may write
$$\alpha K_Y + \alpha D_Y \simq \alpha f^*(K_X+D) + \alpha \sum a_i E_i,$$
$$(1-\alpha) K_Y + (1-\alpha) B_Y\simq (1-\alpha) f^*(K_X+B)+(1-\alpha) \sum b_i E_i,$$
where $a_i\geq -1$, $b_i\geq -1$, since $(X,D)$, $(X,B)$ are log canonical.
Adding the two equivalences, we obtain
$$K_Y + \alpha D_Y +(1-\alpha)B_Y\simq f^*(K_X+\alpha D + (1-\alpha) B) + \sum (\alpha a_i +(1-\alpha)b_i)E_i.$$
Observe that
$$\alpha a_i + (1-\alpha) b_i \geq -\alpha -(1-\alpha) = -1$$
so $(X,\alpha D + (1-\alpha)B)$ is log canonical.
\end{proof}
Often, we will use the following corollary:
\begin{cor}
\label{cor:higher-dim-convexity}
Let $X$ be a $\bbQ$-factorial variety, $D$ and $B$ effective $\bbQ$-divisors on $X$, with $D\simq B$. If $(X,D)$ is not log canonical but $(X,B)$ is log canonical, then there is $\alpha\in [0,1)\cap\bbQ$ such that the pair
$$(X,D':=\frac{1}{1-\alpha}(D-\alpha B))$$
is not log canonical, $D'$ is effective and the support of $D'$ does not contain at least one of the irreducible components of $B$. Observe that $D'\simq D$.
\end{cor}
\begin{proof}
Suppose $(X,D')$ is log canonical. Then $(X,(1-\alpha)D'+\alpha B)$ is log canonical by Lemma \ref{lem:higher-dim-convexity-forward}, but this is impossible, since $(1-\alpha)D'+\alpha B=D$.

We may write $D=\sum_{i=1}^n d_i D_i +\Omega$, $B=\sum^n_{i=1}b_i D_i$ where $\Omega$ is an effective $\bbQ$-divisor such that for all $i$, $d_i\geq 0$, $b_i>0$ and $D_i\not\subseteq \Omega$. We need to show how to choose $\alpha$. If $d_i=0$ for some $i$ then take $\alpha=0$. Hence $D'=D$ and $B_i\not\subseteq\Supp(D')$. Now, assume all $d_i>0$. Let
$$\alpha=\min\left\{\frac{d_i}{b_i}\setsep b_i\neq 0\right\}.$$
Observe that $\alpha<1$, since otherwise
$$0\leq D-B \simq 0$$ and $D=B$. Hence, for some $i$ we have $B_i\not\subseteq\Supp(D')$, finishing the proof.
\end{proof}

\begin{cor}[log convexity]
\label{cor:higher-log-convexity}
Let $X$ be a $\bbQ$-factorial variety. Suppose $(X,A+D_1)$ is not log canonical and $(X,A+D_2)$ is log canonical where $A,D_1,D_2$ are effective $\mathbb{Q}$-divisors, such that $D_1\simq D_2$. There is $\lambda \in [0,1)\bigcap \mathbb{Q}$ such that 
$$D_3=\frac{1}{1-\lambda} (D_1-\lambda D_2),$$
is effective, $\Supp(D_3)$ does not contain at least one of the components of $D_2$ and $(X,A+D_3)$ is not log canonical.
\end{cor}
\begin{proof}
By Corollary \ref{cor:higher-dim-convexity}, the pair $$(X,\frac{1}{1-\lambda}(A+D_1-\lambda(A+D_2)))$$
is not log canonical. This is the same pair as in the statement.
\end{proof}

\begin{lem}
\label{lem:pairs-fixed-boundary-lcs}
Let $0<\beta\leq 1 $ and $0<\lambda\leq 1$. Let $X$ be a non-singular Fano variety and $\Delta\sim-K_X$ be a non-singular irreducible and reduced hypersurface. Let $D\simq-K_X$ be an effective $\bbQ$-divisor such that
\begin{equation}
(X,(1-\beta)\Delta+\lambda\beta D)
\label{eq:pairs-fixed-boundary-lcs}
\end{equation}
is not log canonical at some $p\in X$ and $\lambda\beta\leq\glct(X)$. Then $p\in \Delta$. and the pair \eqref{eq:pairs-fixed-boundary-lcs} is log canonical in codimension $1$.
\end{lem}
\begin{proof}
Since $X$ and $\Delta$ are non-singular, then $(X, (1-\beta)\Delta + \lambda\beta \Delta)$ is log canonical. In particular, by Corollary \ref{cor:higher-log-convexity}, we may assume that $\Delta\not\subseteq\Supp(D)$.

If $(X,(1-\beta)\Delta+\lambda \beta D)$ is log canonical at all points in $\Delta$, then $(X,\lambda \beta D)$ is not log canonical, which is impossible since $\lambda\beta\leq \glct(X)$ and $D\simq-K_X$. Therefore $p\in \Delta$.

If the pair \eqref{eq:pairs-fixed-boundary-lcs} is not log canonical along a hypersurface $D_i$, then $D_i=\Delta$, by the same argument. However, then $1-\beta>1$, which is impossible.
\end{proof}

\section{Tigers and cats}
In this section we develop a setting to classify pairs $(X, D)$ which are not log canonical, where $X$ is a normal Fano variety and $D$ an effective $\bbQ$-divisor with $D\simq -K_X$. When the propoposed setting is possible the classification is very precise, giving an accurate desription of $\Supp(D)$An an application we obtain an easy way to compute $\glct(X, \Delta)$ and $\alpha(X, (1-\beta)\Delta)$. When this classification is possible for $X$, we have coined the expression \emph{the Cat Property holds on $X$} for reasons given below.

Our original motivation to develop this theory was to study birational maps between del Pezzo fibrations. The classification of non log-canonical pairs plays an important role in this setting, as hinted in \cite{Cheltsov-birationally-rigid-del-Pezzo-fibrations}. While unsuccessfully working on this problem, the author realised that when the Cat Property holds, the computation of the dynamic $\alpha$--invariant becomes almost trivial, as shown below in Lemma \ref{lem:cat-property-easy-dynamic-alpha}.

The author found out later that for a non-singular complex del Pezzo surface $S$ of low degree, the Cat Property plays an important role when showing there is no non-trivial $\bbG_a$-action on $\Aut(\hat{S})$, the automorphism group of the affine cone of $S$. The study of this problem has been pioneered by T. Kishimoto, Y. Prokhorov and M. Zaidenberg. In \cite{Kishimoto-Prokhorov-Zaidenberg1}, the authors showed that such a $\bbG_a$--action exists when $\deg S \geq 4$. In \cite{Kishimoto-Prokhorov-Zaidenberg-general} a criterion of existence for the $\bbG_a$--action is provided. The study of this problem for $3$-folds has been started in \cite{Kishimoto-Prokhorov-Zaidenberg-3fold}. In \cite{Kishimoto-Prokhorov-Zaidenberg2}, the case of non-singular del Pezzo surfaces of low degrees is studied, proving the non-existence of a non-trivial $\bbG_a$ action when $1\leq \deg S\leq 2$. The Cat Property is implicitly discussed. In particular it follows that if the Cat Property holds for a smooth cubic surface $S$, there is no non-trivial $\bbG_a$-action on $\Aut(\hat S)$. While the author showed that the Cat Property holds on cubic surfaces (Theorem \ref{thm:del-Pezzo-cat-degree-3}), the same result was proved independently by I. Cheltsov, J. Park and J. Won in \cite{Cheltsov-Park-Won-Sancho}. J. Park and J. Won originally arrived to this problem motivated by a \emph{semi-local} version of Tian's $\alpha$--invariant proposed by F. Ambro in \cite{Ambro-Minimal-log-discrepancy}. 

In \cite{Keel-McKernan-book} Keel and \McKernan introduced the concept of \textit{tiger}, possibly as an inside joke following Miles Reid's general elephant conjecture. A general elephant, as coined by Miles Reid in \cite{Reid-YPG}, is a general element of the anti-canonical linear system and therefore expected to have good singularities. A tiger would be somewhat the opposite, an element in the anti-canonical class with very bad singularities. Colorful names aside, the advantage of this term is that it allow us to use just one word to summarise a long list of conditions:
\begin{dfn}
\label{dfn:tiger}
Let $X$ be a $\bbQ$-factorial projective Fano variety with klt singularities. Let $(X,D)$ be a log pair where $D$ is an effective $\bbQ$-divisor with $D\simq-K_X$. If $(X,D)$ is not log canonical we will call $D$ a \textbf{tiger}.
\end{dfn}
Note that for a tiger $D$ we have $\lct(X,D)<1$. Actually, in \cite{Keel-McKernan-book}, tigers included also any effective $\bbQ$-divisor $D\simq-K_X$ such that $(X,D)$ is not klt (i.e. $\lct(X,D)\leq 1$). In \cite{Keel-McKernan-book} such $D$ is actually called a \textit{special tiger}, reserving the name \textit{tiger} for a divisor $E$ over $X$ with discrepancy $a(X,D,E)\leq -1$. We will stick to Definition \ref{dfn:tiger} since it fits better our purposes.

We are interested in tigers in $\vert-mK_X\vert$ for $m$ small, which will be called \textit{cats}. In order to define a cat precisely, we need the following:
\begin{dfn}
\label{dfn:cat-index}
Suppose $X$ satisfies Conjecture \ref{conj:Tian-glct-stabilisation}, i.e. there is an effective $\bbQ$-divisor $D$ such that $D\simq-K_X$ and $\glct(X)=\lct(X,D)$. We define the \textbf{Cat index} of $X$ as:
$$\Cat(X)=\min\{m\in \bbZ_{>0} \setsep \glct(X)=\lct(X,\frac{1}{m}D) \text{ for } D\in -mK_X\vert \}.$$
If $X$ does not satisfy Conjecture \ref{conj:Tian-glct-stabilisation}, then define $\Cat(X)=\infty$.
\end{dfn}

\begin{exa}
\label{exa:Cat-index}
From Theorem \ref{thm:del-Pezzo-glct-charp}, we have $\glct(\bbP^2)=\frac{1}{3}$. Let $L$ b e a line in $\bbP^2$. Then $\lct(\bbP^2,3L)=\frac{1}{3}$. Since $3L\in \vert-K_{\bbP^2}$, then $\Cat(\bbP^2)=1$. In particular $D_1=3L$ or $D_2=\{y^2z-x^3=0\}$ are cats since $D_1,D_2\in \vert-K_{\bbP^2}\vert$ but $(\bbP^2,D_1)$ and $(\bbP^2,D_2)$ are not log canonical.

\end{exa}

\begin{dfn}
\label{dfn:cat}
If $D$ is a tiger of $X$ with $mD\in \vert-mK_X\vert$ and $m\leq \Cat(X)$, then we say that $D$ is a \textbf{cat}.
\end{dfn}

\begin{dfn}
\label{dfn:cat-property}
We say that a set $U\subseteq X$ satisfies the \textbf{Cat Property}, or that the Cat Property holds on $U$, if for all tigers $D\simq-K_X$ such that $(X,D)$ is not log canonical on $U$, there is an effective $\bbQ$-divisor $T\simq-K_X$ with $mT\in \vert-mK_X\vert$ for some $m\leq \Cat(X)$, such that $\Supp(T)\subseteq\Supp(D)$.
\end{dfn}
We are mostly interested in the case $U=X$.

\begin{lem}
\label{lem:Cat-property-not-lc-has-cats}
Let $X$ be a Fano variety satisfying the Cat Property such that $\exists D$ a tiger of $X$. Then there is a cat $T$ with $\Supp(T)\subseteq\Supp(D)$.
\end{lem}
\begin{proof}
If the Cat Property holds on $X$, by Corollary \ref{cor:higher-dim-convexity} we may assume that $(X,T)$ in Definition \ref{dfn:cat-property} is not log canonical (i.e.~$T$ is a cat). Indeed, suppose that for all effective $\bbQ$-divisors $T\simq-K_X$ such that $mT\in \vert-mK_X\vert$ with $m\leq \Cat(X)$ the pair $(X,T)$ is log canonical. Then we may subtract $\Supp(T)$ from $\Supp(D)$ using Corollary \ref{cor:higher-dim-convexity}  to obtain another tiger $D'\simq-K_X$. Since $(X,D')$ is not log canonical and for all effective $\bbQ$-divisors $T\simq-K_X$ with $mT\in \vert-mK_X\vert$, with $m\leq \Cat(X)$ we have $\Supp(T)\not\subset\Supp(D)$, we contradict the Cat Property.
\end{proof}

\begin{lem}
\label{lem:del-Pezzo-Cat-index}
Let $S$ be a non-singular del Pezzo surface. Then $\Cat(S)=1$.
\end{lem}
\begin{proof}
This is the first step of the computation of $\glct(S)$ where we find a divisor $D\in vert-K_S\vert$ such that $\glct(S)=\lct(S,D)$. See \cite{CheltsovLCTdP} or \cite{JMGlctCharP}.
\end{proof}

By far, most Fano varieties do not seem to satisfy the Cat Property. For instance, in the following two examples $\Cat(X)<\infty$ but we can find sequences of tigers $D_m$ with irreducible support where $\min\{m\in \bbN\setsep mD_m\in \vert-mK_X\vert\}$. 
\begin{exa}
\label{exa:plane-not-cat-property}
Consider in $S=\bbP^2$ the curve given by $C_m=\{zy^m+x^{m+1}=0\}$. Locally around $p=(0:0:1)$ we have $\lct_p(S,C_m)=\frac{1}{m}+ \frac{1}{m+1}=\frac{2m+1}{m(m+1)}$ by Lemma \ref{lem:Igusa}. Let $D_m=\frac{3}{m+1}C_m\simq -K_S$. Then
$$\lct(S,D_m)=\frac{m+1}{3m}\lct_p(S,C_m)=\frac{2m+1}{3m}<1$$
for all $m>1$. Note that $C_m$ is irreducible and by Lemma \ref{lem:del-Pezzo-Cat-index}, $\Cat(\bbP^2)=1$, so one can always find a tiger $(\bbP^2,D_m)$ such that no $D'\in \vert-K_S\vert$ satisfies $D'\subset \Supp(D_m)$.
\end{exa}

\begin{exa}
\label{exa:F1-not-cat-property}
Let $\pi:\widetilde S =\bbF_1\ra \bbP^2$ be the blow-up of $p=(0:0:1)$ with exceptional curve $E\cong \bbP^1$ and let $D_m$ be the $\bbQ$-divisor in $\bbP^2$ from the previous example. As we will see later, by Lemma \ref{lem:log-pullback-preserves-lc}, the pair
\begin{equation}
\big(\widetilde S,D'_m=\widetilde D_m +(\frac{3}{m+1}-1)E\big)
\label{eq:CatDeg9Aux1}
\end{equation}
is not log canonical, where $\widetilde D_m$ is the strict transform of $D_m$. Moreover $D'_m\simq-K_{\bbF_1}$. The exceptional curve $E$ is given in local coordinates $(x,y)$ by the equation $y=0$. The curve $\widetilde C_m$, the strict transform of the curve $C_m$ in $\bbP^2$ used in Example \ref{exa:plane-not-cat-property}, is given locally by the equation $y=x^m$. Therefore $\widetilde D_m \cdot E =\mult_pD_m = \frac{3m}{m+1}>2$ if and only if $m>2$. Notice also that $\mult_pD_m <3$. Hence, for $m>2$, the coefficient of $E$ in \eqref{eq:CatDeg9Aux1} is bigger than $1$, but $\Supp(\widetilde D_m)$ is irreducible and contains no cat in it. Finally the pair in \eqref{eq:CatDeg9Aux1} is a tiger, since $D'_m\simq-K_{\widetilde \bbF_1}$ but $\lct(\bbF_1, D'_m)\leq \lambda_m$ for $\lambda_m=1/(\frac{3m}{m+1}-1)=\frac{m+1}{2m-1}<1$ for $m>2$.
\end{exa}

This is not the only kind of bad behaviour. We can also find simple tigers $D$ whose support has simple normal crossings but $(X,D)$ is not log canonical in codimension $1$, see example \ref{exa:del-Pezzo-no-cat-codimension1}. On the other hand all non-singular del Pezzo surfaces $S$ with $K_S^2\leq 3$ satisfy the Cat Property (see Lemmas \ref{lem:del-Pezzo-cat-degree-1}, \ref{lem:del-Pezzo-cat-degree-2} and Theorem \ref{thm:del-Pezzo-cat-degree-3}). This is also the case for any general point $p\in S_4$, a del Pezzo surface of degree $4$ (see Lemma \ref{lem:del-Pezzo-cat-degree-4-generic}). Based on the behaviour in dimension $2$, we formulate the following conjecture:
\begin{conj}
\label{conj:cat}
Let $X$ be a non-singular Fano variety of dimension $n\geq 2$. $X$ satisfies the Cat property if and only if all effective $\bbQ$-divisors $D\simq-K_X$ are log canonical in codimension~$1$.
\end{conj}
The conjecture is verified in dimension $2$ in Corollary \ref{cor:cat-conjecture-dim2}. Unfortunately the proof uses heavily the classification of del Pezzo surfaces and its birational geometry and it does not provide any deep insight on how the proof should work in higher dimensions. However, should Conjecture \ref{conj:cat} hold, we could split the study of tigers in two distinct classes, depending on whether the Cat Property holds or not. This is evidenced by Chapter \ref{chap:dynamic}. Indeed, when computing $\alpha(S, (1-\beta)C)$ for $S$ a non-singular del Pezzo surface and $C$ a smooth elliptic curve, the methods used are very different, depending onwhether the Cat Property holds or not.

Let us explore a few applications of the Cat Property. The following is straight forward from the definition of Cat index:
Suppose $X$ satisfies the Cat Property. Then it is straight forward to compute the global log canonical threshold, since we just need to compute $\lct(X,\frac{1}{m}D)$ for very singular $D\in \vert -mK_X\vert$ and $0<m\leq \Cat(X)$.
\begin{obs}
\label{obs:cat-property-easy-glct}
Let $X$ be a Fano variety on which the cat property holds. Then
$$\glct(X)=\min\{\lct(X,D)\setsep mD\in \vert-mK_X\vert,\ m\leq \Cat(X)\}.$$
\end{obs}

Unfortunately so far all proofs to show that a variety satisfies the Cat Property are generalisations of the computation of the global log canonical threshold and in any case we need to know the global log canonical threshold \emph{a priori} in order to define $\Cat(X)$, which is required to define the Cat Property. Nevertheless, when present, the Cat Property gives us control over all tigers of $X$, which comes useful when studying the dynamic alpha--invariant of a complex Fano variety:
\begin{lem}
\label{lem:cat-property-easy-dynamic-alpha}
Let $X$ be a complex Fano variety with $\Cat(X)<\infty$. Suppose the Cat Property holds on $X$. Suppose $\exists \Delta\in \vert-K_X\vert$ a smooth effective divisor. Then
$$\alpha(X,(1-\beta)\Delta)=\lct(X,(1-\beta)\Delta,\beta T)$$
for some effective $T\simq -K_X$ such that $mT\in \vert-mK_X\vert$ for some $m\leq \Cat(X)$.
\end{lem}
\begin{proof}
There are effective $\bbQ$-divisors $D\simq-K_X$ such that
\begin{equation}
\omega:=\alpha(X, (1-\beta)\Delta) \leq \lct(X, (1-\beta)\Delta, \beta D)\leq \lct(X, (1-\beta)\Delta, \beta \Delta)=\lct(X,\Delta)=1,
\label{eq:cat-property-easy-dynamic-alpha-aux1}
\end{equation}
given that $\Delta$ is smooth. We proceed by contradiction. In particular we may assume $\omega<1$. Suppose, for contradiction, that any effective $\bbQ$-divisor $D\simq-K_X$ for which \eqref{eq:cat-property-easy-dynamic-alpha-aux1} holds satisfies
$$\min\{m\setsep mD\in \vert-mK_X\vert\}>\Cat(X).$$
Let $D$ be one of those $\bbQ$-divisors.

Suppose that $\forall T$, effective $\bbQ$-divisor with $T\simq -K_X$ and $mT\in \vert-mK_X\vert$ for some $m\leq \Cat(X)$ we have $\Supp(T)\not\subseteq\Supp(D)$. Then, since $(X, (1-\beta)\Delta+\lambda\beta D)$ is not log canonical for $\lambda>\omega$, in particular $(X, (1-\beta)\Delta+\beta D)$ is not log canonical, since $\omega<1$, but $(1-\beta)\Delta+\beta D\simq -K_X$. Hence there is some $T\simq-K_X$ with $mT\in \vert-mK_X\vert$ for some $d\leq \Cat(X)$ and $\Supp(T)\subseteq\Supp(D)$, by the Cat Property, a contradiction.

Therefore $\exists T_i\in \vert-n_iK_X\vert$ with $n_i\leq \Cat(X)$ such that $\Supp(T_i)\subseteq\Supp(D)$ for $1\leq i\leq k$ for some positive number $k=k(D)$ depending on $D$. By assumption
$$\lct(X, (1-\beta)\Delta, \frac{\beta}{n_i}T_i)>\lct(X, (1-\beta)\Delta,\beta D)\geq\omega.$$
Therefore the pair
$$(X, (1-\beta)\Delta+\frac{\lambda\beta}{n_i}T_i)$$
is log canonical for all $1\leq i\leq k(D)$ and some $\lambda>\omega$ and the pair
$$(X, (1-\beta)\Delta+\lambda\beta D)$$
is not log canonical. In particular $\frac{1}{n_i}T_i\neq D$, $\forall i$.

Let $D_0=D$. By Corollary \ref{cor:higher-log-convexity} $\exists a_1\in [0,1)\cap \bbQ$ such that
$$D_1:=\frac{1}{1-a_1}(D_0-\frac{a_1}{n_1}T_1)\simq -K_X$$
is effective, $T_1\not\subset\Supp(D_1)$ and $(X, (1-\beta)\Delta+\lambda\beta D_1)$ is not log canonical for some $\lambda>\omega$.

By Corollary \ref{cor:higher-log-convexity}, for each $2\leq i \leq k(D)=k$, $\exists a_i\in [0,1)\cap \bbQ$ such that 
$$D_i:=\frac{1}{1-a_i}(D_{i-1}-\frac{a_i}{n_i}T_i)\simq-K_X$$
is effective, $T_1,\ldots, T_i\not\subset\Supp(D_i)$ and the pair
$$(X, (1-\beta)\Delta + \lambda\beta D_i)$$
is not log canonical for some $\lambda>\omega$. In particular, the pair
$$(X,(1-\beta)\Delta+\beta D_k)$$
is not log canonical, since we can choose $1\geq\lambda>\omega$. The pair $(X, (1-\beta) \Delta +\beta \Delta)=(X,\Delta)$ is log canonical. Therefore, by Corollary \ref{cor:higher-dim-convexity}, the pair $(X, D_k)$ is not log canonical. By construction, $\Supp(B)\not\subseteq\Supp(D_K)$ for all effective $\bbQ$-divisors $B\simq-K_X$ such that $mB\in \vert-mK_X\vert$ for some $m\leq \Cat(X)$. This contradicts the Cat Property.
\end{proof}
\begin{cor}
\label{cor:cat-property-easy-dynamic-alpha}
Let $X$ be a complex Fano variety with $\Cat(X)<\infty$. Suppose the Cat Property holds on $X$. Let $\Delta\in \vert-K_X\vert$ be a smooth effective divisor. Then
$$\alpha(X,(1-\beta)\Delta)=\min\{1,\min\{ \lct(X,(1-\beta)\Delta,\beta T)\setsep T \text{ is a cat of } X\}\}.$$
\end{cor}

\section{Log pairs on surfaces and intersection theory}
\label{sec:surfaces-intersection-theory}

Most of the surfaces we will deal with will be smooth. However, occasionally, we will deal with surfaces with canonical singularities. We first introduce Du Val singularities.
\begin{dfn}[{\cite[Def. 4.4]{KollarMori}}]
\label{dfn:duVal-sing}
A normal surface singularity $(0\in S)$ with minimal resolution $f\colon S'\ra S$ is called a Du Val singularity if and only if $K_{S'}\cdot E_i=0$ for every exceptional curve $E_i\subset S'$.
\end{dfn}
Du Val singularities have a very explicit classification:
\begin{thm}[{\cite[Thm 4.22]{KollarMori}}]
\label{thm:DuVal-classification}
Every Du Val singularity has embedding dimension $3$. Up to a local analytic change of coordinates, the following is a complete list of Du Val singularities:
\begin{itemize}
	\item[(A)]. The singularity $A_n\ (n\geq 1)$ has equation $x^2+y^2+z^{n+1}=0$ and dual graph with $n$ vertices:
		$$\xymatrix{ {\circ} \ar@{-}[r] & \cdots \ar@{-}[r] & {\circ} }$$
	\item[(B)]. The singularity $D_n\ (n\geq 4)$ has equation $x^2+y^2z+z^{n-1}=0$ and dual graph with $n$ vertices:
		$$\xymatrix{ & &
{\circ} & \\ {\circ} \ar@{-}[r] & \ar@{-}[r] \cdots \ar@{-}[r] &{\circ} \ar@{-}[r] \ar@{-}[u] & {\circ}}$$
	\item[(C)]. The singularity $E_6$ (respectively $E_7$, respectively $E_8$) has equation $x^2+y^3+z^4=0$ (respectively $x^2+y^3+yz^3=0$, respectively $x^2+y^3+z^5=0$) and dual graph with $6$ (respectively $7$, respectively $8$) vertices:
		$$\xymatrix{ & &
{\circ} & &\\ {\circ} \ar@{-}[r] & \ar@{-}[r] \cdots \ar@{-}[r] &{\circ} \ar@{-}[r] \ar@{-}[u] &{\circ} \ar@{-}[r] & {\circ}}$$
\end{itemize}

If $S'$ is a smooth surface and $E\subset S'$ a collection of proper rational $(-2)$-curves whose dual graph is listed above, then $E\subset S'$ is the minimal resolution of a surface $0\in S$ which has the corresponding Du Val singularity at $0$.
\end{thm}

\begin{thm}[{\cite[Thm 4.20]{KollarMori}}]
\label{thm:duVal-equivalence}
Let $(0\in S)$ be the germ of a normal surface singularity. The following are equivalent:
\begin{itemize}
	\item[(i)] $(0\in S)$ is canonical.
	\item[(ii)] $(0\in S)$ is Du Val.
\end{itemize}
\end{thm}

\begin{nota}
\label{nota:surface-basics}
Let $S$ be a surface with canonical singularities. By \cite[4.11, 4.19]{KollarMori} $S$ is $\bbQ$-factorial. Let $f:\widetilde S\ra S$ be a birational morphism and $D$ be a $\bbQ$-divisor in $S$ with proper transform $\widetilde D$. Then we can write the \textbf{log pullback} of $(S,D)$ by $f$ as
$$K_{\widetilde S}+ \widetilde D +\sum^{r}_{i=1}a_iE_i\equiv f^*(K_S+D),$$
where $E_i$ are exceptional curves ($E_i\cong \bbP^1, E_i^2<0)$ and $a_i$ are rational numbers.

Often $f\colon \widetilde S \ra S$ will be the blow-up of a point $p$ with exceptional curve $E$. Other times $f$ will be the minimal log resolution of $(S,D)$. This will be clear from the context, when not explicitly stated. We will denote the strict transform of any $\bbQ$-divisor $B$ in $\widetilde S$ by $\widetilde B$.

\end{nota}

\begin{lem}
\label{lem:log-pullback-preserves-lc}
The log pair $(S,D)$ is log canonical if and only if 
		\begin{equation}
						(\widetilde S, \widetilde D + \displaystyle{\sum^{r}_{i=1}} a_iE_i)
						\label{eq:log-pullback-preserves-lc}
		\end{equation}
is log canonical. In particular when $f\colon \widetilde S \ra S$ is the blow-up of a point $p\in S$ with exceptional divisor $E$, the pair $(S, D)$ is log canonical at $p$ if and only if
$$(\widetilde S, \widetilde D +(\mult_pD-1)E)$$
is log canonical for all $q\in E$.
\end{lem}
\begin{proof}
If $(\widetilde S, \widetilde D + \sum a_i E_i)$ is log canonical, there is a log resolution $g\colon \bar S \ra \widetilde S$ of $(\widetilde S, \widetilde D + \sum a_i E_i)$ with exceptional divisors $F_j$ such that its log pullback is
$$g^*(K_{\widetilde S }+ \widetilde D +\sum a_i E_i)=K_{\bar S} + \bar D +\sum a_i \bar E_i + \sum b_j F_j$$
where $a_i\leq 1$ and $b_j\leq 1$, $\forall i,j$. Since $f\circ g$ is a log resolution for $(S, D)$ with log pullback
$$(f\circ g)^*(K_S+ D)=K_{\bar S} + \bar D + \sum a_i \bar E_i +\sum b_i \bar F_i,$$
the pair $(S, D)$ is log canonical.

If $(S,D)$ is log canonical then there is a log resolution $f'\colon \bar S\ra S$ consisting of blow-ups. Since there is a birational map $\bar S \dra S$, there is a smooth surface $\hat S$ and birational morphisms $g'\colon \hat S \ra \bar S$ and $g\colon \hat S \ra \widetilde S $, such that $f\circ g=f'\circ g'$. In particular, by the previous implication $(\widetilde S, \widetilde D + \sum a_i E_i)$ is log canonical.
\end{proof}

In this section we deal with a pair $(S,D)$ (or $(S,\omega D)$ for some $\omega \in \bbQ\cap [0,1]$) where $S$ is a non-singular surface. Let $p\in S$ and $D\simq -K_S$ be an effective $\bbQ$-divisor such that $(S,D)$ (respectively $(S,\omega D)$) is not log canonical at $p$.

Lemma \ref{lem:higher-dim-convexity-forward} and corollaries \ref{cor:higher-dim-convexity} and \ref{cor:higher-log-convexity} have the following $2$-dimensional versions that we will use often:
\begin{lem}
\label{lem:convexity-forward}
Let $S$ be a surface with canonical singularities, $D$ and $B$ be effective $\bbQ$-divisors on $S$. If $(S,D)$ and $(S,B)$ are log canonical then, for all $\alpha\in [0,1]\cap\bbQ$, the pair
$$(S,\alpha D+(1-\alpha)B)$$
is log canonical.
\end{lem}
\begin{lem}[Convexity]\label{lem:convexity}
Given $S$ non-singular (at $p$), let $D,B$ be effective $\bbQ$-divisors on $S$ such that $(S,B)$ is log canonical (at $p$) and $(S,D)$ is not log canonical (at $p$). Then, for all $\alpha \in [0,1)\bigcap \bbQ$, the pair
$$(S, D' = \frac{1}{1-\alpha}(D-\alpha B))$$
is not log canonical (at $p$).
Moreover if $D\simq B$, then $ D'\simq D$ and we can choose $\alpha$ such that $\exists B_i$ irreducible curve ($p\in B_i$) in the support of $B$ with $B_i \not\subset \Supp( D')$ where $D'$ is effective.
\end{lem}

\begin{lem}[Log-convexity]
\label{lem:log-convexity}
Let $X$ be a $\bbQ$-factorial variety. Suppose $(X,A+D_1)$ is not log canonical and $(X,A+D_2)$ is log canonical where $A,D_1,D_2$ are effective $\mathbb{Q}$-divisors, such that $D_1\simq D_2$. There is $\lambda \in [0,1)\bigcap \mathbb{Q}$ such that 
$$D_3=\frac{1}{1-\lambda} (D_1-\lambda D_2),$$
is effective, $\Supp(D_3)$ does not contain at least one of the components of $D_2$ and $(X,A+D_3)$ is not log canonical.

Let $S$ be a smooth surface. Suppose $(S,A+D_1)$ is not log canonical and $(S,A+D_2)$ is log canonical where $A,D_1,D_2$ are effective $\mathbb{Q}$-divisors, such that $D_1\simq D_2$. There is $\lambda \in [0,1)\bigcap \mathbb{Q}$ such that 
$$D_3=\frac{1}{1-\lambda} (D_1-\lambda D_2),$$
is effective, $\Supp(D_3)$ does notcontain at least one of the components of $D_2$ and $(S,A+D_3)$ is not log canonical.
\end{lem}

\subsection{Classical local inequalities}
The following result is well known and it can be found (when the ground field is $\bbC$) on \cite{CheltsovLCTdP}. Our proof for algebraically closed fields can also be found in \cite{JMGlctCharP}.
\begin{lem}\label{lem:adjunction}
Let $S$ be a non-singular surface, $D$ be an effective $\bbQ$-divisor and $C$ be an irreducible curve on the surface $S$. We may write $D = mC + \Omega$, where $m\geq 0$ is a rational number, and $\Omega=\sum a_i \Omega_i$ is an effective $\bbQ$-divisor such that $C \not\subset \Supp(\Omega)$. Suppose the pair $(S,D)$ is not log canonical at some point $p\in S$ such that $p\in C$. The following are true:
\begin{itemize}
	\item[(i)] $\mult_p D>1$.
	\item[(ii)] If $C \subset \LCS(S,D)$, then $m \geq 1$. In particular, if $D$ is not log canonical along $C$, then $m>1$.
	\item[(iii)] If $m \leq 1$ and $p\in C$ with $C$ non-singular at $p$, then $C\cdot \Omega>1$.
\end{itemize}
\end{lem}
\begin{proof}
Part (ii) is straight forward from the definition of $\LCS(S,D)$. For (i) and (iii) suppose $(S, D)$ is not log canonical. Consider $f\colon \widetilde S \ra S$, the minimal log resolution of $(S,D)$, where the components of $f^{-1}(D)$ have simple normal crossings. By Lemma \ref{lem:log-pullback-preserves-lc}, the pair
$$(\widetilde S, \widetilde D + \sum a_i E_i)$$
is not log canonical. We do induction onthe number $N$ of exceptional divisors of $f$.

For the induction hypothesis we assume that (i) and (iii) hold if the minimal log resolution of a pair consists on at most $N$ blow-ups. Suppose the log resolution of $(S,D)$ consists of $(N+1)$ blow-ups. Let $\sigma S_1 \ra S$ be the blow-up of $p$ with exceptional divisor $E_1$. Since the minimal log resolution is unique, $f$ factors through $S_1$, i.e. there is a birational morphism $g\colon \widetilde S \ra S_1$ consistsing of $N$ blow-ups, such that $f=\sigma \circ g$. By Lemma \ref{lem:log-pullback-preserves-lc}, the pair
\begin{equation}
(S_1, D_1 + (\mult_p D -1)E_1)
\label{eq:adjunction-proof}
\end{equation}
is not log canonical at some $q\in E_1$, where $D_1$, $C_1$ and $\Omega_1$ are the strict transforms of $D$ $C$ and $\Omega$, respectively.

We will prove (i) first, and then (iii). For (i), in the initial step of induction, $D$ is smooth at $p$, so we can assume $D=aD_1$ around $p$. Since $(S,D)$ is not log canonical, $a>1$, so $\mult_p(D)=a>1$. If $D$ is not smooth, then the pair \eqref{eq:adjunction-proof} is not log canonical and its log resolution consists of $N$ blow-ups. Therefore we may apply the induction hypothesis to show
$$1<\mult_qD_1 + (\mult_pD -1)\leq 2\mult_pD-1$$
which implies $\mult_pD>1$.

For (iii) we observe that $\Supp(D)$ is not smooth at $p$. If it was, then $D=mC$ near $p$ and $m>1$ since $(S,D)$ is not log canonical. The initial step for the induction occurs when \eqref{eq:adjunction-proof} is already a log resolution. Then
$$1<\mult_pD-1=\mult_p\Omega+m-1\leq \mult_p\Omega \leq C\cdot \Omega,$$
proving the claim.

Suppose $\mult_pD-1<1$. Then the pair \eqref{eq:adjunction-proof} is not log canonical at some point $q\in E_1$ and log canonical near $q$. The log resolution of the pair \eqref{eq:adjunction-proof} consists of $N$ blow-ups and we can assume (iii) is verified for \eqref{eq:adjunction-proof} by the induction hypothesis, where we substitute $C$ by $C_1$ or $E_1$. If $q\in C_1$, then by the induction hypothesis
$$1<C_1\cdot (\Omega_1+(\mult_pD-1)E_1)=C\cdot \Omega+m-1 \leq C\cdot \Omega$$
since $m\leq 1$. If $q\not\in C_1$, then the pair
$$(S_1, \Omega_1 + (\mult_pD -1)E_1)$$
is not log canonical at $q$ and by the induction hypothesis we have
$$1<E_1\cdot \Omega_1 = \mult_p\Omega\leq C\cdot \Omega.$$
\end{proof}

\subsection{New local inequalities}
In this section we prove a few theorems on the local behaviour of pairs which fail to be log canonical. Let $(S, D=\sum a_i D_i)$ be a log pair where $S$ is a smooth surface and $D$ is an effective $\bbQ$-divisor.

Lemma \ref{lem:adjunction} gives a basic characterisation of when $(S,D)$ is not log canonical. Either $\exists a_i>1$, or at some point $p$ the curves in $\Supp(D)$ are very singular with $1\geq a_i\gg 0$ for enough $D_i$ such that $p\in D_i$. In the first case we say that $(S,D)$ is \emph{not log canonical along} $D_i$ or $(S,D)$ is \emph{not log canonical in codimension} $1$. If all $a_i\leq 1$ but $(S,D)$ is not log canonical, then it is not log canonical in a finite number of points by Lemma \ref{lem:adjunction} (ii). If $p$ is one of those points we say that $p$ is \emph{a locus of non-log canonical singularities} for $(S,D)$ or, if no confusion is likely, we simply say that $p$ is \emph{isolated}.

The general philosophy is that we should treat separately the cases for which $(S,D)$ is not log canonical in codimension $1$ and those for which $(S,D)$ has an isolated locus of non-log canonical singularities. However, in both cases it will be useful to bound the coefficients $a_i$. Furthermore, when $p$ is an isolated locus of non-log canonical singularities of $(S,D)$, we will also be interested in bounding $\mult_p D$.

If $(S,D)$ is not log canonical in codimension $1$, bounding the coefficients of $D$ is relatively simple using intersection theory. All the theorems in this section deal with the isolated case. We will use intersection theory both in the statements and the proofs of these results to bound the multiplicities and coefficients. All these results make frequent use of Lemma \ref{lem:adjunction} (i) and (iii). In some way, they can be seen as applications of that simple fact. While stated in a general setting, these results will come useful in different particular situations, according to the intersection theory of the curves involved.

Theorem \ref{thm:pseudo-inductive-blow-up} is rather elementary. Although it has a technical statement, it reduces to mimick the procedure involved when finding the minimal log resolution of a pair $(S, (1-\beta)C + \lambda\beta D)$ which is log canonical in codimension $1$, where $C$ is a smooth curve and $\lambda, \beta$ are arbitrary coefficients.

The next result is Theorem \ref{thm:inequality-Cheltsov}. This Theorem reflects the nature of how these local inequalities are found. Once we know which statement we want to prove (providing it is true), the proof is easy to obtain by induction. However finding a strong statement is not easy. Theorem \ref{thm:inequality-Cheltsov} is a generalisation of a result of D. Kosta, \cite[Thm 2.21]{Kosta-thesis}, which had a non-symmetric statement. While we provide our own proof, the credit to finding the statement of the Theorem, which is the difficult part, belongs to I. Cheltsov. Theorem \ref{thm:inequality-Cheltsov} is only useful for us other when studying the intersection of two lines, mainly in order to shorten the use of Theorem \ref{thm:pseudo-inductive-blow-up}. However the Theorem is really powerful and I. Cheltsov, in \cite{Cheltsov-Trento-Proceedings}, has used it to reprove many known results on global log canonical thresholds in a more efficient way. Another inequality that implies Kosta's original result can be found in \cite{Cheltsov-Kosta}. Their result and Theorem \ref{thm:inequality-Cheltsov} do not imply one another. A generalisation of both Theorem \ref{thm:inequality-Cheltsov} and \cite{Cheltsov-Kosta}, if it exists, remains yet to be found.

Finally, the last result of this section is Theorem \ref{thm:inequality-local-blowup-bound}. We use this Theorem in the computation of the dynamic alpha--invariant of del Pezzo surfaces of degrees $7$--$9$. The proof is the most involved, which is why we have left it at the end. While proving this result we noticed we could get, as a by-product, a lower estimate on the number of exceptional curves in the minimal log resolution of a given log pair. This was as an unexpected result and it suggests that the number of exceptional curves in the minimal log resolution of a log pair and the log canonical threshold of the pair are closely related.

In the following Theorem we could take $\lambda=1$ but since all the applications use a number $0<\lambda\leq 1$, we include $\lambda$ for the reader's convenience.
\begin{thm}
\label{thm:pseudo-inductive-blow-up}
Let $S=S_0$ be a surface which is non-singular at $q=q_0\in C$, where $C=C^0$ is a smooth curve. Let $0<\beta\leq 1$ be a rational number, $0<\lambda\leq 1$ and $D^0$ be an effective $\bbQ$-divisor such that
\begin{equation}
(S_0, (1-\beta)C^0+\lambda\beta D^0)
\label{eq:pseudo-inductive-blow-up-basicpair}
\end{equation}
is a pair which is not log canonical at $q=q_0$ but is log canonical near $q$.

Let $f_1\colon S_1 \ra S_0$ be the blow-up of $q_0$ with exceptional divisor $F_1$. Let $C^1$ $D^1$ be the strict transforms of $C^0$, $D^0$, respectively. 

Let $i\geq 2$. Let $f_i\colon S_i\ra S_{i-1}$ be the blow-up of $q_{i-1}=C^{i-1}\cap F_{i-1}$ with exceptional divisor $F_i$. Denote by $A^i$ the strict transform in $S^i$ of any $\bbQ$-divisor $A^{i-1}$ of $S_{i-1}$. In particular let $D^i, C^i, F^i_{i-1}, F^i_j$ be the strict transforms of $D^{i-1}, C^{i-1}, F_{i-1}, F^{i-1}_j$, respectively. Let $m_i=\mult_{q_i}D^{i}$ for $i\geq 0$.

\begin{itemize}
	\item[(i)] If $i=1$ then the pair
$$(S_1, (1-\beta)C^1+\lambda\beta D^1 + (\lambda\beta m_0 -\beta) F_1)$$
is not log canonical at some $r_1\in F_1$.

If $i\geq 2$, the pair
\begin{equation}
(S_{i-1}, (1-\beta)C^{i-1}+ \lambda\beta D^{i-1} + (\lambda\beta(\sum_{j=0}^{i-2}m_j)-(i-1)\beta)F_{i-1})
\label{eq:pseudo-inductive-blow-up-pair}
\end{equation}
is not log canonical at $q_{i-1}=C^{i-1}\cap F_{i-1}$ and
\begin{equation}
\lambda\beta(\sum_{j=0}^{i-2}m_j)-(i-1)\beta\leq 1,
\label{eq:pseudo-inductive-blow-up-condition1}
\end{equation}
then the pair
\begin{equation}
(S_i, (1-\beta)C^i+ \lambda\beta D^i + (\lambda\beta(\sum_{j=0}^{i-2}m_j)-(i-1)\beta)F_{i-1}^i+(\lambda\beta(\sum_{j=0}^{i-1}m_j)-i\beta)F_i)
\label{eq:pseudo-inductive-blow-up-logpullback}
\end{equation}
is not log canonical at some $r_i\in F_i$ and is log canonical near $r_i$.

Observe that $F_{i-1}^i\cap C^i=\emptyset$, since $C_{i-1}$ is smooth at $q_{i-1}$ for all $i\geq 1$.
	\item[(ii)] Suppose $i=1$. If $\lambda\beta m_0-\beta\leq 1$, then $r_1=q_1=F_1\cap C^1$.
	
	Suppose $i=2$. If $\lambda\beta m_0\leq 1$ or $\lambda\beta(m_0+m_1)\leq 1$ and in addition to condition (i) we have
	$$\lambda\beta(\sum_{j=0}^{i-1}m_j)-i\beta\leq 1,$$
	then $r_i$ is an isolated locus of non-klt singularities and $r_i\in F_i\cap (F_{i-1}^i\cup C^i)$. 
	\item[(iii)] Let $i\geq 2$. Suppose the pair \eqref{eq:pseudo-inductive-blow-up-pair} is not log canonical at $q_{i-1}=C^{i-1}\cap F_{i-1}$, the inequality
	$$\lambda\beta(\sum_{j=0}^{i-1}m_j)-(i-1)\beta\leq 1$$
	holds, and $\lambda\beta m_0\leq 1$ or  $\lambda\beta(m_0+m_1)-\beta\leq 1$ holds. Then the pair \eqref{eq:pseudo-inductive-blow-up-logpullback} is not log canonical only at $r_i=q_i=F_i\cap C^i$.
	\item[(iv)] Let $i\geq 2$. Suppose that in addition to (i), inequality 
	$$\lambda\beta((\sum_{j=0}^{i-3} + m_j) + 2m_{i-2})-i\beta\leq 1$$
	holds where the sum is $0$ for $i=2$. Suppose further that $\lambda\beta m_0\leq 1$ or  $\lambda\beta(m_0+m_1)-\beta\leq 1$ holds. Then the pair \eqref{eq:pseudo-inductive-blow-up-logpullback} is not log canonical only at $r_i=q_i=F_i\cap C^i$.
\end{itemize}
\end{thm}
\begin{proof}
Since the pair \eqref{eq:pseudo-inductive-blow-up-basicpair} is log canonical near $q_0$, we may assume it is log canonical in codimension $1$. By Lemma \ref{lem:log-pullback-preserves-lc}, the pair
\begin{equation}
(S_1,(1-\beta)C^1+\lambda\beta D^1 + (\lambda\beta m_0 -\beta)F_1)
\label{eq:pseudo-inductive-blow-up-proof-first-logpullback}
\end{equation}
is not log canonical at some $r_1\in F_1$. This proves statement (i) for $i=1$.

Condition (ii) for $i=1$ states that $\lambda\beta m_0-\beta\leq 1$, which is precisely the condition for \eqref{eq:pseudo-inductive-blow-up-proof-first-logpullback} to be log canonical along $F_1$ by Lemma \ref{lem:adjunction} (ii). Moreover, if $r_1\not\in C^1$, then the pair
\begin{equation}
(S_1,\lambda\beta D^1 + (\lambda\beta m_0-\beta)F_1)
\label{eq:pseudo-inductive-blow-up-proof-noC}
\end{equation}
is not log canonical. We apply Lemma \ref{lem:adjunction} (iii) on $F_1$ to get
$$1<\lambda\beta D^1\cdot F_1 =\lambda\beta m_0\leq 1,$$
a contradiction.

Now suppose $i\geq 2$. We prove first part (i). If the pair \eqref{eq:pseudo-inductive-blow-up-basicpair} is not log canonical at $q_{i-1}$ and \eqref{eq:pseudo-inductive-blow-up-condition1} holds then \eqref{eq:pseudo-inductive-blow-up-basicpair} is log canonical near $q_{i-1}$ and by Lemma \ref{lem:log-pullback-preserves-lc}, the pair \eqref{eq:pseudo-inductive-blow-up-logpullback} is not log canonical at some $r_i\in F_i$. Since  $C^0$ is smooth at $q_0$ and we are blowing-up smooth points in $C^k$, $C^k$ is smooth at $q_k$ for all $0\leq k\leq i$. Therefore $C^k\cdot F_k=1$ for all $0\leq k\leq i$. In particular $C^k\cdot F_{k-1}^k=0$ for all $k<i$, so $C^i\cap F^i_{i-1}=\emptyset$. Since $C^i\neq F^i_{i-1}$ and both curves are irreducible at $q_i$, we proved (i).

For (ii), suppose $\lambda\beta(\sum^{i-1}_{j=0}m_j)-i\beta\leq 1$. This implies that $r_i$ is an isolated locus of non-klt singularities for the pair \eqref{eq:pseudo-inductive-blow-up-logpullback} (see Lemma \ref{lem:adjunction} (ii)). Moreover if $r_i\in F_i\setminus (F^i_{i-1}\cup C^i)$, then the pair
$$(S_i, \lambda\beta D^i+((\lambda\beta\sum_{j=0}^{i-1}m_j-i\beta)F_i))$$
is not log canonical at $r_i$. If $\lambda\beta m_0\leq 1$, then
$$1<\lambda \beta D^i \cdot F_i = \lambda \beta m_i \leq \lambda\beta m_0\leq 1$$
by Lemma \ref{lem:adjunction} (iii), giving a contradiction. Therefore $\lambda\beta(m_0+m_1)-\beta\leq 1$ and $\lambda\beta m_0>1$. Then $\lambda\beta m_1-\beta\leq 1-\lambda\beta m_0<0$, so $\lambda\beta m_k\leq \lambda\beta m_1<\beta$ for $1\leq k\leq i-1$. We obtain a contradiction via Lemma \ref{lem:adjunction} (i):
\begin{align*}
1&<\mult_{r_i}(\lambda\beta D^i+((\lambda\beta\sum_{j=0}^{i-1}m_j)-i\beta)F_i)\\
&=\lambda\beta \mult_{r_i}D^i+\lambda\beta(\sum_{j=0}^{i-1}m_j)-i\beta\\
&\leq\lambda\beta(m_{i-1}+\sum_{j=0}^{i-1}m_j)-i\beta\\
&\leq\lambda\beta (m_0+m_1)-\beta \leq 1.
\end{align*}
This finishes the proof of (ii).

Notice that the hypothesis of (iii) implies (i) and (ii):
\begin{align*}
\lambda\beta(\sum^{i-2}_{j=0}m_j)-(i-2)\beta\leq \lambda\beta(\sum^{i-1}_{j=0} m_j)-(i-1)\beta&\leq 1\\
\lambda\beta(\sum^{i-1}_{j=0}m_j)-i\beta\leq \lambda\beta(\sum^{i-1}_{j=0}) m_j-i\beta+\beta&\leq 1.
\end{align*}
Therefore $r_i\in F_i\cap(F^i_{i-1}\cup C^i)$. If $r_i=F^i_{i-1}\cap F_i$, then $r_i\not\in C^i$ and the pair
\begin{equation}
(S_i, \lambda\beta D^i + ((\lambda\beta \sum_{j=0}^{i-2}m_j)-(i-1)\beta)F^i_{i-1}+(\lambda\beta(\sum^{i-1}_{j=0} m_j)-i\beta)F_i)
\label{eq:pseudo-inductive-blow-up-noCpair}
\end{equation}
is not log canonical at $r_i$. Applying Lemma \ref{lem:adjunction} (iii) with $F_i$ we obtain
$$1<(\lambda\beta D^i + ((\lambda\beta\sum_{j=0}^{i-2}m_j)-(i-1)\beta)F^{i}_{i-1})\cdot F_i = \lambda\beta \sum^{i-1}_{j=0}m_j - (i-1)\beta\leq 1,$$
which is absurd.

For (iv), observe that the hypotheses imply the hypotheses of part (ii) of this lemma:
\begin{align*}
\lambda\beta(\sum^{i-1}_{j=0} m_j)-i\beta = \lambda\beta (\sum^{i-3}_{j=0}m_j+m_{i-2} + m_{i-1})-i\beta\leq\lambda\beta ((\sum^{i-3}_{j=0}m_j)+2 m_{i-2})-i\beta\leq 1.
\end{align*}
where we use $m_j\leq m_k$ for all $k\leq j$. Hence $r_i\in F_i\cap (C^i \cup F_{i-1}^i)$. If $r_i = F_i\cap F^{i}_{i-1}$, then the pair \eqref{eq:pseudo-inductive-blow-up-noCpair} is not log canonical at $r_i$. Then, by Lemma \ref{lem:adjunction} (iii) applied with $F^i_{i-1}$ we get a contradiction:
$$1<(\lambda\beta D^i + (\lambda\beta(\sum^{i-1}_{j=0} m_j)-i\beta)F_i)\cdot F^i_{i-1}= \lambda\beta(\sum^{i-3}m_j + 2m_{i-2})-i\beta\leq 1.$$
Therefore $r_i=q_i=C^i\cap F_{i-1}^i$.
\end{proof}

\begin{thm}[See \cite{Cheltsov-Trento-Proceedings}]
\label{thm:inequality-Cheltsov}
Let $S$ be a surface and $p\in S$ be a non-singular point. Let
$$(S, a_1C_1+a_2C_2+\Omega)$$
be a log pair which is not log canonical at $p\in S$ but log canonical near $p$. Suppose that $(C_1\cdot C_2)\vert_p=1$, $C_1, C_2$ are smooth at $p$ and $C_1, C_2\not\subseteq\Supp(\Omega)$. Suppose that $a_1>0$, $a_2>0$ and $0<\mult_p\Omega\leq 1$. Then
$$(\Omega\cdot C_1)\vert_p>2(1-a_2)\qquad \text{ or }\qquad (\Omega\cdot C_2)\vert_p>2(1-a_1).$$
\end{thm}
\begin{proof}
First observe that if $a_1\geq 1$ or $a_2\geq 1$ the statement is trivial. Hence, assume that $a_1,a_2<1$. Hence the pair
$$(S,D=a_1C_1+a_2C_2+\Omega)$$
is log canonical in codimension $1$ but not at $p$. 
Therefore there is a birational morphism $f\colon\hat S \ra S$ with exceptional divisors $E_i$ which is biregular away from $p$, consisting on $N$ blow-ups of points infinitely near $p$ and such that the log pullback
$$f^*(K_S+D)\simq K_{\hat S}+ \hat D +\sum^{N}_{i=1}b_iE_i$$
satisfies $b_N>1$ where $\hat D $ is the strict transform of $D$ in $\hat S$. We will proceed by induction on $N$. By the nature of induction, it is enough to study the situation for the first blow-up $\sigma\colon \widetilde S \ra S$ of $p$ with exceptional divisor $E$. By Lemma \ref{lem:log-pullback-preserves-lc} the pair
$$(\widetilde S, a_1\widetilde C_1 + a_2\widetilde C_2+\widetilde \Omega + (\mult_p\Omega+a_1+a_2-1)E)$$
is not log canonical at some point $q\in E$. If $N=1$ then  $\mult_p\Omega+a_1+a_2>2$. In this case, Lemma \ref{lem:adjunction} (iii) gives
$$1<\widetilde C_1\cdot (\widetilde \Omega+(a_1+a_2+\mult_p\Omega-1)E+a_2\widetilde C_2)\vert_{\widetilde C_1\cap E}$$
and
$$1<\widetilde C_2\cdot (\widetilde \Omega+(a_1+a_2+\mult_p\Omega-1)E+a_1\widetilde C_1)\vert_{\widetilde C_2\cap E},$$
which implies
$$(C_1\cdot \Omega)\vert_p>2-a_1-a_2\qquad (C_2\cdot \Omega)\vert_p>2-a_1-a_2,$$
since $\widetilde C_1\cdot \widetilde C_2=0$. If $a_1\geq a_2$ then the second inequality gives $(C_2\cdot \Omega)>2-a_1-a_1\geq 2(1-a_1)$. Conversely, if $a_1\leq a_2$, then the first inequality gives $(C_1\cdot \Omega)\vert_q>2(1-a_2)$.

Therefore $N>1$ and $\mult_p\Omega+a_1+a_2-1\leq 1$. If $q\not\in \widetilde C_1\cup\widetilde C_2$, then Lemma \ref{lem:adjunction} (iii) implies
$$\mult_p\Omega=E\cdot\widetilde \Omega>1,$$
which is impossible. Hence either $q\in \widetilde C_1$ or $q\in \widetilde C_2$. Without loss of generality, suppose the former. Then the pair
$$(\widetilde S, a_1\widetilde C_1+(\mult_p\Omega+a_1+a_2-1)E_1+\widetilde \Omega)$$
is not log canonical at $q\in \widetilde C_1$. But the function $\widetilde f \colon \hat S \ra \widetilde S$ by the factorisation of $f=\sigma\circ \widetilde f$ where $E_1=E$, consists of $N-1$ blow-ups. Therefore, by the induction hypothesis either
$$(C_1\cdot \Omega)\vert_p-\mult_p\Omega = (\widetilde C_1\cdot\widetilde \Omega)\vert_q>2(1-\mult_p\Omega-a_1-a_2+1)=2(1-a_2)+(2-2\mult_p\Omega-a_1)$$
or
$$(C_2\cdot \Omega)\vert_p\geq \mult_p\Omega\geq E_1\cdot \Omega>2(1-a_1).$$
The Lemma follows, since $a_1+\mult_p\Omega\leq 2-a_2\leq 2$, so the first equation gives $(C_1\cdot\Omega)\vert_p>2(1-a_2)$.
\end{proof}

\begin{thm}
\label{thm:inequality-local-blowup-bound}
Let $S$ be a non-singular surface and $C$ an irreducible curve, non-singular at $p\in C$. Let $\Delta$ be an effective $\bbQ$-divisor such that $C\not \subseteq \Supp(\Delta)$. Suppose the pair
\begin{equation}
(S,(1-\beta)C+\Delta),\qquad 0<\beta\leq 1
\label{eq:inequality-local-blowup-bound-pair}
\end{equation}
is not log canonical at $p\in C$ but it is log canonical in codimension $1$ near $p$. If
$$\mult_p\Delta\leq\min\{1,\frac{1}{n}+\beta\}\text{ for some } n\in \bbN,$$
then
$$(\Delta\cdot C)\vert_p>1+n\beta.$$
Moreover if $N$ is the minimum number of blow-ups required to resolve the pair \eqref{eq:inequality-local-blowup-bound-pair}, then
$$N>\max\{n\in \bbN \setsep \mult_p\Delta\leq \frac{1}{n}+\beta\}.$$
\end{thm}
\begin{proof}
The case $\beta=1$ is trivial by Lemma \ref{lem:adjunction} (i), since then
$$(C\cdot \Delta)\vert_p\geq\mult_p\Delta>1,$$
and the hypothesis is vacuous. Hence we may assume $\beta<1$. First note that the pair \eqref{eq:inequality-local-blowup-bound-pair} cannot non-singular support at $p$, since if this was the case, then we could assume $\Delta=aZ$ where $Z$ is an irreducible curve with $p\in Z$, $Z$ smooth at $p$. But \eqref{eq:inequality-local-blowup-bound-pair} not being log canonical and being non-singular at $p$ implies that $a>1$ and then
$$1<a=\mult_p\Delta\leq 1,$$
a contradiction.

Consider the minimal log resolution of the pair \eqref{eq:inequality-local-blowup-bound-pair} at $p$. It consists of $N\geq 1$ blow-ups $\sigma_i\colon S_i\ra S_{i-1}$ of points $p_{i-1}\in E_{i-1}$ with exceptional curves $E_i$ where $p_0=p$ and $S_0=S$ where $p_i$ is a point in which the log pullback of the pair \eqref{eq:inequality-local-blowup-bound-pair} via the composition $\sigma_1\circ\cdots\circ\sigma_i$ is not log canonical by Lemma \ref{lem:log-pullback-preserves-lc}. Denote by $\Delta_i$ and $C_i$ the strict transforms $S_i$ of the divisors $\Delta$ and $C$ respectively. Let $m_i=\mult_{p_i}\Delta_i$.

We claim that after each blow-up the log pullback of the pair \eqref{eq:inequality-local-blowup-bound-pair} is not log canonical only at $p_i=C_i\cap E_i$. As a consequence the exceptional locus of $\sigma_1\circ\cdots \circ \sigma_i$ is a chain of exceptional divisors. We will prove this assertion and the rest of the lemma by induction on the number of blow-ups. The initial step is when no blow-up are needed to achieve simple normal crossings, which is proven already. Assume we have blown-up $p$ and points infinitely close to $p$ at most $k-1$ times where $k\leq n$.

\textbf{Induction hypothesis:} For $1\leq l\leq k-1$ the log pullback of the pair $(S,(1-\beta)C_0+\Delta_0)$ via $\sigma:=\sigma_1\circ\cdots\circ\sigma_l$  is
\begin{equation}
(S_l,(1-\beta)C_l+\Delta_l+\sum^l_{i=1}((\sum^i_{j=1}m_{j-1})-i\beta)E_i)
\label{eq:inequality-local-blowup-bound-log-pullback}
\end{equation}
which is not log canonical only at $p_l=E_l\cap C_l$ and 
$$\sum^l_{j=1}m_{l-1}-l\beta\leq 1$$
(i.e. the pair \eqref{eq:inequality-local-blowup-bound-log-pullback} is log canonical in codimension $1$).

Notice that since $C_0=C$ is smooth, then $C_l$ is smooth and since $p_i\in C_i$ for all $i$, then
$$C_l\cdot E_i=\delta_{il}=\begin{dcases}
																				1		& i=l,\\
																				0		& i\neq l.
																	\end{dcases}
																	$$
Suppose further that
$$\quad E_i\cdot E_j=\begin{dcases}
		1		&	\vert i-j\vert =1,\\
		-1	&	i=j,\\
		0		&	\text{otherwise,}
	\end{dcases}
$$
for $1\leq i<j\leq k-1$.
Furthermore, given that $p_i\in C_i\cap E_i$ we have that
$$C_l\sim\sigma^*(C)-\sum^l_{i=1} E_i,$$
$$\Delta_l\simq\sigma^*(\Delta)-\sum^l_{i=1}((\sum_{j=1}^i m_{j-1})-i\beta)E_i,$$
for $1\leq l\leq k-1$. This is the end of the induction hypothesis.

Blowing up $p_{k-1}$ we obtain that the pair
\begin{equation}
(S_k,(1-\beta)C_k+\Delta_k +\sum^k_{i=1}((\sum^i_{j=1}m_{j-1})-i\beta)E_i)
\label{eq:inequality-local-blowup-bound-induction-pair}
\end{equation}
is not log canonical at some $p_k\in E_k$ by Lemma \ref{lem:log-pullback-preserves-lc}. We distinguish four cases.

\textbf{Case 1}: The pair \eqref{eq:inequality-local-blowup-bound-induction-pair} is not log canonical in codimension $1$. By the induction hypothesis and Lemma \ref{lem:adjunction} (ii) we have
$$1<(\sum^k_{j=1}m_{j-1})-k\beta\leq km_0-k\beta,$$
but then we obtain a contradiction, since this inequality implies
$$\frac{1}{n}+\beta\geq \mult_{p}\Delta=m_0>\frac{1}{k}+\beta$$
using the lemma's assumption in the statement, so $k>n$ but we assumed $k\leq n$.

\textbf{Case 2}: The point $p_k\in E_k\setminus (C_k\cup E_{k-1})$. This implies that the pair
$$(S,\Delta_k + ((\sum_{j=1}^km_{j-1})-k\beta)E_k)$$
is not log canonical at $p_k$, since $p_k\not \in E_i$ for $i\neq k$. By case 1 we have $\sum_{i=1}^k (m_{j-1})-k\beta\leq 1$. Hence, applying Lemma \ref{lem:adjunction} (iii) with $E_k$, we obtain
$$1\geq \min\{1,\frac{1}{n}+\beta\}\geq \mult_p \Delta_0 =m_0\geq m_{k-1}=\Delta_k\cdot E_k>1$$
which is absurd. Therefore $p_k\in E_k\cap (C_k\cup E_{k-1})$.

\textbf{Case 3}: The point $p_k=E_{k-1}\cap E_k$. Not this case only makes sense for $k\geq 2$. The pair
$$(S_k,\Delta_k + \sum_{i=k-1}^k ((\sum_{j=1}^i m_{j-1})-i\beta)E_i)$$
is not log canonical at $p_k$. Applying Lemma \ref{lem:adjunction} (iii) with $E_{k-1}$ we have that
\begin{align*}
1	&<E_{k-1}\cdot (\Delta_k + ((\sum_{j=1}^k m_{j-1})-k\beta)E_k)\\
	&=E_{k-1}\cdot \Delta_{k-1}-m_{k-1}+\sum_{j=1}^km_{j-1}-k\beta\\
	&=m_{k-2}-m_{k-1}+\sum_{j=1}^k m_{j-1}-k\beta\\
	&\leq(k-1)m_{0}+m_0-k\beta = km_{0}-k\beta,
\end{align*}
where we slightly abused the notation in the first equality when identifying $E_{k-1}\subset S_{k-1}$ with its strict transform in $S_k$. Reordering and dividing by $k$ we obtain a contradiction:
$$\frac{1}{n}+\beta\geq m_0=\mult_p\Delta>\frac{1}{k}+\beta$$
since we assumed that $n\geq k$.

\textbf{Case 4}: The point $p_k=C_k\cap E_k$. Since we reached contradictions in each of the previous cases, this is the only possibility, which together with case 1 not being possible, proves the induction. Moreover, note that since $p_k\not\in E_i$ for $i\neq k$, then the pair
$$(S_k,(1-\beta)C_k+\Delta_k + ((\sum_{j=1}^k m_{j-1})-k\beta) E_k)$$
is not log canonical at $p_k$. By applying Lemma \ref{lem:adjunction} (iii) with $C_k$, we obtain
\begin{align*}
1	&<\left.\left(C_k\cdot \left(\Delta_k+\left(\left(\sum_{j=1}^k m_{j-1}\right)-k\beta\right)E_k\right)\right)\right\vert_{p_k}\\
	&=(C_k\cdot\Delta_k)\vert_{p_k}+\sum_{j=1}^k m_{j-1}-k\beta\\
	&=(C\cdot\Delta)\vert_{p}-\sum_{j=1}^k m_{j-1}+\sum^k_{j-1} m_{j-1}-k\beta.
\end{align*}
Thus $(C\cdot \Delta)\vert_{p}>1+k\beta$. If $k=n$ the first assertion of the lemma is proven. For the second notice that $k<N$, where $N$ is the number of blow-ups in the minimal log resolution by case 1, since all the discrepancies are smaller or equal than $1$. Hence we have not achieved a log resolution yet. If $k<n$ we can repeat the inductive step blowing-up $p_k\in E_k$ until $k=n$.
\end{proof}

\chapter{Log canonical thresholds of del Pezzo surfaces revisited}
\label{chap:del-Pezzo}
This chapter has three sections. In the first one we introduce and give several properties of del Pezzo surfaces, the main class of varieties studied in this Thesis. The second section studies the Cat Property we introduced in the previous chapter. As an application, we prove several cases of Theorem \ref{thm:del-Pezzo-glct-charp}. The last section proves Theorem \ref{thm:del-Pezzo-glct-charp} when $\deg S=4$.
\section{Basic properties of del Pezzo surfaces.}
\label{sec:delPezzo-properties}
In this section we survey several well known results on del Pezzo surfaces. We provide our own proofs for most of them, given that we have not found proofs adequate to our needs in the literature. However, most of the results are either well known or simple exercises. For those results stated without proof we provide an adequate reference. The author learnt most of the material in this section from \cite{Kollar-Rational-Curves-Algebraic-Varieties}, \cite{Beau}, \cite{DemazureDelPezzo}, \cite{ManinCubicForms} and \cite{DolgTopicsClass}.

The first result is standard and will be used freely in the rest of this Thesis.
\begin{lem}[{Genus formula, see \cite[V.1 Ex. 1.3]{HartshorneAG}}]
\label{lem:genus-formula}
Let $S$ be a non-singular surface and $D$ an effective divisor in $S$. Then
$$K_S\cdot D + D^2=2p_a(D)-2$$
where $p_a(D)$ is the \textbf{arithmetic genus} of $D$. When $D$ is a smooth reduced and irreducible curve, then $p_a(D)=g(D)$.

In the particular case of $\bbP^2$, if $C$ is a curve of degree $d$, we have
$$p_a(C)=\frac{(d-1)(d-2)}{2}.$$
\end{lem}

\begin{thm}[{Nakai-Moishezon criterion \cite[V.I.10]{HartshorneAG}}]
\label{thm:Nakai-Moishezon-criterion}
A divisor $D$ on a surface $S$ is ample if and only if $D^2>0$ and $D\cdot C>0$ for all irreducible curves $C$ in $S$.
\end{thm}
This theorem suggests the following definitions:
\begin{dfn}
\label{dfn:del-Pezzo-degree}
A \textbf{del Pezzo surface} $S$ over an algebraically closed field $k$ is a non-singular surface whose anticanonical divisor, $-K_S$ is ample. Given any effective $\bbQ$-divisor $D\neq 0$, its \textbf{anticanonical degree} (or just degree) is the positive rational number defined by
$$\deg (D)=(-K_S)\cdot D.$$
If $S$ has at worst canonical singularities and $D$ is an effective divisor, then $\deg(D)$ is a positive integer. The \textbf{degree} of $S$ is the positive integer
$$\deg(S)=(K_S)^2.$$
\end{dfn}
We will call effective divisors of degrees $2,3,\ldots$ conics, cubics... respectively, unless the curve we are dealing with is in $\bbP^2$. The term \textbf{line} is reserved for irreducible curves with self-intersection $(-1)$. When $S=\bbP^2$, a curve of degree $d$ (or a line, conic, cubic) will be the vanishing locus of a homogeneous polynomial of degree $d$, as usual. The following lemma characterises lines rather precisely. Conics will be characterised in Lemma \ref{lem:del-Pezzo-all-conics-rational} when $\deg S\geq 3$.
\begin{lem}
\label{lem:del-Pezzo-lines-numerical}
Let $S$ be a non-singular del Pezzo surface. Then every irreducible curve with a negative self-intersection number is exceptional, i.e. if $C$ is an irreducible curve with $C^2<0$ then $C^2=-1$ and $C\cong \bbP^1$. Moreover $\deg C=1$. 
\end{lem}
\begin{proof}
The arithmetic genus $p_a(C)$ of $C$, and the geometric genus $g(C)$ satisfy
$$p_a(C)\geq g(C)\geq 0 \quad \text{and}\quad  p_a(C)=g(C)\iff C \text{ is non-singular}.$$
Since $C$ is irreducible, it is enough to show that $p_a(C)=0$ and $C^2=-1$. By the genus formula and using Theorem \ref{thm:Nakai-Moishezon-criterion} with the ample divisor $(-K_S)$ we obtain
$$p_a(C)-2=K_S\cdot + C^2\leq K_S\cdot C -1\leq -2.$$
Thus $p_a(C)=0$. But then
$$-2=p_a(C)-2=K_S\cdot C + C^2\leq -1 +C^2$$
so $C^2=-1$, using again Theorem \ref{thm:Nakai-Moishezon-criterion}. It follows from the genus formula that $\deg C=1$. In particular all $(-1)$-curves are lines.
\end{proof}

\begin{lem}
\label{lem:del-Pezzo-nefcurves}
Let $S$ be a non-singular del Pezzo surface. If $C\subset S$ is an irreducible curve such that $(-K_S)\cdot C>1$, then $C$ is nef. Moreover if $-K_S\cdot C>2$ and $C$ is reduced, then $C$ is nef and big.
\end{lem}
\begin{proof}
By the genus formula, if $C$ is not a line, then $C$ is nef, since
$$C^2=-K_S\cdot C+2p_a(C)-2\geq 2+2p_a(C)-2\geq 0$$
where $g$ is the genus of $C$. From the same equation we see that if its del Pezzo degree is bigger than $2$, then $C^2>0$ and therefore $C$ is big and nef.
\end{proof}

\begin{dfn}
A set of distinct points $\{p_1,\ldots ,p_r\}$ on $\bbP_k^2$ with $r\leq 8$ are in \textbf{general position} if no three of them lie on a line, no six of them lie on a conic and a cubic containing $7$ points, one of them double, does not contain the eighth one.
\end{dfn}
We can classify del Pezzo surfaces:
\begin{thm}[{\cite[Chapter IV, Theorems 24.3, 24.4, 26.2]{ManinCubicForms}}]
\label{thm:del-Pezzo-classification-models}
Let $S$ be a non-singular del Pezzo surface of degree $d$. Then $1\leq d\leq 9$ and either $S=\bbP^1\times \bbP^1$ ($\deg S=8$) or $S$ is a blow-up of $\bbP^2$ in $9-d$ points in general position:
$$\pi\colon S\lra \bbP^2.$$
Conversely, any blow-up of $\bbP^2$ in $9-d$ points in general position, for $1\leq d\leq 9$ is a del Pezzo surface of degree $d$. We call the morphism $\pi$ a \textbf{model} of $S$.
\end{thm}
\begin{nota}
\label{nota:del-Pezzos}
Let $S\neq \bbP^1\times \bbP^1$ be a non-singular del Pezzo surface and $\pi\colon S \ra \bbP^2$ be a model of $S$ contracting $n=9-\deg S$ $(-1)$-curves (lines). We will denote these curves by $E_1,\ldots, E_n$ and $p_i=\pi(E_i)$ their images in $\bbP^2$.
\end{nota}
\begin{cor}
For $S\neq \bbP^1\times\bbP^1$ a non-singular del Pezzo surface of degree $d$, we have
$$\Pic(S)\cong \bbZ^{10-d},$$
with generators $\pioplane{1}, E_1, \ldots, E_{9-d}$.

For $S=\bbP^1\times\bbP^1$, $\Pic (S)\cong \bbZ^2$ with generators the fibres of each ruling $\bbP^1\times\bbP^1\ra \bbP^1$.
\end{cor}
\begin{cor}
\label{cor:del-Pezzo-anticanonical}
The anticanonical divisor of a non-singular del Pezzo surface $S\neq \bbP^1\times\bbP^1$ of degree $d$ is
$$-K_S\sim\pioplane{3}-\sum^{9-d}_{i=1} E_i.$$
For $S=\bbP^1\times\bbP^1$, let $L_1,L_2$ be the class of a fibre of each of the rulings $S\ra\bbP^1$. Then
$$-K_S\sim 2L_1+2L_2.$$
\end{cor}

We also have the following useful characterisation of del Pezzo surfaces of low degree.
\begin{thm}
\label{thm:del-Pezzo-classification-embeddings}
Let $S$ be a del Pezzo surface with at worst canonical singularities. Assume that $K_S^2\leq 4$. Denote by
$$S_{d_1,\ldots,d_l}\subset\bbP(a_1,\ldots,a_n)$$
a complete intersection of hypersurfaces of weighted degree $d_1,\ldots,d_l$.
Then
$$S\cong \mathrm{Proj}\left(\sum_{m\geq 0} H^0\left(S,\calO_S\left(-mK_S\right)\right)\right)$$
can be described as follows:
\begin{itemize}
	\item[(i)] If $K_S^2=1$, then $S\cong S_6\subset \bbP(1,1,2,3)$ and $S$ is a $2\colon 1$ cover of the singular quadric cone $\bbP(1,1,2)\subset\bbP^3$ branched at the vertex of the cone and at a sextic curve not passing through the vertex. If $S$ is smooth, then the branching curve is smooth.
	\item[(ii)] If $K_S^2=2$, then $S\cong S_4\subset \bbP(1,1,1,2)$ and $\vert-K_S\vert$ gives a morphism $\phi\colon S \ra \bbP^2$, which realises $S$ as a $2\colon 1$ cover of $\bbP^2$ branched at a quartic curve. If $S$ is smooth, then the branching curve is smooth.
	\item[(iii)]If $K_S^2=3$, then $S\cong S_3\subset \bbP^3$. 
	\item[(iv)]If $K_S^2=4$, then $S\cong S_{2,2}\subset \bbP^4$, the complete intersection of two quadric hypersurfaces. If $S$ is smooth then these hypersurfaces can be chosen to be smooth.
\end{itemize}
Conversely, any weighted complete intersection as above is a del Pezzo surface of the expected degree $K_S^2$. Moreover $-K_S$ is very ample if $\deg S\geq 3$. For $\deg S=2$, $-2K_S$ is very ample. For $\deg S = 1$, $-3K_S$ is very ample.
\end{thm}
\begin{proof}
See \cite[III, Thm 3.5]{Kollar-Rational-Curves-Algebraic-Varieties} for the smooth case and \cite{DemazureDelPezzo} for the rest. See as well \cite{DolgTopicsClass} for a very detailed account with historical notes going back to \cite{DuVal1}.
\end{proof}
No further characterisation is needed for high degrees:
\begin{lem}
\label{lem:del-Pezzo-no-moduli}
For $5\leq d \leq 7$ and $d=9$ there is a unique non-singular del Pezzo surface of degree $d$ up to isomorphism. There are precisely two del Pezzo surfaces of degree $8$ up to isomorphism. These are $\bbF_1=\bbP_{\bbP^1}(\calO_{\bbP^1}(1)+\calO_{\bbP^1})$ (the blow-up of $\bbP^2$ at one point) and $\bbP^1\times\bbP^1$.
\end{lem}
\begin{proof}
Let $S$, $S'$ be two non-singular del Pezzo surfaces of degree $5\leq d\leq 9$, and suppose that $S,S'\not\cong \bbP^1$. By Theorem \ref{thm:del-Pezzo-classification-models}, we may assume there are models $\pi\colon S \ra \bbP^2$ and $\pi'\colon S'\ra\bbP^2$ which are blow-up of $\bbP^2$ in points $p_1,\ldots,p_l$ and $p_1',\ldots,p_l'$, respectively, with exceptional divisors $E_i =\pi^{-1}(p_i)$ and $E_i'=(\pi')^{-1}(p_i')$ where $l=9-d$, $0\leq l \leq 4$. Since $l\leq 4$, there is $\sigma \in \mathrm{PGL}(3,k)$ an isomorphism such that $\sigma(p_i)=p_i'$ for $1\leq i \leq l$. Let $\sigma'$ be the inverse of $\sigma$. The morphisms $\sigma$ and $\sigma'$ lift to $\widetilde\sigma\colon S \ra S'$, sending $E_i$ to $E_i'$ and $\widetilde \sigma'\colon S' \ra S$, sending $E_i'$ to $E_i$, respectively. Clearly $\pi\circ \widetilde\sigma'\circ\widetilde\sigma=\pi$, which is bijective away from $E_i$ and $\widetilde\sigma'\circ\widetilde\sigma(E_i)=E_i$ for all $i$, so $\widetilde\sigma\circ\widetilde\sigma'$ is an isomorphism. We conclude that $S$ and $S'$ are isomorphic. It is clear than $\bbF_1$ and $\bbP^1\times \bbP^1$ are not isomorphic since the first one has a $(-1)$-curve and the latter does not (see Lemma \ref{lem:del-Pezzo-lines9-d} below).
\end{proof}

Models are useful ways of understanding del Pezzo surfaces. Since models are defined by contractions of lines from a del Pezzo surface $S$ to $\bbP^2$, in order to choose a good model for $S$, it is important to understand the lines in $S$, as the following result illustrates.
\begin{lem}
\label{lem:del-Pezzo-lines9-d}
The surface $\bbP^1\times \bbP^1$ has no $(-1)$-curves. Let $S\neq \bbP^1\times \bbP^1$ be a non-singular del Pezzo surface of degree $1\leq d \leq 9$. Then the number of $(-1)$-curves (lines) of $S$ is given by Table \ref{tab:del-Pezzo-lines}.
\begin{table}[!ht]
\centering
\begin{tabular}{|c||c|c|c|c|c|c|c|c|c|}
\hline
$\deg S$ &$1$ &$2$ &$3$ &$4$ &$5$ &$6$ &$7$ &$8$ &$9$ \\
\hline
$n=9-\deg S$ &$8$ &$7$ &$6$ &$5$ &$4$ &$3$ &$2$ &$1$ &$0$ \\
\hline
Number of Lines &$240$ &$56$ &$27$ &$16$ &$10$ &$6$ &$3$ &$1$ &$0$  \\
\hline
\end{tabular}
\caption{Lines on a non-singular del Pezzo surface $S\neq \bbP^1\times \bbP^1$.}
\label{tab:del-Pezzo-lines}
\end{table}

Given a model of $S$, $\pi\colon S \ra \bbP^2$, and $n=9-\deg S$, with $0\leq n \leq 8$, the lines of $S$ belong to the rational classes in Table \ref{tab:del-Pezzo-lines-classes}.
\begin{table}[!ht]
\centering
\begin{tabular}{|c|c|c|c|c|}
\hline
Line									&Rational class 																												&Indices rank											&Rank of $n$			&Number 					\\
											&																																				&																	&									&of lines					\\
\hline
$E_i$									&$E_i$																			 														&$1\leq i\leq n$									&$1\leq n\leq 8$ 	&$n$							\\
\hline
$L_{ij}$							&$\pioplane{1}-E_i-E_j$																									&$1\leq i<j\leq n$								&$2\leq n\leq 8$ 	&$\binom{n}{2}$		\\
\hline
$C_{i_1\cdots i_5}$		&$\pioplane{2}-E_{i_1}-\cdots-E_{i_5}$																	&$1\leq i_1<\cdots<i_5\leq n$			&$5\leq n\leq 8$ 	&$\binom{n}{5}$		\\
\hline
$Q_{i_1\cdots i_7,j}$	&$\displaystyle{\pioplane{3}-E_{i_j}-\sum_{k=1}^nE_{i_k}}$							&$1\leq i_1<\cdots<i_7\leq n$			&$7\leq n\leq 8$ 	&$7$ if $n=7$			\\
											&																																				&$j=i_k,\ k\in \{1,\ldots,7\}$		&									&$56$ if $n=8$	  \\
\hline
$R_{ijk}$							&$\displaystyle{\pioplane{4}-E_i-E_j-E_k-\sum_{l=1}^8E_l}$							&$1\leq i<j<k\leq n$							&$n=8$					 	&$\binom{8}{3}=56$\\
\hline
$T_{ij}$							&$\displaystyle{\pioplane{5}-\sum_{\substack{l=1\\ l\neq i,j}}^8E_l-\sum_{l=1}^8E_l}$		&$1\leq i<j\leq 8$&$n=8$					 	&$\binom{8}{2}=28$\\
\hline
$Z_i$									&$\displaystyle{\pioplane{6}-E_i-2\sum_{l=1}^8E_l}$											&$1\leq i\leq 8$									&$n=8$					 	&$8$							\\
\hline
\end{tabular}
\caption{Rational classes of lines on a non-singular del Pezzo surface $S\neq \bbP^1\times \bbP^1$ of degree $\deg S=9-n$.}
\label{tab:del-Pezzo-lines-classes}
\end{table}
$E_i$ are the exceptional divisors, which are mapped by $\pi$ to points $p_i$. $L_{ij}$ is the strict transform of the unique line in $\bbP^2$ passing through $p_i$ and $p_j$. $C_{i_1\cdots i_5}$ is the strict transform of the unique conic in $\bbP^2$ passing through $p_{i_1},\ldots,p_{i_5}$. $Q_{i_1\cdots i_7,j}$ is the strict transform of the unique cubic in $\bbP^2$ passing through $p_{i_1},\ldots, p_{i_7}$ with a double point at $p_{i_j}$. $R_{ijk}$ is the strict transform of the unique quartic in $\bbP^2$ passing through $p_1,\ldots, p_8$ with double points at $p_i, p_j, p_k$. $T_{ij}$ is the strict transform of the unique quintic in $\bbP^2$ passing through $p_1,\ldots, p_8$ with double points at all $p_l$ but $p_i, p_j$. $Z_i$ is the strict transform of the unique quartic in $\bbP^2$ passing through $p_1,\ldots, p_8$ with double points at all $p_l$ and a triple point at $p_i$.
\end{lem}
\begin{proof}
If $S=\bbP^1\times\bbP^1$, then $-K_S\sim2F_1+2F_2$ where $F_1,F_2$ are the classes of the fibre for each of the rulings $S\ra \bbP^1$. Since $-K_S$ is ample, for any irreducible curve $C$, we have $-K_S\cdot C>0$, but $-K_S\cdot C=2((F_1+F_2)\cdot C)\geq 2$. Therefore $S$ has no lines.

If $S\neq \bbP^1\times \bbP^1$, let $\pi\colon S \ra \bbP^2$ be a model. It follows from Corollary \ref{cor:del-Pezzo-anticanonical} that all the curves in the statement with classes in Table \ref{tab:del-Pezzo-lines-classes} are lines. We want to show these are all the lines. Suppose there is a line $L\subset S$. If $L\neq E_i$ for some $i$, then $L\sim \pioplane{d}-\sum_{i=1}^n a_i E_i$, where $d\geq 1$ and $a_i\geq 0$. Then
\begin{equation}
1=-K_S\cdot L =3d-\sum a_i,\qquad -1=L^2=d^2-\sum a_i^2.
\label{eq:del-pezzo-lines-classes-proof}
\end{equation}
By Schwarz's inequality, for positive integers $x_i,y_i$,
$$(\sum_{i=1}^n (x_iy_i))^2\leq (\sum_{i=1}^n x_i^2)\cdot (\sum_{i=1}^ny_i^2)$$
holds. In particular, if $x_i=1$ and $y_i=a_i$, since $\sum a_i=3d-1$ and $\sum a_i^2=d^2+1$ from \eqref{eq:del-pezzo-lines-classes-proof}. Then
$$(3d-1)^2\leq n(d^2+1)$$
which implies
$$(9-n)d^2-6d-(n-1)\leq 0.$$
Factorising the quadratic equation of variable $d$, we obtain that for $n\leq 1$, then $d<1$. If $n\leq 4$, then $d<2$. If $n\leq 6$, then $d<3$. If $n=7$, then $d<4$ and for $n=8$, $d<7$. This leaves a finite number of possibilities for $a_i$ to satisfy \eqref{eq:del-pezzo-lines-classes-proof}. Testing them one by one gives the values in Table \ref{tab:del-Pezzo-lines-classes}. In particular, the last column of the table allows us to compute the values in \ref{tab:del-Pezzo-lines}
\end{proof}

In principle when the $\deg S$ is low enough we can find a model that satisfies our needs.
\begin{lem}
\label{lem:del-Pezzo-good-model}
Let $S$ be a non-singular del Pezzo surface of degree $5\leq d\leq 6$ and let $L$ be a line in $S$. We can find a model $\gamma\colon S \ra \bbP^2$ such that $L=E_1$ under this model.
\end{lem}
\begin{proof}
Let $\pi\colon S \ra \bbP^2$ be a model of $S$. By Lemma \ref{lem:del-Pezzo-lines9-d}, under this model, $L=E_i$ or $L=L_{ij}$. In the first case, an obvious relabel of the $E_i$ gives $L=E_1$. It is enough to consider $L=L_{12}$, by the same argument.

We may define a model $\gamma\colon S\ra \bbP^2$ by contracting the disjoint curves $F_1=L_{12}, F_2=L_{13}, F_3=L_{23}$, since these lines are disjoint, which is easy to check from their rational classes, which are described in Table \ref{tab:del-Pezzo-lines-classes}.
\end{proof}
In fact Lemma \ref{lem:del-Pezzo-good-model} can be extended to $\deg S \leq 4$ and the proof is essentially the same, only longer since there are more lines to consider.

\begin{lem}
\label{lem:del-Pezzo-lines-through-a-point}
Let $S$ be a non-singular del Pezzo surface and $p$ a point in $S$. If there are three lines $L_1,L_2,L_3$ such that $\bigcap L_i=\{p\}$, then $K_S^2\leq 3$. The point $p$ is called an \textbf{Eckardt point}.
\end{lem}
\begin{proof}
It follows from Lemma \ref{lem:del-Pezzo-lines9-d}, that when $\deg S\geq 4$, there are no three lines $L_1,L_2,L_3$ in Table \ref{tab:del-Pezzo-lines-classes} such that $L_1\cdot L_2=L_2\cdot L_3=L_1\cdot L_3=1$. This is a necessary condition for the existence of an Eckardt point.
\end{proof}
Eckardt points are named after F. E. Eckardt, who studied them in \cite{Eckardt-points}. We have coined the following definition, which will turn useful when analysing del Pezzo surfaces of high degree.
\begin{dfn}
\label{dfn:del-Pezzo-pseudo-Eckardt-point}
Let $S$ be a non-singular complex del Pezzo surface and $L_1,L_2$ two lines intersecting at a point $p$. We call $p$ a \textbf{pseudo-Eckardt point}.
\end{dfn}
The reason for this name is due to the following result.
\begin{lem}
\label{lem:del-Pezzo-pseudo-Eckardt-point}
Let $S$ be a non-singular del Pezzo surface and of degree $K_S^2\geq 4$ and $p$ a pseudo-Eckardt point of $S$. There are $\deg S-3$ points $\{p_i\}^{\deg S-3}_{i=1}$ in $S$ such that their blow-up $\sigma\colon \widetilde S\ra S$ is a del Pezzo surface of degree $3$, $\sigma$ is an isomorphism near $p$ and $\sigma^{-1}(p)$ is an Eckardt point.
\end{lem}

To prove Lemma \ref{lem:del-Pezzo-pseudo-Eckardt-point} we need an auxiliary result which will become useful later for different reasons. Theorem \ref{thm:del-Pezzo-classification-models} implies that del Pezzo surfaces are rational. The following result applies to del Pezzo surfaces:
\begin{prop}
\label{prop:rational-surfaces-sections-rational-curves}
For $S$ a non-singular rational surface and $C$ an effective divisor in $S$ with arithmetic genus $p_a(C)=0$, we have
\begin{equation}
h^0(S,\calO_S(C))\geq(-K_S)\cdot C.
\label{eq:sectionsC}
\end{equation}
\end{prop}
\begin{proof}
By Serre Duality:
$$h^2(S,\calO_S(C))=h^0(S,\calO_S(K_S-C))=0,$$
since $S$ is rational. By the Riemann-Roch theorem:
$$h^0(S,\calO_S(C))\geq \frac{1}{2}C\cdot(C-K_S)+1=-K_S\cdot C +p_a(C),$$
where we use the genus formula.
\end{proof}
\begin{proof}[Proof of Lemma \ref{lem:del-Pezzo-pseudo-Eckardt-point}]
Since $S$ has lines, then $S\neq\bbP^1\times\bbP^1$. Let $p=L_1\cap L_2$ be the intersection of two lines in $S$. Let $K_S^2=d$. Let $\calC = \vert -K_S -L_1-L_2\vert$. Observe that $\calC$ is not empty, since by Theorem \ref{thm:del-Pezzo-classification-embeddings}, $-K_S$ is very ample, giving an embedding $S\hra \bbP^N$ in  which $L_1,L_2$ are lines in $\bbP^2$. Therefore we may chooose a hyperplane section $H$ containing $L_1,L_2$ and $H=C+L_1+L_2$ for $C\in \calC$ an effective divisor. For $C\sim -K_S-L_1-L_2$, we have
$$C^2=K_S^2-2\deg L_1-2\deg L_2+L_1^2+L_2^2+2L_1\cdot L_2=d-4,$$
$$K_S\cdot C =-d+2,$$
Hence by the genus formula
$$2p_a(C)-2=K_S\cdot C + C^2=-d+2+d-4=-2,$$
so $p_a(C)=0$ and we can apply Proposition \ref{prop:rational-surfaces-sections-rational-curves} to obtain 
$$h^0(S,\calO_S(C))\geq d-2\geq 2.$$
Let $C\in \calC$ be a non-singular irreducible element. Take $\Gamma=\{p_i\}_{i=1}^{d-3}$ general points in $C\subset S$. The generality condition means that all $p_i$ lie in no line and that given a model $\pi\colon S \ra \bbP^2$ the points $$\pi(\Gamma)\cup (\bigcup_{i=1}^d\pi(E_i))$$ are in general position. Let $\sigma\colon \widetilde S \ra S$ be the blow-up of $\Gamma$. By our choice of $\Gamma$, the surface $\widetilde S$ is a non-singular cubic del Pezzo surface. Let $F_i$ be the exceptional divisors of $\Gamma$. The strict transform $\widetilde C$ of $C$ satisfies
$$\widetilde C \sim \sigma^*(C)-\sum F_i,\qquad (\widetilde C)^2 = C^2-(d-3)=-1.$$
Notice that $\sigma$ is an isomorphism around $p$ and $\sigma^{-1}(p)=\widetilde L_1\cap\widetilde L_2\cap\widetilde C$, where $\widetilde L_i$ is the strict transform of $L_i$. Therefore $\sigma^{-1}(p)$ is an Eckardt point.
\end{proof}

We proved Proposition \ref{prop:rational-surfaces-sections-rational-curves} in an attempt to generalise the following:
\begin{lem}[{see \cite[III, Cor. 3.2.5]{Kollar-Rational-Curves-Algebraic-Varieties}}]
\label{lem:del-Pezzo-sections-anticanonical}
Let $S$ be a non-singular del Pezzo surface. Then
$$h^0(S,\calO_S(-mK_S))=\frac{m(m+1)}{2}(K_S^2)+1$$
for $m\geq 0$.
\end{lem}

\begin{lem}
\label{lem:del-pezzo-blowdown-dP}
Let $S$ be a non-singular del Pezzo surface and $E\cong \bbP^1$ a $(-1)$-curve (line). Let $\gamma\colon S \ra \bar S$ be the contraction of $E$. Then $\bar S $ is a non-singular del Pezzo surface.
\end{lem}
\begin{proof}
Since $-K_S$ is ample, for all irreducible curves $C$, $-K_S\cdot C>0$. Let $\bar C$ be an irreducible curve in $\bar S$ with multiplicity $m\geq 0$ at $p=\gamma(E)$ and strict transform $C$ in $S$. Then
$$-K_{\bar S}\cdot \bar C =\gamma^*(-K_{\bar S})\cdot \gamma^*(C)=\gamma^*(-K_{\bar S})\cdot C=(-K_S+E)\cdot C=-K_S\cdot C +m\geq -K_S\cdot C>0.$$
Therefore $-K_{\bar S}$ is ample by Nakai-Moishezon criterion  (Theorem \ref{thm:Nakai-Moishezon-criterion}).
\end{proof}

\begin{lem}
\label{lem:del-Pezzo-uniquemodel}
Let $S\neq \bbP^1\times\bbP^1$ be a non-singular del Pezzo surface of degree $7\leq d \leq 9$. Then there is a unique model $\pi:S\ra\bbP^2$ up to isomorphism in $\bbP^2$.
\end{lem}
\begin{proof}
If $S=\bbP^2$ the statement is trivial. If $S=\bbF_1$, the blow-up of $\bbP^2$ at one point, $S$ has only one line by Lemma \ref{lem:del-Pezzo-lines9-d} and the statement is trivial too. If $\deg S=7$, by Lemma \ref{lem:del-Pezzo-lines9-d}, there are $3$ lines, namely $E_1,E_2,L_{12}$ with intersection matrix given by $E_1\cdot L_{12}=E_2\cdot L_{12}=1$ and $E_1\cdot E_2=0$. If we contract $L_{12}$ to obtain $\gamma \colon S\ra \bar S$, the images $\bar{E_1}, \bar{E_2}$ of $E_1$ and $E_2$ satisfy $\bar E_i^2=0$. Therefore $S\cong \bbP^1\times \bbP^1$. The only possibility to obtain a model is to contract $E_1$ and $E_2$ and $\pi$ is unique up to an isomorphism of $\bbP^2$ interchanging $p_1$ and $p_2$.
\end{proof}

We can classify all conics for degrees bigger than $2$ with the following result.
\begin{lem}
\label{lem:del-Pezzo-all-conics-rational}
Let $S$ be a non-singular del Pezzo surface of degree $\deg S\geq 3$. Then any integral curve $C$ with $(-K_S\cdot C)\leq 2$ has arithmetic genus $p_a(C)=0$.
\end{lem}
\begin{proof}
By Theorem \ref{thm:del-Pezzo-classification-embeddings} $-K_S$ is very ample. Therefore $\vert -K_S\vert$ gives an embedding $\phi \colon S \hra \bbP^n$ for some $n\geq 3$. The curves $\phi(C)$ have degree $(-K_S\cdot C)$ in $\bbP^n$, understood as the intersection with a general hyperplane section. Therefore $\phi(C)$ can be projected isomorphically onto a plane preserving the degree $(-K_S\cdot C)=d\leq 2$. The genus formula for $\bbP^2$ finishes the proof: $p_a(C)=\frac{(1-d)(2-d)}{2}=0$.
\end{proof}

\begin{lem}
\label{lem:del-Pezzo-deg2-all-lines-rational}
Let $S$ be a non-singular del Pezzo surface of degree $2$. Then any curve $C$ with $(-K_S\cdot C)=1$ has arithmetic genus $p_a(C)=0$.
\end{lem}
\begin{proof}
By Theorem \ref{thm:del-Pezzo-classification-embeddings} $-2K_S$ is very ample. Therefore $\vert -2K_S\vert$ gives an embedding $\phi \colon S \hra \bbP^n$ for some $n$. The curves $\phi(C)$ have degree $(-2K_S\cdot C)=2$ in $\bbP^n$, understood as the intersection with a general hyperplane section. Therefore $\phi(C)$ can be projected isomorphically onto a plane preserving the degree $d=(-2K_S\cdot C)=2$. The genus formula for $\bbP^2$ finishes the proof: $p_a(C)=\frac{(1-d)(2-d)}{2}=0$.
\end{proof}

For del Pezzo surfaces of degree $2$, we can characterise the lines as bitangents to the smooth quartic curve in $\bbP^2$ given by Theorem \ref{thm:del-Pezzo-classification-embeddings}. More precisely:
\begin{lem}[{see \cite[Section 8.7.1]{DolgTopicsClass}}]
\label{lem:del-Pezzo-deg2-lines}
Let $S$ be a del Pezzo surface of degree $2$ with at worst Du Val singularities and let $\phi\colon S \ra \bbP^2$ be the morphism given by $\vert-K_S\vert$ as in Theorem \ref{thm:del-Pezzo-classification-embeddings} which realises $S$ as a $2:1$ cover of $\bbP^2$ branched at a quartic curve $Q$ of $\bbP^2$. Then the lines of $\bbP^2$ are mapped $2:1$ to bitangent lines to $Q$.
\end{lem}

\section{The Cat Property for del Pezzo surfaces of low degree}
\label{sec:delpezzo-cats}
In this section we study the Cat Property on non-singular del Pezzo surfaces. We will show that the Cat Property holds on a non-singular del Pezzo surface $S$ if and only if $1\leq \deg S\leq 3$. We will classify the cats of $S$ precisely. The proofs are valid over algebraically closed fields in all characteristics. However in characteristic $2$ the number of cats increases, although the classification is essentially the same. For instance, when the ground field $k$ has $\charac(k)=2$, a line and a conic curve in the cubic surface $S_3\subset\bbP^3$ always intersect in a tacnode. This is the reason the classification of cats in lemmas \ref{lem:del-Pezzo-cat-list-degree-2} and \ref{lem:del-Pezzo-cat-list-degree-3} is given in terms of local intersections.

\begin{lem}
\label{lem:delPezzo-lc-cod1-deg1-to-3}
Let $S$ be a non-singular del Pezzo surface with $d=K_S^2\leq 3$. Then for all effective $\bbQ$-divisors $D\simq-K_S$ the pair $(S,D)$ is log canonical in codimension $1$.
\end{lem}
\begin{proof}
Suppose $\Supp(\LCS(S,D))$ contains an irreducible curve $C$. Then by Lemma \ref{lem:adjunction} (ii) we may write $D=mC+\Omega$ with $m>1$ and $C\not\subseteq \Supp(\Omega)$. We bound $\deg (C)=-K_S\cdot C$:
\begin{equation}
K_S^2=-K_S\cdot D=m(-K_S\cdot C)+(K_S\cdot\Omega)>\deg C\geq 1,
\label{eq:delPezzo-lc-cod1-deg1-to-3-bound}
\end{equation}
since $m>1$ and $-K_S$ is ample. In particular this proves the Lemma when $K_S^2=1$. We will prove cases $K_S^2=2,3$ using the fact that $\deg C$ is a positive natural number.

Suppose $K_S^2=2$. Then \eqref{eq:delPezzo-lc-cod1-deg1-to-3-bound} implies $\deg C=1$. Lemma \ref{lem:del-Pezzo-deg2-all-lines-rational} implies $p_a(C)=0$. By the genus formula $C^2=-1$, so $C$ is a line. It is straight forward to see that $(-K_S-C)^2=-1$ and $(-K_S)\cdot (-K_S-C)=1$. Therefore $p_a(-K_S-C)=0$. By Proposition \ref{prop:rational-surfaces-sections-rational-curves}, $h^0(S, \calO_S(-K_S-C))\geq 1$. Hence we can choose an effective curve $L\sim -K_S-C$. In fact, repeating the argument above with $L$ instead of $C$ it follows that $L$ is a line. Then, since $L+C\sim -K_S$ we obtain that $L\cdot C=2$.

If $L$ and $C$ intersect in two points, then $(S, L+C)$ is log canonical. By Lemma \ref{lem:convexity} we may assume $L\not\subset \Supp(D)$, which gives a contradiction: $1=D\cdot L \geq 2m>2$. Hence, $L$ and $C$ must intersect at one point $p$. Write $D=mC + m'L + \Omega$ with $C, L\not\subseteq \Supp(\Omega)$ and $m'\geq 0$. We may bound $m'$:
$$2=D\cdot (-K_S)\geq m+m'>1+m',$$
so $m'<1$. But this gives a contradiction:
$$1=D\cdot L \geq 2m-m'>2-m'>1,$$
proving the Lemma when $K_S^2=2$.

Suppose $K_S^2=3$. From \eqref{eq:delPezzo-lc-cod1-deg1-to-3-bound} we have that $2\geq \deg C\geq 1$. By Lemma \ref{lem:del-Pezzo-all-conics-rational}, $p_a(C)=0$. Therefore, the genus formula tells us that if $\deg C =1$ then $C^2=-1$, and if $\deg C =2$, then $C^2=0$. We treat each case separately.

If $\deg C=1$, then $(-K_S-C)\cdot (-K_S)=2$ and $(-K_S-C)^2=3-2\deg C +C^2=0$. By the genus formula $p_a(-K_S-C)=0$ and Proposition \ref{prop:rational-surfaces-sections-rational-curves} gives that $h^0(S, \calO_S(-K_S-C))\geq 2$. Therefore we may choose an effective divisor $Q\sim-K_S-C$ with $\deg Q=2$. Observe that $Q\cdot C=2$. Since $\vert-K_S-C\vert$ is at least a pencil, we may choose $Q$ to be irreducible, since otherwise $Q=L_1+L_2$, the union of two lines, with one of them moving free. Therefore $S$ would have an infinite number of lines, contradicting Lemma \ref{lem:del-Pezzo-lines9-d}. By Lemma \ref{lem:del-Pezzo-nefcurves}, $Q$ is nef, and 
$$2=Q\cdot D =m(Q\cdot C)+Q\cdot \Omega\geq 2m>2,$$
a contradiction.

Therefore $\deg C=2$. Then $(-K_S-C)\cdot (-K_S)=1$ and $(-K_S-C)^2=3-2\deg C +C^2=-1$. By the genus formula $p_a(-K_S-C)=0$ and Proposition \ref{prop:rational-surfaces-sections-rational-curves} gives that $h^0(S, \calO_S(-K_S-C))\geq 1$. Therefore we may choose an effective divisor $L\sim-K_S-C$ with $\deg L=1$. Observe that $L\cdot C=2$. Since the degree of an effective divisor is a positive natural number, $L$ is irreducible. We may write $D=mC+m'L+\Omega$ with $m'\geq 0$ and $C, L, \not\subseteq \Supp(\Omega)$. Then
$$3=D\cdot (-K_S)\geq 2m+m'>2+m',$$
giving $m'<1$. This implies a contradiction:
$$1=D\cdot L\geq 2m-m'>2-m'>1,$$
a contradiction, finishing the proof.
\end{proof}

\begin{lem}
\label{lem:del-Pezzo-cat-degree-1}
Let $S$ be a non-singular del Pezzo surface of degree $1$. Then $S$ satisfies the Cat Property.
\end{lem}
\begin{proof}
Let $D$ be a tiger of $S$, i.e. $(S,D)$ is not log canonical. It is enough to show that $C\subset \Supp(D)$ for a cat $C\in \vert-K_S\vert$. We have that $\mult_p(D)>1$ by Lemma \ref{lem:adjunction} (i). Consider a curve $C\in \vert-K_S\vert$ through $p$. If $C$ is not a cat with centre $p$, then by Lemma \ref{lem:convexity} we may assume $C\not\subset\Supp(D)$, but then
$$1=C\cdot D\geq \mult_pC\cdot \mult_pD>1.$$
If $C$ is a cat and $C\not\subset\Supp(D)$, we also obtain a contradiction.
\end{proof}
\begin{lem}
\label{lem:del-Pezzo-cat-list-degree-1}
Let $S$ a del Pezzo surface of degree $1$. The cats of $S$ are cuspidal rational curves $C\in\vert-K_S\vert$.
\end{lem}
\begin{proof}
Let $C\in \vert-K_S\vert$ be a cat. Since $-K_S\cdot C =1$, the curve $C$ is irreducible. Since $(S, C)$ is not log canonical, then $C$ must have a singularity at a point $p$, which cannot be nodal (normal crossings). Let $\pi:S\ra \bbP^2$ be a model of $S$ with exceptional divisors $E_1,\ldots,E_8$. Observe that $C\cdot E_i=-K_S\cdot E_i=1$ $\forall i$. Hence $\bar C = \pi(C)$ is smooth at all $p_i=\pi(E_i)$. In particular $p\not\in E_i$ and $\pi$ is an isomorphism near $p$. The curve $\bar C$ is an irreducible curve of degree $3$ with at least one singular point $\pi(p)$ which is not nodal. By the classification of cubic curves in $\bbP^2$, the point $\pi(p)$ must be a cuspidal point. In fact, $\pi(p)$ is the only point at which $\bar C$ is singular.
\end{proof}
\begin{rmk}
\label{rmk:delPezzo-deg1-nocats}
The general non-singular surface $S$ of degree $1$ has $\glct(S)=1$. Therefore, in that case, $S$ has no cats.
\end{rmk}

\begin{lem}
\label{lem:del-Pezzo-cat-degree-2}
Let $S$ be a non-singular del Pezzo surface of degree $2$. Then $S$ satisfies the Cat Property.
\end{lem}
\begin{proof}
It is enough to show that if $D$ is a tiger of $S$ and $(S,D)$ is not log canonical, then $T\subset \Supp(D)$ for some curve $T\in \vert-K_S\vert$ such that $(S,T)$ is not log canonical. Suppose this is not the case. By Lemma \ref{lem:delPezzo-lc-cod1-deg1-to-3} $(S,D)$ is not log canonical at an isolated $p\in \LCS(S,D)$. Consider the linear system $\calT\subset \vert-K_S\vert$ given by all curves singular at $p$. If $\exists T\in \calT$, then we may assume that $T\not\subset\Supp(D)$ either because $T$ is a cat or by Lemma \ref{lem:convexity} (when $(S,T)$ is log canonical). If $T$ is irreducible, then $T\not\subset \Supp(D)$ and we obtain a contradiction: $2=T\cdot D\geq 2\mult_pD>2$. Hence $T=L_1+L_2$ is the union of two lines intersecting at $p$, and such that $L_i^2=-1$ and $L_1\cdot L_2=2$. If $(L_1\cdot L_2)\vert_p=1$ then $(S,L_1+L_2)$ is log canonical and we may assume by Lemma \ref{lem:convexity} that $L_1\not\subset \Supp(D)$. If $(L_1\cdot L_2)\vert_{p}=2$, then $L_1+L_2$ is a cat, so without loss of generality $L_1\not\subset \Supp(D)$. We obtain a contradiction by Lemma \ref{lem:adjunction} (i):
$$1=D\cdot L_1\geq \mult_pD>1.$$

Hence all $T\in \vert -K_S\vert$ with $p\in T$ are smooth at $p$. Let $\pi:\widetilde S \ra S$ be the blow-up of $p$ with exceptional divisor $E$. By Lemma \ref{lem:log-pullback-preserves-lc} the log pullback
$$(\widetilde S, \widetilde D + (\mult_pD-1)E)$$
is not log canonical at some $q\in E$. We may take a general $C\in \vert-K_S\vert$ passing through $p$ such that $C\not\subseteq\Supp(D)$ by Lemma \ref{lem:del-Pezzo-sections-anticanonical}. Then
$$\mult_p(D)\leq D\cdot C=K_S^2=2.$$

The linear system $\calL=\vert\pi^*(-K_S)-E\vert$ is a pencil by Lemma \ref{lem:del-Pezzo-sections-anticanonical}. Take $\widetilde C\in \calL$ such that $q\in \widetilde C$ and let $C=\pi_*(\widetilde C)$. Then $p\in C\in\vert-K_S\vert$ and its strict transform $\widetilde C$ contains $q$. Then
\begin{equation}
\widetilde C \cdot \widetilde D = C\cdot D -\mult_pD = 2-\mult_p D.
\label{eq:del-Pezzo-cat-degree-2-aux1}
\end{equation}
If $C$ is irreducible, then $C$ is a curve of genus $1$ smooth at $p$ and by Lemma \ref{lem:convexity} we may assume that $C\not\subset \Supp(D)$. Then
$$\widetilde C \cdot \widetilde D \geq \mult_q \widetilde D >2-\mult_pD $$
by Lemma \ref{lem:adjunction} (i). Together with \eqref{eq:del-Pezzo-cat-degree-2-aux1}, this gives a contradiction. Hence $C$ is reducible and smooth at $p$.

Write $C=L_1+L_2$. We have that $L_1,L_2\neq E$, since otherwise $\bar C = \pi_*(C)\in \vert-K_S\vert$ would be singular at $p$. Hence we may assume $p\in L_1$, $q\in \widetilde L_1$ and $p\not\in L_2$. The line $L_1\subset\Supp(D)$, since otherwise $1=D\cdot L_1 \geq \mult_pD>1$, by Lemma \ref{lem:adjunction} (i). We write $D=mL_1 + \Omega$, where $m>0$ and $L_1\not\subset\Supp(D)$. Since $(S, L_1+L_2)$ is log canonical, by Lemma \ref{lem:convexity}, we may assume that $L_2\not\subset\Supp(\Omega)$. We observe that $m\leq \frac{1}{2}$. Indeed:
$$1=D\cdot L_2\geq 2m + L_2\cdot \Omega\geq 2m.$$
Recall that $(\widetilde S, \widetilde D + (\mult_pD -1)E)$ is not log canonical at $q=\widetilde L_1\cap E$. By Lemma \ref{lem:adjunction} (iii) applied to this pair with $\widetilde L_1$, we obtain a contradiction:
\begin{align*}
1&<\widetilde L_1\cdot (\widetilde \Omega + (\mult_pD-1)E)\\
 &=L_1\cdot \Omega - \mult_p\Omega+m+\mult_p\Omega-1\\
 &=L_1\cdot (D-mL_1)-1+m\\
 &=2m\leq 1.
\end{align*}
Therefore $\Supp(D)$ must contain a cat.
\end{proof}
\begin{lem}
\label{lem:del-Pezzo-cat-list-degree-2}
Let $S$ be a non-singular del Pezzo surface of degree $2$. The cats of $S$ are cuspidal rational curves in $\vert -K_S\vert$ or two lines $L_1+L_2\sim -K_S$ intersecting at a single point.
\end{lem}
\begin{proof}
Since $\Cat(S)=1$ by \cite{JMGlctCharP}, the cats of $S$ are elements $D\in \vert -K_S\vert$ such that $(S, D)$ is not log canonical. In particular, $D$ is singular and $\deg D=2$.

If $D\in \vert-K_S\vert$ is irreducible, we may choose a model $\pi\colon S \ra \bbP^2$ which contracts the lines $E_1,\ldots, E_7$. Since $D\cdot E_i=1$ $\forall i$, the image $\bar D = \pi_*(D)\sim-K_{\bbP^2}$ is an irreducible cubic curve in $\bbP^2$, smooth at $p_1,\ldots, p_7$, where $p_i=\pi(E_i)$. Since $1=D\cdot E_i\geq\mult_{E_i\cap D} D$, the morphism $\pi$ is an isomorphism around $\Sing(D)$. The pair $(\bbP^2, \bar D)$ is not log canonical and $\bar D$ is an irreducible cubic curve. Therefore $\bar D$ has precisely one singularity, a cuspidal point.

If $D\in\vert-K_S\vert$ is reducible, then $D=L_1+L_2$, the union of two lines. Observe that
$$L_1\cdot L_2=-K_S\cdot L_1-L_1^2=2.$$
The lines $L_1$ and $L_2$ must intersect at precisely one point, since otherwise $L_1+L_2$ would have simple normal crossings and $(S, L_1+L_2)$ would be log canonical.
\end{proof}

\begin{lem}
\label{lem:del-Pezzo-cat-list-degree-3}
Let $S\subset \bbP^3$ be a non-singular del Pezzo surface of degree $3$ embedded by $\vert-K_S\vert$. Consider elements in $\vert-K_S\vert$ corresponding to hyperplane sections. Given any point $p\in S$, let $T\in\vert-K_S\vert$ be the unique section which is singular at $p$, corresponding to the tangent hyperplane section. We can classify each point $p\in S$ according to the singularities of $T$:
\begin{itemize}
	\item[(IA)] The curve $T$ is irreducible, rational, and has a nodal singularity at $p$.
	\item[(IB)] The curve $T$ is irreducible, rational, and has a cuspidal singularity at $p$.
	\item[(IIA)] The curve $T$ is reducible, $T=L+C$, the union of a line $L$ and a conic $C$, intersecting with simple normal crossings, one of them at $p$.
	\item[(IIB)] The curve $T$ is reducible, $T=L+C$, the union of a line $L$ and a conic $C$, intersecting only at $p$ with a tacnodal singularity.
	\item[(IIIA)] The curve $T$ is reducible, $T=L_1+L_2+L_3$, the union of three lines intersecting with simple normal crossings, where $L_1\cap L_2=\{p\}$ and $p\not \in L_3$.
	\item[(IIIB)] The curve $T$ is reducible, $T=L_1+L_2+L_3$, the union of three lines intersecting only at $p$. We say that $p$ is an \textbf{Eckardt point}.
\end{itemize}
In the above cases, the pair $(S,T)$ is log canonical for (IA), (IIA) and (IIIA) but not for the rest.
\end{lem}
Note that a general cubic surface does not have Eckardt points.
\begin{proof}
Since $T$ corresponds to a tangent hyperplane section of $S\subset \bbP^3$, a non-singular surface of degree $3$, the section $T$ corresponds to a reduced cubic curve in a hyperplane $H$ of $\bbP^3$, i.e. $H\cong \bbP^2$. Therefore $T$ is a reduced singular plane cubic curve. If $T$ is irreducible, it must have precisely one singularity, which must be nodal or cuspidal, corresponding to cases (IA) and (IB). Note that nodal singularities are log canonical, but if $T$ has a cuspidal singularity, then $\lct(\bbP^2, T)=\frac{5}{6}$ by Example \ref{exa:P2cuspidal}.

If $T$ is reducible, since $T$ is a plane cubic curve, it can only split as a conic and a line or $3$ lines. The Lemma follows.
\end{proof}
Numerically subcases A and B in each case of Lemma \ref{lem:del-Pezzo-cat-list-degree-3} are the same, i.e. the intersection matrix is the same. The only difference is in the singularity type at $p$. For this reason sometimes we will omit the letter when it makes essentially no difference.

We will need an auxiliary Lemma on del Pezzo surfaces of degree $2$ with very mild singularities before we can prove that the Cat Property is satisfied on smooth cubic surfaces.
\begin{lem}
\label{lem:del-Pezzo-cat-degree-3-aux-deg2}
Let $S$ be a non-singular del Pezzo surface of degree $2$ with at most two $A_1$ points and let $\phi\colon S \ra \bbP^2$ be the generically finite $2\colon 1$ morphism branched at a quartic curve $Q$ of $\bbP^2$ given by $\vert -K_S\vert$ as stated in Theorem \ref{thm:del-Pezzo-classification-embeddings}. 

Let $D\simq-K_S$ be an effective $\bbQ$-divisor which is log canonical in codimension $1$. The pair $(S,D)$ is log canonical $\forall p\in S$ such that $\phi(p)\not\in Q$.
\end{lem}
The above Lemma gives some strong indication that del Pezzo surfaces with one or two $A_1$ points should satisfy the Cat Property. Since we did not need such a strong result we decided to restrict to the above statement.
\begin{proof}[Proof of Lemma \ref{lem:del-Pezzo-cat-degree-3-aux-deg2}]
Suppose for contradiction that $(S,D)$ is not log canonical at some $p\in \Supp(D)$ such that $\phi(p)\not\in Q$. In particular $S$ is smooth at $p$. Indeed, by \cite{DuVal1}, \cite{DuVal2}, \cite{DuVal3} or the more recent \cite[Section 8.7.1]{DolgTopicsClass}, the singularities of $S$ are mapped to singularities of the curve $Q$. Moreover $S$ is smooth if and only if $Q$ is smooth, $S$ has one (two) $A_1$ points if and only if $Q$ has one (two) nodal singularities.

Curves $L_\lambda\in\vert-K_S\vert$ containing $p$ are precisely the preimages by $\phi$ of lines $\bar L_\lambda$ containing $\phi(p)$. It follows that $\{L_\lambda\}$ is a pencil. Since $\phi(p)\not\in Q$, then all $L_\lambda$ are smooth at $p$. Take $L_\lambda$ general, then $L_\lambda \not\subset \Supp(D)$. Therefore
\begin{equation}
\mult_pD \leq D\cdot L_\lambda=2.
\label{eq:del-Pezzo-cat-degree-3-aux-deg2-bound1}
\end{equation}
Let $\sigma\colon \widetilde S \ra S$ be the blow-up of $p$ with exceptional curve $E\cong\bbP^1$, $E^2=-1$ since $S$ is smooth at $p$. Let $\widetilde Z$ be the strict transform in $\widetilde S$ of any $\bbQ$-divisor $Z$ in $S$. By Lemma \ref{lem:log-pullback-preserves-lc}, the pair
$$(\widetilde S, \widetilde D + (\mult_pD -1)E)$$
is not log canonical at some $q\in E$ but it is log canonical near $q$ by \eqref{eq:del-Pezzo-cat-degree-3-aux-deg2-bound1}. By Lemma \ref{lem:adjunction} (i) applied to this pair, we obtain
\begin{equation}
\mult_pD + \mult_q\widetilde D >2.
\label{eq:del-Pezzo-cat-degree-3-aux-deg2-bound2}
\end{equation}
Let $L\in \{L_\lambda\}$ be the unique element such that $q\in \widetilde L$. By Lemma \ref{lem:convexity}, we may assume that $L\not\subseteq \Supp(D)$, since $(S,L)$ is smooth at $p$. If $L$ is irreducible, then
$$2-\mult_p D = \widetilde L \cdot \widetilde D \geq \mult_q \widetilde D,$$
contradicting \eqref{eq:del-Pezzo-cat-degree-3-aux-deg2-bound2}.

Therefore $L=L_1+L_2$ the union of two lines with $p\in L_1$ and $q\in \widetilde L_1$ but $p\not\in L_2$ since $L$ is smooth at $p$. Moreover $L_1\cdot L_2=(2-\frac{x}{2})$ where $0\leq x \leq 2$ is the number of $A_1$ points in $S$ that lie in $L_1\cup L_2$. By Lemma \ref{lem:convexity}, we may assume that either $L_1\not\subset\Supp(D)$ or $L_2\not\subset\Supp(D)$. However $L_1\subset\Supp(D)$ since otherwise
$$1=L_1\cdot D \geq \mult_pD>1,$$
by Lemma \ref{lem:adjunction} (i). Therefore we may write $D=aL_1+\Omega$ with $L_1, L_2\not\subset\Supp(\Omega)$ and $a>0$. On one hand
\begin{equation}
1=L_2\cdot D =(2-\frac{x}{2})a + L_2\cdot \Omega\geq (2-\frac{x}{2})a.
\label{eq:del-Pezzo-cat-degree-3-aux-deg2-bound3}
\end{equation}
On the other hand, since $(\widetilde S, a \widetilde L_1 + \widetilde \Omega + (\mult_p D-1)E)$ is not log canonical at $q\not\in E$, then
\begin{align*}
 1&<\widetilde L_1\cdot (\widetilde \Omega +(\mult_p D-1)E)\\
	&=L_1\cdot \Omega-\mult_p\Omega + a + \mult_p\Omega-1\\
	&=L_1\cdot(D-aL_1)+a-1\\
	&=(1-L_1^2)a\\
	&=(2-\frac{1}{2}x)a,
\end{align*}
which contradicts \eqref{eq:del-Pezzo-cat-degree-3-aux-deg2-bound3}.
\end{proof}

\begin{thm}
\label{thm:del-Pezzo-cat-degree-3}
Let $S$ a non-singular del Pezzo surface of degree $3$. Then $S$ satisfies the Cat Property. In particular, the cats of $S$ are the tangent hyperplane sections from Lemma \ref{lem:del-Pezzo-cat-list-degree-3} in cases (IB), (IIB) and (IIIB).
\end{thm}
\begin{proof}
Suppose $(S,D)$ is not log canonical, where $D$ is a tiger. By Lemma \ref{lem:delPezzo-lc-cod1-deg1-to-3} the locus of log canonical singularities $\LCS(S,D)$ consists of isolated points. Let $p\in S$ be one of those points. Let $T\in \vert-K_S\vert$ be the tangent hyperplane section at $p$, classified depending on $p$ as in Lemma \ref{lem:del-Pezzo-cat-list-degree-3}. Note that $\mult_p(T)=2$ for all cases apart from (IIIB) where $\mult_pT=3$. We need to prove that in cases (IB), (IIB) and (IIIB) we have $T\subseteq\Supp(D)$, i.e. exactly when $(S,T)$ is not log canonical, and therefore, \textit{post factum}, $T$ is a cat. Furthermore, we need to prove that in cases (IA), (IIA) and (IIIA), the pair $(S,D)$ is in fact log canonical, i.e. we must achieve a contradiction. Note that, by Lemma \ref{lem:del-Pezzo-cat-list-degree-3} we may assume that $T\neq D$.

First of all notice that if $L$ is a line with $p\in L$, then $L\subset \Supp(D)$, since otherwise
\begin{equation}
1=L\cdot D \geq \mult_pD>1,
\label{eq:del-Pezzo-cat-degree3-aux1}
\end{equation}
by Lemma \ref{lem:adjunction} (i), which is absurd. In particular when $p$ is of type (IIIB) this finishes the proof, since all components of $T$ are in $\Supp(D)$.

Let $\sigma:\widetilde S \ra S$ be the blow-up at $p$ with exceptional divisor $E$. Denote by $\widetilde Q=(\sigma^{-1})_*(Q)$, the strict transform of any $\bbQ$-divisor $Q$ in $S$. The strict transform $\widetilde T$ of $T$ contains exactly one irreducible smooth rational component $\widetilde F$ such that $(\widetilde F)^2=-1$. Indeed, in case (I) we have $F=T$, since $T$ is irreducible. In case (II) we have $F=C$. Finally, in case (IIIA) we have $F=L_3$. Note that in all three cases $(\widetilde F)^2=-1$ and its genus is $g(\widetilde F)=0$. Therefore $F$ is contractible. In general terms $\widetilde S$ is a weak del Pezzo surface, i.e. for all irreducible curves $Q\subset \widetilde S$ we have $-K_{\widetilde S}\cdot Q\geq 0$. In fact $\widetilde S$ is a del Pezzo surface of degree $2$ in case (I). In case (II) $-K_{\widetilde S}\cdot Q=0$ if and only $Q=\widetilde L$, and in case (III) we have $-K_{\widetilde S}\cdot Q=0$ if and only if $Q=\widetilde{L_1},\widetilde{L_2}$. Note that in (II) and (III) $(\widetilde Q)^2=-2$, i.e. the only curves in which $-K_{\widetilde S}$ fails to be ample are the strict transforms of lines passing through $p$.

Let $\bar \sigma : \widetilde S \ra \bar S$ be the contraction of $F$. The surface $\bar{S}$ is smooth since $F$ is a rational $(-1)$-curve. For any $\bbQ$-divisor $\widetilde B$ in $\widetilde S$, denote by $\bar{B}:=\bar{\sigma}_*(\widetilde B)$. Let $\bar{E}:=\bar{\sigma}_*(E)$ and $\bar{p}:=\bar{\sigma}(F)$. The class $-K_{\bar S}$ is ample, since the only $(-2)$-curves $Q$ described above, intersect $F$ with $Q\cdot F=1$ and therefore $\bar{\sigma}_*(Q)$ are rational $(-1)$-curves. Also $K_{\bar S}^2=3$, so $\bar{S}$ is a non-singular del Pezzo surface of degree $3$. In fact, $S$ is abstractly isomorphic to $\bar{S}$ and the composition
$$\bar{\sigma}\circ \sigma^{-1}:S \ldra \bar{S}\cong S$$
is known as the \textit{Geiser involution}, since $\sigma^2=\text{Id}_S$. The Geiser involution is described in detail using this language in \cite[Section 2.6]{Corti-Pukhlikov-Reid}. Roughly speaking we have substituted the curve $F$ with $\bar{E}$. In particular, $T':=\bar{\sigma}_*((\sigma^{-1})_*(T))+ \bar{E}\sim-K_{\bar S}$ has $\bar{p}=\bar{\sigma}(F)$ the same singularity type as $p$, in the sense of Lemma \ref{lem:del-Pezzo-cat-list-degree-3}.

We will consider the log pullback of $D$ under $\sigma$ and $\bar{\sigma}$. By Lemma \ref{lem:log-pullback-preserves-lc}, the pairs
$$(\widetilde S, \widetilde D + (\mult_pD-1)E)\quad \text{and} \quad (\bar{S}, \bar{D}+(\mult_pD-1)\bar{E})$$
are not log canonical at some point $q\in E$, and $\bar{q}=\bar{\sigma}(q)\in \bar{E}=\bar{\sigma}_*(E)$, respectively. Note that
$$\widetilde D +(\mult_pD-1)E \simq -K_{\widetilde S} \quad \text{ and }\quad D':=\bar{D} + (\mult_pD-1)\bar{E}\simq-K_{\bar{S}}.$$
Since $(S,D)$ is not log canonical at $p$, then $\mult_pD-1>0$, by Lemma \ref{lem:adjunction} (i). Hence $D'$ is effective and has non-zero coefficient for $\bar{E}$. On the other hand, by Lemma \ref{lem:delPezzo-lc-cod1-deg1-to-3} we have that $1\geq \mult_pD-1$. Therefore
$$\LCS(\widetilde S, \widetilde D + (\mult_pD-1)E)\quad \text{ and } \quad \LCS(\bar S, \bar D + (\mult_pD-1)\bar E)$$
consists of isolated points, $q$ and $\bar q$, respectively.

We will show that $q\in \widetilde T \cap E$. Recall that $\widetilde D + (\mult_p D-1)E\simq-K_{\widetilde S}$. Suppose $q\not\in \widetilde T\cap E$. We will show that
$$(\widetilde S, \widetilde D +(\mult_pD-1)E)$$
is log canonical at $q\not\in E$. Let $\gamma \colon \widetilde S \ra \hat S$ be the contraction of all $(-2)$-curves in $\widetilde S$. We distinguish the following cases:
\begin{itemize}
	\item[(I)] Since $p\not\in L$ for any line $L\subset S$, the surface $\widetilde S$ is a del Pezzo surface. Therefore $-K_{\widetilde S}$ is ample and $\gamma$ is the identity. 
	\item[(II)] Recall $p\in L\subset T$. Therefore $\widetilde L^2=-2$, $-K_{\widetilde S}\cdot \widetilde L=0$ and $\gamma$ contracts $\widetilde L$ to a point $r_1$, which is a singularity of type $A_1$ is $\hat S$.
	\item[(III)] Recall $p=L_1\cap L_2$, $T=L_1+L_2+L_3$. Therefore $\widetilde L_1^2=L_2^2=-2$, $-K_{\widetilde S}\cdot \widetilde L_1=-K_{\widetilde S}\cdot \widetilde L_1=0$ and $\gamma$ contracts $\widetilde L_1, \widetilde L_2$ to points $r_1, r_2$, respectively which are du Val singularities of type $A_1$ in $\hat S$.
\end{itemize}
The surface $\hat S$ is a del Pezzo surface of degree $2$ with at most $2$ points of type $A_1$. By Theorem \ref{thm:del-Pezzo-classification-embeddings}, there is a generically finite $2:1$ morphism $\phi\colon \hat S \ra \bbP^2$ branched at a quartic curve $Q\subset \bbP^2$. The morphism $\phi$ is given by $\vert-K_{\hat S}\vert$. For any $\bbQ$-divisor $Z$ of $\widetilde S$, denote by $\hat Z=\gamma_*(Z)$. Since $q\not\in \widetilde T$, the morphism $\gamma$ is an isomorphism near $q$. Therefore the pair
$$(\hat S, \hat D + (\mult_p D -1)\hat E)$$
is log canonical at $\hat q =\gamma (q)$ if and only if $(\hat S, \hat D + (\mult_p D -1)\hat E)$ is log canonical at $q$. Observe that
$$\hat D + (\mult_p D -1)\hat E= \sigma_*(\widetilde D +(\mult_pD-1) E)\simq \sigma_*(-K_{\widetilde S})=-K_{\hat S}.$$
We distinguish the three cases:
\begin{itemize}
	\item[(I)]  The divisor
	$$\hat T + \hat E=\gamma_*(\widetilde T + E)\sim \gamma_*(-K_{\widetilde S})\sim-K_{\hat S}.$$
	Therefore $\phi(\hat T)=\phi(\hat E)=H$, a line in $\bbP^2$. Since $\hat T\cdot (-K_{\hat S})=\hat E\cdot (-K_{\hat S})=1$, by Lemma \ref{lem:del-Pezzo-deg2-lines}, $H$ is bitangent to $Q$ at points $\phi(r)$ where $r\in \hat E\cap \hat T$. Since $q\not\in \widetilde T$, then $\hat q\not\in \hat T$ and by Lemma \ref{lem:del-Pezzo-cat-degree-3-aux-deg2} $(\hat S, \hat D + (\mult_p D -1)\hat E)$ is log canonical at $\hat q$.
	\item[(II)] The divisor
	$$\hat C + \hat E=\gamma_*(\widetilde C+\widetilde L + E)\sim \gamma_*(-K_{\widetilde S})\sim-K_{\hat S}.$$
	Therefore $\phi(\hat C)=\phi(\hat E)=H$, a line in $\bbP^2$. Since $\hat C\cdot (-K_{\hat S})=\hat E\cdot (-K_{\hat S})=1$, by Lemma \ref{lem:del-Pezzo-deg2-lines}, $H$ is bitangent to $Q$ at points $\phi(r)$ where $r\in \hat E\cap \hat C$. Since $q\not\in \widetilde C\cup \widetilde L$, then $\hat q\not\in \hat C$ and by Lemma \ref{lem:del-Pezzo-cat-degree-3-aux-deg2} $(\hat S, \hat D + (\mult_p D -1)\hat E)$ is log canonical at $\hat q$.
	\item[(III)]  The divisor
	$$\hat L_3 + \hat E=\gamma_*(\widetilde L_1+\widetilde L_2+\widetilde L_3+ E)\sim \gamma_*(-K_{\widetilde S})\sim-K_{\hat S}.$$
	Therefore $\phi(\hat L_3)=\phi(\hat E)=H$, a line in $\bbP^2$. Since $\hat L_3\cdot (-K_{\hat S})=\hat E\cdot (-K_{\hat S})=1$, by Lemma \ref{lem:del-Pezzo-deg2-lines}, $H$ is bitangent to $Q$ at points $\phi(r)$ where $r\in \hat E\cap \hat L_3$. Since $q\not\in \widetilde L_1\cup \widetilde L_2\cup \widetilde L_3$, then $\hat q\not\in \hat L_3$ and by Lemma \ref{lem:del-Pezzo-cat-degree-3-aux-deg2} $(\hat S, \hat D + (\mult_p D -1)\hat E)$ is log canonical at $\hat q$.
\end{itemize}
Since $\gamma$ is an isomorphism near $q$, the pair $(\widetilde S, \widetilde D + (\mult_pD-1)E)$ is an isomorphism away from $\widetilde T \cap E$. Therefore we may assume $q\in \widetilde T \cap E$.

If $p$ is of type (I), then $T\subseteq\Supp(D)$ since otherwise we obtain a contradiction, using $\mult_p D>1$ and Lemma \ref{lem:adjunction} (i) applied to the pair $(\widetilde S, \widetilde D + (\mult_p D -1)E)$:
$$2-\mult_pD<\mult_q\widetilde D \leq \widetilde T \cdot \widetilde D = T\cdot D -2\mult_pD=3-2\mult_pD\leq 2-\mult_pD.$$
Similarly, if $p$ is of type (II), then $C\subseteq \Supp(D)$, since otherwise
$$2-\mult_pD<\mult_q\widetilde D \leq \widetilde C \cdot \widetilde D = C\cdot D -\mult_pD=2-\mult_pD,$$
again by means of Lemma \ref{lem:adjunction} (i). In particular these two absurdities imply that $p$ is not of types (IA), (IB), (IIA) or (IIB), since in these cases we just proved $T\subseteq\Supp(D)$. Hence, let us assume that $p$ is of type (IIIA). We write
$$D=c_1L_1+c_2L_2+c_3L_3+\Omega \text{ where }L_1,L_2,L_3\not\subseteq \Supp(\Omega),\ c_1,c_2>0 \text{ and } c_3\geq 0,$$
where $c_1, c_2>0$ because of \eqref{eq:del-Pezzo-cat-degree3-aux1}. Write $\Omega=\sum a_i\Omega_i$ where $a_i>0$ for all $i$. Let $m:=\mult_p\Omega$.

We may assume that $b_1\geq b_2$ and $c_1\geq c_2\geq c_3$. Indeed, if this was not the case we may subtract one component of $T$ from $D$ by convexity (Lemma \ref{lem:convexity}).

The pushforward via $\bar{\sigma}$ of the log pullback via $\sigma$ of the above divisors in $\bar{S}$ corresponds to:
$$(\bar{S}, D':=c_1\bar L_1 +c_2\bar L_2 + (c_1+c_2+m-1)\bar{E}+\bar{\Omega}),$$
where $\bar{E},\bar{L_i}\not\subseteq \Supp(\bar\Omega)$ and $\deg(\bar{E})=\ \deg(\bar{L_i})=1$. This pair is not log canonical. Write $\bar\Omega=\sum a_i\bar\Omega_i$. Note that the coefficients $a_i$ for $\bar\Omega_i=\bar\sigma_*\circ\sigma^{-1}_*(\Omega_i)$ in $\bar\Omega$ are the same as for $\Omega_i$ in $\Omega$. Furthermore, note that we are in in the same situation as before we applied the Geiser Transform, i.e. we have a tiger not log canonical at a point $p$ of type (IIIA). Therefore, we may substitute $S$ for $\bar S$, $p$ for $\bar{q}$, $D$ for $D'$, $\Omega$ for $\bar{\Omega}$ $L_1$ for $\bar{L_1}$, $L_2$ for $\bar{L_2}$ and $L_3$ for $\bar{E}$. We could apply again the Geiser involution, repeating the process \textit{ad infinitum} in an inductive fashion. The only assumption we had made is that $(S,D)$ is not log canonical at $p$. We obtain as a result that $(\bar{S}, D')$ is not log canonical at $\bar{q}$. We claim the following:
\begin{clm}
\label{clm:del-Pezzo-cat-list-degree-3-aux2}
For $\Omega$ a $\bbQ$-divisor in $S$ and $\bar\Omega$ a $\bbQ$-divisor in $\bar S$ as above
$$\deg \Omega>\deg \bar{\Omega}$$
holds.
\end{clm}
Observe that, as we already mentioned, $\deg \Omega=\sum a_i\deg \Omega_i$ and $\deg \bar{\Omega}=\sum a_i\deg \bar{\Omega}_i$ share the same coefficients. On the other hand, by construction, no curve in $\Supp(\Omega)$ is contracted when applying the Geiser involution. Also, since $-K_S$ and $-K_{\bar{S}}$ are ample, $\deg \Omega_i, \deg \bar{\Omega}_i\in \bbN_{>0}$. If we consider the set $\bigcup \{\sum a_i \deg\Omega_i\}$ indexed after applying the Geiser transform a countable number of times, we notice that it satisfies the Descending Chain Condition. But this contradicts the claim, finishing the proof.
\end{proof}
The use of Claim \eqref{clm:del-Pezzo-cat-list-degree-3-aux2} is a key step to the proof of the Theorem and it was suggested to the author by I. Cheltsov.
\begin{proof}[Proof of Claim \ref{clm:del-Pezzo-cat-list-degree-3-aux2}]
We will prove $\deg \Omega-\deg \bar{\Omega}>0$. Note that $\deg D=\deg D'=K_S^2=3$ and let $m:=\mult_p\Omega$. 
We have $\mult_pD =c_1+c_2+m$ and then
\begin{equation}
\deg\Omega-\deg \bar{\Omega}=-(c_1+c_2+c_3)+(c_1+c_2+(c_1+c_2+m-1))=c_1+c_2+m-1-c_3
\label{eq:del-Pezzo-cat-list-degree-3-aux2-IIIA}
\end{equation}
However, since $(S,L_1+L_2+L_3)$ is log canonical, then $L_1+L_2+L_3\neq D$ and by Lemma \ref{lem:convexity}, applied to $(S,D)$ and $L_1+L_2+L_3$ we obtain that the pair
$$\left(S,\frac{1}{1-c_3}\left(D-c_3\left(L_1+L_2+L_3\right)\right)\right)=\left(S,\frac{1}{1-c_3}\left(\left(c_1-c_3\right)L_1+(c_2-c_3)L_2+\Omega\right)\right)$$
is not log canonical where we use the fact that $c_1\geq c_2\geq c_3$. Hence using Lemma \ref{lem:adjunction} (i) we have
$$c_1-c_3+c_2-c_3+m=(c_1-c_3)\mult_pL_1+(c_2-c_3)\mult_pL_2+\mult_p\Omega>1-c_3,$$
so $c_1+c_2+m-1-c_3>0$ and from \eqref{eq:del-Pezzo-cat-list-degree-3-aux2-IIIA} the claim is proven.
\end{proof}

Using Observation \ref{obs:cat-property-easy-glct} and the previous results in this section, we can conclude:
\begin{cor}
\label{cor:del-Pezzo-glct-low-degrees}
Let $S$ be a non-singular del Pezzo surface with $K_S^2\leq 3$. Then
$$\omega=\glct(S)=
			\begin{dcases}
					1								&\text{ when } K_S^2=1 \text{ and } \vert -K_S\vert \text{ has no cuspidal curves}\\
					\frac{5}{6}			&\text{ when } K_S^2=1 \text{ and } \vert -K_S\vert \text{ has a cuspidal curve}\\
					\frac{5}{6} 		&\text{ when } K_S^2=2 \text{ and } \vert -K_S\vert \text{ has no tacnodal curves}\\
					\frac{3}{4}			&\text{ when } K_S^2=2 \text{ and } \vert -K_S\vert \text{ has a tacnodal curve}\\
					\frac{3}{4}			&\text{ when } K_S^2=3 \text{ and } \forall C \in \vert -K_S\vert,\ C \text{ has no Eckardt points}\\
					\frac{2}{3}			&\text{ when } K_S^2=3 \text{ and } \exists C \in \vert -K_S\vert \text{ with an Eckardt point.}
			\end{dcases}
$$
\end{cor}
This result can be found in \cite{CheltsovLCTdP} when the ground field is $k=\bbC$ and for algebraically closed fields in \cite{JMGlctCharP}. The proof of Lemma \ref{lem:del-Pezzo-cat-degree-2} is a generalisation of the computation of $\glct(S)$ when $S$ is a del Pezzo surface of degree $2$ in \cite{JMGlctCharP}.

Not all del Pezzo surfaces satisfy the Cat Property. We have seen counter-examples for $\bbP^2$ (Example \ref{exa:plane-not-cat-property}) and $\bbF_1$ (Example \ref{exa:F1-not-cat-property}). However, a better counter-example which gives evidence for Conjecture \ref{conj:cat} is available for degrees 4--9:
\begin{exa}
\label{exa:del-Pezzo-no-cat-codimension1}
Let $S$ be a non-singular del Pezzo surface of degree $4\leq\deg S \leq 9$. We will construct an effective $\bbQ$-divisor $D\simq-K_S$ such that $mD\in \vert-mK_S\vert$ only for $m>\Cat(S)=1$ and $(S,D)$ is not log canonical. In particular, the Cat Property does not hold for $S$.

If $S=\bbP^1\times\bbP^1$, let $F_1$, $F_2$, be two general fibres of each of the projections $S\ra  \bbP^1$. Observe that for any $C\sim -K_S-F_1-F_2$, we have $C^2=2$ and $\deg C=4$, so by the genus formula $p_a(C)=0$. Applying Proposition \ref{prop:rational-surfaces-sections-rational-curves} we have that $h^0(S, \calO_S(-K_S-F_1-F_2))\geq 4$, so we may assume $C$ to be effective, $F_1,F_2\not\subset C$. Let
$$D:=\frac{3}{2}C + \frac{1}{2}(F_1+F_2)\simq-K_S.$$

Suppose $S\neq \bbP^1\times\bbP^1$. Let $\pi\colon S \ra \bbP^2$ be a model of $S$ with exceptional divisors $E_1,\ldots,E_n$ where $n=9-\deg S$, $0\leq n \leq 5$. Let $p_i=\pi(E_i)$. Let $\bar C\sim\pioplane{2}$ be a smooth conic in $\bbP^2$ passing through the points $p_1,\ldots,p_n$. This curve exists since $h^0(\bbP^2, \pioplane{2})=\binom{4}{2}=6>5$. Let $C$ be the strict transform of $\bar C $ in $S$. Then $C\sim\pioplane{2}-\sum_{i=1}^nE_i$. Let
$$D:=\frac{3}{2}C + \frac{1}{2}\sum_{i=1}^n E_i\simq-K_S.$$

In both cases, $(S,D)$ is clearly not log canonical, since $\lct(S,D)\leq\frac{2}{3}$. Moreover $D\not\in \vert-K_S\vert$ but $2D\in \vert-2K_S\vert$.

Finally, observe that $(S,D)$ is not log canonical in codimension $1$. This condition coincides with the absence of the Cat Property and together with the previous results, it is the evidence behind Conjecture \ref{conj:cat}.

Observe that this situation is not possible when $\deg S =3$, since there is no conic passing through $6$ points in general position.
\end{exa}

However the Cat Property is satisfied locally when $\deg S=4$:
\begin{lem}
\label{lem:del-Pezzo-cat-degree-4-generic}
Let $S$ be a non-singular del Pezzo surface of degree $4$ over an algebraically closed field $k$. The surface $S$ satisfies the Cat Property for all $p\in S$ such that $p$ lies in no line.
\end{lem}

\begin{lem}
\label{lem:del-Pezzo-cat-degree-4-generic-list}
Let $S$ be a non-singular del Pezzo surface of degree $3$ and $p\in S$, such that it does not lie in a line. The cats of $S$ at $p$ are cuspidal rational curves in $\vert -K_S\vert$ and curves of the form
$$A+B$$
two conics such that $A+B\sim-K_S$ and $(A\cdot B)\vert_p=2$.
\end{lem}
We provide a joint proof of lemmas \ref{lem:del-Pezzo-cat-degree-4-generic} and \ref{lem:del-Pezzo-cat-degree-4-generic-list}.
\begin{proof}[Proof of lemmas \ref{lem:del-Pezzo-cat-degree-4-generic} and \ref{lem:del-Pezzo-cat-degree-4-generic-list}]
Suppose $(S, D)$ is not log canonical at $p\in S$, where $p$ is a point which is not contained in any line. Let $\sigma\colon \widetilde S \ra S$ be the blow-up of $p$ with exceptional divisor $E$. The surface $\widetilde S$ is a non-singular del Pezzo surface of degree $3$ and the curve $E$ is a line in $\widetilde S$. By Theorem \ref{thm:del-Pezzo-cat-degree-3}, the surface $\widetilde S$ satisfies the Cat Property. Denote by $\widetilde B$ the strict transform in $\widetilde S$ of any $\bbQ$-divisor $B$ in $S$. By Lemma \ref{lem:log-pullback-preserves-lc}, the pair
\begin{equation}
(\widetilde S, \widetilde D +(\mult_p D-1)E)
\label{eq:del-Pezzo-cat-degree-4-generic-proof}
\end{equation}
is not log canonical at some $q\in E$. By Lemma \ref{lem:delPezzo-lc-cod1-deg1-to-3}, the pair \eqref{eq:del-Pezzo-cat-degree-4-generic-proof} is log canonical in codimension $1$. Since
$$K_{\widetilde S}+ \widetilde D +(\mult_p D -1)E \simq \sigma^*(D+K_S)\simq 0$$
we deduce that $\widetilde D +(\mult_p D-1)E\simq -K_{\widetilde S}$. Therefore, by the Cat Property on $\widetilde S$, $\exists T\in \vert-K_{\widetilde S}\vert$ such that $(\widetilde S,T)$ is not log canonical and $T\subseteq\Supp(\widetilde D + (\mult_p D -1)E)$. By Lemma \ref{lem:del-Pezzo-cat-list-degree-3}, since $q\in E$, either:
\begin{itemize}
	\item[(i)] The curve $T=E+\widetilde C$ where $\widetilde C$ is a conic such that $\widetilde C\cap E=\{q\}$.
	\item[(ii)] The point $q$ is an Eckardt point of $\widetilde S$ and $T=\widetilde L_1+\widetilde L_2+E$ where $\widetilde L_1$ and $\widetilde L_2$ are lines such that $\widetilde L_1\cap \widetilde L_2 \cap E=q$.
\end{itemize}
In case (i), let
$$C=\sigma_*(\widetilde C)\sim\sigma_*(\widetilde C + E)\sim \sigma_*(-K_{\widetilde S})\sim -K_S.$$
Since $\mult_pC = \widetilde C \cdot E=2$ and $\widetilde C \cap E=q$, then $C$ is a cuspidal rational curve.

If $T=\widetilde L_1+\widetilde L_2+E$ (case (ii)), let $A=\sigma_*(\widetilde L_1)$ and $B=\sigma_*(\widetilde L_2)$. The curves $A$ and $B$ are conics in $S$. Then
$$A+B=\sigma_*(\widetilde L_1+\widetilde L_2)\sim\sigma_*(\widetilde L_1+\widetilde L_2+E)\sim\sigma_*(-K_{\widetilde S}) \sim -K_S.$$
Moreover $A\cap B =p$ and since $A+B\sim-K_S$, then $A\cdot B=2$.
\end{proof}

Lemmas \ref{lem:delPezzo-lc-cod1-deg1-to-3}, \ref{lem:del-Pezzo-cat-degree-1}, \ref{lem:del-Pezzo-cat-degree-2}, Theorem \ref{thm:del-Pezzo-cat-degree-3} and Example \ref{exa:del-Pezzo-no-cat-codimension1} can be summarised in the following elegant statement, which proves Conjecture \ref{conj:cat} for non-singular surfaces.
\begin{cor}
\label{cor:cat-conjecture-dim2}
Let $S$ be a non-singular del Pezzo surface. Then $S$ satisfies the Cat property if and only if  all effective $\bbQ$-divisors $D$ with $D\simq-K_S$ are log canonical in codimension $1$.
\end{cor}

\section{Del Pezzo surface of degree $4$. Curves of low degree and log canonical pairs}
\label{sec:delPezzo-4-curves}
As we mentioned before, the proof of Theorem \ref{thm:del-Pezzo-glct-charp} for smooth complex del Pezzo surfaces does not extend to surfaces over algebraically closed fields when $2\leq \deg S\leq 4$. Corollary \ref{cor:del-Pezzo-glct-low-degrees} deals with the cases $1\leq \deg S\leq 3$. In this section we prove the case $\deg S=4$. To do this, we first classify low degree curves on $S$. Then, we construct effective anticanonical $\bbQ$-divisors with certain good properties. These $\bbQ$-divisor are used to provide a proof of the Theorem in Section \ref{sec:delPezzo-4-proof}.

\subsection{Curves of low degree and models of $S$}
Let $\pi\colon S\ra \bbP^2$ be the blow-up at points $p_1,\ldots,p_5\in \bbP^2$ in general position. Let $E_1,\ldots,E_5$ be the exceptional divisors. Recall $-K_S\sim \pioplane{3}-\sum^5_{i=1}E_i$ and $E_i^2=-1$.
\begin{table}[!ht]%
\begin{center}
\begin{tabular}{|c|c|c|c|c|c|}
\hline
Linear system $\calL\calS$																																			&$\deg C$	&$C^2$	&Fix $p$	&Fix $q$	&$C'$\\
\hline \hline
$\vert E_{i}\vert$																																										&$1$			&$-1$		&N					&N			&$E_i$\\
\hline
$\calL_{ij}= \left \vert \pioplane{1} -E_i-E_j\right\vert$																						&$1$			&$-1$		&N					&N			&$L_{ij}$\\
\hline
$\displaystyle{\calC_0= \left \vert \pioplane{2} -\sum^5_{i=1}E_i \right\vert}$											&$1$			&$-1$		&N					&N			&$C_0$\\
\hline
$\calB_{i}= \left \vert \pioplane{1} -E_i\right\vert$																								&$2$			&$0$		&Y					&N			&$B_i$\\
\hline
$\displaystyle{\calA_i = \left \vert \pioplane{2} -\sum^5_{\substack{j=1\\j\neq i}}E_j \right\vert}$	&$2$			&$0$		&Y					&N			&$A_i$\\
\hline
$\displaystyle{\calQ_i = \left \vert \pioplane{3} -E_i -\sum^5_{\substack{j=1}}E_j\right\vert}$			&$3$			&$1$		&Y					&Y			&$Q_i$\\
\hline
$\calR= \left \vert \pioplane{1}\right\vert$																													&$3$			&$1$		&Y					&Y			&$R$\\
\hline
$\calR_{ijk}= \left \vert \pioplane{2} -E_i-E_j-E_k \right\vert$																			&$3$			&$1$		&Y					&Y			&$R_{ijk}$\\
\hline
\end{tabular}
\end{center}
\caption{Catalogue of curves of low degree in $S$.}
\label{tab:delPezzo-4-lowdegree}
\end{table}

Observe Table \ref{tab:delPezzo-4-lowdegree}. In the first column we have defined certain complete linear systems $\calL\calS$ in $S$. Let $C\sim\calL\calS$ be any divisor. Its numerical properties ($C^2, \deg(C)$) are the same for any divisor in a given $\calL\calS$ and are easy to compute. We list them in the second and third columns of table \ref{tab:delPezzo-4-lowdegree}. Note that, by the genus formula, $p_a(C)=0$ in all cases. If $\deg C =1$, then by Proposition \ref{prop:rational-surfaces-sections-rational-curves}, $h^0(\calL\calS)\geq 1$. As we saw in Lemma \ref{lem:del-Pezzo-lines9-d} there is only a finite number of lines in a del Pezzo surface, so $h^0(\calL\calS)=1$ and we can find a unique curve $C'\in \calL \calS$. The notation for each particular $C'$ is in the last column of the table. 

If $\deg C=2$, then by Proposition \ref{prop:rational-surfaces-sections-rational-curves}, $h^0(\calL\calS)\geq 2$. Take $\calL\calS'\subset \calL\calS$ to be the sublinear system fixing $p$. Then $h^0(\calL\calS')\geq 1$ and we can find a curve $C'\in\calL\calS$ with $p\in C'$. The notation for each particular $C'$ is in the last column of the table. 

When the curve $C'$ is irreducible, we can realise it as the strict transform of an irreducible curve in $\bbP^2$ via the model $\pi$. For instance $L_{ij}$ is the strict transform of the unique line through $p_i$ and $p_j$. $C_0$ is the strict transform of the unique conic through all $p_i$. $B_i$ is the strict transform of a line passing through $p_i$ and $A_i$ is the strict transform of a conic through all $p_j$ but $p_i$. The last three rows of Table \ref{tab:delPezzo-4-lowdegree} deal with cubics and they are treated in Lemma \ref{lem:del-Pezzo-deg4-cubics}.

In order to understand the geometry of $S$ we need to understand which are its curves of low degree and how they intersect each other. We have just constructed some of these curves in Table \ref{tab:delPezzo-4-lowdegree}. In this section, among other properties of the low degree curves constructed above, we will show that the lines in Table \ref{tab:delPezzo-4-lowdegree} are all the lines in $S$. Furthermore, we will show that the conics in Table \ref{tab:delPezzo-4-lowdegree} are all the conics in $S$ passing through a given point $p$. Finally, there is more than one model $S\ra \bbP^2$ that characterises $S$ as a blow-up of the plane in $5$ points. We will also show how we can choose a model adequate to our needs.

\begin{lem}
\label{lem:del-Pezzo-deg4-good-curves-smooth}
Let $S$ be a non-singular del Pezzo surface of degree $4$ and $C'$ a curve as in Table \ref{tab:delPezzo-4-lowdegree}. Suppose $C'$ is irreducible. Then $C'$ is non-singular.
\end{lem}
\begin{proof}
Suppose $C'\neq Q_i$. Then, $\pi(C')$ is an irreducible curve of degree $1$ or $2$ in $\bbP^2$. Therefore $\pi(C')$ is smooth. Since $S$ is just the blow-up of smooth points of $\bbP^2$, if $\pi(C')$ is a smooth curve of $\bbP^2$, then its strict transform $C'$ is a smooth curve in $S$.

The irreducible curve $\pi(Q_i)$ is an irreducible cubic curve in $\bbP^2$ with multiplicity $2$ at $p_i$. Its strict transform $Q_i$ in $S$ must be smooth, since it is enough to blow-up $S$ at $p_i$ once to resolve $\pi(Q_i)$.
\end{proof}

\begin{lem}
\label{lem:del-Pezzo-deg4-lines-list}
The $16$ lines in Table \ref{tab:delPezzo-4-lowdegree} are all the lines in $S$. The intersection of these lines are:
\begin{align*}
&E_i\cdot E_j = -\delta_{ij}, \qquad L_{ij}\cdot E_i = L_{ij}\cdot E_j = 1,\qquad C_0\cdot E_i=1,\qquad C_0^2=-1, \\
&C_0\cdot L_{ij}=0,\qquad {L_{ij}\cdot L_{kl}=\left\{ \begin{array}{rl}
																-1 &\text{if } i=k \text{ and } j=l, \\
																0 & \text{if only two subindices are equal,} \\
																1 & \text{if none of the subindices are equal.} \\
														\end{array} \right.}
\end{align*}
\end{lem}
\begin{proof}
See Lemma \ref{lem:del-Pezzo-lines9-d} where $C_0=C_{12345}$.
\end{proof}

\begin{lem}
\label{lem:del-Pezzo-deg4-model-line-choice}
Given a line $L\subset S$, we can choose a model $\gamma\colon S\ra \bbP^2$ such that $L=E_1$. If $p=L_1\cap L_2$, the intersection of two lines, we can choose $\gamma$ such that $L_1=E_1,\ L_2=L_{12}$.
\end{lem}
\begin{proof}
We construct $\gamma\colon S \ra \bbP^2$ by contracting $5$ disjoint exceptional curves $F_i$ (i.e. $F_i\cdot F_j=0$ if $i\neq j$). Let $F_1=L$. 
\begin{itemize}
	\item[(i)] If $F_1=E_1$, take $F_2=E_2, F_3 =L_{34}, F_4=L_{35}, F_5=L_{45}$.
	\item[(ii)] If $F_1=C_0$, take $F_j=L_{1j}$.
	\item[(iii)] If $F_1=L_{12}$, take $F_i=L_{1\ (i+1)}$ for $2\leq i\leq 4$, $F_5=C_0$.
\end{itemize}
Obvious relabelling exhausts all possibilities for $L$ among the $16$ lines in Lemma \ref{lem:del-Pezzo-deg4-lines-list}. By Castelnuovo contractibility criterion \cite[V.5.7]{HartshorneAG} we can contract each $F_i$, leaving every other point intact. The image of $\gamma$ is $\bbP^2$, because the relative minimal model of $S$ once $5$ exceptional curves are contracted is unique. For the second part we can assume already $L_1=E_1$ and run this lemma again. In that case we are in case (i) above and we are done.\end{proof}
In a similar fashion to Lemma \ref{lem:del-Pezzo-deg4-lines-list} we can show:
\begin{lem}
\label{lem:del-Pezzo-deg4-conics}
If $C$ is an irreducible conic in $S$ passing through $p$, then $C=A_i$ or $C=B_i$, with $\pi(C)$ either a conic through all marked points but $p_i$ or a line through $p$ and $p_i$, respectively.
\end{lem}
\begin{proof}
By Lemma \ref{lem:del-Pezzo-all-conics-rational}, $C$ has arithmetic genus $p_a(C)=0$. By the genus formula, this implies $C^2=0$. Let $\bar C =\pi_*(C)\subset\bbP^2$, $\bar C\sim \oplaned$ for some $d\geq 1$. Hence $C\sim \pioplaned-\sum a_i E_i$ for $a_i\geq 0$. This gives
\begin{equation*}
0=C^2=d^2-\sum a_i^2,\ 2=(-K_S)\cdot C=3d-\sum a_i.
\label{eq:del-Pezzo-deg4-conics-proof}
\end{equation*}
given that $a_i$ are non-negative integers $\sum a_i^2\geq \sum a_i$. Hence
$$0=d^2-\sum a_i^2\leq d^2-\sum a_i = d^2-3d+2=(d-1)(d-2),$$
so $d=1$ or $d=2$. The only possibilities for $a_i$ for the second equation in \eqref{eq:del-Pezzo-deg4-conics-proof} to hold are $\sum a_i=2$ when $d=1$ and $\sum a_i=4, 5$, when $d=2$. All these possibilities are classified in Table \ref{tab:delPezzo-4-lowdegree}.
\end{proof}

\begin{lem}
\label{lem:del-Pezzo-deg4-model-conic-choice}
Given $C$ an irreducible conic in $S$, $p\in C$, we can choose a model $\gamma\colon S\ra \bbP^2$ such that under that model the curve $C$ can be realised as $C=A_i$ for any $i$ in Table \ref{tab:delPezzo-4-lowdegree}, unless $p\in E_1$ in which case $i\neq 1$.
\end{lem}
\begin{proof}
If $p\in L$, a line in $S$, assume $L=E_1$ by Lemma \ref{lem:del-Pezzo-deg4-model-line-choice}. We have $C\neq A_1$ since otherwise  $$0=A_1\cdot E_1 =C \cdot E_1 \geq \mult_p(C)\cdot \mult_p(E_1)=1,$$
a contradiction.

If $C=B_1$, take $F_i$ and $\gamma:S\ra \bbP^2$ as in the proof of Lemma \ref{lem:del-Pezzo-deg4-model-line-choice}, case (i). Because $C$ is irreducible, $\overline C=\gamma(B_1)=\oplaned$ by the genus formula on $\bbP^2$. Moreover:
$$B_1\sim\gamma^*(\oplaned)-\sum_{i=1}^5 (F_i\cdot B_1) F_i= \gamma^*(\oplaned)-F_1-F_3-F_4-F_5,$$
and $2=B_1 \cdot (-K_S)=3d -4,$
so $d=2$.
Therefore under the new blow-up $C$ is $A_2$. By obvious relabelling of the $F_j$ we can consider $C=A_i$ with $i\neq 1$.

If $C=B_i$, with $i\neq 1$, then $p\not\in E_1$ since $C$ is irreducible. If $C=B_2$, the same choice of $F_i$ gives us $C=A_1$ under the new blow-up. If $C=B_i$ for $i=3,4,5$ take $F_1=E_1, F_2=E_i, F_3=L_{jk}, F_4=L_{jl}, F_5=L_{kl}$ for different $j,k,l\in \{1,\ldots,5\} \setminus \{ i\}$ and $C=A_i$ under $\gamma$.
\end{proof}

\begin{lem}
\label{lem:del-Pezzo-deg4-conics-tangency}
Let $p\in L$, where $L$ is a line, and let $C_1,C_2$ be distinct irreducible conics passing through $p$. Then $C_1$ and $C_2$ intersect normally at $p$.
\end{lem}
\begin{proof}
By Lemma \ref{lem:del-Pezzo-deg4-model-conic-choice} we may assume that $E=L_1$ and $C_1=A_i$ for some $2\leq i\leq 5$. Without loss of generality we may assume that $C_1=A_2$. Note that $C_2\neq A_1, B_j$, for $j>1$ since $E_1\cdot A_1=E_1\cdot B_i=0$ and $C_2$ being irreducible would give a contradiction:
$$0=E_1\cdot C_2\geq \mult_pE_1\cdot \mult_pC_2\geq1.$$
By Lemma \ref{lem:del-Pezzo-deg4-conics} we have that $C_2=B_1$ or $C_2=A_i$ for $i\neq 1,2$. In both cases $C_2\cdot A_2=1$, obtaining simple normal crossings at $p$:
$$1=C_2\cdot A_2\geq (C_2\cdot A_2)\vert_p.$$
\end{proof}

\begin{lem}
\label{lem:del-Pezzo-deg4-cubics}
Let $\calL\calS$ be a complete linear system of degree $3$ as in the last three rows of Table \ref{tab:delPezzo-4-lowdegree}. Let $\pi \colon \widetilde S \ra S$ be the blow-up of some point $p\in S$ with exceptional curve $E\subset \widetilde S$. Let $q\in E$. Then $\exists\ C' \in \calL\calS$, a curve with $p\in C'$ satisfying one of the following:
\begin{itemize}
	\item[(i)] $C'$ is smooth at $p$ and its strict transform $\widetilde C'\sim\sigma^*(C)-E$ passes through $q$.
	\item[(ii)] $C'$ is reducible and two of its components intersect at $p$. One of this components is a line $L$. By Lemma \ref{lem:del-Pezzo-deg4-model-line-choice} we choose a model $\pi\colon S \ra \bbP^2$ such that $L=E_1$. Then either:
	\begin{itemize}
	\item[(a)] $C'=E_1+C$ for $C$ an irreducible conic in Table \ref{tab:delPezzo-4-lowdegree} passing through $p$.
	\item[(b)] $C'=E_1+L_{12}+L$ for $L$ a line not passing through $p$.
\end{itemize}
\end{itemize}
Case (ii)(a) is possible only if $\calL\calS=\calR_{ijk}$ for $(i,j,k)\in \{(2,3,4),(2,4,5),(3,4,5)\}$ or $\calL\calS=\calR$. Case (ii)(b) is possible only if $\calL\calS=\calR, \calQ_2,\calR_{2jk},\calR_{12k}$.\newline
In case (ii) (a) we can find $C$:
\begin{itemize}
	\item If $\calL\calS=\calR$, then $C=B_1$.
	\item If $\calL\calS=\calR_{234}$, then $C=A_5$.
	\item If $\calL\calS=\calR_{245}$, then $C=A_3$.
	\item If $\calL\calS=\calR_{345}$, then $C=A_2$.
\end{itemize}
In case (ii) (b) we can find $L$:
\begin{itemize}
	\item If $\calL\calS=\calR$, then $L=E_2$.
	\item If $\calL\calS=\calQ_2$, then $L=C_0$.
	\item If $\calL\calS=\calR_{2jk}$, then $L=L_{jk}$.
	\item If $\calL\calS=\calR_{12k}$, then $L=L_{1k}$.
\end{itemize}
\end{lem}
In each case, denote $C'$ by the letter in the last column of Table \ref{tab:delPezzo-4-lowdegree}. Note that in case (i), $C'$ may still be reducible, but it is irreducible around $p$.

\begin{proof}
Let $\calL\calS'=\{D\in \calL\calS \setsep p\in \Supp(D)\}$ and let $\widetilde{\calL\calS'}=\vert\sigma^*(\calL\calS')-E\vert$. By Proposition \ref{prop:rational-surfaces-sections-rational-curves}
$$h^0(\widetilde{\calL\calS'})=h^0(\calL\calS')=h^0(\calL\calS)-1\geq 2,$$ so we can choose $B\in \widetilde{\calL\calS'}$ an effective divisor passing through $q$. If $E\not\subset\Supp(B)$, then let $C'=\sigma_*(B)$ and $B=\widetilde C' \sim \sigma^*(C')-E$ where $B\cdot E=1$. Clearly, this is case (i) in the statement.

Conversely, if $E\subset \Supp(B)$, let $B=A+bE$ where $E\not\subset\Supp(A)$, $b\geq 1$ is an integer and $A$ is effective. Then $\widetilde C'=A=B-bE\sim\sigma^*(C')-(b+1)E$ for $C'=\sigma_*(B)=\sigma_*(A)$ and $C'$ is singular at $p$. $C'$ is reducible, since otherwise $p_a(\widetilde C')<p_a(C')=0$, which is impossible. Note that if $C'$ is reducible, then $C'=L+F$ for $L$ a line and $F$ a possibly reducible conic. By Lemma \ref{lem:del-Pezzo-deg4-model-line-choice} we may choose a model such that $L=E_1$. This is case (ii) in the statement which can only split in subcases (a) and (b).

In case (b) we can assume the second line is $L_{12}$ by Lemma \ref{lem:del-Pezzo-deg4-model-line-choice}.

We prove (a). If $C'\in \calR$, then $C\sim\pioplane{1}-E_1=B_1$. If $\calR\calS=\calR_{ijk}$, then $C\sim\pioplane{2}-E_i-E_j-E_k-E_1$ which is not an irreducible conic in Lemma \ref{lem:del-Pezzo-deg4-conics} unless it is one of the cases in the statement.

We prove (b). If $\calL\calS=\calR$, then $L\sim\pioplane{1}-E_1-L_{12}=E_2$. If $\calL\calS=Q_i$, then
$$L\sim\pioplane{3}-E_i-\sum^5_{j=1}E_j-E_1-L_{12}\sim\pioplane{2}+E_2-E_i-\sum^5_{j=1}E_j,$$
which is not a line in Lemma \ref{lem:del-Pezzo-deg4-lines-list} unless $i=2$, in which case $L=C_0$. Finally, if $\calL\calS=R_{ijk}$, then
$$L\sim\pioplane{2}-E_i-E_j-E_k-E_1-L_{12}\sim\pioplane{1}-E_i-E_j-E_k+E_2$$ which is not a line unless $i=1,j=2, 3\leq k\leq 5$ or $i=2,3\leq j<k\leq 5$, finishing the proof.
\end{proof}

\subsection{Auxiliary $\bbQ$-divisors}
In this section we will use the rational curves constructed in the previous section for non-singular del Pezzo surfaces of degree $4$ to show the existence of certain effective anti-canonical $\bbQ$-divisors with good local properties and controlled singularities. These $\bbQ$-divisors are used in the proof of Theorem \ref{thm:del-Pezzo-glct-charp}, when $K_S^2=4$.
\begin{lem}
\label{lem:del-Pezzo-deg4-complementary-hyperplane-section}
Given an integral curve $C\subset S$ with $\deg C \leq 2$, there is an irreducible curve $Z$ such that $Z+C\in \vert-K_S\vert$.
\end{lem}
\begin{proof}
Given $p\in C$, denote by $\sigma\colon\widetilde{S} \ra S$ the blow-up at $p$ and $\widetilde C$ the strict transform of $C$.

If $\deg C=1$ we can assume $C=E_1$ by Lemma \ref{lem:del-Pezzo-deg4-model-line-choice}. Consider
$$\calQ_1 = \vert\pioplane{3}-2E_1-E_2-\cdots -E_5\vert.$$
Choose $p\in E_1$ not passing through any other line and $q=\sigma^{-1}(p)\subset \widetilde S$, $q\not\in \widetilde E_1$ a general point. Our choice of $p$ and $q$ defines a curve $Z=Q_1\in \calQ_1$ with $p\in Z$ as in section \ref{sec:delPezzo-4-curves} which is irreducible, since $q$ is general. Indeed, only a finite number of conics pass through $p$ so their strict transforms in $\widetilde S$ cannot pass through $q$ due to the generality condition, but $q\in Q_1$ and has degree $3$, so it is irreducible.

If $\deg C=2$ by Lemma \ref{lem:del-Pezzo-deg4-model-conic-choice} assume $C=A_1$. Choose $p\in A_1$ such that $p$ is not in any line and take $Z=B_1\in \calB_1$ as in section \ref{sec:delPezzo-4-curves}. $Z$ is irreducible by the proof of Lemma \ref{lem:del-Pezzo-deg4-aux-divisors-G-Basic}, subcase 1.

In both cases we have
$$C+Z\sim -K_S.$$
\end{proof}

The following two lemmas are needed in the proof of theorems \ref{thm:del-Pezzo-glct-charp}. We provide a joint proof.
\begin{lem}
\label{lem:del-Pezzo-deg4-aux-divisors-G-Basic}
Let $S$ be a non-singular del Pezzo surface of degree $4$. Let $p\in S$. There is an effective $\bbQ$-divisor $G=\sum{g_i G_i}$ in $S$, $G\simq -K_S$ such that
\begin{itemize}
	\item[(i)] $(S,\frac{2}{3}G)$ is log canonical,
	\item[(ii)] $p\in G_i \ \forall G_i$,
	\item[(iii)] $\deg G_i\leq 2 \ \forall G_i$,
	\item[(iv)] all $G_i$ are irreducible and smooth.
\end{itemize}
\end{lem}

\begin{lem}
\label{lem:del-Pezzo-deg4-aux-divisors-H}
Let $S$ be a non-singular del Pezzo surface of degree $4$. Let $p\in S$, $q\in E\subset \widetilde S\stackrel{\sigma}{\lra} S$, where $E$ is the exceptional curve in the blow-up $\sigma$ of $S$ at $p$. Suppose $p$ belongs to, at most, one line. There is $H=\sum{h_i H_i}$, an effective $\bbQ$-divisor in $S$, with $H\simq -K_S$, such that:
\begin{itemize}
	\item[(i)] $(S,\frac{2}{3}H)$ is log canonical,
	\item[(ii)] $p\in H_i\  \forall H_i$,
	\item[(iii)] $\deg H_i\leq 3\ \forall H_i$,
	\item[(iv)] all $H_i$ are irreducible and smooth,
	\item[(v)] the point $q\in \widetilde H_i$, the strict transform of $H_i$ via $\sigma,\ \forall H_i$ such that $\deg H_i>1$.
\end{itemize}
\end{lem}

\begin{proof}[Proof of lemmas \ref{lem:del-Pezzo-deg4-aux-divisors-G-Basic} and \ref{lem:del-Pezzo-deg4-aux-divisors-H}]
We will construct these $\bbQ$-divisors explicitly, by case analysis, depending on the position of $p\in S$ and $q\in E$, the exceptional divisor of the blow-up of $p$. In order to do this we use curves from Table \ref{tab:delPezzo-4-lowdegree}, which are lines, conics and cubics. These were constructed depending on $p$ and $q$ and were possibly reducible.

Conditions (ii) and (iii) in lemma \ref{lem:del-Pezzo-deg4-aux-divisors-H} will be clear by construction, as well as condition (v), for $H$.

We will check log canonicity (condition (i)). Not it is enough to show that $\mult_p H\leq \frac{3}{2}$, since if $(S, \frac{2}{3}H)$ is not log canonical, then $\mult_pH>\frac{3}{2}$ by Lemma \ref{lem:adjunction} (i).

The biggest task will be to prove that the curves chosen for each particular case are irreducible in each situation (condition (iv)). Ultimately this is the reason for our break down into cases. Smoothness follows from Lemma \ref{lem:del-Pezzo-deg4-good-curves-smooth}.

\begin{case}
\label{case:1}{\ \\}
\assump{$p$ is not in any line}In particular $p\not\in E_i$ for all $i$. Let $G=A_1+B_1\sim -K_S$, and $(S,\frac{2}{3}G)$ is log canonical, since $A_1$ and $B_1$ intersect either in a tacnodal point or with simple normal crossings. Since all curves $C\subset \Supp(G)$ are conics, if $C=L_a+L_b$, the sum of two lines, then $p$ is one line, contradicting \refAssump.

\begin{subcase}
\label{subcase:1.1}{\ \\}
\subassump{$q$ is not in the strict transform of conics in $S$ passing through $p$}Let $H=\frac{1}{2}R+ \frac{1}{6}\sum_{i=1}^5Q_i \simq -K_S.$
Again, \refAssump and \refSubassump assure that $R$ and $Q_i$ are irreducible, and therefore smooth, so we just need to check
\[2\mult_p(H) = 2\left[\frac{1}{2}+5\cdot \frac{1}{6}\right]=\frac{16}{6}<3.\]
\end{subcase}
\begin{subcase}\label{subcase:1.2}{\ \\}
\subassump{The point $q\in \widetilde C$, the strict transform of a conic in $S$}By Assumption 1, $q\not\in \widetilde L$, for $L$ a line in $S$. Without loss of generality assume $C=A_1$, which is irreducible (use Lemma \ref{lem:del-Pezzo-deg4-model-conic-choice}). Observe that given a conic $C'\neq A_1$ with $p\in C'$ in Lemma \ref{lem:del-Pezzo-deg4-conics}, then $A_1\cdot C'=1$ unless $C'=B_1$. If $(A_1\cdot B_1)\vert_p=1$ then $A_1$ is the only conic such that $q\in \widetilde A_1$. Suppose this is the case. Let
\[H=\frac{1}{2}A_1 + \frac{1}{2}R_{125}+ \frac{1}{2}R_{134}\simq -K_S.\]
The point $q\in \widetilde R_{125}\cap \widetilde R_{124}$ by Lemma \ref{lem:del-Pezzo-deg4-cubics}, since $p$ does not lie in a line. We need to show $R_{125}$ and $R_{134}$ are irreducible. By relabelling, it is enough to show it for $R_{125}$. Suppose $R_{125}=C_a+L_b$ where $C_a$ is a conic and $L_b$ is a line. $p\in C_a$ and $p\not\in L_b$, by \refAssump. In particular $C_a=A_1$. Then
$$L_b\sim R_{125}-A_1=-E_1+E_3+E_4$$
which is not a line by Lemma \ref{lem:del-Pezzo-deg4-lines-list}. Finally
\[2\mult_p(H) = 2\left[\frac{1}{2}+\frac{1}{2}+\frac{1}{2}\right]=3.\]

Suppose $(A_1\cdot B_1)\vert_p=2$. Then $q=\widetilde B_1\cap \widetilde A_1$. Let
$$H=A_1+B_1\sim -K_S.$$
Clearly $A_1,B_1$ are irreducible by the assumption an $(S, \frac{2}{3} H)$ is log canonical by Case 1.
\end{subcase}
\end{case}
\begin{case}
\label{case:2}
Suppose $p\in L$, a line in $S$ and no other line. By Lemma \ref{lem:del-Pezzo-deg4-model-line-choice} we can consider $L=E_1$.
\forcenewline\assump{$p\in E_1$ and $p\not\in L$, any other line different than $E_1$}Take
\[G=\frac{1}{3}\sum^5_{j=2}A_j+ \frac{1}{3}B_1 + \frac{2}{3}E_1 \simq -K_S.\]
$B_1$ is irreducible. Since, if it was not irreducible, then
\[\pi^*(\oplane{1})-E_1\sim B_1=L_a+E_1,\]
where $p\not\in L_a\sim B_1-E_1\sim \pi^*(\oplane{1})-2E_1$, a line in $S$, but there is no such a line in $S$ by Lemma \ref{lem:del-Pezzo-deg4-lines-list}. 

The curves $A_j$ are irreducible too. If they were not irreducible, then
\[\pi^*(\oplane{2})-\sum_{\substack{k=1\\k\neq j}}^5 E_k\simq A_j=L_b+E_1,\]
where
$$p\not\in L_b=A_j-E_1 \sim \pi^*(\oplane{2})-2E_1-\sum_{\substack{k=2\\k\neq j}}^5 E_k$$
is a line, but there is no such a line in $S$, by Lemma \ref{lem:del-Pezzo-deg4-lines-list}. 
\begin{clm}
$(S,\frac{2}{3}G)$ is log canonical.
\end{clm}
\begin{proof}
Since
$$E_1\cdot A_j=B_1\cdot E_1 = A_j\cdot B_1 =A_j\cdot A_k=1,\ j\neq 1, k\neq j, k\neq 1,$$
they intersect each other transversely so we blow up once to obtain simple normal crossings:
$$\sigma^*(\lambda G+K_{\widetilde S})\simq\lambda \widetilde G + \left(\left(4\cdot \frac{1}{3}+ \frac{1}{3}+ \frac{2}{3}\right) \lambda-1\right) F_1= \lambda \widetilde G + (\frac{7}{10}\lambda-1)F_1,$$
and $\lambda=\frac{2}{3}$ makes $\disc(S,\lambda D)\geq -1$.
\end{proof}
\begin{subcase}
\label{subcase:2.1}\forcenewline
\subassump{$q\not\in \widetilde C$ for $C$ any line or conic in $S$}In particular $q\not\in \widetilde E_1$. Let
$$H=\frac{1}{8}\sum_{2\leq j<k\leq 5}R_{1jk} + \frac{1}{8}\sum^5_{i=2}Q_i+ \frac{1}{4}E_1\simq -K_S$$
$q\in \widetilde Q_i, \widetilde R_{1jk}$ by Lemma \ref{lem:del-Pezzo-deg4-cubics} since $p$ does not lie in two lines. All $R_{1jk}$ and $Q_i$ are irreducible, since otherwise they would split in a conic $C_a$ and a line $L_b$ and either $q\in\widetilde C_a$ or $q\in \widetilde L_b$, contradicting the assumption. Moreover
$$2(\mult_p(H)) = 2\left(\frac{1}{8}\cdot 6 + \frac{1}{8}\cdot 4+ \frac{1}{4}\cdot 1\right) = 3.$$
\end{subcase}
\begin{subcase}
\label{subcase:2.2}\forcenewline
\subassump{$q\in \widetilde C$, for some conic $C$ in $S$ but $q\not \in \widetilde L$, for all lines $L$ in $S$}In particular $C$ is irreducible and $q\not \in \widetilde E_1$. By lemmas \ref{lem:del-Pezzo-deg4-conics} and \ref{lem:del-Pezzo-deg4-model-conic-choice} we can assume that 
$$C\sim\pi^*(\oplane{2})-\sum_{\substack{i=1\\i\neq k}}^5 E_i~\sim A_k, \text{ for } k\neq1.$$
where $p\in C=A_k,\ k\neq 1$, with $q\in \widetilde C$. Moreover, since $q\in \widetilde A_k$, \refSubassump assures it is irreducible.

Without loss of generality, suppose $k=5$. Suppose there is another conic $C'$ in $S$ such that $p\in C'$, $q\in \widetilde C'$ and $C'\neq A_5$. Since $p\in C'\cap E_1$, by Lemma  \ref{lem:del-Pezzo-deg4-conics} either $C'=B_1$ or $C'=A_j$, for $j\neq 1,5$. However $A_5\cdot B_1=1$, $A_i\cdot A_5=1$ for $i\neq 1,5$. Therefore in both cases $C'$ and $A_5$ intersect transversely and $q\not\in \widetilde C'$. Let
\begin{equation}
H=\frac{3}{5} A_5 + \frac{1}{5}\left(R_{125}+R_{135}+R_{145}\right)+ \frac{1}{5}Q_5 + \frac{2}{5}E_1\simq -K_S.
\label{eq:2.2}
\end{equation}
$q\in \widetilde Q_5, \widetilde R_{1jk}$ by Lemma \ref{lem:del-Pezzo-deg4-cubics} since $p$ does not lie in two lines.

We already know that $A_5$ and $E_1$ are irreducible. Suppose $Q_5$ is reducible. Then $Q_5=A_5+L_b$, where $L_b$ is a line. But then
$$L_b \sim Q_5-A_5\sim \pioplane{1}-2E_5$$
which is not one of the lines in $S$, by Lemma \ref{lem:del-Pezzo-deg4-lines-list}.

If $R_{125}$ is reducible, then there is a line $L_a$ such that
$$L_a \sim R_{125}-A_5 \sim -E_5+E_4+E_3$$
which is not a line in Lemma \ref{lem:del-Pezzo-deg4-lines-list}. By relabelling, it is clear that $R_{135}$ and $R_{145}$ are irreducible too.

\begin{lem}\label{lem:lc2.2}
If $p\in A_5\cap E_1$ and $q\in \widetilde A_5$, $q\not\in \widetilde L$, for any line $L$ in $S$, then $(S,\frac{2}{3}H)$ is log canonical, for $H$ as in \eqref{eq:2.2}.
\end{lem}
\begin{table}[htb]%
\begin{center}
\begin{tabular}{|r||c|c|c|c|c|c|c|}
\hline
										&$\widetilde {A_5}$	&$\widetilde {R_{125}}$	&$\widetilde {R_{135}}$	&$\widetilde {R_{145}}$	&$\widetilde {Q_5}$	&$\widetilde{E_1}$ &$F_1$	\\
\hline \hline
$\widetilde {A_5}$	& &1&1&1&1&0&1\\
\hline
$\widetilde {R_{125}}$	& & &1&1&1&0&1\\
\hline
$\widetilde {R_{135}}$	& & & &1&1&0&1\\
\hline
$\widetilde {R_{145}}$	& & & & &1&0&1\\
\hline
$\widetilde {Q_{5}}$	& & & & & &0&1\\
\hline
$\widetilde {E_{1}}$	& & & & & & &1\\
\hline
\end{tabular}
\caption{Intersection numbers for Lemma \ref{lem:lc2.2}.}
\label{tab:intersectNumbers}
\end{center}
\end{table}
\begin{proof}
Let $\sigma_0:S_0\ra S$ be the blow up at $p$ with exceptional divisor $F_1$ with $q\in F_1$. Table \ref{tab:intersectNumbers} gives the intersection numbers in $S_0$. Since all curves in Table \ref{tab:intersectNumbers} intersect normally and pass through $q$, we just need to blow up this point to obtain simple normal crossings. Let $\sigma:\widetilde S \ra S$ be the composition of both blowing up maps and $F_2$ be the second exceptional divisor. Then:
\begin{align*}
\sigma^*(\lambda H+K_{\widetilde S})\simq\lambda \widetilde H &+ \left(\left(\frac{3}{5}+3\cdot \frac{1}{5}+ \frac{1}{5}+ \frac{2}{5}\right)\lambda -1\right) F_1+\\
&\left(\left(\frac{7}{5}+\frac{9}{5}\right)\lambda -2\right) F_2,
\end{align*}
and for $\lambda=2/3$, $(S,\lambda H)$ is log canonical.
\end{proof}

\end{subcase}

\begin{subcase}
\label{subcase:2.3}{ \ \\ }
Suppose that in \refAssump  $q\in \widetilde L$ for some line $L$ in $S$. Then $L=E_1$.
\forcenewline
\subassump{$q\in \widetilde E_1$}Suppose $q\in \widetilde C$ where $C$ is a conic in $S$. As in case \ref{subcase:2.1} we can assume, by using Lemma \ref{lem:del-Pezzo-deg4-conics} and Lemma \ref{lem:del-Pezzo-deg4-model-conic-choice} that $C= A_5\sim\pi^*(\oplane{2})-\sum_{i=1}^4E_i$.
$C$ is irreducible, since otherwise $L_a\sim A_5-E_1\sim \pioplane{2} -2E_1-\sum_{i=2}^4E_i,$ would be a line.

Since $C$ is irreducible, it intersects $E_1$ transversely at $p$ and since $q\in \widetilde E_1$, then $q\not\in \widetilde C$. So $q$ does not belong to the strict transform of any conic. Now, take
$$H=Q_1+E_1\sim -K_S,$$
and, providing $Q_1$ is irreducible, $(S,\frac{2}{3}H)$ is log canonical, since $Q$ and $E_1$ intersect each other at worst at a tacnodal point. Since $p$ is not the intersection of two lines, $q\in \widetilde Q_1$ by Lemma \ref{lem:del-Pezzo-deg4-cubics}. Since $q$ is not on the strict transform of any conic, if $Q_1$ is not irreducible, then $Q_1=E_1+C_a$, where $C_a$ is a (possibly irreducible) conic such that $q\not\in\widetilde C_a$, $p\in C_a$. But
$$C_a=Q_1-E_1 \sim \pioplane{3} -3 E_1-\sum^5_{i=2}E_i$$
is not an irreducible conic, by Lemma \ref{lem:del-Pezzo-deg4-conics}. If it was the union of two lines, by Lemma \ref{lem:del-Pezzo-deg4-lines-list}, one of them should be $C_0$, but then 
$$L_b=C_a-C_0\sim \pioplane{1}-2E_1$$
is not one of the lines in Lemma \ref{lem:del-Pezzo-deg4-lines-list}.
\end{subcase}
\end{case}
\end{proof}

\subsection{Proof of Theorem \ref{thm:del-Pezzo-glct-charp} in degree $4$}
\label{sec:delPezzo-4-proof}
We prove Theorem \ref{thm:del-Pezzo-glct-charp} in the case $K_S^2=4$.
\begin{clm}
\label{clm:del-Pezzo-deg4-upperbound}
$$\glct(S)\leq \frac{2}{3}$$
\end{clm}
\begin{proof}
Take $p=E_1\cap L_{12}$ and the conic $A_2$, which is irreducible (see case $3$ in the proof of Lemma \ref{lem:del-Pezzo-deg4-aux-divisors-G-Basic}) and isomorphic to $\bbP^1$. Consider $G=E_1+L_{12}+A_2\sim K_S$. We are done, since $\lct_p(S, G)=2/3$, and
$$\glct(S)=\inf\{\lct_r(S,D) \setsep r\in S,\ D\simq K_S\}\leq\frac{2}{3}.$$
\end{proof}
We need to show $\glct(S)\geq \frac{2}{3}$. We proceed by contradiction. Suppose there is an effective $\bbQ$-divisor
$$D=\sum{d_i D_i}\simq -K_S, \qquad d_i>0 \ \forall i$$
such that $(S,\lambda D)$ is not log canonical for some $\lambda<\frac{2}{3}$. Then $\LCS(S,\lambda D)\neq \emptyset$. 

\begin{lem}
\label{lem:deg4LCSpoints}
$\LCS (S,\lambda D)$ contains only isolated points.
\end{lem}
\begin{proof}
If $C\subset \LCS(S,\lambda D)$, where $C$ is a curve, then $C=D_i$ for some $D_i$ such that $\lambda d_i\geq 1$, by Lemma \ref{lem:adjunction} (ii) i.e. $d_i>\frac{3}{2}$. Then
$$4=-K_S \cdot D = \sum d_i \deg(D_i)>\frac{3}{2}\deg(D_i),$$
so $\deg (D_i) \leq 2$.
Using Lemma \ref{lem:del-Pezzo-deg4-complementary-hyperplane-section} choose a curve $Z$ such that $D_i+Z$ is cut out by a hyperplane section of $S$ passing through $D_i$ such that $Z$ is irreducible. We have $D_i+Z\sim -K_S\simq D$. Hence
\begin{equation}
\deg Z=Z\cdot D = (-K_S+D_i)(-K_S)=4-\deg D_i.
\label{eq:deg4hypA}
\end{equation}
In particular $\deg Z \geq 2$, so $Z\cdot D_j\geq 0$ for all irreducible $D_j$ (since only lines can have negative self-intersection). Then
\begin{equation}
(Z\cdot D)\geq d_i(Z\cdot D_i)=d_i(-K_S-D_i)\cdot D_i =d_i\left( \deg D_i -(\deg D_i-2) \right)=2d_i
\label{eq:deg4hypB}
\end{equation}
by the genus formula, since all lines and conics in $S$ are rational. But  then \eqref{eq:deg4hypA} and \eqref{eq:deg4hypB} give a contradiction:
$$3\geq 4-\deg (D_i)\geq 2 d_i>3.$$
\end{proof}

Let $p\in \LCS(S,\lambda D)$, i.e. the pair $(S, \lambda D)$ is not log canonical at some point $p$.
\begin{lem}
The point $p$ is not in the intersection of two lines.
\end{lem}
\begin{proof}
Suppose for contradiction that $p$ is a pseudo-Eckardt point. By Lemma \ref{lem:del-Pezzo-deg4-model-line-choice} we may choose $\pi\colon S \ra \bbP^2$ such that $p=E_1\cap L_{12}$. For $L=E_1,L_{12}$ we have that $L\subseteq\Supp(D)$ since otherwise
$$1=L\cdot D \geq \mult_pD>\frac{1}{\lambda}>1,$$
by Lemma \ref{lem:adjunction} (i). Hence we may write $D=aE_1+bL_{12}+\Omega$ where $a,b>0$ and $E_1,L_{12}\not\subseteq\Supp(\Omega)$.

Observe that the curve $A_2$ in Table \ref{tab:delPezzo-4-lowdegree} with $p\in A_2$ is irreducible, since otherwise there would be lines passing through $p$ with rational classes
$$A_2-E_1\sim\pioplane{2}-2E_1-E_3-E_4-E_5,\qquad \text{ or } \qquad A_2-L_{12}\sim\pioplane{1}+E_2-E_3-E_4-E_5,$$
which is impossible by Lemma \ref{lem:del-Pezzo-deg4-lines-list}. Since
$$A_2\cdot E_1=A_2\cdot L_{12}=L_{12}\cdot E_1=1,$$
the pair $(S, \lambda(A_2+E_1+L_{12}))$ is log canonical for $\lambda\leq \frac{2}{3}$ (see proof of Claim \ref{clm:del-Pezzo-deg4-upperbound}). Therefore, by Lemma \ref{lem:convexity} we may assume that $A_2\not\subset\Supp(D)$. We conclude
\begin{equation}
2\geq D\cdot A_2\geq a+b+\mult_p\Omega\geq a+b.
\label{eq:del-Pezzo-glct-4-proof-pEckardt}
\end{equation}
Now observe that
$$1=E_1\cdot D \geq -a+b+\mult_p\Omega,$$
$$1=L_{12}\cdot D \geq a-b+\mult_p\Omega,$$
and adding these two equations it follows that $\mult_p\Omega\leq 1$. The hypotheses of Theorem \ref{thm:inequality-Cheltsov} are satisfied. Therefore one of the following holds:
$$2(1-\lambda a)<L_{12}\cdot (\lambda\Omega)=\lambda(1-a+b)$$
$$2(1-\lambda b)<E_1\cdot (\lambda\Omega)=\lambda(1+a-b).$$
Since the roles of $a$ and $b$ are symmetric, it is enough to disprove the latter equation to obtain a contradiction. Indeed, the last inequality implies
$$2<\lambda(1+a+b)<3\lambda<2$$
by \ref{eq:del-Pezzo-glct-4-proof-pEckardt}, a contradiction.
\end{proof}

Let $G$ be the effective $\bbQ$-divisor in Lemma \ref{lem:del-Pezzo-deg4-aux-divisors-G-Basic}. Recall that $(S,\lambda G)$ is log canonical and that all irreducible components $G_j\subset \Supp(G)$ satisfy $p\in G_j$.
By Lemma \ref{lem:convexity}, we can assume $\exists G_j\subset \Supp (G)$ an irreducible curve such that $G_j\not\subset \Supp (D)$. Then
$$2\geq \deg G_j= (-K_S)\cdot G_j= D\cdot G_j\geq \mult_p (D) \cdot \mult_p(G_j)\geq \mult_p(D).$$
Therefore, we have bounded the multiplicity of $D$ at $p$:
\begin{equation}
2\geq \mult_p(D) \geq \frac{3}{2}.
\label{eq:boundD}
\end{equation}

Let $\sigma : \widetilde S \lra S$ be the blow-up of $p$ with exceptional divisor $E$. Applying Lemma \ref{lem:log-pullback-preserves-lc} to $(S,\lambda D)$, there is some $q\in E$ such that the pair
$$(\widetilde S,\lambda \widetilde D +(\lambda\mult_p D -1)E)$$
is not log canonical at some point $q\in E$. By \eqref{eq:boundD}, the pair is log canonical near $q\in E$.

Since $\lambda\mult_p D-1\leq 1$ by \eqref{eq:boundD}, applying Lemma \ref{lem:adjunction} (i) to this pair we obtain, that for some $q\in E$:
\begin{equation}
\mult_q(\widetilde D)+\mult_pD>3.
\label{eq:boundDtransform}
\end{equation}

Given $p\in S$ and $q\in E\subset \widetilde S$ as above, we apply Lemma \ref{lem:del-Pezzo-deg4-aux-divisors-H} to obtain an effective $\bbQ$-divisor $H$ on $S$ such that $(S,\lambda H)$ is log canonical with $p\in H_j$ for all irreducible components $H_j\subset\Supp(H)$ and $q\in \widetilde H_j$ whenever $\deg H_j>1$.

By Lemma \ref{lem:convexity} we may assume that $\exists H_i$ such that $H_i\not\subset\Supp(D)$.
\begin{clm}
If $\deg H_i=1$, then $H_i\subset \Supp (D)$.
\end{clm}
\begin{proof}
Suppose $H_i\not\subset \Supp(D)$. Then
$$\frac{3}{2}\leq\mult_p(D)\leq\mult_p(D)\cdot\mult_p(H_i)\leq D\cdot H_i= (-K_S)\cdot H_i=1$$
which is a contradiction.
\end{proof}
Hence, we can assume $\exists H_j\not\subset \Supp (D)$ and $q\in \widetilde H_j$, $p\in H_j$ and $2\leq\deg(H_j)\leq 3$. Then
$$\widetilde H_j \cdot \widetilde D = H_j \cdot D - \mult_p(H_j)\mult_p (D) \leq 3-\mult_p(D).$$
But $\widetilde H_j\not\subset \Supp(\widetilde D)$, so 
$$3-\mult_p(D)\geq \widetilde H_j \cdot \widetilde D\geq \mult_	q (\widetilde D),$$
contradicting \eqref{eq:boundDtransform}. This completes the proof.

\chapter{Dynamic $\alpha$-invariant of smooth del Pezzo surfaces}
\label{chap:dynamic}
Let $S$ be a smooth del Pezzo surface over $k=\bbC$. In this chapter we compute $\alpha(S, (1-\beta)C)$ when $C\in\vert-K_S\vert$ is a smooth curve. In the first section we show which other $C$ we could take. Each of the sections afterwards deals with each non-singular del Pezzo surfaces when $C\in\vert-K_S\vert$, with decreasing degree.

\section{The dynamic $\alpha$-invariant on del Pezzo surfaces}
In order to apply Theorem \ref{thm:Jeffres-Mazzeo-Rubinstein}, we need an effective $\bbQ$-divisor $D$ with simple normal crossings such that $-(K_S+(1-\beta)D)$ is ample for $0<\beta\ll 1$. An obvious choice is to take $D\sim-K_S$. The only other choice when $D$ is a smooth divisor is to take $D\cong\bbP^1$. However not all rational curves satisfy the ampleness condition. The following Theorem classifies which are the possibilities when $D$ is an irreducible smooth curve other than $D\sim-K_S$.
\begin{thm}
\label{thm:del-Pezzo-dynamic-rational-classification}
Let $S$ be a non-singular del Pezzo surface and $C\cong \bbP^1$ a non-singular curve in $S$. The following are equivalent
\begin{itemize}
	\item[(i)] $-(K_S+(1-\beta)C)$ is ample for $0<\beta \ll1$,
	\item[(ii)] $0<-K_S\cdot C \leq K_S^2 -2$.
\end{itemize}
\end{thm}
\begin{proof}
We first prove (i) implies (ii). Consider the short exact sequence
\begin{equation}
0\lra \calO_S(-K_S-2C)\lra \calO_S(-K_S-C)\lra \calO_C(-(K_S+C)\vert_C)\lra 0.
\label{eq:goodDelPezzoSES}
\end{equation}
Since $-(K_S+(1-\beta)C)$ is ample for sufficiently small $\beta$, then $-(K_S+C)$ is nef. Suppose it is also big. Then so is $-2(K_S+C)$. Hence by Kawamata-Viehweg vanishing theorem (see Theorem \ref{thm:KawamataViehweg})
$$H^1(S,\calO_S(-K_S-2C))=H^1(K_S+(-2(K_S+C)))=0.$$
The long exact sequence in cohomology of \eqref{eq:goodDelPezzoSES} is exact:
\begin{align*}
0\lra H^0(S, \calO_S(-K_S-2C))&\lra H^0(S, \calO_S(-K_S-C))\lra\\
															\lra &H^0(C, \calO_C(-(K_S+C)\vert_C)) \lra  H^1(S, \calO_S(-K_S-2C))=0.
\end{align*}
By adjunction 
$$h^0(\calO_C(-(K_S+C)\vert_C))=h^0(\calO_C(-K_C))=h^0(\calO_{\bbP^1}(2))=3$$
so we can lift a global section $s\vert_C\in H^0(C, \calO_C(-(K_S+C)\vert_C))$ to $s\in H^0(S, \calO_S(-K_S-C))$  such that its vanishing locus is $\div(s)=L\sim-K_S-C$.

By Nakai-Moishezon criterion (Theorem \ref{thm:Nakai-Moishezon-criterion}), since $-(K_S +(1-\beta)C)$ is ample, then
\begin{align*}
0&<(-K_S-(1-\beta)C)\cdot L \\
	&=-K_S\cdot(-K_S-C)-(1-\beta)C\cdot(-K_S-C)\\
	&=K^2_S+K_S\cdot C +(1-\beta)K_S\cdot C +(1-\beta)C^2\\
	&=K_S^2+(2-\beta)K_S\cdot C +(1-\beta)(-K_S\cdot C -2)
\end{align*}
by the genus formula applied to $C\cong \bbP^1$:
\begin{equation}
C^2=-K_S\cdot C + 2p_a(C)-2=-K_S\cdot C -2.
\label{eq:del-Pezzo-dynamic-rational-classification-proof-genus}
\end{equation}
Therefore
\begin{align*}
0&<(-K_S-(1-\beta)C)\cdot L \\
	&=K_S^2+(2-\beta)K_S\cdot C +(1-\beta)(-K_S\cdot C -2)\\
	&=K_S^2+(2-\beta+\beta-1)K_S\cdot C -2+2\beta\\
	&=K_S^2+K_S\cdot C-2+2\beta,
\end{align*}
Observe that $K_S^2, K_S\cdot C\in \bbZ$ since $S$ is smooth. If $0<\beta<\frac{1}{2}$, then 
$$0=\ceil*{-2\beta}\leq \ceil*{K_S^2+K_S\cdot C -2}=K^2_S+K_S\cdot C -2.$$
In particular this proves (ii) when $-2(K_S+C)$ is big.

Now suppose that $-2(K_S+C)$ is not big. Since it is nef, then, using \eqref{eq:del-Pezzo-dynamic-rational-classification-proof-genus}.
$$0=(K_S+C)^2=K_S^2+2K_S\cdot C +C^2=K_S^2+K_S\cdot C -2,$$
which upon rearrangement gives the desired result.

Now we will prove (ii) implies (i). For degrees 1 and 2 there are no rational curves satisfying (ii). Indeed,
$$K_S^2-2\leq 0<-K_S\cdot C.$$
If $K_S^2=3$, then $C$ is a line. Consider  divisors $Q\sim-K_S-C$. Since $Q\cdot (-K_S)=K_S^2-\deg C\geq 2$, by Proposition \ref{prop:rational-surfaces-sections-rational-curves}, $h^0(S, \calO_S(-K_S-C))\geq 2$. Therefore we can assume $Q$ is an effective divisor. 

Since by impossing linear conditions, we may assume that $Q\in \calQ$, a pencil of conics, we may choose $Q$ to be irreducible. Indeed, if this was not the case, all $Q\in \calQ$ would split as the union of two lines, but since $\calQ$ is a pencil of conics, there would be a infinite number of lines, contradicting Lemma \ref{lem:del-Pezzo-lines9-d}. 

Since $Q$ is an irreducible conic, by Lemma \ref{lem:del-Pezzo-nefcurves}, $Q$ is nef. Suppose there is an irreducible curve $E\subset S$ such that 
\begin{equation}
0\geq -(K_S+(1-\beta)C)\cdot E = (Q+\beta C)\cdot E \geq \beta (C\cdot E).
\label{eq:GoodDelPezzoAmpleRational-aux}
\end{equation}
If $C=E$, then $C\cdot E=-1$ and $(-K_S)\cdot C =1$, giving a contradiction:
$$0\geq -(K_S+(1-\beta)C)\cdot C=\deg C-(1-\beta)C^2=1+(1-\beta)\geq 1>0.$$
If $C\neq E$, then \eqref{eq:GoodDelPezzoAmpleRational-aux} implies $C\cdot E=0$ and $(-(K_S+(1-\beta)C)\cdot E=\deg E>0$, which is also a contradiction.

We have proven that (ii) implies (i) when $K_S^2\leq 3$. We will carry out an induction on the degree of $S$. Suppose (ii) for $K_S^2>3$. If $C$ is  a $(-1)$-curve, let $\pi\colon \widetilde S \ra S$ be the blow-up of a general point $p\not\in C$. The surface $\widetilde S$ is a del Pezzo surface. Let $\widetilde C$ be the strict transform of $C$ in $S$ and let $E$ be the exceptional divisor. Then $\widetilde C \cong \mathbb{P}^1$ and $-K_{\widetilde S}\cdot \widetilde C=1$. Hence $0<-K_{\widetilde S}\cdot \widetilde C =1\leq K_{\widetilde S}^2-2$ since $K_{\widetilde S}^2\geq 3$. By the induction hypothesis $-(K_{\widetilde S} +(1-\beta) \widetilde C)$ is ample $\forall\beta\in (0,\beta_0)$ where $0<\beta_0\ll 1$. Let $Q\subset S$ be any irreducible curve and $\widetilde Q$ be its strict transform in $\widetilde S$. Let $m=\mult_pQ\geq 0$. If $Q=C$, then $m=0$ and 
$$-(K_S+(1-\beta)C)\cdot Q = \deg C -(1-\beta)C^2=\deg C +(1-\beta)>0.$$
Suppose $Q\neq C$. Then
\begin{align*}
&-(K_S+(1-\beta)C)\cdot Q=\pi^*(-(K_S+(1-\beta)C))\cdot (\widetilde Q + mE) \\
=&-(K_{\widetilde S}-E+(1-\beta)\widetilde C) \\
=& -(K_{\widetilde S} + (1-\beta)\widetilde C)\cdot\widetilde Q + E\cdot \widetilde Q >E\cdot \widetilde Q =m\geq 0,
\end{align*}
by the inductive hypothesis. 
Hence, by Nakai-Moishezon criterion $-(K_S+(1-\beta)C)$ is ample, since $p\not\in C$ and $-K_S$ is ample.

Now suppose $C$ is not a $(-1)$-curve. Let $\pi\colon \widetilde S \ra S$ be the blow-up of a general point $p\in S$. Let $E$ be the exceptional divisor of $\pi$. Since $p$ is general and $C^2\geq 0$, then $\widetilde S$ is a del Pezzo surface. Note that since $C$ is smooth, $\widetilde C\cdot E=1$. Hence
$$\widetilde C \cdot (-K_{\widetilde S})=-C\cdot K_S-1\leq K_S^2-3=K_{\widetilde S}-2,$$
and the pair $(\widetilde S, \widetilde C)$ satisfies (ii).

By the inductive hypothesis $-(K_{\widetilde S} + (1-\beta)\widetilde C)$ is ample for $0<\beta\ll 1$. Let $Q\subset S$ be any irreducible curve and $\widetilde Q$ be its strict transform in $\widetilde S$. By Nakai-Moishezon criterion $-(K_{\widetilde S} + (1-\beta)\widetilde C)\cdot \widetilde Q>0$. Therefore
\begin{align*}
&\ -(K_S +(1-\beta)C)\cdot Q \\
&=\pi^*(-(K_S+(1-\beta)C))(\widetilde Q +mE)\\
&=-(K_{\widetilde S}-E +(1-\beta)\widetilde C +(1-\beta)E)\widetilde Q \\
&=-(K_{\widetilde S}+(1-\beta)\widetilde C)\cdot \widetilde Q + \beta \widetilde Q \cdot E>0,
\end{align*}
and by Nakai-Moishezon criterion, the divisor class $-(K_S +(1-\beta)C)$ is ample.
\end{proof}

\begin{cor}
\label{cor:del-Pezzo-dynamic-rational-classification}
Let $S$ be a non-singular del Pezzo surface and $C\cong \bbP^1$ a non-singular curve in $S$. The following are equivalent
\begin{itemize}
	\item[(i)] $-(K_S+(1-\beta)C)$ is ample $\forall \beta\in (0,1]$,
	\item[(ii)] $0<-K_S\cdot C \leq K_S^2 -2$.
\end{itemize}
\end{cor}
\begin{proof}
By Theorem \ref{thm:del-Pezzo-dynamic-rational-classification} it is enough to show that the following are equivalent:
\begin{itemize}
	\item[(i)] $-(K_S+(1-\beta)C)$ is ample $\forall \beta\in (0,1]$,
	\item[(ii)] $-(K_S+(1-\beta)C)$ is ample for $0<\beta \ll1$.
\end{itemize}
Clearly (i) implies (ii). To see that (ii) implies (i) observe that if two $\bbQ$-divisors $A,B$ are ample, then $\forall \alpha\in [0,1]$ the $\bbQ$-divisor $\alpha A + (1-\alpha) B$ is ample, since ample divisors form a cone.

Let $A=-(K_S + (1-\beta_0)C$ for $0<\beta_0\ll1$ such that $A$ is ample and let $B=-K_S$, which is ample. Given $\beta\in (0,1]$, we can choose $\alpha=\frac{1-\beta}{1-\beta_0}<1$ whenever $\beta_0\leq \beta$, which is always the case since $\beta_0$ can be chosen to be arbitrarily small. It is clear that $\alpha>0$. Then the $\bbQ$-divisor
$$\alpha A + (1-\alpha) B = -(K_S + \alpha(1-\beta_0)C)=-(K_S + (1-\beta)C)$$
is ample.
\end{proof}

\section{Projective plane}
\begin{lem}
\label{lem:del-Pezzo-deg9-dynamic-upbound}
Let $S=\bbC\bbP^2$ and $C\sim -K_S$ be a smooth cubic curve. Then
		\begin{equation}
				\alpha(S,(1-\beta)C)\leq\omega:=
						\begin{dcases}
								1 												&\text{ for } 0<\beta\leq \frac{1}{6},\\
								\frac{1+3\beta}{9\beta}		&\text{ for } \frac{1}{6}\leq \beta\leq \frac{2}{3},\\
								\frac{1}{3\beta}		&\text{ for } \frac{2}{3}\leq \beta\leq 1.
						\end{dcases}
			\label{eq:del-Pezzo-dynamic-alpha-degree-9-upbound}
		\end{equation}
\end{lem}
\begin{proof}
It is a well known fact (see \cite[Ex. 5.24]{FultonAlgebraicCurves}) that any smooth and reduced cubic curve $C$ in $S$ can be given local coordinates $(x,y)$ such that around some point $p=(0,0)\in C$ its equation is
$$y^2-x(x-1)(x-\epsilon)=0,\qquad \epsilon\neq 0,1.$$
Such a point $p$ is called a \emph{flex point}, since the local intersection of $C$ with the tangent line $L$ at $p$ is $3$. In fact, there are $9$ such points in a cubic curve (see \cite[Ex. 5.23, Corollary]{FultonAlgebraicCurves}).

Let $L$ be the line with local equation $y=0$. Then
$$(C\cdot L)\vert_p=\dim_\bbC\left( \frac{\bbC[x,y]}{\langle y^2-x(x-1)(x-\epsilon),y\rangle }\right)=\dim_\bbC(\bbC\oplus\bbC\langle x\rangle \oplus\bbC\langle x^2\rangle )=3.$$
The minimal log resolution $f\colon \widetilde S \ra S$ of the log pair $(S,D=(1-\beta)C+\lambda\beta(3L))$ consists of $3$ consecutive blow-ups. The log pullback is:
$$f^*(K_S+(1-\beta)C+\lambda\beta(3L))\simq K_{\widetilde S}+(1-\beta)\widetilde C +\lambda\beta (3 \widetilde L)+ (3\lambda\beta-\beta)E_1 + (6\lambda\beta -2\beta)E_2 + (9\lambda\beta-3\beta)E_3.$$
We conclude
\begin{align*}
\alpha(S,(1-\beta)C)&\leq \min\{\lct(S,(1-\beta)C,\beta C),\lct(S,(1-\beta)C,3\beta L)\}\\
=&\min\{1,\frac{1}{3\beta},\frac{1+\beta}{3\beta},\frac{1+2\beta}{6\beta},\frac{1+3\beta}{9\beta}\}=\omega,
\end{align*}
giving the desired result.
\end{proof}

\begin{thm}
\label{thm:del-Pezzo-dynamic-alpha-degree-9}
Let $S=\bbP^2$ and $C\sim -K_S$ be a smooth cubic curve. Then
		\begin{equation}
				\alpha(S,(1-\beta)C)=\omega:=
						\begin{dcases}
								1 												&\text{ for } 0<\beta\leq \frac{1}{6},\\
								\frac{1+3\beta}{9\beta}		&\text{ for } \frac{1}{6}\leq \beta\leq \frac{2}{3},\\
								\frac{1}{3\beta}		&\text{ for } \frac{2}{3}\leq \beta\leq 1.
						\end{dcases}
			\label{eq:del-Pezzo-dynamic-alpha-degree-9}
		\end{equation}
\end{thm}
\begin{proof}
By Lemma \ref{lem:del-Pezzo-deg9-dynamic-upbound} $\alpha(S,(1-\beta)C)\leq\omega$. If $\alpha(S,(1-\beta)C)<\omega$, then there is an effective $\bbQ$-divisor $D\simq-K_S$, such that 
\begin{equation}
(S,(1-\beta)C+\lambda\beta D)
\label{eq:del-Pezzo-dynamic-alpha-degree-9-proof-pair}
\end{equation}
is not log canonical at some point $p\in S$ for some $\lambda<\omega$. We have
\begin{equation}
\lambda\beta<\begin{dcases}
								\beta 												&\text{ for } 0<\beta\leq \frac{1}{6},\\
								\frac{1+3\beta}{9}						&\text{ for } \frac{1}{6}\leq \beta\leq \frac{2}{3},\\
								\frac{1}{3}										&\text{ for } \frac{2}{3}\leq \beta\leq 1.
						\end{dcases}\leq
						\begin{dcases}
								\frac{1}{6}										&\text{ for } 0<\beta\leq \frac{1}{6},\\
								\frac{1}{3}										&\text{ for } \frac{1}{6}\leq \beta\leq 1.
						\end{dcases}\leq \frac{1}{3}
\label{eq:del-Pezzo-dynamic-alpha-degree-9-proof-basic-inequality}
\end{equation}
Since $\glct(S)=\frac{1}{3}$, by Lemma \ref{lem:pairs-fixed-boundary-lcs} the pair \eqref{eq:del-Pezzo-dynamic-alpha-degree-9-proof-pair} is not log canonical at $p\in C$ and it is log canonical in codimension $1$.

\textbf{Case 1: Suppose $\beta\leq \frac{2}{3}$.} We have
\begin{equation}
3\lambda\beta<\frac{1}{3}+\beta.
\label{eq:del-Pezzo-dynamic-alpha-degree-9-proof-basic-inequality-2}
\end{equation}
Indeed, for $\beta\leq \frac{1}{6}$ we have
$$3\lambda\beta<3\beta=\beta+2\beta\leq \beta + \frac{1}{3}.$$
For $\frac{1}{6}\leq \beta\leq \frac{2}{3}$ we have
$$3\lambda\beta<\frac{1+3\beta}{3}=\frac{1}{3}+\beta\leq 1.$$
Take a general line $L$ through $p$. Then $L\not\subseteq\Supp(D)$ and
$$\lambda\beta\mult_pD\leq \lambda\beta L \cdot D \leq 3\lambda\beta<\frac{1}{3}+\beta\leq 1,$$
by \eqref{eq:del-Pezzo-dynamic-alpha-degree-9-proof-basic-inequality-2}. Applying Theorem \ref{thm:inequality-local-blowup-bound} with $n=3$ we obtain
\begin{equation}
9\lambda\beta= (\lambda\beta C\cdot D)>1+3\beta.
\label{eq:del-Pezzo-dynamic-alpha-degree-9-proof-basic-inequality-3}
\end{equation}
If $\beta\leq \frac{1}{6}$, then we have $\lambda<1$ and \eqref{eq:del-Pezzo-dynamic-alpha-degree-9-proof-basic-inequality-3} gives a contradiction: $1\geq 6\beta>1$.
If $\frac{2}{3}\geq \beta > \frac{1}{6}$, then $\lambda\beta<\frac{1+3\beta}{9}$ and \eqref{eq:del-Pezzo-dynamic-alpha-degree-9-proof-basic-inequality-3} gives a contradiction as well:
$$1+3\beta>9\lambda\beta>1+3\beta.$$
\textbf{Case 2: $\frac{2}{3}\leq \beta\leq 1$.} Then $\lambda<\frac{2}{3\beta}$.

Let $\pi\colon \widetilde S \ra S$ be the blow-up of $p$ with exceptional curve $E$. The pair
\begin{equation}
(\widetilde S, (1-\beta)\widetilde C+\lambda\beta\widetilde D + (\lambda\beta\mult_pD-\beta)E)
\label{eq:del-Pezzo-dynamic-alpha-degree-9-proof-log-pullback}
\end{equation}
is not log canonical at some $q\in E$. Observe that $\lambda\beta\mult_pD-\beta\leq 1$. In fact, if this was not the case then
$$\mult_pD>\frac{1+\beta}{\lambda\beta}>3(1+\beta)>3,$$
by \eqref{eq:del-Pezzo-dynamic-alpha-degree-9-proof-basic-inequality}. But then, taking a general line $L\subset S$ through $p$ we obtain
$$3=L\cdot D\geq \mult_pD>3$$
which is absurd.

We claim that $q\in \widetilde C$. If this was not the case, the pair 
$$(\widetilde S, (\lambda\beta)\widetilde D + (\lambda\beta\mult_pD-\beta)E)$$
would not be log canonical at $q\in E$, but then by Lemma \ref{lem:adjunction} (iii) we would obtain the following inequality:
$$\frac{1}{3}\mult_pD>\lambda\beta\mult_pD=E\cdot(\lambda\beta D)>1$$
where we use \eqref{eq:del-Pezzo-dynamic-alpha-degree-9-proof-basic-inequality}. Now we take a general line $L$ through $p$ to obtain a contradiction:
$$3=L\cdot D\geq \mult_p D>3.$$

Hence $q=E\cap C$. Let $L$ be the unique line through $p$ such that $\widetilde L \cap E =\{q\}$. Note $3L\sim-K_S$. By the proof of Lemma \ref{lem:del-Pezzo-deg9-dynamic-upbound} the pair $(S, (1-\beta)C + \lambda\beta 3L)$ is log canonical. Since $\mult_pD\leq 3$, the pair
$$(S, (1-\beta)C+(\lambda\beta\mult_pD) L)$$
is also log canonical, by \eqref{eq:del-Pezzo-dynamic-alpha-degree-9-proof-basic-inequality}. By lemma \ref{lem:log-pullback-preserves-lc} the pair
\begin{equation}
(\widetilde S, (1-\beta)\widetilde C + (\lambda\beta\mult_pD- \beta)E +(\lambda\beta \mult_pD )\widetilde L)
\label{eq:del-Pezzo-dynamic-alpha-degree-9-proof-line-pullback}
\end{equation}
is log canonical. We apply Lemma \ref{lem:log-convexity} to pairs \eqref{eq:del-Pezzo-dynamic-alpha-degree-9-proof-log-pullback} and \eqref{eq:del-Pezzo-dynamic-alpha-degree-9-proof-line-pullback} with $(1-\beta)\widetilde C + (3\lambda\beta - \beta)E$ as fixed boundary to obtain a pair 
\begin{equation}
(\widetilde S, (1-\beta)\widetilde C + \lambda \beta \widetilde D' + (\lambda\beta\mult_pD -\beta)E)
\label{eq:del-Pezzo-dynamic-alpha-degree-9-proof-log-pullback-convexity}
\end{equation}
which is not log canonical at $q\in E$ where $E, \widetilde L, \widetilde C\not\subseteq\Supp(\widetilde D')$ and such that $\widetilde D'\simq \widetilde D$ with $\widetilde D'$ effective. By linear equivalence we deduce that $\mult_pD =\mult_p D'$ where $D'=\pi_*(\widetilde D')$. On one hand we have
\begin{equation}
3-\mult_pD'=\widetilde D' \cdot \widetilde L \geq \mult_q \widetilde D'.
\label{eq:del-Pezzo-dynamic-alpha-degree-9-proof-basic-inequality-4}
\end{equation}
On the other hand, by Lemma \ref{lem:adjunction} (i) applied to the pair \eqref{eq:del-Pezzo-dynamic-alpha-degree-9-proof-log-pullback-convexity} we deduce
$$1-2\beta+\lambda\beta(\mult_pD'+\mult_q\widetilde D')>1$$
which implies
$$\mult_pD'+\mult_q\widetilde D'>\frac{2\beta}{\lambda\beta}>\frac{2\beta}{3}\geq 4$$
where we first use \eqref{eq:del-Pezzo-dynamic-alpha-degree-9-proof-basic-inequality} and then $\beta\geq \frac{2}{3}$. This gives a contradiction by the means of \eqref{eq:del-Pezzo-dynamic-alpha-degree-9-proof-basic-inequality-4}, finishing the proof.
\end{proof}

\section{Smooth quadric surface $\bbP^1\times\bbP^1$}
\begin{lem}
\label{lem:del-Pezzo-dynamic-alpha-degree-8-F0-upbound}
Let $S=\bbP^1\times\bbP^1$ and $C\in\vert-K_S\vert$ a smooth elliptic curve. Then
		\begin{equation}
				\alpha(S,(1-\beta)C)\leq
						\begin{dcases}
								1 												&\text{ for } 0<\beta\leq \frac{1}{4},\\
								\frac{1+2\beta}{6\beta}		&\text{ for } \frac{1}{4}\leq \beta\leq 1.
						\end{dcases}
			\label{eq:del-Pezzo-dynamic-alpha-degree-8-F0-upbound}
		\end{equation}
\end{lem}
\begin{proof}
$S$ has two rulings $S\ra \bbP^1$. Let $F_1$ be a fibre of any ruling $S\ra \bbP^1$ such that $(F_1\cdot C)\vert_p=2$ at a point $p$. Let $F_2$ be the fibre through $p$ of the other ruling. Observe that $(F_2\cdot C)\vert_p=1$, since $F_2\cdot F_1=1$. Notice that $2(F_1+F_2)\sim-K_S$. the minimal log resolution $f\colon \widetilde S \ra S$ of $(S,(1-\beta)C+\lambda\beta(2F_1+2F_2)$ is obtained after two consecutive blow-ups. We compute the log pullback:
\begin{align*}
f^*(K_S+(1-\beta)C+\lambda\beta(2F_1+2F_2))\simq K_{\widetilde S}&+ (1-\beta)\widetilde C +\lambda\beta(2\widetilde F_1+2\widetilde F_2)\\
&+(4\lambda\beta-\beta)E_1+(6\lambda\beta-2\beta)E_2.
\end{align*}
Therefore
\begin{align*}
\alpha(S,(1-\beta)C&\leq\min\{\lct(S,(1-\beta)C,\beta C),\lct(S,(1-\beta)C,\beta (2F_1+2F_2)\}\\
&=\min\{1,\frac{1+\beta}{4\beta},\frac{1+2\beta}{6\beta}\}=\min\{1,\frac{1+2\beta}{6\beta}\},
\end{align*}
finishing the proof.
\end{proof}

\begin{thm}
\label{thm:del-Pezzo-dynamic-alpha-degree-8-F0}
Let $S=\bbP^1\times\bbP^1$ and $C\in\vert-K_S\vert$ a smooth elliptic curve. Then
		\begin{equation}
				\alpha(S,(1-\beta)C)=\omega=
						\begin{dcases}
								1 												&\text{ for } 0<\beta\leq \frac{1}{4},\\
								\frac{1+2\beta}{6\beta}		&\text{ for } \frac{1}{4}\leq \beta\leq1.
						\end{dcases}
			\label{eq:del-Pezzo-dynamic-alpha-degree-8-F0}
		\end{equation}
\end{thm}
\begin{proof}
By Lemma \ref{lem:del-Pezzo-dynamic-alpha-degree-8-F0-upbound} $\alpha(S, (1-\beta)C) \leq \omega$. If $\alpha(S,(1-\beta)C)<\omega$, then there is an effective $\bbQ$-divisor $D\simq-K_S$ such that
\begin{equation}
(S,(1-\beta)C+\lambda\beta D)
\label{eq:del-Pezzo-dynamic-alpha-degree-8-F0-proof-pair}
\end{equation}
is not log canonical for some $0<\lambda<\omega$. First observe that
\begin{equation}
\lambda\beta<\omega\beta = \begin{dcases}
								\beta 								&\text{ for } 0<\beta\leq \frac{1}{4},\\
								\frac{1+2\beta}{6}		&\text{ for } \frac{1}{4}\leq \beta\leq1
								\end{dcases}\leq\begin{dcases}
								\frac{1}{4}		&\text{ for } 0<\beta\leq \frac{1}{4},\\
								\frac{1}{2}		&\text{ for } \frac{1}{4}\leq \beta\leq1
								\end{dcases}\leq \frac{1}{2}.
\label{eq:del-Pezzo-dynamic-alpha-degree-8-F0-proof-inequality0}
\end{equation}
Since $\glct(S)=\frac{1}{2}$, by Lemma \ref{lem:pairs-fixed-boundary-lcs}, the pair \eqref{eq:del-Pezzo-dynamic-alpha-degree-8-F0-proof-pair} is not log canonical at an isolated $p\in C$ and is log canonical in codimension $1$. Let $F_1, F_2$ be the fibres through $p$ of each of the rulings $S\ra \bbP^1$. By the proof of Lemma \ref{lem:del-Pezzo-dynamic-alpha-degree-8-F0-upbound}, the pair
$$(S,(1-\beta)C + 2\lambda\beta(F_1+F_2))$$
is log canonical at $p$. By Lemma \ref{lem:log-convexity} applied to this pair and \eqref{eq:del-Pezzo-dynamic-alpha-degree-8-F0-proof-pair}, we may assume, without loss of generality, that $F_1\not\subseteq\Supp(D)$. This implies that 
\begin{equation}
\mult_pD\leq D\cdot F_1=(-K_S)\cdot F_1=2.
\label{eq:del-Pezzo-dynamic-alpha-degree-8-F0-proof-inequality1}
\end{equation}
As a consequence
\begin{equation}
\lambda\beta\mult_pD<\frac{1}{4}+\beta.
\label{eq:del-Pezzo-dynamic-alpha-degree-8-F0-proof-inequality2}
\end{equation}
Indeed if $0<\beta\leq \frac{1}{4}$, we have
$$\lambda\beta\mult_pD<\beta\mult_pD\leq 2\beta\leq \beta + \frac{1}{4}.$$
On the contrary, if $\frac{1}{4}\leq\beta\leq 1$, using \eqref{eq:del-Pezzo-dynamic-alpha-degree-8-F0-proof-inequality1} we obtain
$$\lambda\beta\mult_pD<\frac{1+2\beta}{6}\cdot 2=\frac{2}{3}\beta+ \frac{1}{3}+ \frac{1}{4}- \frac{1}{4}= \frac{1}{4} + \frac{2}{3}\beta + \frac{1}{3}\cdot \frac{1}{4}\leq \beta + \frac{1}{4},$$
and \eqref{eq:del-Pezzo-dynamic-alpha-degree-8-F0-proof-inequality2} is proven for all $0<\beta\leq 1$. Notice that using \eqref{eq:del-Pezzo-dynamic-alpha-degree-8-F0-proof-inequality0} and \eqref{eq:del-Pezzo-dynamic-alpha-degree-8-F0-proof-inequality1} we obtain
$$\lambda\beta\mult_pD<1.$$
Together with \eqref{eq:del-Pezzo-dynamic-alpha-degree-8-F0-proof-inequality2} we are in the hypotheses of Theorem \ref{thm:inequality-local-blowup-bound} for $n=4$, i.e.
$$\lambda\beta\mult_pD\leq \min \{1,\beta+ \frac{1}{4}\}$$
and we conclude
\begin{equation}
8\lambda\beta=(\lambda\beta D\cdot C)\vert_p>1+4\beta.
\label{eq:del-Pezzo-dynamic-alpha-degree-8-F0-proof-inequality3}
\end{equation}
If $0<\beta\leq \frac{1}{4}$, then $\lambda<1$ and \eqref{eq:del-Pezzo-dynamic-alpha-degree-8-F0-proof-inequality3} becomes
$$8\beta>1+4\beta$$
implying $\beta>\frac{1}{4}$, a contradiction. On the other hand, if $\frac{1}{4}\leq \beta\leq 1$, then \eqref{eq:del-Pezzo-dynamic-alpha-degree-8-F0-proof-inequality3} implies
$$8\cdot \frac{1+2\beta}{6}>8\lambda\beta>1+4\beta,$$
giving as a result $\beta<\frac{1}{4}$, which is also impossible, finishing the proof.
\end{proof}

\section{Hirzebruch surface $\bbF_1$}
\begin{lem}
\label{lem:del-Pezzo-dynamic-alpha-degree-8-F1-upbound}
Let $S=\bbF_1$ and $C\in\vert-K_S\vert$ smooth. Let $E$ be the unique curve in $S$ such that $E^2=-1$. Let $\gamma \colon S\ra\bbP^1$ be the ruling given by the linear system $\vert \pi^*(L)-E\vert$ where $\pi:S\ra \bbP^2$ is the contraction of $E$. Let $F$ be the unique fibre of $\gamma$ passing through $r=E\cap C$. If $F\cap C=r$ (i.e. $(F\cdot C)\vert_p=2$), then
	\begin{equation}
			\alpha(S,(1-\beta)C)\leq\omega_1:=
					\begin{dcases}
							1 												&\text{ for } 0<\beta\leq \frac{1}{6},\\
							\frac{1+2\beta}{8\beta}		&\text{ for } \frac{1}{6}\leq \beta\leq \frac{5}{6},\\
							\frac{1}{3\beta}					&\text{ for } \frac{5}{6}\leq \beta\leq 1.\\
					\end{dcases}
		\label{eq:del-Pezzo-dynamic-alpha-degree-8-F1-upbound-special}
	\end{equation}
Otherwise
	\begin{equation}
			\alpha(S,(1-\beta)C)\leq\omega_2:=
					\begin{dcases}
							1 												&\text{ for } 0<\beta\leq \frac{1}{4},\\
							\frac{1+\beta}{5\beta}		&\text{ for } \frac{1}{4}\leq \beta\leq \frac{2}{3},\\
							\frac{1}{3\beta}					&\text{ for } \frac{2}{3}\leq \beta\leq 1.\\
					\end{dcases}
		\label{eq:del-Pezzo-dynamic-alpha-degree-8-F1-upbound-generic}
	\end{equation}
\end{lem}
\begin{proof}
Recall that $C\sim-K_S\sim3F+2E$ and $F\cdot E =1$, $F^2=0$. Let $D=3F+2E_1$. Let $f\colon \widetilde S \ra S$ be the minimal log resolution of $$(S,(1-\beta)C+\lambda\beta(3F+2E)).$$ Since $F\cdot C=2$, if $F\cdot C=\{r\}$, then the minimal log resolution consists of two consecutive blow-ups and the log pullback is
\begin{align*}
f^*(K_S+(1-\beta)C+&\lambda\beta(3F+2E))\simq K_{\widetilde S}+(1-\beta)\widetilde C +\lambda \beta(3\widetilde F +2\widetilde E)\\
&+(5\lambda\beta-\beta)F_1+(8\lambda\beta-2\beta)F_2,
\end{align*}
where $F_1,F_2$ are exceptional divisors of $f$. Therefore
\begin{align*}
\alpha(S,(1-\beta)C)&\leq\min\left\{\lct\left(S,\left(1-\beta\right)C,\beta C\right),\lct\left(S,\left(1-\beta\right)C,\beta\left(3F+2E\right)\right)\right\}\\
&=\min\left\{1,\frac{1}{2\beta},\frac{1}{3\beta},\frac{1+\beta}{5\beta},\frac{1+2\beta}{8\beta}\right\}=\omega_1
\end{align*}
giving the desired result.

Conversely, if $(F\cdot C)\vert_r=1$, then the minimal log resolution consists of one blow-up. In this case, the log pullback is
$$f^*(K_S+(1-\beta)C+\lambda\beta(3F+2E))\simq K_{\widetilde S}+(1-\beta)\widetilde C +\lambda \beta(3\widetilde F +2\widetilde E)+(5\lambda\beta-\beta)F_1.$$
Hence
\begin{align*}
\alpha(S,(1-\beta)C)&\leq\min\left\{\lct\left(S,\left(1-\beta\right)C,\beta C\right),\lct\left(S,\left(1-\beta\right)C,\beta\left(3F+2E\right)\right)\right\}\\
&=\min\left\{1,\frac{1}{2\beta},\frac{1}{3\beta},\frac{1+\beta}{5\beta},\right\}=\omega_2,
\end{align*}
finishing the proof.
\end{proof}

We need the following auxiliary result:
\begin{lem}
\label{lem:del-Pezzo-dynamic-alpha-degree-8-F1-inequality}
The following inequalities
\begin{itemize}
	\item[(i)] $2\omega_2\beta\leq\frac{1}{4}+\beta$
	\item[(ii)] $1+4\beta\geq 8\omega_2\beta$
\end{itemize}
hold for $\omega_1, \omega_2$ as in the statement of Lemma \ref{lem:del-Pezzo-dynamic-alpha-degree-8-F1-upbound}.
\end{lem}
\begin{proof}
We prove (i):
If $0<\beta\leq \frac{1}{4}$, then $\omega_2=1$ and
$$2\omega_2\beta\leq 2\beta=\beta+\beta\leq \frac{1}{4} + \beta.$$
If $\frac{1}{4}\leq \beta\leq \frac{2}{3}$, then $\omega_2= \frac{1+\beta}{\beta}$ and
$$2\omega_2\beta=\frac{1}{5}(2+2\beta)=\beta- \frac{3}{5}\beta+ \frac{2}{5}\leq \frac{1}{4}+\beta.$$
If $\frac{2}{3}\leq \beta\leq 1$, then $\omega_2= \frac{1}{3\beta}$ and
$$2\omega_2\beta = \frac{2}{3}\leq \beta <\frac{1}{4}+\beta.$$

We prove (ii):
If $0<\beta\leq \frac{1}{4}$, then $\omega_2=1$ and
$$8\omega_2\beta=8\beta=4\beta+4\beta\leq 1+4\beta.$$
If $\frac{1}{4}\leq \beta\leq \frac{2}{3}$, then $\omega_2=\frac{1+\beta}{5\beta}$ and 
$$8\omega_2\beta = 8 \frac{1+\beta}{5}= 1+4\beta + \frac{3}{5} - \frac{12}{5}\beta\leq 1+4\beta - \frac{12}{5}\cdot\frac{2}{3} + \frac{9}{15}< 1+4\beta.$$
If $\frac{2}{3}\leq \beta \leq 1$, then $\omega_2=\frac{1}{3\beta}$ and 
$$8\omega_2\beta=\frac{8}{3} = 1+ \frac{5}{3} \leq 1+ \frac{5}{2}\beta < 1 + 4\beta.$$
\end{proof}

\begin{thm}
\label{thm:del-Pezzo-dynamic-alpha-degree-8-F1}
Let $S=\bbF_1$ and $C\in\vert-K_S\vert$ smooth. Let $E$ be the unique curve in $S$ such that $E^2=-1$. Let $\gamma\colon S\ra\bbP^1$ be the ruling given by the linear system $\vert \pioplane{1}-E\vert$ for $\pi:S\ra \bbP^2$, the contraction of $E$. Let $F$ be the unique fibre of $\gamma$ passing through $r=E\cap C$. If $F\cap C=r$ (i.e. $(F\cdot C)\vert_r=2$), then
	\begin{equation}
			\alpha(S,(1-\beta)C)=\omega_1:=
					\begin{dcases}
							1 												&\text{ for } 0<\beta\leq \frac{1}{6},\\
							\frac{1+2\beta}{8\beta}		&\text{ for } \frac{1}{6}\leq \beta\leq \frac{5}{6},\\
							\frac{1}{3\beta}					&\text{ for } \frac{5}{6}\leq \beta\leq 1.\\
					\end{dcases}
		\label{eq:del-Pezzo-dynamic-alpha-degree-8-F1-special}
	\end{equation}
Otherwise
	\begin{equation}
			\alpha(S,(1-\beta)C)=\omega_2:=
					\begin{dcases}
							1 												&\text{ for } 0<\beta\leq \frac{1}{4},\\
							\frac{1+\beta}{5\beta}		&\text{ for } \frac{1}{4}\leq \beta\leq \frac{2}{3},\\
							\frac{1}{3\beta}					&\text{ for } \frac{2}{3}\leq \beta\leq 1.\\
					\end{dcases}
		\label{eq:del-Pezzo-dynamic-alpha-degree-8-F1-generic}
	\end{equation}
\end{thm}
\begin{obs}
\label{obs:del-Pezzo-dynamic-alpha-degree-8-F1}
The first case of Theorem \ref{thm:del-Pezzo-dynamic-alpha-degree-8-F1} is special, since not always we have $F\cap C = F\cap E$ for any fibre $F$, whereas the second case is clearly the general one. Since the del Pezzo surfaces of degree $8$ do not have moduli by Lemma \ref{lem:del-Pezzo-no-moduli}, which case we are in depends only on the choice of curve $C\sim-K_{\bbF_1}$.
\end{obs}

\begin{proof}[Proof of Theorem \ref{thm:del-Pezzo-dynamic-alpha-degree-8-F1}]
By Lemma \ref{lem:del-Pezzo-dynamic-alpha-degree-8-F1-upbound} in each case we have $\alpha(S,(1-\beta)C)\leq \omega_i$. Suppose that $\alpha(S,(1-\beta)C)<\omega_i$. Then there is an effective $\bbQ$-divisor $D\simq-K_S$ such that the pair
\begin{equation}
(S,(1-\beta)C+\lambda\beta D)
\label{eq:del-Pezzo-dynamic-alpha-degree-8-F1-proof-pair}
\end{equation}
is not log canonical at some $p\in S$ where $0<\lambda<\omega_i$.

\textbf{Step 1: We reduce to the case $p\in C$.}

First observe that $\omega_1\leq \omega_2$ and 
\begin{equation}
\omega_1\beta\leq\omega_2\beta = \begin{dcases}
								\beta 								&\text{ for } 0<\beta\leq \frac{1}{4},\\
								\frac{1+\beta}{5}			&\text{ for } \frac{1}{4}\leq \beta\leq \frac{2}{3}\\
								\frac{1}{3}						&\text{ for } \frac{2}{3}\leq \beta\leq1
								\end{dcases}\leq\begin{dcases}
								\frac{1}{4}						&\text{ for } 0<\beta\leq \frac{1}{4},\\
								\frac{1}{3}						&\text{ for } \frac{1}{4}\leq \beta \leq\frac{2}{3}\\
								\frac{1}{3}						&\text{ for } \frac{2}{3}\leq \beta \leq 1\\
								\end{dcases}\qquad\leq \frac{1}{3}.
\label{eq:del-Pezzo-dynamic-alpha-degree-8-F1-proof-inequality0-generic}
\end{equation}
Hence, in each case $\lambda\beta\leq \glct(S)$ by Theorem \ref{thm:del-Pezzo-glct-charp}. Applying Lemma \ref{lem:pairs-fixed-boundary-lcs} we conclude that the pair \eqref{eq:del-Pezzo-dynamic-alpha-degree-8-F1-proof-pair} is log canonical in codimension $1$ and not log canonical at an isolated $p\in C$. By Theorem \ref{thm:del-Pezzo-classification-models} and Lemma \ref{lem:del-Pezzo-lines9-d} there is a unique $(-1)$-curve $E$ and a unique model $\pi\colon S\ra \bbP^2$ which contracts $E$ to a point $o$.

\textbf{Step 2:  We show that if $E\not\subset \Supp(D)$ or $F\not\subset \Supp(D)$, then $(S, (1-\beta)C + \omega_2\beta D)$ is log canonical at $r=E\cap C\cap F$.}

Suppose $(S, (1-\beta)C + \omega_2\beta D)$  is not log canonical at $r$.

If $E\not\subset \Supp(D)$ then, by Lemma \ref{lem:adjunction} (i),
\begin{equation*}
1=E\cdot D \geq \mult_rD>\frac{1-(1-\beta)}{\omega_2\beta}=\frac{1}{\omega_2}\geq 1,
\end{equation*}
which is a contradiction.

If $F\not\subset\Supp(D)$, then 
\begin{equation}
1\geq\frac{2}{3}\geq 2\omega_2\beta=F\cdot(\omega_2\beta D)\geq\omega_2\beta\mult_rD.
\label{eq:del-Pezzo-dynamic-alpha-degree-8-F1-proof-bound-mult3}
\end{equation}
By Lemma \ref{lem:del-Pezzo-dynamic-alpha-degree-8-F1-inequality} (i)
$$2\omega_2\beta\leq\frac{1}{4}+\beta.$$
Hence, by \eqref{eq:del-Pezzo-dynamic-alpha-degree-8-F1-proof-bound-mult3}
$$\omega_2\beta\mult_r D \leq \min\{1,\frac{1}{4}+\beta\}.$$
We apply Theorem \ref{thm:inequality-local-blowup-bound} with $n=4$ to conclude
$$8\omega_2\beta= (C\cdot (\omega_2\beta D))\vert_r > 1+4\beta.$$
This contradicts Lemma \ref{lem:del-Pezzo-dynamic-alpha-degree-8-F1-inequality} (ii).

\textbf{Step 3: If $(F\cdot C)\vert_r=1$, where $r=E\cap C\cap F$, then $(S, (1-\beta)C + \omega_2\beta D)$ is log canonical at $r$. }

Suppose $(S, (1-\beta)C + \omega_2\beta D)$ is not log canonical.

Observe that if the pair \eqref{eq:del-Pezzo-dynamic-alpha-degree-8-F1-proof-pair} is not log canonical at $r$ we have
\begin{equation}
\mult_pD>\frac{\beta}{\omega_2\beta}=\frac{1}{\omega_2}>1,
\label{eq:del-Pezzo-dynamic-alpha-degree-8-F1-proof-3-bound-mult}
\end{equation}
by Lemma \ref{lem:adjunction} (i), since $r\in C$.
We have
$$3F+2E\sim-K_S.$$
The pair
$$(S,(1-\beta)C + \omega_2\beta(3F+2E))$$
is log canonical by the proof of Lemma \ref{lem:del-Pezzo-dynamic-alpha-degree-8-F1-upbound}. By Lemma \ref{lem:log-convexity}, we may assume that either $E\not\subseteq\Supp(D)$ or $F\not\subseteq\Supp(D)$. These are precisely the hypotheses of Step 2. Step 3 follows.

\textbf{Step 4: We reduce to steps 2 and 3.}

Suppose $p\in C$ and $p\not\in E$. In particular $p\neq r$. Suppose
$$(S, (1-\beta)C + \lambda\beta D)$$
is not log canonical at $p$. Let $L_p$ be the only fibre of $\gamma\colon S\ra\bbP^1$ passing through $p$. Recall that $L_p\sim\pioplane{1}-E$. Moreover $L_p$ is irreducible. Let $H$ be the unique element in $\vert\pioplane{1}\vert$ such that $p\in H$ and $(H\cdot C)\vert_p\geq 2$. The curve $H$ is the strict transform via $\pi$ of the unique line in $\bbP^2$ tangent to $\pi(C)$ at $\pi(p)$. Observe that $\deg H=3$.

\textbf{Case (a): Suppose that $H$ is reducible.} Then it splits as the union of a line and a conic. However, the only line in $S$ is $E$, so $H=E+L_p$ with $(C\cdot L_p)\vert_p=2$. 

Let $L$ be the strict transform of a general line $\hat{L}\subset \bbP^2$ through $\pi(p)\in \bbP^2$. Then $L$ is a cubic and $L\not\subseteq \Supp(D)$, since the sublinear system of $\vert\hat L \vert$ fixing $\pi(p)$ is a pencil. Then
\begin{equation}
1>3\lambda\beta=\lambda\beta D \cdot L \geq \lambda\beta \mult_pD
\label{eq:del-Pezzo-dynamic-alpha-degree-8-F1-proof-bound-a-mult1}
\end{equation}
by \eqref{eq:del-Pezzo-dynamic-alpha-degree-8-F1-proof-inequality0-generic}. Let $\sigma\colon\widetilde S \ra S$ be the blow-up of $p$ with exceptional divisor $F_1$. Since $p\not\in E$, the surface $\widetilde S $ is a del Pezzo surface. By Lemma \ref{lem:log-pullback-preserves-lc} the pair
$$(\widetilde S, (1-\beta)\widetilde C + \lambda\beta \widetilde D + (\lambda\beta\mult_p D -\beta) F_1)$$
is not log canonical at some $q\in F_1$. By \eqref{eq:del-Pezzo-dynamic-alpha-degree-8-F1-proof-bound-a-mult1} we have that such $q$ is isolated, i.e. the discrepancy of the pair \eqref{eq:del-Pezzo-dynamic-alpha-degree-8-F1-proof-pair} along $F_1$ is smaller or equal than $1$. Note that
\begin{align}
			&\widetilde D + (\mult_p D -1)F_1 \nonumber \\
\simq & \sigma^*(D)-(\mult_pD)F_1 + (\mult_pD-1)F_1 \label{eq:del-Pezzo-dynamic-alpha-degree-8-F1-proof-pair-correction} \\
\simq &\sigma^*(-K_S)-F_1\simq-K_{\widetilde S}. \nonumber
\end{align}
Let $\bar\sigma\colon \widetilde S \ra \bar S $ be the contraction of $\widetilde E$. Observe that since $\widetilde E^2=\sigma^*(E)^2=E^2=-1$, then $\bar S$ is a non-singular del Pezzo surface by Lemma \ref{lem:del-pezzo-blowdown-dP}. By Lemma \ref{lem:del-Pezzo-no-moduli} we have $S\cong \bar S$ and since $\widetilde E\cap F_1=\emptyset$, the morphism $\bar\sigma$ is an isomorphism around $q$. Notice that since $\lambda<1$, the $\bbQ$-divisor $\beta(1-\lambda)F_1$ is effective. Therefore by Lemma \ref{lem:disc-monotonous}, the pair
\begin{equation}
(\bar S, (1-\beta)\bar{C} + \lambda \beta (\bar D + (\mult_pD-1)\bar F_1))
\label{eq:del-Pezzo-dynamic-alpha-degree-8-F1-proof-pair-correction3}
\end{equation}
is not log canonical at $\bar q = \bar\sigma(q)\in \bar F_1$ where $\bar C = \bar\sigma_*(\widetilde C)$, $\bar D = \bar\sigma_*(\widetilde D)$ and $\bar F_1 = \bar\sigma_*(F_1)$ and $\bar L_p=\bar\sigma_*(\widetilde L_p)$. By \eqref{eq:del-Pezzo-dynamic-alpha-degree-8-F1-proof-pair-correction}
$$\bar D + (\mult_pD-1)\bar F_1 \simq\sigma_*(-K_{\widetilde S })\sim-K_{\bar S}.$$
Therefore \eqref{eq:del-Pezzo-dynamic-alpha-degree-8-F1-proof-pair-correction3} is log canonical around $\bar q=\bar F_1\cap \bar C$ by Step 1. Hence $\bar q=\bar F_1 \cap \bar C$ where $\bar F_1$ is the only $(-1)$-line in $\bar S \cong \bbF_1$.

Now, observe that since $(L_p\cdot C)\vert_p=2$, then $q=\widetilde L_p\cap \widetilde C \cup F_1$ and $\widetilde L_p\cdot \widetilde C=1$. Therefore, since $\bar\sigma$ is an isomorphism near $q$, $(\bar L_p\cdot \bar C)\vert_{\bar q}=1$. Moreover $(\bar F_1)^2=-1$. Therefore we may substitute $\bar S$ for $S$, $\bar C$ for $C$, $\bar L_p$ for $F$, $\bar q$ for $r$ and $\bar F_1$ for $E$ and $\bar D + (\mult_pD-1)\bar F_1$ for $D$. We are in the situation of Step 3, giving a contradiction, since $\lambda<\omega_i\leq \omega_2$.

\textbf{Case (b): Suppose that the curve $H\sim\pioplane{1}$ is irreducible.} Recall that $3\geq (H\cdot C)\vert_p\geq 2$, $p\in H$, $p\not\in E$, and $L_p\sim\pioplane{1}-E$ is the only fibre of $\gamma \colon\bbF_1\ra \bbP^1$ passing through $p$.

We claim that $(L_p\cdot C)\vert_p=1$. Observe that $\hat C =\pi_*(C)$ is a smooth cubic in $\bbP^2$ and $\hat L_p=\pi_*(L_p)$ and $\hat H =\pi_*(H)$ are smooth lines passing through $\hat p=\pi(p)\neq \pi(E)$. Moreover $\hat H$ is tangent to to $\hat C$ at $\hat p$. Since there is only one tangent line at $C$ at $\pi(p)$, if $(L_p\cdot C)\vert_p=2$, then $\hat L_p$ is tangent to $\hat C$ at $\hat p$. But then $\hat H=\hat L_p$ and $H=L_p+E_1$, a contradiction with our assumption. Therefore $(L_p\cdot C)\vert_p=1$.

Observe that $2H+L_p\sim-K_S$. The pair
\begin{equation*}
(S, (1-\beta)C +\omega_2\beta(2H+L_p))
\end{equation*}
is log canonical at $p$ by Lemma \ref{lem:appendix-F1-pair1}. By Lemma \ref{lem:log-convexity} we may assume that either $H\not\subset\Supp(D)$ or $L_p\not\subset\Supp(D)$.

Suppose $L_p\not\subset\Supp(D)$. Then
\begin{equation}
1\geq\frac{2}{3}\geq 2\omega_2\beta=L_p\cdot(\omega_2\beta D)\geq\omega_2\beta\mult_rD.
\label{eq:del-Pezzo-dynamic-alpha-degree-8-F1-proof-lp-bound-mult3}
\end{equation}
By Lemma \ref{lem:del-Pezzo-dynamic-alpha-degree-8-F1-inequality} (i)
$$2\omega_2\beta\leq\frac{1}{4}+\beta.$$
Hence, by \eqref{eq:del-Pezzo-dynamic-alpha-degree-8-F1-proof-lp-bound-mult3}
$$\omega_2\beta\mult_r D \leq \min\{1,\frac{1}{4}+\beta\}.$$
We apply Theorem \ref{thm:inequality-local-blowup-bound} with $n=4$ to conclude
$$8\omega_2\beta= (C\cdot (\omega_2\beta D))\vert_r > 1+4\beta.$$
This contradicts Lemma \ref{lem:del-Pezzo-dynamic-alpha-degree-8-F1-inequality} (ii).

Hence we conclude that $H\not\subseteq \Supp(D)$. Then
\begin{equation}
1>3\lambda\beta=\lambda\beta D \cdot H \geq \lambda\beta \mult_pD
\label{eq:del-Pezzo-dynamic-alpha-degree-8-F1-proof-H-bound-mult1}
\end{equation}
by \eqref{eq:del-Pezzo-dynamic-alpha-degree-8-F1-proof-inequality0-generic}. Let $\sigma\colon\widetilde S \ra S$ be the blow-up of $p$ with exceptional divisor $F_1$. Since $p\not\in E$, the surface $\widetilde S $ is a del Pezzo surface. By Lemma \ref{lem:log-pullback-preserves-lc} the pair
$$(\widetilde S, (1-\beta)\widetilde C + \lambda\beta \widetilde D + (\lambda\beta\mult_p D -\beta) F_1)$$
is not log canonical at some $q\in F_1$. By \eqref{eq:del-Pezzo-dynamic-alpha-degree-8-F1-proof-3-bound-mult} we have that such $q$ is isolated, i.e. the discrepancy of the pair \eqref{eq:del-Pezzo-dynamic-alpha-degree-8-F1-proof-pair} along $F_1$ is smaller or equal than $1$. Note that
\begin{align}
			&\widetilde D + (\mult_p D -1)F_1 \nonumber \\
\simq & \sigma^*(D)-(\mult_pD)F_1 + (\mult_pD-1)F_1 \label{eq:del-Pezzo-dynamic-alpha-degree-8-F1-proof-pair-correction2} \\
\simq &\sigma^*(-K_S)-F_1\simq-K_{\widetilde S}. \nonumber
\end{align}
Let $\bar\sigma\colon \widetilde S \ra \bar S $ be the contraction of $\widetilde E$. Observe that since $\widetilde E^2=\sigma^*(E)^2=E^2=-1$, then $\bar S$ is a non-singular del Pezzo surface by Lemma \ref{lem:del-pezzo-blowdown-dP}. By Lemma \ref{lem:del-Pezzo-no-moduli} we have $S\cong \bar S$ and since $\widetilde E\cap F_1=\emptyset$, the morphism $\bar\sigma$ is an isomorphism around $q$. Notice that since $\lambda<1$, the $\bbQ$-divisor $\beta(1-\lambda)F_1$ is effective. Therefore by Lemma \ref{lem:disc-monotonous}, the pair
$$(\bar S, (1-\beta)\bar{C} + \lambda \beta (\bar D + (\mult_pD-1)\bar F_1))$$
is not log canonical at $\bar q = \bar\sigma(q)\in \bar F_1$ where $\bar C = \bar\sigma_*(\widetilde C)$, $\bar D = \bar\sigma_*(\widetilde D)$ and $\bar F_1 = \bar\sigma_*(F_1)$ and $\bar H=\bar\sigma_*(\widetilde H)$. By \eqref{eq:del-Pezzo-dynamic-alpha-degree-8-F1-proof-pair-correction2}
$$\bar D + (\mult_pD-1)\bar F_1 \simq\sigma_*(-K_{\widetilde S })\sim-K_{\bar S}.$$
Hence the pair \eqref{eq:del-Pezzo-dynamic-alpha-degree-8-F1-proof-pair-correction3} is log canonical in codimension $1$ and $\bar q\in \bar C $, by step 1. Hence $\bar q=\bar F_1 \cap \bar C$ where $\bar F_1$ is the only $(-1)$-line in $\bar S \cong \bbF_1$.

Since $q=\widetilde H\cap \widetilde C \cup F_1$ and $\bar\sigma$ is an isomorphism near $q$, we have
$$(\bar H\cdot \bar C)\vert_{\bar q}=(\widetilde H \cdot \widetilde C)\vert_q = (H\cdot C)\vert_p-1\geq 1.$$
Moreover $(\bar F_1)^2=-1$. 

Since $E\cdot H =0$, then $(\bar H)^2=(\widetilde H)^2=H^2-1=0$. Therefore we may substitute $\bar S$ for $S$, $\bar C$ for $C$, $\bar H$ for $F$, $\bar q$ for $r$ and $\bar F_1$ for $E$ and $\bar D + (\mult_pD-1)\bar F_1$ for $D$ we are in the situation of Step 2, where $F\not\subset\Supp(D)$, giving a contradiction, since $\lambda<\omega_i\leq \omega_2$. This finishes the proof.

\end{proof}

\section{Del Pezzo surface of degree $7$}
\begin{nota}
\label{nota:del-Pezzo-deg7}
Let $S$ be a non-singular del Pezzo surface of degree $7$. By Lemma \ref{lem:del-Pezzo-no-moduli}, $S$ is unique up to isomorphism. By Lemma \ref{lem:del-Pezzo-uniquemodel} there is a unique morphism $\pi:S\ra \bbP^2$, up to isomorphism in $\bbP^2$, that contracts two $(-1)$-curves $E_1,E_2$ to points $p_1,p_2$ in $\bbP^2$. By Lemma \ref{lem:del-Pezzo-lines9-d} there is a unique line $L\neq E_1,E_2$ with
$$L\sim\pioplane{1}-E_1-E_2,$$
corresponding to the strict transform of the unique line in $\bbP^2$ passing through $p_1$ and $p_2$. We have $L\cdot E_1=L\cdot E_2=1$. Let $C$ be a smooth curve, $C\in\vert-K_S\vert$. The curve $C$ intersects each of $E_1,E_2$ and $L$ at precisely one point. At most two of these points coincide.

Let $L_i$ be the unique curve
$$L_i\sim\pioplane{1}-E_i,$$
containing $r_i=E_i\cap C$, for $i=1,2$, which is precisely the strict transform of the unique line $\bar{L_i}$ in $\bbP^2$ tangent to $\bar{C}=\pi_*(C)$ at $p_i$. If $L\cap C \neq r_1,r_2$, then let $r=L\cap C$ and $R$ be the unique curve passing through $r$ such that 
$$R\sim \pioplane{1}$$
and $R$ is tangent to $C$ at $r$, i.e. $(C\cdot R)\vert_r\geq 2$.

Alternatively, we can realise $R$ as the strict transform of the unique line in $\bbP^2$ tangent to $\pi(C)$ at $\pi(r)$. since $\pi$ is an isomorphism around $r$, then
$$(\pi(L)\cdot\pi(C))\vert_r=(C\cdot L)\vert_r=C\cdot L=1.$$
In particular, $R\neq L$.

The different intersections of these curves will give different values for $\alpha(S, (1-\beta)C)$. These are described precisely in Lemma \ref{lem:del-Pezzo-dynamic-alpha-degree-7} and Theorem \ref{thm:del-Pezzo-dynamic-alpha-degree-7}, but we refer to the reader to figures in Table \ref{tab:deg7} for a graphical interpretation.
\end{nota}
\begin{table}[!t]%
\centering
\begin{tabular}{c|c}
	\begin{tikzpicture}[scale=0.8]
					\tikzset{mypoints/.style={fill=white,draw=black,thick}}
					\def\ptsize{2.0pt}
					\def\etx{-2.0} \def\ety{2.0} 
					\coordinate (ET) at (\etx,\ety);
					\coordinate[label=below:{$E_1$}] (EB) at (\etx,-1*\ety);
					\draw[name path=lineE,black, very thick] (ET)--(EB);
					\coordinate (FT) at (-1*\etx,\ety);
					\coordinate[label=below:{$E_2$}]  (FB) at (-1*\etx,-1*\ety);
					\draw[name path=lineF,black, very thick] (FT)--(FB);
					\coordinate (EF1) at (\etx-1.0,\ety-1.0);
					\coordinate (EF2) at (-1*\etx+1.0,\ety-1.0);
					\draw[name path=lineF,blue, very thick] (EF1)--(EF2);
					\coordinate (startC) at (\etx-1.0,-1*\ety);
					\coordinate (PSE) at (intersection of ET--EB and EF1--EF2);
					\coordinate (midC) at (-1*\etx,\ety-0.5);
					\coordinate (FinC) at (-1*\etx+1.0,\ety);
					\draw [red, xshift=4cm] plot [smooth, tension=1] coordinates { (startC) (PSE) (midC) (FinC)};
					\draw[red](-1*\etx-2.0,\ety-0.3) node[left=1pt,fill=white]{$C$};
					\draw[blue](-1*\etx-2.0,\ety-1.2) node[left=1pt,fill=white]{$L$};
					
					\coordinate (caption) at (0.0,-1*\ety);
					
					\node [below=1cm, align=flush center,text width=8cm] at (caption)
        {
            $\alpha(S, (1-\beta)C)=\omega_4$
        };

	\end{tikzpicture}

& \begin{tikzpicture}[scale=0.8]

 \tikzset{mypoints/.style={fill=white,draw=black,thick}}
	\def\ptsize{2.0pt}
	\def\etx{-2.0} \def\ety{2.0} 
	\coordinate (ET) at (\etx,\ety);
	\coordinate[label=below:{$E_1$}] (EB) at (\etx,-1*\ety);
	\draw[name path=lineE,black, very thick] (ET)--(EB);
	\coordinate (FT) at (-1*\etx,\ety);
	\coordinate[label=below:{$E_2$}]  (FB) at (-1*\etx,-1*\ety);
	\draw[name path=lineF,black, very thick] (FT)--(FB);
	\coordinate (EF1) at (\etx-1.0,\ety-1.0);
	\coordinate (EF2) at (-1*\etx+1.0,\ety-1.0);
	\draw[name path=lineF,blue, very thick] (EF1)--(EF2);
	\coordinate (startC) at (\etx-1.0,-1*\ety+2.0);
	\coordinate (Ri) at (\etx,- 1*\ety+2.0);
	\coordinate (PSE) at (intersection of ET--EB and EF1--EF2);
	\coordinate (tangency) at (\etx,-1*\ety+1.0);
	\coordinate (midL) at (\etx+1.0, \ety-1.0);
	\coordinate (midC) at (-1*\etx,\ety-0.5);
	\coordinate (FinC) at (-1*\etx+1.0,\ety);
	\draw [name path=curveC, red, xshift=4cm] plot [smooth, tension=1] coordinates { (startC) (tangency) (midL) (midC) (FinC)};
	\draw[red](-1*\etx-2.0,\ety-0.3) node[left=1pt,fill=white]{$C$};
	\draw[blue](-1*\etx-2.0,\ety-1.2) node[left=1pt,fill=white]{$L$};
	\path [name intersections={of = curveC and lineE}];
	\coordinate[label=below right:{$r_1$}] (R1) at (intersection-1);
	\coordinate (Focus) at (\etx-1.0,-1.3);
	\coordinate[label=below:{$L_1$}] (Focus2) at ($(Focus)!4.0!(R1)$);
	\draw[very thick,green!50!black!50] (Focus)--(R1)--(Focus2);
	\coordinate (caption) at (0.0,-1*\ety);
					
					\node [below=1cm, align=flush center,text width=8cm] at (caption)
        {
            $\alpha(S, (1-\beta)C)=\omega_3$
        };
	\end{tikzpicture}\\
	
\hline
\\ [-2.5ex]
	\begin{tikzpicture}[scale=0.8]

 \tikzset{mypoints/.style={fill=white,draw=black,thick}}
	\def\ptsize{2.0pt}
	\def\etx{-2.0} \def\ety{2.0} 
	\coordinate (ET) at (\etx,\ety);
	\coordinate[label=below:{$E_1$}] (EB) at (\etx,-1*\ety);
	\draw[name path=lineE,black, very thick] (ET)--(EB);
	\coordinate (FT) at (-1*\etx,\ety);
	\coordinate[label=below:{$E_2$}]  (FB) at (-1*\etx,-1*\ety);
	\draw[name path=lineF,black, very thick] (FT)--(FB);
	\coordinate (EF1) at (\etx-1.0,\ety-1.0);
	\coordinate (EF2) at (-1*\etx+1.0,\ety-1.0);
	\draw[name path=lineL,blue, very thick] (EF1)--(EF2);
	\coordinate (startC) at (\etx-1.0,-1*\ety+2.0);
	\coordinate (Ri) at (\etx,- 1*\ety+2.0);
	\coordinate (PSE) at (intersection of ET--EB and EF1--EF2);
	\coordinate (tangency) at (\etx,-1*\ety+1.0);
	\coordinate (midL) at (\etx+1.5, \ety-1.0);
		\coordinate (FinC) at (-1*\etx+1.0,\ety-0.5);
	\draw [name path=curveC, red, xshift=4cm] plot [smooth, tension=1] coordinates { (startC) (tangency) (midL) (FinC)};
	\draw[red](-1*\etx-0.7,\ety-0.3) node[left=1pt,fill=white]{$C$};
	\draw[blue](-1*\etx-2.0,\ety-1.2) node[left=1pt,fill=white]{$L$};
	\path [name intersections={of = curveC and lineE}];
	\coordinate[label=below right:{$r_1$}] (R1) at (intersection-1);
	\coordinate (Focus) at (\etx-1.3, -0.2);
	\coordinate[label=below:{$L_1$}] (Focus2) at ($(Focus)!1.8!(R1)$);
	\draw[very thick,green!50!black!50] (Focus)--(R1)--(Focus2);
	\path [name intersections={of = curveC and lineL}];
	\coordinate[label=above left:{$r$}] (R) at (intersection-1);
	\coordinate (FocusB) at (\etx+0.3, 0.2);
	\coordinate[label=above:{$R$}] (FocusB2) at ($(FocusB)!1.8!(R)$);
	\draw[very thick,orange!50!black!50] (FocusB)--(R)--(FocusB2);
	
	\coordinate (caption) at (0.0,-1*\ety);
					
					\node [below=1cm, align=flush center,text width=8cm] at (caption)
        {
            $\alpha(S, (1-\beta)C)=\omega_2$
        };
	\end{tikzpicture}

&	\begin{tikzpicture}[scale=0.8]

 \tikzset{mypoints/.style={fill=white,draw=black,thick}}
	\def\ptsize{2.0pt}
	\def\etx{-2.0} \def\ety{2.0} 
	\coordinate (ET) at (\etx,\ety);
	\coordinate[label=below:{$E_1$}] (EB) at (\etx,-1*\ety);
	\draw[name path=lineE,black, very thick] (ET)--(EB);
	\coordinate (FT) at (-1*\etx,\ety);
	\coordinate[label=below:{$E_2$}]  (FB) at (-1*\etx,-1*\ety);
	\draw[name path=lineF,black, very thick] (FT)--(FB);
	\coordinate (EF1) at (\etx-1.0,\ety-1.0);
	\coordinate (EF2) at (-1*\etx+1.0,\ety-1.0);
	\draw[name path=lineL,blue, very thick] (EF1)--(EF2);
	\coordinate (startC) at (\etx-1.0,-1*\ety+2.0);
	\coordinate (Ri) at (\etx,- 1*\ety+2.0);
	\coordinate (PSE) at (intersection of ET--EB and EF1--EF2);
	\coordinate (tangency) at (\etx,-1*\ety+1.0);
	\coordinate (midL) at (\etx+1.5, \ety-1.0);
	\coordinate (FinC) at (-1*\etx+1.0,\ety-0.5);
	\draw [name path=curveC, red, xshift=4cm] plot [smooth, tension=1] coordinates { (startC) (tangency) (midL) (FinC)};
	\draw[red](-1*\etx-0.7,\ety-0.3) node[left=1pt,fill=white]{$C$};
	\draw[blue](-1*\etx-2.0,\ety-1.2) node[left=1pt,fill=white]{$L$};
	\path [name intersections={of = curveC and lineE}];
	\coordinate[label=below right:{$r_1$}] (R1) at (intersection-1);
	\coordinate (Focus) at (\etx-1.3, -0.2);
	\coordinate[label=below:{$L_1$}] (Focus2) at ($(Focus)!1.8!(R1)$);
	\draw[very thick,green!50!black!50] (Focus)--(R1)--(Focus2);
	\path [name intersections={of = curveC and lineL}];
	\coordinate[label=above right:{$r$}] (R) at (intersection-1);
	\coordinate (FocusB) at (\etx+0.3, \ety);
	\coordinate[label=above:{$R$}] (FocusB2) at ($(FocusB)!1.8!(R)$);
	\draw[very thick,orange!50!black!50] (FocusB)--(R)--(FocusB2);
	\coordinate (caption) at (0.0,-1*\ety);
					
					\node [below=1cm, align=flush center,text width=8cm] at (caption)
        {
            $\alpha(S, (1-\beta)C)=\omega_1$
        };
	
	\end{tikzpicture}

\end{tabular}
\caption{Arrangements of lines and intersection of relevant curves in the non-singular del Pezzo surface of degree $7$.}
\label{tab:deg7}
\end{table}

\begin{lem}
\label{lem:del-Pezzo-dynamic-alpha-degree-7}
Let $S$ be the non-singular del Pezzo surface of degree $7$ and $C\in\vert-K_S\vert$ be a smooth curve. We follow notation \ref{nota:del-Pezzo-deg7}.
\begin{itemize}
	\item[(i)] If $C$ contains a pseudo-Eckardt point, then
		\begin{equation}
			\alpha(S,(1-\beta)C)\leq\omega_4:=
					\begin{dcases}
							1 												&\text{ for } 0<\beta\leq \frac{1}{4},\\
							\frac{1+\beta}{5\beta}		&\text{ for } \frac{1}{4}\leq \beta\leq \frac{2}{3},\\
							\frac{1}{3\beta}					&\text{ for } \frac{2}{3}\leq \beta\leq 1.\\
					\end{dcases}
		\label{eq:del-Pezzo-dynamic-alpha-degree-7-special-upbound}
	\end{equation}
	\item[(ii)] If $C$ contains no pseudo-Eckardt point but $(C\cdot L_i)\vert_{r_i}=2$, for some $i=1,2$, then
		\begin{equation}
			\alpha(S,(1-\beta)C)\leq\omega_3:=
					\begin{dcases}
							1 												&\text{ for } 0<\beta\leq \frac{1}{4},\\
							\frac{1+2\beta}{6\beta}		&\text{ for } \frac{1}{4}\leq \beta\leq \frac{1}{2},\\
							\frac{1}{3\beta}					&\text{ for } \frac{1}{2}\leq \beta\leq 1.\\
					\end{dcases}
		\label{eq:del-Pezzo-dynamic-alpha-degree-7-special3-upbound}
	\end{equation}
	\item[(iii)] If $C$ contains no pseudo-Eckardt point, $(C\cdot L_i)\vert_{r_i}=1$, for both $i=1,2$ and $(R\cdot C)\vert_{r}=3$, then
		\begin{equation}
			\alpha(S,(1-\beta)C)\leq\omega_2:=
					\begin{dcases}
							1 												&\text{ for } 0<\beta\leq \frac{1}{4},\\
							\frac{1+3\beta}{7\beta}		&\text{ for } \frac{1}{4}\leq \beta\leq \frac{4}{9},\\
							\frac{1}{3\beta}					&\text{ for } \frac{4}{9}\leq \beta\leq 1.\\
					\end{dcases}
		\label{eq:del-Pezzo-dynamic-alpha-degree-7-special2-upbound}
	\end{equation} 
	\item[(iv)] If $C$ contains no pseudo-Eckardt point, $(C\cdot L_i)\vert_{r_i}=1$, for both $i=1,2$ and $(R\cdot C)\vert_{r}\leq 2$, then
		\begin{equation}
			\alpha(S,(1-\beta)C)\leq\omega_1:=
					\begin{dcases}
							1 												&\text{ for } 0<\beta\leq \frac{1}{3},\\
							\frac{1}{3\beta}					&\text{ for } \frac{1}{3}\leq \beta\leq 1.\\
					\end{dcases}
			\label{eq:del-Pezzo-dynamic-alpha-degree-7-generic-upbound}
		\end{equation}
\end{itemize} 
\end{lem}
\begin{proof}
First observe that $\omega_4\leq\omega_3\leq\omega_2\leq\omega_1$ which will explain the statement logic once we find an effective $\bbQ$-divisor $D\simq -K_S$for each case, such that $\lct(S, (1-\beta)C, \beta D)=\omega_i$.

Suppose $C$ has a pseudo-Eckardt point $p$. Without loss of generality let $p=E_1\cap L$. Observe that $3L+2E_1+2E_2\sim-K_S$ and let $f\colon \widetilde S \ra S$ be the minimal log resolution of the pair $(S,(1-\beta)C+\lambda\beta(3L+2E_1+2E_2))$, consisting of the blow-up of $p$ with exceptional divisor $F_1$. The log pullback is
$$f^*(K_S+(1-\beta)C+\lambda\beta( 3L+2E_1+2E_2))\simq K_{\widetilde S}+(1-\beta)\widetilde C +\lambda\beta (3\widetilde L + 2\widetilde E_1 + 2\widetilde E_2) + (5\lambda\beta-\beta)F_1.$$
Therefore
\begin{align*}
\alpha(S,(1-\beta)C)&\leq\min\left\{\lct\left(S,\left(1-\beta\right)C,\beta C\right),\lct\left(S,\left(1-\beta\right)C,\beta\left(3L+2E_1+2E_2\right)\right)\right\}\\
&=\min\left\{1,\frac{1}{2\beta},\frac{1}{3\beta},\frac{1+\beta}{5\beta}\right\}=\omega_4.
\end{align*}

Suppose $C$ contains no pseudo-Eckardt point. Since $C\cdot L= C\cdot E_i=L\cdot E_i=1$  and $E_1\cdot E_2=0$, the pair
$$(S, (1-\beta)C + \lambda\beta (3L+2E_1+2E_2))$$
has simple normal crossings. Hence
\begin{align*}
\alpha(S,(1-\beta)C)&\leq\min\left\{\lct\left(S,\left(1-\beta\right)C,\beta C\right),\lct\left(S,\left(1-\beta\right)C,\beta\left(3L+2E_1+2E_2\right)\right)\right\}\\
&=\min\left\{1,\frac{1}{2\beta},\frac{1}{3\beta},\right\}=\omega_1.
\end{align*}

Suppose $(C\cdot L_i)\vert_{p_i}=2$ for some $i=1,2$. Without loss of generality assume $i=1$. Observe that $2L_1+2E_1+L\sim-K_S$ and let $f\colon \widetilde S \ra S$ be the minimal log resolution of the pair $(S,(1-\beta)C+\lambda\beta(2L_1+2E_1+L))$, consisting of two consecutive blow-ups of $r_1$ with exceptional divisors $F_1, F_2$. The log pullback is
$$f^*(K_S+(1-\beta)C+\lambda\beta (2L_1+2E_1+L))\simq K_{\widetilde S}+(1-\beta)\widetilde C +\lambda\beta (2\widetilde L_1 + 2\widetilde E_1 + \widetilde L) + (4\lambda\beta-\beta)F_1+(6\lambda\beta-2\beta)F_2.$$
Therefore
\begin{align*}
&\alpha(S,(1-\beta)C)\\
\leq&\min\{\lct\left(S,\left(1-\beta\right)C,\beta C\right),\lct\left(S,\left(1-\beta\right)C,\beta\left(3L+2E_1+2E_2\right)\right),\\
& \qquad\qquad \lct\left(S,\left(1-\beta\right)C,\beta\left(2L_1+2E_1+L\right)\right)\}\\
=&\min\left\{1,\frac{1}{2\beta},\frac{1}{3\beta},\frac{1+\beta}{5\beta}, \frac{1+2\beta}{6\beta}, \frac{1+\beta}{4\beta}\right\}=\omega_3.
\end{align*}

Finally, suppose $(R\cdot C)\vert_{p_i}=3$. Observe that $L+2R\sim-K_S$ and let $f\colon \widetilde S \ra S$ be the minimal log resolution of the pair $(S,(1-\beta)C+\lambda\beta(L+2R))$, consisting of three consecutive blow-ups of $r$ with exceptional divisors $F_1, F_2, F_3$. The log pullback is
$$f^*(K_S+(1-\beta)C+\lambda\beta (L+2R))\simq K_{\widetilde S}+(1-\beta)\widetilde C +\lambda\beta (\widetilde L+2\widetilde R) + (3\lambda\beta-\beta)F_1+(5\lambda\beta-2\beta)F_2+(7\lambda\beta-3\beta)F_3.$$
Therefore
\begin{align*}
&\alpha(S,(1-\beta)C)\\
\leq&\min\{\lct\left(S,\left(1-\beta\right)C,\beta C\right),\lct\left(S,\left(1-\beta\right)C,\beta\left(3L+2E_1+2E_2\right)\right),\\
&\qquad\qquad \lct\left(S,\left(1-\beta\right)C,\beta\left(L+2R\right)\right)\}\\
=&\min\left\{1,\frac{1}{2\beta},\frac{1}{3\beta},\frac{1+\beta}{3\beta},\frac{1+2\beta}{5\beta},\frac{1+3\beta}{7\beta}\right\}=\omega_2.
\end{align*}
\end{proof}

\begin{thm}
\label{thm:del-Pezzo-dynamic-alpha-degree-7}

Let $S$ be the non-singular del Pezzo surface of degree $7$ and $C\in\vert-K_S\vert$ be a smooth curve. We follow notation \ref{nota:del-Pezzo-deg7}.
\begin{itemize}
	\item[(i)] If $C$ contains a pseudo-Eckardt point, then
		\begin{equation}
			\alpha(S,(1-\beta)C)=\omega_4:=
					\begin{dcases}
							1 												&\text{ for } 0<\beta\leq \frac{1}{4},\\
							\frac{1+\beta}{5\beta}		&\text{ for } \frac{1}{4}\leq \beta\leq \frac{2}{3},\\
							\frac{1}{3\beta}					&\text{ for } \frac{2}{3}\leq \beta\leq 1.\\
					\end{dcases}
		\label{eq:del-Pezzo-dynamic-alpha-degree-7-special}
	\end{equation}
	\item[(ii)] If $C$ contains no pseudo-Eckardt point but $(C\cdot L_i)\vert_{r_i}=2$, for some $i=1,2$, then
		\begin{equation}
			\alpha(S,(1-\beta)C)=\omega_3:=
					\begin{dcases}
							1 												&\text{ for } 0<\beta\leq \frac{1}{4},\\
							\frac{1+2\beta}{6\beta}		&\text{ for } \frac{1}{4}\leq \beta\leq \frac{1}{2},\\
							\frac{1}{3\beta}					&\text{ for } \frac{1}{2}\leq \beta\leq 1.\\
					\end{dcases}
		\label{eq:del-Pezzo-dynamic-alpha-degree-7-special3}
	\end{equation}
	\item[(iii)] If $C$ contains no pseudo-Eckardt point, $(C\cdot L_i)\vert_{r_i}=1$, for both $i=1,2$ and $(R\cdot C)\vert_{r}=3$, then
		\begin{equation}
			\alpha(S,(1-\beta)C)=\omega_2:=
					\begin{dcases}
							1 												&\text{ for } 0<\beta\leq \frac{1}{4},\\
							\frac{1+3\beta}{7\beta}		&\text{ for } \frac{1}{4}\leq \beta\leq \frac{4}{9},\\
							\frac{1}{3\beta}					&\text{ for } \frac{4}{9}\leq \beta\leq 1.\\
					\end{dcases}
		\label{eq:del-Pezzo-dynamic-alpha-degree-7-special2}
	\end{equation} 
	\item[(iv)] If $C$ contains no pseudo-Eckardt point, $(C\cdot L_i)\vert_{r_i}=1$, for both $i=1,2$ and $(R\cdot C)\vert_{r}\leq 2$, then
		\begin{equation}
			\alpha(S,(1-\beta)C)=\omega_1:=
					\begin{dcases}
							1 												&\text{ for } 0<\beta\leq \frac{1}{3},\\
							\frac{1}{3\beta}					&\text{ for } \frac{1}{3}\leq \beta\leq 1.\\
					\end{dcases}
			\label{eq:del-Pezzo-dynamic-alpha-degree-7-generic}
		\end{equation}
\end{itemize}
\end{thm}

In the rest of this section we prove Theorem \ref{thm:del-Pezzo-dynamic-alpha-degree-7} by case analysis.

First observe that by Lemma \ref{lem:del-Pezzo-dynamic-alpha-degree-7}, $\alpha(S,(1-\beta)C)\leq \omega_i$. Suppose $\alpha(S,(1-\beta)C)<\omega_i$. Then there is an effective $\bbQ$-divisor $D\simq-K_S$ such that the pair
\begin{equation}
(S,(1-\beta)C+\lambda\beta D)
\label{eq:del-Pezzo-dynamic-alpha-degree-7-proof-pair}
\end{equation}
is not log canonical at some $q=q_0\in S$ for $\lambda<\omega_i$ in each case.

Observe that
\begin{equation}
\omega_4\leq\omega_3\leq\omega_2\leq\omega_1
\label{eq:del-Pezzo-dynamic-alpha-degree-7-proof-thresholds-ineq}
\end{equation}

\begin{lem}
\label{lem:del-Pezzo-dynamic-alpha-degree-7-proof-in-C}
One has that the point $q\in C$, the pair $(S,(1-\beta)C+\lambda\beta D)$ is log canonical in codimension $1$ and $\lambda\omega_i<\frac{1}{3}$.
\end{lem}
\begin{proof}
First we claim that
\begin{equation}
\lambda\beta<\omega_i\beta\leq \frac{1}{3}=\glct(S)
\label{eq:del-Pezzo-dynamic-alpha-degree-7-proof-inequality0}
\end{equation}
where the last equality comes from Theorem \ref{thm:del-Pezzo-glct-charp}.
If $\lambda<\omega_1$, then inequality \eqref{eq:del-Pezzo-dynamic-alpha-degree-7-proof-inequality0} is trivial. For the rest $\omega_i$ it follows from \eqref{eq:del-Pezzo-dynamic-alpha-degree-7-proof-thresholds-ineq}.

By Lemma \ref{lem:pairs-fixed-boundary-lcs} the pair \eqref{eq:del-Pezzo-dynamic-alpha-degree-7-proof-pair} is log canonical in codimension $1$ and not log canonical at some isolated $q\in C$.
\end{proof}

Let $f_1\colon S_1\ra S$ be the blow-up of $q$ with exceptional divisor $F_1$. Let $A^1$ be the strict transform of any $\bbQ$-divisor $A$ in $S$. By Lemma \ref{lem:log-pullback-preserves-lc} the pair
\begin{equation}
(S_1, (1-\beta) C^1 + \lambda\beta D^1 + (\lambda\beta\mult_qD - \beta)F_1)
\label{eq:del-Pezzo-dynamic-alpha-degree-7-proof-bad-pair-blowup}
\end{equation}
is not log canonical at some point $t_1\in F_1$. Let $m_1=\mult_{t_1}D$ and $m=m_0=\mult_qD$.

\begin{lem}
\label{lem:del-Pezzo-dynamic-alpha-degree-7-proof-bad-pair-blowup-inC}
If $\mult_qD\leq 3$, then
$$(S_1, (1-\beta) C^1 + \lambda\beta D^1 + (\lambda\beta\mult_qD - \beta)F_1)$$
is log canonical in codimension $1$ and not log canonical at $t_1=q_1:=C^1\cap F_1$.
\end{lem}
\begin{proof}
Suppose $t_1\not\in C^1$, then the pair
$$(S_1, \lambda\beta D^1 + (\lambda\beta\mult_qD - \beta)F_1)$$
is not log canonical at some $t_1\in F_1$. Since $\lambda\beta\mult_qD\leq 1$, then the locus of log canonical singularities of the pair consists of isolated points. By Lemma \ref{lem:adjunction} (i) we have
$$1\geq\lambda\beta\mult_qD=\lambda\beta D^1\cdot F_1>1,$$
a contradiction.
\end{proof}

For the rest of the proof of Theorem \ref{thm:del-Pezzo-dynamic-alpha-degree-7} we rule out the different positions of $q$, the point at which
$$(S, (1-\beta)C + \lambda\beta D )$$
is not log canonical. We will apply Theorem \ref{thm:pseudo-inductive-blow-up} several times on successive blow-ups of $q$ and points over $q$. 

Let $S_0=S$, $C_0=C$, $D_0=D$ and $q_0=q$ as in Theorem \ref{thm:pseudo-inductive-blow-up}. Let $i\geq 1$ and let $f_i\colon S_i\ra S_{i-1}$ be the blow-up of the point $q_{i-1}=C^i\cap F_{i-1}$ with exceptional curve $F_i$. Let $A^{i-1}$ or $A$ be any $\bbQ$-divisor in $S_{i-1}$. We will denote its strict transform in $S^i$ by $A^i$. Let $m_i=\mult_{q_i} D_i$. Recall from \eqref{eq:del-Pezzo-dynamic-alpha-degree-7-proof-bad-pair-blowup} that the pair
$$(S_1, (1-\beta)C^1 + \lambda\beta D^1 + (\lambda\beta m_0-\beta)F_1)$$
is not log canonical at some $t_1\in F_1$.

\begin{lem}
\label{lem:del-Pezzo-dynamic-alpha-degree-7-proof-in-line}
The point $q\in E_1\cup E_2\cup L$.
\end{lem}
\begin{proof}
Suppose $q$ is not contained in any $(-1)$-curve. First observe it is enough to show that the pair $(S,(1-\beta)C + \lambda\beta D)$ is log canonical at $q$ for $\lambda<\omega_1$, by \eqref{eq:del-Pezzo-dynamic-alpha-degree-7-proof-thresholds-ineq}.

Let $Z\sim\pioplane{1}$ be the strict transform of a general line in $\bbP^2$ through $\pi(q)$. Then $Z\cdot (-K_S)=3$ and $Z\not\subseteq\Supp(D)$, so
\begin{equation}
3=Z\cdot D\geq\mult_qD.
\label{eq:del-Pezzo-dynamic-alpha-degree-7-proof-multBound}
\end{equation}
By Lemma \ref{lem:del-Pezzo-dynamic-alpha-degree-7-proof-bad-pair-blowup-inC} the pair \eqref{eq:del-Pezzo-dynamic-alpha-degree-7-proof-bad-pair-blowup} is not log canonical at $t_1=q_1=C^1\cap F_1$.

We want to apply Theorem \ref{thm:pseudo-inductive-blow-up} for increasing $i$. Observe that for $i=2$ the hypothesis of Theorem \ref{thm:pseudo-inductive-blow-up} (ii) is satisfied by \eqref{eq:del-Pezzo-dynamic-alpha-degree-7-proof-multBound}, since then $\lambda\beta m_0\leq 1$, using Lemma \ref{lem:del-Pezzo-dynamic-alpha-degree-7-proof-in-C}. Moreover, this is a necessary condition to apply Theorem \ref{thm:pseudo-inductive-blow-up} (ii)-(iv) whenever $i\geq 1$ and we will assume it from now onwards without mentioning it for the rest of the proof.  In particular, by Theorem \ref{thm:pseudo-inductive-blow-up} (ii) when $i=2$, the pair
\begin{equation}
(S_2, (1-\beta)C^2+\lambda\beta D^2 + (\lambda\beta m_0 -\beta)F_1^2 + (\lambda\beta(m_0+m_1)-2\beta)F_2)
\label{eq:del-Pezzo-dynamic-alpha-degree-7-proof-genericpoint-pair2}
\end{equation} 
is not log canonical at some $t_2\in F_2$. 

Let $H_1$ and $H_2$ be the unique curves such that
$$H_i\sim\pioplane{1}-E_i$$
with $q\in H_i$. They correspond to the strict transform of the unique lines in $\bbP^2$ containing $\pi(q)$ and $p_i$, respectively. Let $H\in \vert \pioplane{1}\vert$ be the unique curve such that $q\in H$ and $q_1\in H^1$.
The curve $H$ is a cubic tangent to $C$ at $q$. Observe that $H+H_1+H_2\sim-K_S$.

\textbf{Case 1: we show that if $H$ is irreducible, then}
$$(S_4, (1-\beta)C^4 + \lambda\beta D^4 +(\lambda\beta(m_0+m_1+m_2+m_3)-4\beta)F_4)$$
\textbf{is not log canonical at $q_4=C^4\cap F_4$.}

If $H$ is irreducible, by Lemma \ref{lem:appendix-delPezzo-deg7-dynamic-lct1}, the pair $(S,(1-\beta)C + \frac{1}{3}(H+H_1+H_2))$ is log canonical. Applying Lemma \ref{lem:log-convexity} we may assume one irreducible component of $(H+H_1+H_2)$ is not in $\Supp(D)$.

\textbf{Case 1a: Suppose $H\not\subseteq\Supp(D)$}. Then $(3-m_0)=H^1\cdot D^1\geq m_1,$ implying
\begin{equation}
m_0+m_1\leq 3
\label{eq:del-Pezzo-dynamic-alpha-degree-7-proof-genericpoint-bound1}
\end{equation}
and $\lambda\beta(m_0+m_1)\leq 1$ by Lemma \ref{lem:del-Pezzo-dynamic-alpha-degree-7-proof-in-C}, which in particular implies the hypothesis of Theorem \ref{thm:pseudo-inductive-blow-up} (ii) and (iii) when $i=2$, thus $t_2=F_2\cap(C^2)=q_2$. Hence we may apply Theorem \ref{thm:pseudo-inductive-blow-up} (i) with $i=3$ to conclude that the pair
$$(S_3, (1-\beta)C^3 + \lambda\beta D^3 +(\lambda\beta(m_0+m_1)-2\beta)F_2 + (\lambda\beta(m_0+m_1+m_2)-3\beta)F_3)$$
is not log canonical at some $t_3\in F_3$. To apply parts (ii) and (iii) of Theorem \ref{thm:pseudo-inductive-blow-up} when $i=3$, it is enough to show that
\begin{equation}
\lambda\beta(m_0+m_1+m_2)-3\beta\leq \lambda\beta(m_0+m_1+m_2)-2\beta\leq 1.
\label{eq:del-Pezzo-dynamic-alpha-degree-7-proof-genericpoint-bound2}
\end{equation}
Since $m_{j+1}\leq m_j$ for all $j$, \eqref{eq:del-Pezzo-dynamic-alpha-degree-7-proof-genericpoint-bound1} implies that $m_{j+1}+m_j\leq 3$ for $j=0,1,2$. Hence adding this $3$ inequalities we obtain $2(m_0+m_1+m_2)\leq 9$, which we use to claim
$$\lambda\beta(m_0+m_1+m_2)-2\beta\leq \frac{9}{2}\lambda\beta -2\beta<\frac{9}{2}\omega_1\beta\leq 1.$$
Indeed, if $\beta\leq \frac{1}{3}$, then
$$\frac{9}{2}\lambda\beta-2\beta<\frac{9}{2}\beta-2\beta=\frac{5}{2}\beta\leq \frac{5}{6}<1.$$
If $\beta>\frac{1}{3}$, then $\lambda<\frac{1}{3\beta}$ and
$$\frac{9}{2}\lambda\beta-2\beta<\frac{3}{2}-2\beta=\frac{3}{2}-\frac{2}{3}=\frac{5}{6}<1.$$
Hence \eqref{eq:del-Pezzo-dynamic-alpha-degree-7-proof-genericpoint-bound2} is proven and Theorem \ref{thm:pseudo-inductive-blow-up} (iii) with $i=3$ implies $t_3=q_3=F_3\cap C^3$. Thus, we may apply Theorem \ref{thm:pseudo-inductive-blow-up}, part (i) with $i=4$ to conclude that the pair
$$(S_4, (1-\beta)C^4 + \lambda\beta D^4 +(\lambda\beta(m_0+m_1+m_2)-3\beta)F_3 + (\lambda\beta(m_0+m_1+m_2+m_3)-4\beta)F_4)$$
is not log canonical at some $t_4\in F_4$. Since $m_j\leq m_{j-1}$ for $j=1,2,3$, then, by \eqref{eq:del-Pezzo-dynamic-alpha-degree-7-proof-genericpoint-bound1} we conclude that 
\begin{equation}
\lambda\beta(\sum_{i=0}^3 m_i)-3\beta\leq 6\lambda\beta-3\beta<6\omega_1\beta-3\beta \leq 1.
\label{eq:del-Pezzo-dynamic-alpha-degree-7-proof-genericpoint-bound3}
\end{equation}
Indeed, if $\beta\leq\frac{1}{3}$, then $6\lambda\beta-3\beta<6\beta-3\beta=3\beta\leq 1$, whereas if $\frac{1}{3}\leq \beta\leq 1$, then $6\lambda\beta-3\beta<2-3\beta\leq 1$. Hence we can use Theorem \ref{thm:pseudo-inductive-blow-up}, parts (ii) and (iii) to conclude that $t_4=q_4=F_4\cap C^4$. Hence the pair
$$(S_4, (1-\beta)C^4 + \lambda\beta D^4 +(\lambda\beta(m_0+m_1+m_2+m_3)-4\beta)F_4)$$
is not log canonical at $q_4$.

\textbf{Case 1b: Suppose $H\subseteq \Supp(D)$}. We may write $D=aH+\Omega$ where $H\not\subseteq\Supp(\Omega)$ and $a>0$. Let $x=x_0=\mult_q\Omega$ and $x_i=\mult_q\Omega^i$. Then $m_0=a+x_0$ and $m_1=a+x_1$. We bound these multiplicities. Notice that $H^2=1$, so $3-a -x_0 = \Omega\cdot H -x_0=\Omega^1\cdot H^1\geq x_1$. Hence:
\begin{equation}
2=D\cdot H_i\geq a+x_0=m_0,  \text{ and } 3\geq a + x_0 + x_1=m_0+x_1.
\label{eq:del-Pezzo-dynamic-alpha-degree-7-proof-genericpoint-bound4}
\end{equation}
Observe that $2H + L \sim -K_S$. By Lemma \ref{lem:appendix-delPezzo-deg7-dynamic-lct3}, the pair
$$(S, (1-\beta)C + \lambda\beta(2H+L))$$
is log canonical. By Lemma \ref{lem:log-convexity} we may assume that $L\not\subseteq\Supp(D)$, since $H\subseteq\Supp(D)$ by assumption. Then
\begin{equation}
1=L\cdot D \geq a
\label{eq:del-Pezzo-dynamic-alpha-degree-7-proof-genericpoint-bound4b}
\end{equation}
Recall from \eqref{eq:del-Pezzo-dynamic-alpha-degree-7-proof-genericpoint-pair2} that the pair
$$(S_2, (1-\beta)C^2+\lambda\beta D^2 + (\lambda\beta m_0 -\beta)F_1^2 + (\lambda\beta(m_0+m_1)-2\beta)F_2)$$
is not log canonical at some $t_2\in F_2$ by Theorem \ref{thm:pseudo-inductive-blow-up} (i) with $i=2$. We claim that
\begin{equation}
\lambda\beta(m_0 + m_1)-\beta<\omega_1\beta(m_0+m_1)-\beta\leq 1.
\label{eq:del-Pezzo-dynamic-alpha-degree-7-proof-genericpoint-bound5}
\end{equation}
This is a necessary condition to apply parts (ii) and (iii) of Theorem \ref{thm:pseudo-inductive-blow-up} with $i\geq 2$ and sufficient when $i=2$, concluding $t_2=F_2\cap C^2=q_2$. Indeed using the first equation in \eqref{eq:del-Pezzo-dynamic-alpha-degree-7-proof-genericpoint-bound4} we see
$$\lambda\beta(m_0+m_1)-\beta = \lambda\beta(2a+x_0+x_1)-\beta'\leq 2\lambda\beta(a+x_0)-\beta\leq 4\lambda\beta-\beta<4\omega_1\beta-\beta.$$
To prove \eqref{eq:del-Pezzo-dynamic-alpha-degree-7-proof-genericpoint-bound5} it is enough to show $4\omega_1\beta-\beta\leq 1$. Indeed, if $0<\beta\leq \frac{1}{3}$, then
$$4\omega_1\beta-\beta\leq4\beta-\beta=3\beta\leq 1,$$
whereas if $\frac{1}{3}\leq \beta\leq 1$, then $4\omega_1\beta-\beta=\frac{4}{3}-\beta\leq 1$.

Hence, using \eqref{eq:del-Pezzo-dynamic-alpha-degree-7-proof-genericpoint-bound5} we may apply Theorem \ref{thm:pseudo-inductive-blow-up} (i) with $i=3$ to deduce that the pair
$$(S_3, (1-\beta)C^3 + \lambda\beta D^3 +(\lambda\beta(m_0+m_1)-2\beta)F_2 + (\lambda\beta(m_0+m_1+m_2)-3\beta)F_3)$$
is not log canonical at some $t_3\in F_3$. 

Observe that since $C^1\cdot H^1=2$ either $(H^1\cdot C^1)\vert_{q_1}=1$ or $(H^1\cdot C^1)\vert_{q_1}=2$.

\textbf{Subcase 1b (i): Suppose $(H^1\cdot C^1)\vert_{q_1}=1$.} Then we have $q_2\not\in H^2$ so $m_2=x_2$ and
\begin{align*}
 &\lambda\beta(m_0+m_1+m_2)-2\beta=\lambda\beta(2a+x_0+x_1+x_2)-2\beta\\
\leq &\lambda\beta(a+x_0+a+x_0+x_1)-2\beta<5\omega_1\beta-2\beta.
\end{align*}
If $\beta\leq \frac{1}{3}$, then $5\omega_1\beta-2\beta =3\beta\leq 1$. If $\beta\geq \frac{1}{3}$, then $5\omega_1\beta-2\beta=\frac{5}{3}-2\beta\leq 1$. Hence
$$\lambda\beta(m_0+m_1+m_2)-2\beta\leq 1$$
holds and we may apply part (iii) of Theorem \ref{thm:pseudo-inductive-blow-up} with $i=3$ to conclude that $t_3=q_3=C^3\cap F_3$. Moreover, since
$$\lambda\beta(m_0+m_1+m_2)-3\beta\leq \lambda\beta(m_0+m_1+m_2)-2\beta\leq 1$$
holds, we may also apply Theorem \ref{thm:pseudo-inductive-blow-up} part (i) with $i=4$ to conclude that the pair
$$(S_4, (1-\beta)C^4 + \lambda\beta D^4 +(\lambda\beta(m_0+m_1+m_2)-3\beta)F_3 + (\lambda\beta(m_0+m_1+m_2+m_3)-4\beta)F_4)$$
is not log canonical at some $t_4\in F_4$.

\textbf{Subcase 1b (ii): Suppose $(H^1\cdot C^1)\vert_{q_1}=2$.} Then we have $q_2\in H^2$ so $m_2=a+x_2$ and
\begin{align*}
 &\lambda\beta(m_0+2m_1)-3\beta=\lambda\beta(3a+x_0+2x_1)-3\beta\\
\leq &3\lambda\beta(a+x_0)-3\beta\leq 6\lambda\beta-3\beta<6\omega_1\beta-3\beta\leq 1,
\end{align*}
by \eqref{eq:del-Pezzo-dynamic-alpha-degree-7-proof-genericpoint-bound4}. The last inequality follows by case analysis. Indeed, if $0<\beta\leq \frac{1}{3}$, then $6\omega_1\beta-3\beta=3\beta\leq 1$. If $\frac{1}{3}\leq\beta\leq 1$, then $6\omega_1\beta-3\beta=2-3\beta\leq 1$.

Now we can apply parts (ii) and (iv) of Theorem \ref{thm:pseudo-inductive-blow-up} with $i=3$ to conclude that $t_3=C^3\cap F_3=q_3$. Moreover, since
$$\lambda\beta(m_0+m_1+m_2)-3\beta\leq \lambda\beta(m_0+m_1+m_2)-2\beta\leq 1$$
we may also apply Theorem \ref{thm:pseudo-inductive-blow-up} part (i) with $i=4$ to conclude that the pair
$$(S_4, (1-\beta)C^4 + \lambda\beta D^4 +(\lambda\beta(m_0+m_1+m_2)-3\beta)F_3 + (\lambda\beta(m_0+m_1+m_2+m_3)-4\beta)F_4)$$
is not log canonical at some $t_4\in F_4$.

\textbf{Finishing case 1b. }Observe that
\begin{align*}
&\lambda\beta(m_0+m_1+2m_2)-4\beta=\\
\leq &\lambda\beta(3a+x_0+x_1+x_2+x_3)-4\beta
\leq 2\lambda\beta(a+x_0+x_1+a)-4\beta\\
\leq&7\lambda\beta-4\beta<7\omega_i\beta-4\beta\leq 7\omega_1\beta-4\beta\leq 1
\end{align*}
where we apply \eqref{eq:del-Pezzo-dynamic-alpha-degree-7-proof-genericpoint-bound4} and \eqref{eq:del-Pezzo-dynamic-alpha-degree-7-proof-genericpoint-bound4b}. The last inequality follows from case analysis. If $0<\beta\leq \frac{1}{3}$, then $7\omega_1\beta-4\beta=3\beta\leq 1$, whereas if $\frac{1}{3}\leq\beta\leq 1$, then $7\omega_1\beta-4\beta=\frac{7}{4}-4\beta\leq 1$.

Hence by Theorem \ref{thm:pseudo-inductive-blow-up} (ii) and (iv) with $i=4$ we obtain that $t_4=F_4\cap C^4=q_4$. In particular, the pair
$$(S_4, (1-\beta)C^4 + \lambda\beta D^4 +(\lambda\beta(m_0+m_1+m_2+m_3)-4\beta)F_4)$$
is not log canonical at $q_4=F_4\cap C^4$.

\textbf{Case 2: we show that if $H$ is reducible, then}
$$(S_4, (1-\beta)C^4 + \lambda\beta D^4 +(\lambda\beta(m_0+m_1+m_2+m_3)-4\beta)F_4)$$
\textbf{is not log canonical at $q_4=C^4\cap F_4$.}

If $H$ is reducible, then it must be the union of a line and a conic. In fact $H=H_i+E_i$ for some $i$. Without loss of generality assume $i=1$. Since $q\not\in E_1$ and
$$2\leq (H\cdot C)\vert_p=(H_1\cdot C)\vert_p\leq H_1\cdot C_1=2,$$
then $q_1=H_1^1\cap C^1\cap F_1$. By Lemma \ref{lem:appendix-delPezzo-deg7-dynamic-lct2}, the pair
$$(S,(1-\beta) C + \omega_1\beta(2H_1+H_2+E_1) ) $$
is log canonical. Then by Lemma \ref{lem:log-convexity} (log convexity) we may assume that some component of $2H_1+H_2+E_1$ is not in $\Supp(D)$.

\textbf{Case 2a: Suppose $H_1\not\subseteq \Supp(D)$.} Then $2-m_0=H^1_1\cdot D^1\geq m_1$, so
\begin{equation}
m_0+m_1\leq 2.
\label{eq:del-Pezzo-dynamic-alpha-degree-7-proof-genericpoint-bound7}
\end{equation}
In particular, this implies $\lambda\beta(m_0+m_1)-\beta\leq 1$, which is enough to satisfy the hypothesis of Theorem \ref{thm:pseudo-inductive-blow-up} parts (ii) and (iii) with $i\geq 2$. Therefore $t_2=F_2\cap C^2=q_2$. Condition (i) of Theorem \ref{thm:pseudo-inductive-blow-up}, when $i=3$ is also satisfied, since
$$\lambda\beta(m_0+m_1)-2\beta\leq \lambda\beta(m_0+m_1)-\beta\leq 1.$$
Hence, the pair
$$(S_3, (1-\beta)C^3 + \lambda\beta D^3 +(\lambda\beta(m_0+m_1)-2\beta)F_2 + (\lambda\beta(m_0+m_1+m_2)-3\beta)F_3)$$
is not log canonical at $t_3\in F_3$. Since $m_{i+1}\leq m_i$ for all $i$, \eqref{eq:del-Pezzo-dynamic-alpha-degree-7-proof-genericpoint-bound7} gives $m_i+m_j\leq 2$ for all $1\leq i<j\leq 2$. Adding the 3 distinct possibilities we get $2(m_0+m_1+m_2)\leq 6$, so $m_0+m_1+m_2\leq 3$. By Lemma \ref{lem:del-Pezzo-dynamic-alpha-degree-7-proof-in-C}, this implies $\lambda\beta(m_0+m_1+m_2)-2\beta\leq 1-2\beta\leq 1$, which precisely gives the hypothesis of parts (ii) and (iii) of Theorem \ref{thm:pseudo-inductive-blow-up} when $i=3$, so $t_3=C^3\cap F_3=q_3$. Since
$$\lambda\beta(m_0+m_1+m_2)-3\beta\leq \lambda\beta(m_0+m_1+m_2)-2\beta\leq 1,$$
then Theorem \ref{thm:pseudo-inductive-blow-up} (i) with $i=4$, the pair
$$(S_4, (1-\beta)C^4 + \lambda\beta D^4 +(\lambda\beta(m_0+m_1+m_2)-3\beta)F_3 + (\lambda\beta(m_0+m_1+m_2+m_3)-4\beta)F_4)$$
is not log canonical at some $t_4\in F_4$. Observe that 
$$\lambda\beta(\sum_{j=0}^3 m_j)-3\beta\leq 2\lambda\beta (m_0+m_1)-3\beta\leq 4\lambda\beta-3\beta<4\omega_1\beta-3\beta$$
by  \eqref{eq:del-Pezzo-dynamic-alpha-degree-7-proof-genericpoint-bound7}.
We claim that $4\omega_1\beta-3\beta\leq 1$. Indeed, if $\beta\leq \frac{1}{3}$, then $4\omega_1\beta-3\beta<4\beta-3\beta\leq \beta\leq 1$ and if $\frac{1}{3}\leq \beta$, then
$$4\omega_1\beta-3\beta\leq \frac{4}{3}-3\beta\leq \frac{4}{3}-1\leq \frac{1}{3}\leq 1.$$
Conditions (ii) and (iii) of Theorem \ref{thm:pseudo-inductive-blow-up} are therefore satisfied for $i=4$ and $t_4=C^4\cap F_4=q_4$. In particular the pair
$$(S_4, (1-\beta)C^4 + \lambda\beta D^4 +(\lambda\beta(m_0+m_1+m_2+m_3)-4\beta)F_4)$$
is not log canonical at $q_4=F_4\cap C^4$.

\textbf{Case 2b: Suppose $H_1\subseteq\Supp(D)$.}

Write $D=aH_1+\Omega$ where $H_1\not\subseteq\Supp(\Omega)$ and $a>0$. Recall that
$$H_1\sim \pioplane{1}-E_1$$
with $q_1=C^1\cap H^1_1$, since $C^1\cdot H^1_1=1$. Let $x_i=\mult_{q_i}\Omega^i$. Then $m_0=a+x_0$, $m_1=a+x_1$ and $m_i=x_i$ for $i\geq 2$. We bound the multiplicities of $D$:
\begin{equation}
2=H_1\cdot D \geq x_0+x_1
\label{eq:del-Pezzo-dynamic-alpha-degree-7-proof-genericpoint-boundA}
\end{equation}

By Lemma \ref{lem:appendix-delPezzo-deg7-dynamic-lct2}, the pair
$$(S, (1-\beta)C + \lambda\beta(2H_1+H_2+E_1))$$
is log canonical. By Lemma \ref{lem:log-convexity} either $E_1\not\subseteq\Supp(D)$ or $H_2\not\subseteq\Supp(D)$. In the first case
$$1=D\cdot E_1 \geq a \text{ and } m_0+m_1 = 2a+x_0+x_1\leq 4$$
by \eqref{eq:del-Pezzo-dynamic-alpha-degree-7-proof-genericpoint-boundA}.
If $H_2\not\subseteq\Supp(D)$, then 
$$2=D\cdot H_2\geq a+x_0 \text{ and } m_0+m_1 \leq 2(a+x_0)\leq 4.$$
In both cases, we have proved
\begin{equation}
m_0+m_1\leq 4.
\label{eq:del-Pezzo-dynamic-alpha-degree-7-proof-genericpoint-boundB}
\end{equation}
Observe that for $0<\beta\leq \frac{1}{3}$ we have
$$4\lambda\beta-\beta<4\omega_1\beta-\beta=3\beta\leq 1$$
and for $\frac{1}{3}\leq\beta\leq 1$ inequality
$$4\omega\beta-\beta<4\omega_1\beta-\beta=\frac{4}{3}-\beta\leq 1$$
also holds. Therefore, using \eqref{eq:del-Pezzo-dynamic-alpha-degree-7-proof-genericpoint-boundB} we have proven
$$\lambda\beta(m_0+m_1)-\beta\leq 4\lambda\beta-\beta\leq 1.$$
This is a necessary condition to apply Theorem \ref{thm:pseudo-inductive-blow-up} (ii)--(iv) for $i\geq 2$. We apply Theorem \ref{thm:pseudo-inductive-blow-up} (ii) with $i=2$ to deduce that $t_2=F_2\cap C^2=q_2$. 

Furthermore, since
$$\lambda\beta(m_0+m_1)-2\beta\leq \lambda\beta(m_0+m_1)-\beta\leq 1,$$
Theorem \ref{thm:pseudo-inductive-blow-up} gives that the pair
$$(S_3, (1-\beta)C^3 + \lambda\beta D^3 +(\lambda\beta(m_0+m_1)-2\beta)F_2 + (\lambda\beta(m_0+m_1+m_2)-3\beta)F_3)$$
is not log canonical at some $t_3\in F_3$. 

Observe that \eqref{eq:del-Pezzo-dynamic-alpha-degree-7-proof-genericpoint-boundA} implies
\begin{equation}
2\geq x_0+x_1\geq2x_2
\label{eq:del-Pezzo-dynamic-alpha-degree-7-proof-genericpoint-boundC}
\end{equation}
so $x_2\leq 1$, but using \eqref{eq:del-Pezzo-dynamic-alpha-degree-7-proof-genericpoint-boundB} the following inequality holds:
\begin{equation}
m_0+m_1+m_2\leq 4+x_2\leq 5.
\label{eq:del-Pezzo-dynamic-alpha-degree-7-proof-genericpoint-boundC2}
\end{equation}
This implies
\begin{equation}
\lambda\beta(m_0+m_1+m_2)-3\beta\leq \lambda\beta(m_0+m_1+m_2)-2\beta\leq 5\lambda\beta-2\beta<5\omega_1\beta-2\beta\leq 1.
\label{eq:del-Pezzo-dynamic-alpha-degree-7-proof-genericpoint-boundD}
\end{equation}
Indeed, for $0<\beta\leq \frac{1}{3}$, we have $5\omega_1\beta-2\beta=3\beta\leq 1$ and for $\frac{1}{3}\leq \beta\leq 1$ we have $5\omega_1\beta-2\beta=\frac{5}{3}-2\beta\leq 1$. 

Inequality \eqref{eq:del-Pezzo-dynamic-alpha-degree-7-proof-genericpoint-boundD} and Theorem \ref{thm:pseudo-inductive-blow-up} (ii) with $i=3$ give that $t_3=C^3\cap F_3=q_3$. Again, \eqref{eq:del-Pezzo-dynamic-alpha-degree-7-proof-genericpoint-boundD} and Theorem \ref{thm:pseudo-inductive-blow-up} (i) with $i=4$ gives that the pair
$$(S_4, (1-\beta)C^4 + \lambda\beta D^4 +(\lambda\beta(m_0+m_1+m_2)-3\beta)F_3^4+(\lambda\beta(m_0+m_1+m_2+m_3)-4\beta)F_4)$$
is not log canonical at some $t_4\in F_4$. By \eqref{eq:del-Pezzo-dynamic-alpha-degree-7-proof-genericpoint-boundC} and \eqref{eq:del-Pezzo-dynamic-alpha-degree-7-proof-genericpoint-boundC2} we obtain
\begin{equation}
\lambda\beta(m_0+m_1+2m_2)-4\beta\leq 6\lambda\beta-4\beta<6\omega_1\beta-4\beta\leq 1.
\label{eq:del-Pezzo-dynamic-alpha-degree-7-proof-genericpoint-boundE}
\end{equation}
We check the last inequality: for $0<\beta\leq \frac{1}{3}$ we have $6\omega_1\beta-4\beta=2\beta\leq 1$ and for $\frac{1}{3}\leq \beta \leq 1$ we have $6\omega_1\beta-4\beta=2-4\beta\leq 1$. Inequality \eqref{eq:del-Pezzo-dynamic-alpha-degree-7-proof-genericpoint-boundE} and Theorem \ref{thm:pseudo-inductive-blow-up} (iv) with $i=4$ give that $t_4=C^4\cap F_4=q_4$.

\textbf{Finishing all cases.} We have that
$$(S_4, (1-\beta)C^4 + \lambda\beta D^4 +(\lambda\beta(m_0+m_1+m_2+m_3)-4\beta)F_4)$$
is not log canonical at $t_4=C^4\cap F_4=q_4$. Applying Lemma \ref{lem:adjunction} (iii) with $C^4$ we deduce
\begin{align*}
1&<C^4\cdot(\lambda\beta D^4 + (\lambda\beta(m_0+m_1+m_2+m_3)-4\beta)F_4)\\
 &=7\lambda\beta-\lambda\beta(\sum^3_{i=0}m_i)+\lambda\beta(\sum^3_{i=0}m_i)-4\beta=7\lambda\beta-4\beta\leq 7\omega_1-4\beta.
\end{align*}
However we claim that $7\omega_1\beta-4\beta\leq 1$ which gives a contradiction. Indeed, if $0<\beta\leq \frac{1}{3}$, then $7\omega_1\beta-4\beta=3\beta\leq 1$, whereas if $\frac{1}{3}\leq \beta\leq 1$, then $7\omega_1\beta-4\beta=\frac{7}{3}-4\beta\leq 1$.
\end{proof}

We prove case (i) of Theorem \ref{thm:del-Pezzo-dynamic-alpha-degree-7}:
\begin{lem}
\label{lem:del-Pezzo-dynamic-alpha-degree-7-proof-case4}
If $C$ contains a pseudo-Eckardt point, then $\alpha(S,(1-\beta)C)=\omega_4$.
\end{lem}
\begin{proof}
Recall Notation \ref{nota:del-Pezzo-deg7}. We may assume that $\lambda<\omega_4$ and assume for contradiction that the pair
$$(S, (1-\beta)C + \lambda\beta D)$$
is log canonical in codimension $1$ and not log canonical at $q\in C$ by Lemma \ref{lem:del-Pezzo-dynamic-alpha-degree-7-proof-in-C}.

Let $p\in C$ be the pseudo-Eckardt point. The surface $S$ has has two pseudo-Eckardt points, both of them in $L$. Since $C\cdot L =1$, only one of them, $p$, lies in $C$. Without loss of generality assume that $p=C\cap E_1$.

By Lemma \ref{lem:del-Pezzo-dynamic-alpha-degree-7-proof-in-line}, the point $q$ where the pair
$$(S, (1-\beta) C +\lambda D)$$
is not log canonical belongs to a line.

\textbf{Case 1: Suppose $q\neq p$.} Then $q=E_2\cap C = r_2\not\in L$. There is a birational morphism $\sigma\colon S \ra \bbP^1\times \bbP^1$ which contracts precisely $L$ to a point. Since $q\not \in L$, the morphism $\sigma$ is an isomorphism around $q$. Let $\bar D = \sigma_*(D)\simq -K_{\bbP^1\times\bbP^1}, \bar C =\sigma_*(C)\sim -K_{\bbP^1\times\bbP^1}$, and $\bar C$ is a smooth curve. Since $\bar{\sigma}$ is an isomorphism around $q$, the pair
$$(\bbP^1\times\bbP^1, (1-\beta)\bar C + \lambda\beta\bar D)$$
is not log canonical at $\bar q=\sigma(q)$. However
$$\lambda<\omega_4\leq \omega_3\leq\alpha(\bbP^1\times\bbP^1, (1-\beta)\bar C)$$
which contradicts Theorem \ref{thm:del-Pezzo-dynamic-alpha-degree-8-F0}.

\textbf{Case 2: $p=q$.} There is a birational morphism $\sigma\colon S \ra \bbF_1$ which contracts $E_2$ to a point. Since $q\not\in E_2$, the morphism $\sigma$ is an isomorphism around $q$. Let $\bar D = \sigma_*(D)\simq-K_{\bbF_1}$ and $\bar C = \sigma_*(C)\sim-K_S$. Observe the curve $\bar C$ is smooth since $C \cdot E_2=1$. The unique $(-1)$-curve of $\bbF_1$ is $\bar E_1=\sigma_*(E_1)$. The curve $\bar L =\sigma_*(L)$ is the unique fibre through the point $\bar q = \bar C \cap \bar E$ of the unique $\bbP^1$-fibration $\gamma\colon\bbF_1\ra \bbP^1$. Observe that $(\bar L \cdot \bar C ) \vert_{\bar q}= (L\cdot C )\vert_{q}=1$, since $\sigma$ is an isomorphism around $q$ and $\bar q$. Therefore, Theorem \ref{thm:del-Pezzo-dynamic-alpha-degree-8-F1} gives that
\begin{equation}
\alpha(\bbF_1, (1-\beta)\bar C)=\omega_4.
\label{eq:del-Pezzo-dynamic-alpha-degree-7-proof-pEckardt}
\end{equation}

On the other hand, since $\sigma$ is an isomorphism near $q$ and $\bar q$, the pair
$$(\bbF_1, (1-\beta)\bar C + \lambda\beta \bar D)$$
is not log canonical at $\bar q = \sigma_*(q)$, but $\lambda<\omega_4$. This contradicts \eqref{eq:del-Pezzo-dynamic-alpha-degree-7-proof-pEckardt}.
\end{proof}
\begin{rmk}
\label{rmk:delPezzo-dynamic-alpha-degree-7-pseudo-Eckardt}
There is an alternative proof to Lemma \ref{lem:del-Pezzo-dynamic-alpha-degree-7-proof-case4} for the case in which $p=q$. The new proof is somewhat independent of the classification of del Pezzo surfaces and the computation of their dynamic alpha--invariants. Since $p=E_1\cap L$ and $(S, (1-\beta)C +\lambda\beta D)$ is not log canonical at $p$ for some $\lambda<\omega_3$ then $E_1, L\subseteq \Supp(D)$, since otherwise
$$1\geq 1-\beta+\lambda\beta=((1-\beta)C + \lambda\beta D)\cdot L>1,$$
by Lemma \ref{lem:adjunction} (i), which is clearly a contradiction. The proof for $E_1$ is the same. Therefore we write $D=aE_1+bL+\Omega$ where $a,b>0$ and $E_1,L\not\subseteq\Supp(D)$. Then
$$1=D\cdot E_1\geq -a+b+\mult_p\Omega\geq -a+b,$$
$$1=D\cdot E_1\geq a-b+\mult_p\Omega\geq a-b.$$
Adding these two equation we see that inequality $1\geq \mult_p \Omega$ holds. Let $Q$ be the strict transform of a general conic in $\bbP^2$ passing through $p_1$ and such that $\pi(Q)\cdot \pi(L)\vert_{p_1}\geq 2$ (i.e. $\pi(Q)$ and $\pi(L)$ are tangent at $p_1$). Equivalently, $Q\in \vert\pioplane{2}-E_1\vert$ is a general element. The curve $Q\not\subseteq\Supp(D)$, since $Q$ is general. Hence 
$$4=Q\cdot D\geq a+b+x_0\geq a+b.$$
Since $E\cdot L=1$ and $\mult_p\Omega\leq 1$, then we may apply Theorem \ref{thm:inequality-Cheltsov} to $(S, (1-\beta)C + \lambda \beta D)$ at $p=L\cap E_1$ and conclude that either
$$2(1-\lambda\beta a)<L\cdot ((1-\beta)C + \lambda\beta\Omega), \text{ or}$$
$$2(1-\lambda\beta b)<E_1\cdot ((1-\beta)C + \lambda\beta\Omega)$$
hold. Since the roles of $a$ and $b$ are symmetric, it is enough to disprove the first equation to achieve a contradiction and finish the proof. Indeed if the first equation holds, then, since $\lambda<\omega_4$ we have
$$2(1-\omega_4\beta a)<1-\beta+\omega_4\beta(1+b-a).$$
In particular
$$1+\beta<\omega_4\beta(1+a+b)\leq 5\omega_4\beta\leq 1+\beta,$$
a contradiction. The last inequality is easy to see. If $0<\beta\leq \frac{1}{4}$, then $5\omega_1\beta=5\beta=4\beta+\beta\leq 1+\beta$. If $\frac{1}{4}\leq\beta\leq \frac{2}{3}$, then $5\omega_1\beta=1+\beta$ and if $\frac{2}{3}\leq \beta\leq 1$, then $5\omega_1\beta=\frac{5}{3}=1+\frac{2}{3}\leq 1+\beta$.
\end{rmk}

We prove case (ii) of Theorem \ref{thm:del-Pezzo-dynamic-alpha-degree-7}:
\begin{lem}
\label{lem:del-Pezzo-dynamic-alpha-degree-7-proof-case3}
If $C$ does not contains a pseudo-Eckardt point but $(C\cdot L_i)\vert_{r_i}=2$ for some $i=1,2$ then $\alpha(S,(1-\beta)C)=\omega_3$.
\end{lem}
\begin{proof}
Recall Notation \ref{nota:del-Pezzo-deg7}. Suppose $\lambda<\omega_3$. By Lemma \ref{lem:del-Pezzo-dynamic-alpha-degree-7-proof-in-line}, the point $q\in C$ where the pair
$$(S, (1-\beta) C +\lambda D)$$
is not log canonical belongs to a line.

\textbf{Case 1: Suppose $q\in E_i$ for $i=1$ or $i=2$}. Without loss of generality suppose $q\in E_1$. Then, by assumption $q\not\in L$. There is a birational morphism $\sigma \colon S \ra \bbP^1\times\bbP^1$ which contracts $L$ to a point. Since $q\not \in L$, the morphism $\sigma$ is an isomorphism around $q$. Let $\bar D = \sigma_*(D)\simq -K_{\bbP^1\times\bbP^1}, \bar C =\sigma_*(C)\sim -K_{\bbP^1\times\bbP^1}$, and $\bar C$ is a smooth curve, since $C\cdot L=1$. Since $\bar{\sigma}$ is an isomorphism around $q$, the pair
$$(\bbP^1\times\bbP^1, (1-\beta)\bar C + \lambda\beta\bar D)$$
is not log canonical at $\bar q=\sigma(q)$. However
$$\lambda<\omega_3=\alpha(\bbP^1\times\bbP^1, (1-\beta)\bar C)$$
which contradicts Theorem \ref{thm:del-Pezzo-dynamic-alpha-degree-8-F0}.

\textbf{Case 2: Suppose $q\not\in E_i$.} Then $q\in L$ by Lemma \ref{lem:del-Pezzo-dynamic-alpha-degree-7-proof-in-line}. There is a birational morphism $\sigma\colon S \ra \bbF_1$ which contracts $E_2$ to a point. Since $q\not\in E_2$, the morphism $\sigma$ is an isomorphism around $q$. Let $\bar D = \sigma_*(D)\simq-K_{\bbF_1}$ and $\bar C = \sigma_*(C)\sim-K_{\bbF_1}$. Observe the curve $\bar C$ is smooth since $\bar C \cdot E_2=1$. The unique $(-1)$-curve of $\bbF_1$ is $\bar E_1=\sigma_*(E_1)$. Let $\bar L=\sigma_*(L)$. Observe that $\sigma$ is an isomorphism around $r_1$, since $r_1\not\in E_2$. The curve $\bar L_1=\sigma_*(L_1)$ is the unique fibre through $\bar r_1=\sigma_*(r_1)=\bar C \cap \bar E_1$ of the unique $\bbP^1$-fibration $\gamma\colon\bbF_1\ra \bbP^1$. Observe that $(\bar L_1 \cdot \bar C ) \vert_{\bar r_1}= (L_1\cdot C )\vert_{r_1}=2$, since $\sigma$ is an isomorphism around $q$ and $\bar q$. Therefore, Theorem \ref{thm:del-Pezzo-dynamic-alpha-degree-8-F1} gives that
\begin{equation}
\alpha(\bbF_1, (1-\beta)\bar C)=\epsilon:=
					\begin{dcases}
							1 												&\text{ for } 0<\beta\leq \frac{1}{6},\\
							\frac{1+2\beta}{8\beta}		&\text{ for } \frac{1}{6}\leq \beta\leq \frac{5}{6},\\
							\frac{1}{3\beta}					&\text{ for } \frac{5}{6}\leq \beta\leq 1.\\
					\end{dcases}\geq\omega_3.
\label{eq:del-Pezzo-dynamic-alpha-degree-7-proof-psemispecial}
\end{equation}

On the other hand, since $\sigma$ is an isomorphism near $q$ and $\bar q$, the pair
$$(\bbF_1, (1-\beta)\bar C + \lambda\beta \bar D)$$
is not log canonical at $\bar q = \sigma_*(q)$, but we may choose $\lambda$ so that $\epsilon\leq\lambda<\omega_3$. This contradicts \eqref{eq:del-Pezzo-dynamic-alpha-degree-7-proof-psemispecial}.
\end{proof}

For the rest of the proof of Theorem \ref{thm:del-Pezzo-dynamic-alpha-degree-7} we rule out the different positions of $q$, the point at which
$$(S, (1-\beta)C + \lambda\beta D )$$
is not log canonical. We will apply Theorem \ref{thm:pseudo-inductive-blow-up} several times on successive blow-ups of $q$ and points over $q$. Since $q\in C$, by Lemma \ref{lem:del-Pezzo-dynamic-alpha-degree-7-proof-case4} we may assume $q\neq E_1\cap L$.

Let $S_0=S$, $C_0=C$, $D_0=D$ and $q_0=q$ as in Theorem \ref{thm:pseudo-inductive-blow-up}. Let $i\geq 1$ and let $f_i\colon S_i\ra S_{i-1}$ be the blow-up of the point $q_{i-1}=C^i\cap F_{i-1}$ with exceptional curve $F_i$. Let $A^{i-1}$ or $A$ be any $\bbQ$-divisor in $S_{i-1}$. We will denote its strict transform in $S^i$ by $A^i$. Let $m_i=\mult_{q_i} D_i$. Recall from \eqref{eq:del-Pezzo-dynamic-alpha-degree-7-proof-bad-pair-blowup} that the pair
$$(S_1, (1-\beta)C^1 + \lambda\beta D^1 + (\lambda\beta m_0-\beta)F_1)$$
is not log canonical at some $t_1\in F_1$.

\begin{lem}
\label{lem:del-Pezzo-dynamic-alpha-degree-7-L}
The point $q\not\in L$.
\end{lem}
\begin{proof}
Suppose $q\in L$. Recall Notation \ref{nota:del-Pezzo-deg7}. If $(C\cdot L_i)\vert_{r_i}=1$ for $i=1$ or $i=2$, then $\lambda<\omega_1$ and if $(C\cdot L_1)\vert_{r_1}=2$, or equivalently for $L_2$, then $\lambda<\omega_2$. Observe that since $q\in C$ and $C\cap L=r$, then $q=r$.

The curve $L\subseteq \Supp(D)$, since otherwise
$$1=L\cdot D \geq \mult_{q} D>\frac{1-(1-\beta)}{\lambda\beta} = \frac{1}{\lambda}>1,$$
by Lemma \ref{lem:adjunction} (i). Therefore, we write $D=aL+\Omega$ where $L\not\subseteq \Supp(\Omega).$ Let $x_i=\mult_{q_i}\Omega^i$ for $i\geq 0$. Since $L\cdot C=1$, then $m_0=a+x_0$ and $m_i=x_i$ for $i\geq 1$.

By Lemma \ref{lem:appendix-delPezzo-deg7-dynamic-lct5}, the pair
$$(S, (1-\beta)C + \lambda\beta (L+2R))$$
is log canonical. By Lemma \ref{lem:log-convexity}, we may assume that $R\not\subseteq\Supp(D)$. Hence
\begin{align}
&3=R\cdot D \geq a+ x_0+x_1=m_0+m_1  &\text{ if } (R\cdot C)\vert_{q_0}=2,\nonumber\\
&3=R\cdot D \geq a+ x_0+x_1+x_2=m_0+m_1 +m_2 &\text{ if } (R\cdot C)\vert_{q_0}=3, 
\label{eq:del-Pezzo-dynamic-alpha-degree-7-proof-L-boundA}
\end{align}
since $R\cdot L=1$ and $3\geq (R\cdot C)\vert_{q_0}\geq 2$. In particular
$$\lambda\beta m_0<\omega_1\beta m_0\leq 3\omega_1\beta\leq 1$$
by Lemma \ref{lem:del-Pezzo-dynamic-alpha-degree-7-proof-in-C}. This is one of the hypotheses of Theorem \ref{thm:pseudo-inductive-blow-up} (ii)-(iv) when$ i\geq2$ ande we will assume it from now onwards. Observe that
$$1=L\cdot D \geq -a+x_0$$
which adding it to \eqref{eq:del-Pezzo-dynamic-alpha-degree-7-proof-L-boundA} implies
\begin{align*}
&4\geq 2x_0+x_1\geq 3x_1 &\text{ if } (R\cdot C)\vert_{q_0}&=2 \text{ and }\\
&4\geq 2x_0+x_1+x_2\geq 4x_2 &\text{ if } (R\cdot C)\vert_{q_0}&=3.
\end{align*}
Therefore
\begin{equation}
\frac{4}{3}\geq m_1= x_1\geq x_2=m_2 \text{ if } (R\cdot C)\vert_{q_0}=2 \text{ and }1\geq x_2\geq x_3=m_3 \text{ if } (R\cdot C)\vert_{q_0}=3.
\label{eq:del-Pezzo-dynamic-alpha-degree-7-proof-L-boundB}
\end{equation}

As we will see, \eqref{eq:del-Pezzo-dynamic-alpha-degree-7-proof-L-boundA} and \eqref{eq:del-Pezzo-dynamic-alpha-degree-7-proof-L-boundB} are very strong inequalities. 

\textbf{Case 1: Suppose $(R\cdot C)\vert_{q_0}=3$.} Then
$$\lambda\beta(\sum_{j=0}^{i-1} m_j)-(i-1)\beta\leq \lambda\beta(\sum_{j=0}^{i-1} m_j)\leq 3\lambda\beta\leq 1$$
for $i\leq 3$. Moreover:
$$\lambda\beta(\sum_{j=0}^{3} m_j)-3\beta<4\omega_i\beta-3\beta\leq 4\omega_1\beta-3\beta\leq1.$$
Indeed, if $0<\beta\leq \frac{1}{3}$, then $4\omega_1\beta-3\beta=\beta\leq 1$, while if $\frac{1}{3}\leq \beta\leq 1$, then $4\omega_1\beta-3\beta=\frac{4}{3}-3\beta\leq 1$. Hence, by Theorem \ref{thm:pseudo-inductive-blow-up} (iii) with $i=1,\ldots, 4$, the pair
$$(S_i, (1-\beta)C^i+ \lambda\beta D^i + (\lambda\beta(\sum_{j=0}^{i-2}m_j)-(i-1)\beta)F_{i-1}^i+(\lambda\beta(\sum_{j=0}^{i-1}m_j)-i\beta)F_i)$$
is not log canonical at $q_i=F_i\cap C^i$ for $i=1,\ldots,4$. In particular, for $i=4$, the pair
$$(S_4, (1-\beta)C^4 + \lambda\beta D^4 + (\lambda\beta(m_0+m_1+m_2+m_3)-4\beta)F_4)$$
is not log canonical at $q_4=F_4\cap C^4$. But then we apply Lemma \ref{lem:adjunction} (iii) with $C^4$ to obtain a contradiction:
\begin{align}
1&<C^4\cdot (\lambda\beta D^4 + (\lambda\beta(m_0+m_1+m_2+m_3)-4\beta)F_4)\nonumber\\
&=7\lambda\beta-4\beta\leq 7\omega_i\beta-4\beta\leq 7\omega_1\beta-4\beta\leq 1.
\label{eq:del-Pezzo-dynamic-alpha-degree-7-proof-L-boundC}
\end{align}
The last inequality is easy to obtain: if $0<\beta\leq\frac{1}{3}$, then $7\omega_1\beta-4\beta=3\beta\leq 1$ and if $\frac{1}{3}\leq\beta\leq 1$, then $7\omega_1\beta-4\beta=\frac{7}{3}-4\beta\leq 1$. Thus, the lemma is proven when $(R\cdot C)\vert_{q_0}=3$.

 \textbf{Suppose that $(R\cdot C)\vert_{q_0}=2$.} Inequalities \eqref{eq:del-Pezzo-dynamic-alpha-degree-7-proof-L-boundA} and \eqref{eq:del-Pezzo-dynamic-alpha-degree-7-proof-L-boundB} together with Lemma \ref{lem:del-Pezzo-dynamic-alpha-degree-7-proof-in-C} imply
\begin{align*}
&\lambda\beta(m_0+m_1)-\beta< 3\omega_1\beta-\beta\leq 3\omega_1\beta\leq 1		&(i=2),\\
&\lambda\beta(m_0+m_1+m_2)-2\beta<\frac{13}{3}\omega_1\beta-2\beta \leq1			&(i=3),\\
&\lambda\beta(m_0+m_1+m_2+m_3)-3\beta<\frac{17}{3}\omega_1\beta-3\beta\leq1		&(i=4).
\end{align*}
Indeed, if $0<\beta\leq \frac{1}{3}$, then $\frac{13}{3}\omega_1\beta-2\beta=\frac{7}{3}\beta\leq \frac{7}{9}\leq 1$ and $\frac{17}{3}\omega_1\beta-3\beta=\frac{8}{3}\beta\leq \frac{8}{9}\leq 1$, while if $\frac{1}{3}\leq \beta\leq 1$, then $\frac{13}{3}\omega_1\beta-2\beta=\frac{13}{9}-2\beta \leq \frac{7}{9}\leq 1$ and $\frac{17}{3}\omega_1\beta-3\beta=\frac{17}{9}-3\beta\leq \frac{8}{9}\leq 1$.

Hence, by Theorem \ref{thm:pseudo-inductive-blow-up} (iii) with $i=1,\ldots, 4$, the pair
$$(S_i, (1-\beta)C^i+ \lambda\beta D^i + (\lambda\beta(\sum_{j=0}^{i-2}m_j)-(i-1)\beta)F_{i-1}^i+(\lambda\beta(\sum_{j=0}^{i-1}m_j)-i\beta)F_i)$$
is not log canonical at $q_i=F_i\cap C^i$ where $i=1,\ldots,4$. In particular, for $i=4$, the pair
$$(S_4, (1-\beta)C^4 + \lambda\beta D^4 + (\lambda\beta(m_0+m_1+m_2+m_3)-4\beta)F_4)$$
is not log canonical at $q_4=F_4\cap C^4$. But then we apply Lemma \ref{lem:adjunction} (iii) with $C^4$ to obtain a contradiction:
\begin{align*}
1&<C^4\cdot (\lambda\beta D^4 + (\lambda\beta(m_0+m_1+m_2+m_3)-4\beta)F_4)\nonumber\\
&=7\lambda\beta-4\beta< 7\omega_1\beta-4\beta\leq 1,
\end{align*}
as in \eqref{eq:del-Pezzo-dynamic-alpha-degree-7-proof-L-boundC}. The last inequality is easy to obtain: if $0<\beta\leq\frac{1}{3}$, then $7\omega_1\beta-4\beta=3\beta\leq 1$ and if $\frac{1}{3}\leq\beta\leq 1$, then $7\omega_1\beta-4\beta=\frac{7}{3}-4\beta\leq 1$.
\end{proof}

Therefore $q\in (E_1\cup E_2)\setminus L$, since we may relabel these exceptional curves, we may assume without loss of generality that $ q\in E_1\cap C$. The following Lemma finishes the proof of Theorem \ref{thm:del-Pezzo-dynamic-alpha-degree-7}.
\begin{lem}
\label{lem:del-Pezzo-dynamic-alpha-degree-7-proof-case1-L}
The point $q\not\in E_1$.
\end{lem}
\begin{proof}
Suppose $q=E_1\cap C$. Notice that $E_1\subseteq\Supp(D)$, since otherwise
$$1=E_1\cdot D \geq \mult_{q} D>\frac{1-(1-\beta)}{\lambda\beta} = \frac{1}{\lambda}>1,$$
by Lemma \ref{lem:adjunction} (i). Therefore, we write $D=aE_1+bL_1+\Omega$ where $E_1, L_1\not\subseteq \Supp(\Omega)$, $a>0$ and $b\geq 0$. Let $x_i=\mult_{q_i}\Omega^i$.

Recall from \eqref{eq:del-Pezzo-dynamic-alpha-degree-7-proof-bad-pair-blowup} that the pair
$$(S_1, (1-\beta)C^1+\lambda\beta D^1+ \lambda\beta(m_0-\beta)F_1)$$
is not log canonical at some $t_1\in F_1$, where $\lambda<\omega_1$ if $(C\cdot L_1)\vert_{r_1}=1$ and $(C\cdot L_1)\vert_{r_1}=2$ if $\lambda<\omega_2$.

By Lemma \ref{lem:appendix-delPezzo-deg7-dynamic-lct6}, the pair
$$(S, (1-\beta)C + \lambda\beta(2L_1+L+2E_1))$$
is log canonical. Observe that $2L_1+L+2E_1\sim-K_S$. By Lemma \ref{lem:log-convexity}, either $L\not\subseteq\Supp(D)$ or $L_1\not\subseteq\Supp(D)$ and either
\begin{equation}
1=L\cdot D \geq a, \text{ or } b=0.
\label{eq:del-Pezzo-dynamic-alpha-degree-7-proof-final-bound1}
\end{equation}
In either case, since $(L_1)^2=0$ we have
\begin{align}
2&=L_1\cdot D \geq a+x_0+x_1, \label{eq:del-Pezzo-dynamic-alpha-degree-7-proof-final-bound2}\\
1&=E_1\cdot D \geq -a+b+x_0. \label{eq:del-Pezzo-dynamic-alpha-degree-7-proof-final-bound3}
\end{align}
Observe that if $a\leq 1$, adding \eqref{eq:del-Pezzo-dynamic-alpha-degree-7-proof-final-bound1} twice and \eqref{eq:del-Pezzo-dynamic-alpha-degree-7-proof-final-bound3}, we obtain $m_0=a+b+x_0\leq 3$ while if $b=0$, then \eqref{eq:del-Pezzo-dynamic-alpha-degree-7-proof-final-bound2} gives $m_0=a+b+x_0+x_1\leq 2<3$. Therefore
$$\lambda\beta m_0-\beta\leq\lambda\beta m_0<3\omega_1\beta\leq 1,$$
by Lemma \ref{lem:del-Pezzo-dynamic-alpha-degree-7-proof-in-C}. This is a necessary condition for Theorem \ref{thm:pseudo-inductive-blow-up} (ii), (iii) and (iv) when $i\geq 2$ and we will assume it from now on when invoking that Lemma.

By Theorem \ref{thm:pseudo-inductive-blow-up} (ii) with $i=1$ we have $t_1=F_1\cap C^1=q_1$. Moreover, by Theorem \ref{thm:pseudo-inductive-blow-up} (i), the pair
$$(S_2, (1-\beta)C^2+\lambda\beta D^2 + (\lambda\beta m_0 -\beta)F_1^2 + (\lambda\beta(m_0+m_1)-2\beta)F_2)$$
is not log canonical at some $t_2\in F_2$. 

\textbf{Case 1:  Suppose that $(C\cdot L_1)_{r_1}=1$.}
Since $(C\cdot L_1)_{r_1}=1$, then $\lambda<\omega_1$. If $b=0$, then
$$\lambda\beta(m_0+m_1)-\beta \leq \lambda\beta(a+x_0+x_1)-\beta<2\omega_1\beta-\beta\leq 4\omega_1\beta-\beta\leq 1$$
by \eqref{eq:del-Pezzo-dynamic-alpha-degree-7-proof-final-bound1}. If $a\leq 1$, then
\begin{equation}
\lambda\beta(m_0+m_1)-\beta \leq \lambda\beta(a+b+2x_0+x_1)-\beta<4\omega_1\beta-\beta\leq 1
\label{eq:del-Pezzo-dynamic-alpha-degree-7-proof-final-bound4}
\end{equation}
by \eqref{eq:del-Pezzo-dynamic-alpha-degree-7-proof-final-bound1}, \eqref{eq:del-Pezzo-dynamic-alpha-degree-7-proof-final-bound2} and \eqref{eq:del-Pezzo-dynamic-alpha-degree-7-proof-final-bound3}. The last inequalities follow by case analysis: if $0<\beta\leq \frac{1}{3}$, then $4\omega_1\beta-\beta=3\beta\leq 1$ while if $\frac{1}{3}\leq\beta\leq1$, then $4\omega_1\beta-\beta=\frac{4}{3}-\beta\leq 1$. Hence, by Theorem \ref{thm:pseudo-inductive-blow-up} (iii) with $i=2$, we have $t_2= F_2\cap C^2=q_2$. Moreover, since 
$$\lambda\beta(m_0+m_1)-2\beta<\lambda\beta(m_0+m_1)-\beta\leq 1$$
then Theorem \ref{thm:pseudo-inductive-blow-up} (i) with $i=3$ gives that the pair
$$(S_3, (1-\beta)C^3 + \lambda\beta D^3 +(\lambda\beta(m_0+m_1)-2\beta)F_2 + (\lambda\beta(m_0+m_1+m_2)-3\beta)F_3)$$
is not log canonical at some point $t_3\in F_3$. 

If $b=0$, then
$$\lambda\beta(m_0+m_1+m_2)-2\beta\leq \lambda\beta(a+x_0+x_1+x_2)-2\beta\leq \lambda\beta(2a+2x_0+2x_1)-2\beta\leq 4\omega_1\beta-2\beta\leq 1$$
by \eqref{eq:del-Pezzo-dynamic-alpha-degree-7-proof-final-bound1} and \eqref{eq:del-Pezzo-dynamic-alpha-degree-7-proof-final-bound4}. If $a\leq 1$, then
$$\lambda\beta(m_0+m_1+m_2)-2\beta\leq \lambda\beta(a+b+2x_0+x_1)-2\beta\leq 4\omega_1\beta-2\beta\leq 1$$
by \eqref{eq:del-Pezzo-dynamic-alpha-degree-7-proof-final-bound1}, \eqref{eq:del-Pezzo-dynamic-alpha-degree-7-proof-final-bound2}, \eqref{eq:del-Pezzo-dynamic-alpha-degree-7-proof-final-bound3} and \eqref{eq:del-Pezzo-dynamic-alpha-degree-7-proof-final-bound4}. Applying Theorem \ref{thm:pseudo-inductive-blow-up} (iii) with $i=3$, we conclude that $t_3=F_3\cap C^3=q_3$.

Now, notice that if $b=0$, then
$$\lambda\beta(m_0+m_1+m_2+m_3)-3\beta<\omega_1 \beta(a+x_0+x_1+x_2+x_3)-3\beta\leq 6\omega_1\beta-3\beta\leq 1$$
by \eqref{eq:del-Pezzo-dynamic-alpha-degree-7-proof-final-bound1} while if $a\leq 1$, then
\begin{equation}
\lambda\beta(m_0+m_1+m_2+m_3)-3\beta<\omega_1 \beta(a+b+x_0+x_1+2x_2)-3\beta\leq 6\omega_1\beta-3\beta\leq 1
\label{eq:del-Pezzo-dynamic-alpha-degree-7-proof-final-bound7}
\end{equation}
by \eqref{eq:del-Pezzo-dynamic-alpha-degree-7-proof-final-bound1}, \eqref{eq:del-Pezzo-dynamic-alpha-degree-7-proof-final-bound2} and \eqref{eq:del-Pezzo-dynamic-alpha-degree-7-proof-final-bound3}. The last part of each inequality follows by case analysis: if $0<\beta\leq\frac{1}{3}$, then $6\omega_1\beta-3\beta=3\beta\leq 1$, while if $\frac{1}{3}\leq \beta\leq 1$, then $6\omega_1\beta-3\beta=2-3\beta\leq 1$.

Hence, by Theorem \ref{thm:pseudo-inductive-blow-up} (iii) with $i=4$, the pair
$$(S_4, (1-\beta)C^4 + \lambda\beta D^4 +(\lambda\beta(m_0+m_1+m_2+m_3)-3\beta)F_3 + (\lambda\beta(m_0+m_1+m_2+m_3)-4\beta)F_4)$$
is not log canonical at $F_4\cap C^4=q_4$. But then, the pair
$$(S_4, (1-\beta)C^4 + \lambda\beta D^4 + (\lambda\beta(m_0+m_1+m_2+m_3)-4\beta)F_4)$$
is not log canonical at $q_4$. Lemma \ref{lem:adjunction} (iii) applied with $C^4$ gives:
\begin{align*}
1&<C^4\cdot(\lambda\beta D^4 + (\lambda\beta(m_0+m_1+m_2+m_3)-4\beta)F_4)\\
 &=7\lambda\beta-\lambda\beta(\sum^3_{i=0}m_i)-+\lambda\beta(\sum^3_{i=0}m_i)-4\beta<7\omega_1\beta-4\beta.
\end{align*}
However we claim that $7\omega_1\beta-4\beta\leq 1$ which gives a contradiction. Indeed, if $0<\beta\leq \frac{1}{3}$, then $7\omega_1\beta-4\beta<3\beta\leq 1$, whereas if $\frac{1}{3}\leq \beta\leq 1$, then $7\omega_1\beta-4\beta<\frac{7}{3}-4\beta\leq 1$.

\textbf{Case 2:  Suppose that $(C\cdot L_1)_{r_1}=2$.}
Notice that $\lambda<\omega_3$ by the statement of Theorem \ref{thm:del-Pezzo-dynamic-alpha-degree-7}. Recall that $D=aE_1+bL_1+\Omega$ and $q_i=C^i\cap F_i$. Then $m_0=a+b+x_0$, $m_1=b+x_1$ and $m_i=x_i$ for $i\geq 2$.

Recall \eqref{eq:del-Pezzo-dynamic-alpha-degree-7-proof-final-bound1}. If $a\leq 1$, then adding \eqref{eq:del-Pezzo-dynamic-alpha-degree-7-proof-final-bound3} and two times \eqref{eq:del-Pezzo-dynamic-alpha-degree-7-proof-final-bound1} we obtain
$$m_0=a+b+x_0\leq 3,$$
whereas if $b=0$, then by \eqref{eq:del-Pezzo-dynamic-alpha-degree-7-proof-final-bound2}, then
$$m_0=a+x_0+\leq a+x_0+x_1\leq 2<3.$$
In both cases we have
$$\lambda\beta m_0-2\beta<6\omega_3\beta-2\beta\leq 1.$$
Indeed, for $0<\beta\leq \frac{1}{4}$, then $6\omega_3\beta-2\beta=4\beta\leq 1$. For $\frac{1}{4}\leq \beta\leq \frac{1}{2}$, we have $6\omega_3\beta-2\beta=1+2\beta-2\beta=1$ and for $\frac{1}{2}\leq \beta\leq 1$, we have $6\omega_3\beta-2\beta=2-2\beta\leq 1$.

Therefore we may apply Theorem \ref{thm:pseudo-inductive-blow-up} (iv) with $i=2$ to conclude that $t_2=C^2\cap F_2=q_2$. Since
$$\lambda\beta(m_0+m_1)-2\beta\leq 2\lambda\beta m_0-2\beta\leq 1$$
we may apply Theorem \ref{thm:pseudo-inductive-blow-up} (i) with $i=3$ to conclude that the pair
$$(S_3, (1-\beta)C^3+\lambda\beta D^3 +(\lambda\beta(m_0+m_1)-2\beta)F_2^3 +(\lambda\beta(m_0+m_1+m_2)-3\beta)F_3)$$
is not log canonical at some $t_3\in F_3$. Observe \eqref{eq:del-Pezzo-dynamic-alpha-degree-7-proof-final-bound1}. If $b=0$, then \eqref{eq:del-Pezzo-dynamic-alpha-degree-7-proof-final-bound2} gives
$$m_0+m_1+m_2=a+x_0+x_1+x_2\leq2(a+x_0+x_1)\leq 4<6 $$
while if $a\leq 1$, then adding \eqref{eq:del-Pezzo-dynamic-alpha-degree-7-proof-final-bound1}, \eqref{eq:del-Pezzo-dynamic-alpha-degree-7-proof-final-bound2} and two times \eqref{eq:del-Pezzo-dynamic-alpha-degree-7-proof-final-bound3}, we obtain
$$m_0+m_1+m_2=a+2b+x_0+x_1+x_2\leq6.$$
Therefore we have proven
$$\lambda\beta(m_0+m_1+m_2)-2\beta<6\omega_3-2\beta\leq 1.$$
The last inequality follows by case analysis. If $0<\beta\leq \frac{1}{4}$, then $6\omega_3\beta-2\beta=4\beta\leq 1$. If $\frac{1}{4}\leq \beta\leq \frac{1}{2}$, then $6\omega_3\beta-2\beta=1+2\beta-2\beta=1$ and if $\frac{1}{2}\leq \beta\leq 1$, then $6\omega_3\beta-2\beta=2-2\beta\leq 1$.

Hence we may apply Theorem \ref{thm:pseudo-inductive-blow-up} (iii) with $i=3$ to obtain that the pair
$$(S_3, (1-\beta)C^3 + \lambda\beta D^3 +(\lambda\beta(m_0+m_1)-2\beta)F_2 + (\lambda\beta(m_0+m_1+m_2)-3\beta)F_3)$$
is not log canonical only at $t_3=F_3\cap C^3=q_3$. In particular, the pair
$$(S_3, (1-\beta)C^3 + \lambda\beta D^3 + (\lambda\beta(m_0+m_1+m_2)-3\beta)F_3)$$
is not log canonical at $q_3$. 

Lemma \ref{lem:adjunction} (iii) applied with $C^4$ gives:
\begin{align*}
1&<C^3\cdot(\lambda\beta D^3 + (\lambda\beta(m_0+m_1+m_2)-3\beta)F_3)\\
 &=7\lambda\beta-\lambda\beta(\sum^2_{i=0}m_i)-+\lambda\beta(\sum^2_{i=0}m_i)-3\beta=7\omega_3\beta-3\beta.
\end{align*}
However we claim that $7\omega_3\beta-3\beta\leq 1$ which gives a contradiction. Indeed, if $0<\beta\leq \frac{1}{4}$, then $7\omega_3\beta-3\beta=4\beta\leq 1$, whereas if $\frac{1}{3}\leq \beta\leq \frac{1}{2}$, then $7\omega_3\beta-3\beta=\frac{7(1+2\beta)}{6}-3\beta=\frac{7-4\beta}{6}\leq 1$. Finally, if $\frac{1}{2}\leq\beta\leq 1$, then $7\omega_3\beta-3\beta=\frac{7}{3}-3\beta\leq \frac{5}{6}<1$.
\end{proof}

Therefore Theorem \ref{thm:del-Pezzo-dynamic-alpha-degree-7} is proven. Together with Theorem \ref{thm:Jeffres-Mazzeo-Rubinstein}, we obtain the following Corollary:
\begin{cor}
Let $S$ be a non-singular del Pezzo surface of degree $7$ and $C$ a general curve in $\vert-K_S\vert$. Then $(S, (1-\beta)C)$ has a K\"ahler--Einstein metric with edge singularities of angle $2\pi\beta$ along $C$ for all $0<\beta<\frac{1}{2}$.
\end{cor}
This result is of special interest, given that $S$ does not accept a K\"ahler--Einstein metric.
\section{Del Pezzo surface of degree $6$}
\label{sec:dynamic-alpha-degree6}
In this section we will follow and extend Notation \ref{nota:del-Pezzos}. Let $S$ be a non-singular del Pezzo surface of degree $6$. Given any model $\pi\colon S \ra \bbP^2$ we have exceptional curves $E_1,E_2,E_3\subset S$ mapping to points $p_1,p_2,p_3\in \bbP^2$, respectively. The other $3$ lines in $S$ (see Lemma \ref{lem:del-Pezzo-lines9-d} correspond to strict transforms of lines in $\bbP^2$ through $p_i,p_j$. We will denote them by
$$L_{ij}\sim\pioplane{1}-E_i-E_j\text{ for } 1\leq i <j\leq 3.$$

\begin{lem}
\label{lem:del-Pezzo-dynamic-alpha-degree-6-upbound}
Let $S$ be a non-singular del Pezzo surface of degree $6$ and $C\in\vert-K_S\vert$ be a smooth curve.
\begin{itemize}
	\item[(i)] If $C$ contains a pseudo-Eckardt point of $S$, then
	\begin{equation}
			\alpha(S,(1-\beta)C)\leq\omega_3:=
					\begin{dcases}
							1 												&\text{ for } 0<\beta\leq \frac{1}{3},\\
							\frac{1+\beta}{4\beta}		&\text{ for } \frac{1}{3}\leq \beta\leq 1.
					\end{dcases}
		\label{eq:del-Pezzo-dynamic-alpha-degree-6-special-upbound}
	\end{equation}
	\item[(ii)]If $C$ contains no pseudo-Eckardt points but there is a model $\pi\colon S\ra \bbP^2$ such that through $p=C\cap E_1$ there is a smooth rational curve $L\sim\pioplane{1} -E_1$ satisfying $(C\cdot L )\vert_{p}=2$, then
				\begin{equation}
						\alpha(S,(1-\beta)C)\leq\omega_2:=
								\begin{dcases}
										1 												&\text{ for } 0<\beta\leq \frac{1}{3},\\
										\frac{1+2\beta}{5\beta}					&\text{ for } \frac{1}{3}\leq \beta\leq \frac{3}{4},\\
										\frac{1}{2\beta}					&\text{ for } \frac{3}{4}\leq \beta\leq 1.
								\end{dcases}
					\label{eq:del-Pezzo-dynamic-alpha-degree-6-semispecial-upbound}
				\end{equation}
	
	\item[(iii)]
						If $C$ contains no pseudo-Eckardt points and for all models $\pi\colon S \ra \bbP^2$ the unique irreducible curve $L\in \vert\pioplane{1} -E_1\vert$ passing through $p=C\cap E_1$ has simple normal crossings with $C$ (i.e. $(C\cdot L)\vert_{p}=E_1$), then
				\begin{equation}
						\alpha(S,(1-\beta)C)\leq\omega_1:=
								\begin{dcases}
										1 												&\text{ for } 0<\beta\leq \frac{1}{2},\\
										\frac{1}{2\beta}					&\text{ for } \frac{1}{2}\leq \beta\leq 1.
								\end{dcases}
					\label{eq:del-Pezzo-dynamic-alpha-degree-6-generic-upbound}
				\end{equation}
\end{itemize}	
\end{lem}
\begin{proof}
If $C$ contains a pseudo-Eckardt point $p$, by Lemma \ref{lem:del-Pezzo-good-model} we can choose a model $\pi\colon S \ra \bbP^2$ such that $p=E_1\cap L$, where $L$ is the other line intersecting $p$ normally. Let $\bar L=\pi_*(L)\subset \bbP^2$. Since $L\neq E_1$, then $\bar L \sim \oplane{d}$ for some $d>0$ and
\begin{equation}
L\sim \pi^*(\bar L)-\sum_{i=1}^3 \mult_{p_i}\bar L \sim \pioplane{d}-\sum_{i=1}^3 a_i E_i
\label{eq:del-Pezzo-dynamic-alpha-degree-6-upbound-line}
\end{equation}
where $p_i=\pi(E_i)$ and $a_i=\mult_{p_i}\bar L =E_i\cdot L$. In particular $a_1=1$. Since $L$ is a line, we have
\begin{align*}
1&=L\cdot(-K_S)=3d-1-a_2-a_3,\\
-1&=L^2=d^2-1-a_2-a_3.
\end{align*}
Subtracting the first equation from the second one we obtain
$$d^2-3d+2=0$$
with roots $d=1,2$. If $d=2$, then $a_2+a_3=4$, so $\exists a_j\geq 2$ since all $a_i\geq 0$ but all irreducible curves of degree $2$ in $\bbP^2$ are smooth. Hence $d=1$ and either $(a_2,a_3)=(1,0)$ or $(a_2,a_3)=(0,1)$. Without loss of generality suppose the former. Then $L=L_{12}$ is the strict transform of the line in $\bbP^2$ through $p_1$ and $p_2$. Let $L_{13}$ be the strict transform of the line in $\bbP^2$ through $p_1$ and $p_3$. Then
$$L_{13}\sim \pioplane{1}-E_1-E_3.$$
Notice that
$$D:=2E_1+2L_{12}+L_{13}+E_2\sim-K_S$$
is a divisor with simple normal crossings. Let $f\colon \widetilde S \ra S $ be the blow-up of $p=E_1\cap L_{12}$ with exceptional divisor $E$. Then $f$ is a log resolution of $(S,(1-\beta)C+\lambda\beta D)$, since
$$C\cdot E_1=C\cdot L_{12 } =C\cdot L_{13}=C\cdot E_2=1.$$
Its log pullback is
$$f^*(K_S+(1-\beta)C+\lambda\beta D)\sim-K_{\widetilde S}+(1-\beta)\widetilde C +\lambda\beta \widetilde D  + (4\lambda\beta-\beta)E$$
and therefore
\begin{align*}
\alpha(S,(1-\beta)C)\leq \min\{\lct(S,(1-\beta)C,\beta C),\lct(S,(1-\beta)C,\beta D)\}
\leq \min\{1,\frac{1}{2\beta},\frac{1+\beta}{4\beta}\}=\omega_3.
\end{align*}
If $C$ contains no pseudo-Eckardt points, the pair
$$(S,(1-\beta)C+\lambda\beta D)$$
has simple normal crossings and we obtain
\begin{align*}
\alpha(S,(1-\beta)C)\leq \min\{\lct(S,(1-\beta)C,\beta C),\lct(S,(1-\beta)C,\beta D)\}\leq \min\{1,\frac{1}{2\beta}\}=\omega_1.
\end{align*}
Finally, suppose $C$ contains no pseudo-Eckardt points but there is an irreducible curve $L\sim\pioplane{1}-E_1$ as in (ii) in the statement, for some model $\pi\colon S \ra \bbP^2$ and such that $(L\cdot C)\vert_p=2$. Let $$D=2L+L_{23}+E_1$$ where $L_{23}$ is the unique line such that $L_{23}\sim\pioplane{1}-E_1-E_2$. Then the pair $(S,(1-\beta)C + \lambda\beta D)$ has simple normal crossings away from $p=C\cap E_1\cap L$ where $(L\cdot C )\vert_{p}=2$. Let $\sigma\colon \widetilde S \ra S$ be the minimal log resolution of $(S,(1-\beta)C + \lambda\beta D)$ with exceptional divisors $F_1,F_2$. The log pullback is
$$\sigma^*(K_S+(1-\beta)C+\lambda\beta D)\simq K_{\widetilde S} + (1-\beta)\widetilde C + \lambda\beta \widetilde D + (3\lambda\beta -\beta) F_1 + (5\lambda\beta - 2\beta) F_2.$$
Therefore, we conclude
\begin{align*}
\alpha(S,(1-\beta)C)&\leq \min\{\omega_1,\lct(S,(1-\beta)C,\beta D)\}\\
&\leq \min\{1,\frac{1}{2\beta},\frac{1+2\beta}{5\beta}, \frac{1+\beta}{3\beta}\}=\min\{1,\frac{1}{2\beta},\frac{1+2\beta}{5\beta}\}=\omega_2.
\end{align*}
\end{proof}

\begin{thm}
\label{thm:del-Pezzo-dynamic-alpha-degree-6}
Let $S$ be a non-singular del Pezzo surface of degree $6$ and $C\in\vert-K_S\vert$ be a smooth curve.
\begin{itemize}
	\item[(i)] If $C$ contains a pseudo-Eckardt point of $S$, then
	\begin{equation}
			\alpha(S,(1-\beta)C)=\omega_3:=
					\begin{dcases}
							1 												&\text{ for } 0<\beta\leq \frac{1}{3},\\
							\frac{1+\beta}{4\beta}		&\text{ for } \frac{1}{3}\leq \beta\leq 1.
					\end{dcases}
		\label{eq:del-Pezzo-dynamic-alpha-degree-6-special}
	\end{equation}
	\item[(ii)]If $C$ contains no pseudo-Eckardt points but there is a model $\pi\colon S\ra \bbP^2$ such that through $p=C\cap E_1$ there is a smooth rational curve $L\sim\pioplane{1} -E_1$ satisfying $(C\cdot L )\vert_{p}=2$, then
				\begin{equation}
						\alpha(S,(1-\beta)C)=\omega_2:=
								\begin{dcases}
										1 												&\text{ for } 0<\beta\leq \frac{1}{3},\\
										\frac{1+2\beta}{5\beta}					&\text{ for } \frac{1}{3}\leq \beta\leq \frac{3}{4},\\
										\frac{1}{2\beta}					&\text{ for } \frac{3}{4}\leq \beta\leq 1.
								\end{dcases}
					\label{eq:del-Pezzo-dynamic-alpha-degree-6-semispecial}
				\end{equation}
	
	\item[(iii)]
						If $C$ contains no pseudo-Eckardt points and for all models $\pi\colon S \ra \bbP^2$ the unique irreducible curve $L\in \vert\pioplane{1} -E_1\vert$ passing through $p=C\cap E_1$ has simple normal crossings with $C$ (i.e. $(C\cdot L)\vert_{p}=E_1$), then
				\begin{equation}
						\alpha(S,(1-\beta)C)=\omega_1:=
								\begin{dcases}
										1 												&\text{ for } 0<\beta\leq \frac{1}{2},\\
										\frac{1}{2\beta}					&\text{ for } \frac{1}{2}\leq \beta\leq 1.
								\end{dcases}
					\label{eq:del-Pezzo-dynamic-alpha-degree-6-generic}
				\end{equation}
\end{itemize}	
\end{thm}
\begin{proof}
By Lemma \ref{lem:del-Pezzo-dynamic-alpha-degree-6-upbound} we have that $\alpha(S,(1-\beta)C)\leq \omega_i$ for each case. If $\alpha(S,(1-\beta)C)<\omega_i$, then $\exists \lambda<\omega_i$ and an effective $\bbQ$-divisor $D\simq-K_S$ such that the pair
\begin{equation}
(S,(1-\beta)C + \lambda \beta D)
\label{eq:del-Pezzo-dynamic-alpha-degree-6-proof-badpair}
\end{equation}
is not log canonical at some $q\in S$. Observe that
\begin{equation}
\omega_3\leq \omega_2\leq\omega_1 \text{ for all } 0<\beta\leq 1.
\label{eq:del-Pezzo-dynamic-alpha-degree-6-proof-omegas}
\end{equation}
By Theorem \ref{thm:del-Pezzo-glct-charp} this implies
$$\lambda\beta<\omega_i\beta\leq \omega_1\beta\leq \frac{1}{2}=\glct(S).$$
Therefore by Lemma \ref{lem:pairs-fixed-boundary-lcs}, the pair \eqref{eq:del-Pezzo-dynamic-alpha-degree-6-proof-badpair} is log canonical in codimension $1$ and $q\in C$.

Assume the following:
\begin{clm}
\label{clm:del-Pezzo-dynamic-alpha-degree-6-proof-pEckardt}
The point $q$ is not a pseudo-Eckardt point.
\end{clm}

We want to apply Theorem \ref{thm:pseudo-inductive-blow-up} in each blow-up of the minimal log resolution of \eqref{eq:del-Pezzo-dynamic-alpha-degree-6-proof-badpair}. We will define a birational morphism $f\colon S_4\ra S_0$ as a sequence of blow-ups over $q$ which will be biregular away from $q$.

Let $f_1\colon S_1\ra S_0$ be the blow-up of $q_0:=q$ with exceptional curve $F_1$. For any $\bbQ$-divisor $A=A^0$ in $S_0$, let $A^1$ be the strict transform of $A^0$. Since $C=C^0$ is smooth at $q=q_0$, then $C^1\cdot F_1=1$ and therefore $C^1\cap F_1=q_1$, a unique point.

For $i\geq 2$, let $f_i\colon S_i \ra S_{i-1}$ be the blow-up of $q_{i-1}$ with exceptional curve $F_i$. Let $A^{i}$ be the strict transform of any $\bbQ$-divisor $A^{i-1}$ of $S^{i-1}$. Let $q_i=C^{i}\cap F_i$. Define $m_i=\mult_{q_i}D^i$ for $i\geq 0$. Define $f\colon S_4\ra S$ by $f=f_1\circ f_2\circ f_3\circ f_4$. Assume the following
\begin{clm}
\label{clm:del-Pezzo-dynamic-alpha-degree-6-proof}
If $\lambda<\omega_2$ or $\lambda<\omega_3$, then the pair
$$(S_3, (1-\beta)C^3+ \lambda\beta D^3+(\lambda\beta(m_0+m_1+m_2)-3\beta)F_3)$$
is not log canonical only at $q_3=F_3\cap C^3$.
If $<\lambda<\omega_1$, then the pair 
$$(S_4, (1-\beta)C^4+ \lambda\beta D^4+(\lambda\beta(m_0+m_1+m_2+m_3)-4\beta)F_4)$$
is not log canonical only at $q_4=F_4\cap C^4$.
\end{clm}
If $\lambda<\omega_1$, then we apply Lemma \ref{lem:adjunction} (iii) to $C^4$ and we obtain
\begin{align*}
1&<C^4\cdot (\lambda\beta D^4 + (\lambda\beta(m_0+m_1+m_2+m_3)-4\beta)F_4)=\\
&=\lambda\beta(C\cdot D -(m_0+m_1+m_2+m_3))+\lambda\beta(m_0+m_1+m_2+m_3)-4\beta=\\
&=6\lambda\beta-4\beta<6\omega_i\beta-4\beta\leq 6\omega_1\beta-4\beta
\end{align*}
by \eqref{eq:del-Pezzo-dynamic-alpha-degree-6-proof-omegas}. This gives a contradiction and finishes the proof, since $6\omega_1\beta-4\beta\leq 1$. Indeed, if $0<\beta\leq \frac{1}{2}$, then
$$6\omega_1\beta-4\beta\leq 6\beta-4\beta=2\beta\leq 1.$$
If $\frac{1}{2}\leq \beta \leq 1$, then
$$6\omega_1\beta-4\beta=3-4\beta\leq 3-2 =1.$$

If $\lambda<\omega_2$ or $\lambda<\omega_3$, then we apply Lemma \ref{lem:adjunction} (iii) to $C^3$ and we obtain
\begin{align*}
1&<C^3\cdot (\lambda\beta D^3 + (\lambda\beta(m_0+m_1+m_2)-3\beta)F_3)=\\
&=\lambda\beta(C\cdot D -(m_0+m_1+m_2))+\lambda\beta(m_0+m_1+m_2)-3\beta=\\
&=6\lambda\beta-3\beta<6\omega_i\beta-3\beta\leq 6\omega_2\beta-3\beta \text{ for }i=2,3.
\end{align*}
by \eqref{eq:del-Pezzo-dynamic-alpha-degree-6-proof-omegas}. This gives a contradiction and finishes the proof since $6\omega_2\beta-3\beta\leq 1$. Indeed, for $\lambda<\omega_2$, if $0<\beta\leq \frac{1}{3}$, then
$$6\omega_2\beta-3\beta\leq 6\beta-3\beta=3\beta\leq 1.$$
If $\frac{1}{3}\leq \beta \leq \frac{3}{4}$, then 
$$6\omega_2\beta-3\beta=\frac{6(1+2\beta)}{5\beta}\beta-3\beta=\frac{6}{5} -\frac{3\beta}{5} < 1.$$
If $\frac{3}{4}\leq \beta \leq 1$, then
$$6\omega_1\beta-3\beta=3-3\beta\leq 3-\frac{9}{4} <1.$$

Finally, suppose $\lambda<\omega_3$, if $0<\beta\leq \frac{1}{3}$, then $6\omega_3\beta-3\beta\leq 6\beta-3\beta=3\beta\leq 1$. If $\frac{1}{3}\leq \beta \leq1$, then $6\omega_3\beta-3\beta=\frac{6(1+\beta)}{4\beta}\beta-3\beta=\frac{3}{2} -\frac{3\beta}{2} \leq 1$.
\end{proof}

\begin{proof}[Proof of Claim \ref{clm:del-Pezzo-dynamic-alpha-degree-6-proof-pEckardt}]
Suppose $(S, (1-\beta)C + \lambda\beta D)$ is not log canonical at a pseudo-Eckardt point $q_0\in C$. Then $\lambda<\omega_3$ and by Lemma \ref{lem:del-Pezzo-good-model} we may assume that $q_0=E_1\cap L_{12}$. Since $E_1$ and $L_{12}$ are lines passing through $p$, then $E_1, L_{12}\subseteq \Supp(D)$, by Lemma \ref{lem:adjunction} (i). For instance, if $E_1\not\subset\Supp(D)$, then
$$1=E_1\cdot D \geq \mult_{q_0} D > \frac{1-(1-\beta)}{\lambda\beta}=\frac{1}{\lambda}>1.$$
Hence, we may write $D=aE_1+bL_{12}+\Omega$ where $a, b>0$ and $E_1, L \not\subseteq\Supp(\Omega)$. Let $x_i=\mult_{q_i}\Omega^i$ and $m_i=\mult_{q_i} D_i$. Observe that $m_0=a+b+x_0$. Let $A$ be the strict transform of a general conic in $\bbP^2$ passing through $p_1,p_3$ such that $q\in A$. Since there is a pencil of such conics, we can choose $A$ such that $A\not\subset D$. Observe that $A\sim\pioplane{2}-E_1-E_3$. Intersecting with $D$ we get the following inequalities:
\begin{align}
4&=A\cdot D \geq a+b, \label{eq:del-Pezzo-dynamic-alpha-degree-6-special-bound1} \\
1&=E_1\cdot D \geq -a+b+x_0, \label{eq:del-Pezzo-dynamic-alpha-degree-6-special-bound2}\\
1&=L_{12}\cdot D \geq a-b+x_0. \label{eq:del-Pezzo-dynamic-alpha-degree-6-special-bound3}
\end{align}
In particular, adding \eqref{eq:del-Pezzo-dynamic-alpha-degree-6-special-bound2} and \eqref{eq:del-Pezzo-dynamic-alpha-degree-6-special-bound3} we obtain
\begin{equation}
1\geq x_0.
\label{eq:del-Pezzo-dynamic-alpha-degree-6-special-bound4}
\end{equation}
Now we may apply Theorem \ref{thm:inequality-Cheltsov} to conclude that either
$$2(1-\lambda\beta a) <E_1\cdot (\lambda\beta\Omega+ (1-\beta)C)$$
holds or 
$$2(1-\lambda\beta b) <L_{12}\cdot (\lambda\beta\Omega+ (1-\beta)C)$$
holds. Since the roles of $E_1$ and $L_{12}$ in this problem are symmetric, it is enough to disprove the latter equation. Indeed, if it holds, then
\begin{align*}
2(1-\lambda\beta a) < L_{12}\cdot (\lambda\beta\Omega+ (1-\beta)C)=\lambda\beta(1-a+b) + (1-\beta).
\end{align*}
Hence, by \eqref{eq:del-Pezzo-dynamic-alpha-degree-6-special-bound1}
$$1+\beta<\lambda\beta(a+b)<4\omega_3\beta.$$
But then
$$\frac{1+\beta}{4\beta}<\omega_3=\min\{1,\frac{1+\beta}{4\beta}\}\leq\frac{1+\beta}{4\beta},$$
which is absurd.
\end{proof}

In the rest of this section we will prove Claim \ref{clm:del-Pezzo-dynamic-alpha-degree-6-proof} by case analysis, depending on the position of $q_0$.
\begin{lem}
\label{lem:del-Pezzo-dynamic-alpha-degree-6-pointgeneric}
If $q_0$ does not belong to any $(-1)$-curve, then
$$(S_3, (1-\beta)C^3 + \lambda\beta D^3 + (\lambda\beta(m_0+m_1+m_2)-3\beta)F_3)$$
is not log canonical only at $q_4=C^4\cap F_4$ for $\lambda<\omega_1$.
$$(S_4, (1-\beta)C^4 + \lambda\beta D^4 + (\lambda\beta(m_0+m_1+m_2+m_3)-4\beta)F_4)$$
is not log canonical only at $q_4=C^4\cap F_4$ for $\lambda<\omega_1$.
\end{lem}
\begin{proof}
Fix a model $\pi\colon S \ra \bbP^2$ for $S$. We will follow the notation at the beginning of this section. Since $q_0$ lies in no $(-1)$-curve and all $(-1)$-curves have degree $1$, we may define irreducible conics
\begin{equation}
L_i\in \vert\pioplane{1}-E_i\vert \text{ with }q_0\in L_i\text{ for } i=1,2,3.
\label{eq:del-Pezzo-dynamic-alpha-degree-6-pointgeneric-cubics}
\end{equation}
Such curves are unique and they correspond to  the strict transform of the unique line in $\bbP^2$ passing through $\pi(q_0)$ and $p_i$. Since there is no curve of degree $1$ through $q_0$ and the point $q_0\in L_i$, then $L_i$ is irreducible. Observe that $L_i^2=0$ and $(-K_S)\cdot L_i=2$. Furthermore
$$L_1+L_2+L_3\sim -K_S.$$
By Lemma \ref{lem:appendix-delPezzo-deg6-dynamic-lct1}  the pair
\begin{equation*}
(S,(1-\beta)C + \lambda\beta (L_1+L_2+L_3))
\end{equation*}
is log canonical. Therefore we may assume by Lemma \ref{lem:log-convexity} that there is $L_i\not\subseteq\Supp(D)$. Without loss of generality assume that $L_i=L_1$. Let $m_0=\mult_qD=\mult_{q_0}D^0$.
\begin{equation}
2=D\cdot L_1 \geq m_0
\label{eq:del-Pezzo-dynamic-alpha-degree-6-pointgeneric-bound-m0up}
\end{equation}
and therefore
\begin{equation}
\lambda\beta m_0 < \omega_i\beta m_0\leq \omega_1\beta m_0\leq 1,
\label{eq:del-Pezzo-dynamic-alpha-degree-6-pointgeneric-bound-m0}
\end{equation}
since $\omega_1\beta\leq \frac{1}{2}$.
Hence, by Theorem \ref{thm:pseudo-inductive-blow-up} (i) with $i=1$, we have that the pair
$$(S_1, (1-\beta)C^1 +\lambda\beta D^1 + (\lambda\beta m_0-\beta)F_1)$$
is not log canonical at some $t_1\in F_1$. Observe \eqref{eq:del-Pezzo-dynamic-alpha-degree-6-pointgeneric-bound-m0} is a necessary condition for Theorem \ref{thm:pseudo-inductive-blow-up} (ii)-(iv) whenever $i\geq 2$. We will assume this condition fromnow onwards. Also, by Theorem \ref{thm:pseudo-inductive-blow-up} (ii) when $i=1$ we conclude that $t_1=F_1\cap C^1=q_1$. But then since
$$\lambda\beta m_0-\beta<\lambda\beta m_0<1,$$
by part (i) of Theorem \ref{thm:pseudo-inductive-blow-up} when $i=2$ we obtain that
\begin{equation}
(S_2, (1-\beta) C^1 + \lambda\beta D^1 + (\lambda\beta m_0-\beta)F^2_1 +(\lambda\beta(m_0+m_1)-2\beta)F_2)
\label{eq:del-Pezzo-dynamic-alpha-degree-6-pointgeneric-badpair2}
\end{equation}
is not log canonical at some $t_2\in F_2$. Observe from \eqref{eq:del-Pezzo-dynamic-alpha-degree-6-pointgeneric-cubics} that $L_i^1\cdot L_j^1=0$ for $1\leq i<j\leq 3$. Therefore these curves are disjoint. We distinguish 3 distinct cases:
\begin{itemize}
	\item[(1)] $q_1=L^1_1\cap F_1$,
	\item[(2)] $q_1\in (L_2^1 \cup L_3^1)\cap F_1$,
	\item[(3)] $q_1\in F_1\setminus (L_1^1\cup L_2^1\cup L_3^1)$.
\end{itemize}
We will analyse each case separately.

\textbf{Case (1): }$q_1=L^1_1\cap F_1$. 

Since $L_1\not\subseteq\Supp(D)$ and $q_1\in L^1_1$, then
\begin{equation}
2=L_1\cdot D\geq m_0+m_1.
\label{eq:del-Pezzo-dynamic-alpha-degree-6-pointgeneric-bound-m0m1}
\end{equation}
This implies
$$\lambda\beta(m_0+m_1)-\beta\leq \lambda\beta(m_0+m_1)\leq 2\omega_1\beta\leq 1,$$
since $\omega\beta\leq \frac{1}{2}$. Using the above inequality, Theorem \ref{thm:pseudo-inductive-blow-up} part (iii) with $i=2$ implies $t_2 = F_2\cap C^2=q_2$ and part (i) when $i=3$ implies the pair
$$(S_3, (1-\beta)C^3 + \lambda\beta D^3 + (\lambda\beta(m_0+m_1) -2\beta)F^3_2 + (\lambda\beta(m_0+m_1+m_2)-3\beta)F_3)$$
is not log canonical at some $t_3\in F_3$. Then \eqref{eq:del-Pezzo-dynamic-alpha-degree-6-pointgeneric-bound-m0up} and \eqref{eq:del-Pezzo-dynamic-alpha-degree-6-pointgeneric-bound-m0m1} imply
\begin{equation}
\lambda\beta(m_0+m_1+m_2)-2\beta<\omega_1\beta(m_0+m_1+m_0)-2\beta=4\omega_1\beta-2\beta\leq 1
\label{eq:del-Pezzo-dynamic-alpha-degree-6-pointgeneric-bound-blowup3}
\end{equation}
since $m_j\leq m_0$ for all $j\geq 0$. Indeed, if $0<\beta\leq \frac{1}{2}$, then
$$4\omega_1\beta-2\beta\leq 4\beta-2\beta=2\beta\leq 1$$
and if $\frac{1}{2}\leq\beta\leq 1$, then
$$4\omega_1\beta-2\beta\leq 2-2\beta\leq 2-1\leq 1.$$
Hence, applying \eqref{eq:del-Pezzo-dynamic-alpha-degree-6-pointgeneric-bound-blowup3} to Theorem \ref{thm:pseudo-inductive-blow-up} (iii) with $i=3$ gives us that $t_3=F_3\cap C^3=q_3$. Since
$$\lambda\beta(m_0+m_1+m_2)-3\beta<\lambda\beta(m_0+m_1+m_2)-2\beta<1,$$
by Theorem \ref{thm:pseudo-inductive-blow-up} (i) with $i=4$, we conclude that
\begin{equation}
(S_4, (1-\beta)C^4 + \lambda\beta D^4 + (\lambda\beta(m_0+m_1+m_2)-3\beta)F_3^4 + (\lambda\beta(m_0+m_1+m_2+m_3)-4\beta)F_4)
\label{eq:del-Pezzo-dynamic-alpha-degree-6-pointgeneric-badpair4}
\end{equation}
is not log canonical at some $t_4\in F_4$. We use \eqref{eq:del-Pezzo-dynamic-alpha-degree-6-pointgeneric-bound-m0m1} to show that
\begin{equation}
\lambda\beta(m_0+m_1+m_2+m_3)-3\beta\leq 4\lambda\beta-3\beta\leq 4\omega_1\beta-3\beta\leq 1
\label{eq:del-Pezzo-dynamic-alpha-degree-6-pointgeneric-bound-blowup4}
\end{equation}
holds. Indeed for $0<\beta\leq \frac{1}{2}$ we have
$$4\omega_1\beta-3\beta\leq 4\beta-3\beta=\beta\leq \frac{1}{2}<1,$$
while for $\frac{1}{2}\leq \beta\leq 1$ we have
$$4\omega_1\beta-3\beta\leq 2 -3\beta\leq \frac{1}{2}<1.$$
Now we may use \eqref{eq:del-Pezzo-dynamic-alpha-degree-6-pointgeneric-bound-blowup4} in Theorem \ref{thm:pseudo-inductive-blow-up} (iii) with $i=4$ to prove that pair \eqref{eq:del-Pezzo-dynamic-alpha-degree-6-pointgeneric-badpair4} is not log canonical only at $t_4=F_4\cap C^4=q_4$, which implies the claim in case (1).

\textbf{Case (2):} $q_1\in (L_2^1\cup L_3^1)\cap F_1$.

If $q_1\in (L_2^1\cup L_3^1)\cap F_1$, we may assume without loss of generality that $q_1=L_2^1\cap F_1$ and $q_1\not\in L_3^1$. We may write $D=aL_2+\Omega$ where $L_1, L_2\not\subseteq\Supp(\Omega)$ and $a\geq 0$. Let $x_i=\mult_{q_i}\Omega^i$ for $i\geq 0$. Then
$$m_0 = a+x_0, m_1=a+x_1, m_i=x_i \text{ for } i\geq 2.$$
Since $L_2\cdot L_1 = 1$ and $(L_2)^2=0$, then
\begin{equation}
2=D\cdot L_1\geq a+x_0=m_0,
\label{eq:del-Pezzo-dynamic-alpha-degree-6-pointgeneric-special-boundm0}
\end{equation}
\begin{equation}
2=D\cdot L_2\geq x_0+x_1.
\label{eq:del-Pezzo-dynamic-alpha-degree-6-pointgeneric-special-boundmx0x1}
\end{equation}
The pair $(S, (1-\beta)C + \lambda\beta(L_{13}+2L_2+E_2))$ is log canonical by Lemma \ref{lem:appendix-delPezzo-deg6-dynamic-lct0}. Therefore by Lemma \ref{lem:log-convexity}, either $L=L_{13}\not\subset\Supp(D)$ or $L=E_2\not\subset\Supp(D)$. In both cases
\begin{equation}
1=D\cdot L\geq a.
\label{eq:del-Pezzo-dynamic-alpha-degree-6-pointgeneric-special-boundm-a}
\end{equation}

From \eqref{eq:del-Pezzo-dynamic-alpha-degree-6-pointgeneric-special-boundm0} we obtain
\begin{equation}
\lambda\beta(2m_0)-2\beta \leq 4\omega_1\beta-2\beta\leq 1
\label{eq:del-Pezzo-dynamic-alpha-degree-6-pointgeneric-special-bound-blowup2}
\end{equation}
where the last inequality follows from the last inequality in \eqref{eq:del-Pezzo-dynamic-alpha-degree-6-pointgeneric-bound-blowup3}. By Theorem \ref{thm:pseudo-inductive-blow-up} (iv) with $i=2$, the pair \eqref{eq:del-Pezzo-dynamic-alpha-degree-6-pointgeneric-badpair2} is not log canonical at $t_2=F_2\cap C^2=q_2$. Since $\lambda\beta(m_0+m_1)-2\beta\leq 2\lambda\beta m_0-2\beta\leq 1$, then Theorem \ref{thm:pseudo-inductive-blow-up} (i) with $i=3$ implies the pair
$$(S_3, (1-\beta)C^3 + \lambda\beta( aL_2^3+\Omega^3 )+ (\lambda\beta(m_0+m_1) -2\beta)F^3_2 + (\lambda\beta(m_0+m_1+m_2)-3\beta)F_3)$$
is not log canonical at some $t_3 \in F_3$. Observe that
$$\lambda\beta(m_0+m_1+m_2)-3\beta<\omega_1\beta(2a+x_0+x_1+x_2)-3\beta\leq 1,$$
by  \eqref{eq:del-Pezzo-dynamic-alpha-degree-6-pointgeneric-special-boundm0}, \eqref{eq:del-Pezzo-dynamic-alpha-degree-6-pointgeneric-special-boundmx0x1} and \eqref{eq:del-Pezzo-dynamic-alpha-degree-6-pointgeneric-special-boundm-a}. In particular the point $t_3$ is isolated. The last inequality follows from case analysis. 
Indeed, if $0<\beta\leq \frac{1}{2}$, then $5\omega_1\beta-3\beta=2\beta\leq 1$ while if $\frac{1}{2}\leq\beta\leq 1$, then $5\omega_1\beta-3\beta=\frac{5}{2}-3\beta\leq 1$.

Suppose $t_3\not\in F_2^3\cup C^3$. Then the pair
$$(S_3, \lambda\beta( aL_2^3+\Omega^3 )+ (\lambda\beta(m_0+m_1+m_2)-3\beta)F_3)$$
is not log canonical at $t_3$. But then, applying Lemma \ref{lem:adjunction} (iii) with $F_3$ we obtain that
$$1<\lambda\beta(aL_2^3+\Omega^3)\cdot F_3<\omega_1\beta x_2\leq 1$$
by \eqref{eq:del-Pezzo-dynamic-alpha-degree-6-pointgeneric-special-boundmx0x1}. This is absurd.

Suppose $t_3= F_2^3\cap F_3$. Since $F_3\cap L_2^3=F_2^3\cap C^3=\emptyset$, the pair
$$(S_3, \lambda\beta( \Omega^3 )+ (\lambda\beta(m_0+m_1) -2\beta)F^3_2 + (\lambda\beta(m_0+m_1+m_2)-3\beta)F_3)$$
is not log canonical at $t_3\in F_3$. By applying Lemma \ref{lem:adjunction} (iii) with $F_2^3$ we obtain
\begin{align*}
1&<(\lambda\beta(\Omega^3)+(\lambda\beta(m_0+m_1+m_2)-3\beta)F_3)\cdot F_2^3\\
&=\lambda\beta(x_1-x_2+a+x_0+a+x_1+x_2)-3\beta\\
&< 5\omega_1\beta-3\beta\leq 1
\end{align*}
by  \eqref{eq:del-Pezzo-dynamic-alpha-degree-6-pointgeneric-special-boundm0}, \eqref{eq:del-Pezzo-dynamic-alpha-degree-6-pointgeneric-special-boundmx0x1} and \eqref{eq:del-Pezzo-dynamic-alpha-degree-6-pointgeneric-special-boundm-a}, giving a contradiction. 

Hence we conclude that $t_3=F_3\cap C^3=q_3$. Moreover, since  $\lambda\beta(m_0+m_1+m_2)-3\beta\leq 1$, as we have seen above, we apply Theorem \ref{thm:pseudo-inductive-blow-up} (i) with $i=4$, to show that
\begin{equation}
(S_4, (1-\beta)C^4 + \lambda\beta D^4 +(\lambda\beta(m_0+m_1+m_2)-3\beta)F_3 + (\lambda\beta(m_0+m_1+m_2+m_3)-4\beta)F_4)
\label{eq:del-Pezzo-dynamic-alpha-degree-6-pointgeneric-special-badpair4}
\end{equation}
is not log canonical at some $t_4 \in F_4$.
Observe that 
\begin{align*}
&\lambda\beta(m_0+m_1+2m_2)-4\beta\leq \lambda\beta(2a+x_0+x_1+2x_2)-4\beta\leq\\
\leq& (\lambda\beta(2a+2x_0+x_0+x_1)-4\beta\leq 6\lambda\beta-4\beta<6\omega_1\beta-4\beta\leq 1
\end{align*}
by \eqref{eq:del-Pezzo-dynamic-alpha-degree-6-pointgeneric-special-boundmx0x1} and \eqref{eq:del-Pezzo-dynamic-alpha-degree-6-pointgeneric-special-boundm0}. The last inequality follows from case analysis. If $0<\beta\leq \frac{1}{2}$, then $6\omega_1\beta-4\beta=6\beta-4\beta=2\beta\leq 1$. If $\frac{1}{2}\leq \beta\leq 1$, then $6\omega_1\beta-4\beta\leq 3-4\beta\leq 3-2=1$. Hence $t_4=C^4\cap F_4=q_4$ and the pair  \eqref{eq:del-Pezzo-dynamic-alpha-degree-6-pointgeneric-special-badpair4} is not log canonical only at $q_4=F_4\cap C^4$, which implies the claim in case (2).

\textbf{Case (3):} $q_1\not\in (L_1^1\cup L_2^1\cup L_3^1)\cap F_1$.

Let $H^1\in \vert f_1^*(\pioplane{1})-F_1\vert$ be the unique member of this linear system which passes through $q_1$. There are two ways to see that such $H^1$ exists. One of them is to realise that $S_1$ is a del Pezzo surface of degree $5$ and let $H^1$ be the strict transform of a line in $\bbP^2$ under the model $\pi\circ f_1$. The other way to see $H^1$ exists is to notice that $-K_{S_1}\cdot H^1=2$ and $(H^1)^2=0$, $p_a(H^1)=0$ and apply Proposition \ref{prop:rational-surfaces-sections-rational-curves} to its complete linear system, to show that is a pencil. $H^1$ is irreducible. Indeed, if $H^1=A+B$ then $A$ and $B$ are curves of del Pezzo degree $1$ and $q_1\in A$. By assumption, the only curve of degree $1$ is $F_1$. But then $B\sim\pioplane{1}-2F_1$ and there is no such line in $S_1$ by Lemma \ref{lem:del-Pezzo-lines9-d}. Let $H=H_0:=(f_1)_*(H^{1})$. Note that $H^1$ is the strict transform of $H$ and $q_0\in H$. Let
$$G^1\in \vert f_1^*(\pioplane{2})-F_1-E_1-E_2-E_3\vert$$
be the unique member of this linear system which passes through $q_1$. There are two ways to see that such $G^1$ exists. One of them is to realise that $S_1$ is a del Pezzo surface of degree $5$ and let $G^1$ be the strict transform of a conic in $\bbP^2$ passing through $p_1,p_2,p_3$ under the model $\pi\circ f_1$. The other way to see that $G^1$ exists is to notice that $-K_{S_1}\cdot G^1=2$ and $(G^1)^2=0$, $p_a(G^1)=0$ and apply Proposition \ref{prop:rational-surfaces-sections-rational-curves} to its complete linear system, to show that is a pencil. $G^1$ is irreducible. Indeed, if $G^1=A+B$ then $A$ and $B$ are curves of del Pezzo degree $1$ and $q_1\in A$. By assumption, the only curve of degree $1$ is $F_1$. But then $B\sim\pioplane{1}-2F_1$ and there is no such line in $S_1$ by Lemma \ref{lem:del-Pezzo-lines9-d}. Let $G=G_0:=(f_1)_*(G^{1})$. Note that $G^1$ is the strict transform of $G$ and $q_0\in G$.

Observe that $G+H\sim -K_S$. By Lemma \ref{lem:appendix-delPezzo-deg6-dynamic-lct2} the pair
\begin{equation}
(S,(1-\beta)C + \lambda\beta (G+H))
\label{eq:del-Pezzo-dynamic-alpha-degree-6-pointgeneric-2-auxpair1}
\end{equation}
is log canonical. Therefore we may assume by Lemma \ref{lem:log-convexity} that either $G\not\subseteq\Supp(D)$ or $H\not\subseteq\Supp(D)$. In any case there is an irreducible curve $Q$ in $S$ with del Pezzo degree $3$ such that $q_0\in Q$, $q_1\in Q^1$ and $Q\not\subset\Supp(D)$. Recall from \eqref{eq:del-Pezzo-dynamic-alpha-degree-6-pointgeneric-bound-m0up}
\begin{equation}
m_0\leq 2.
\label{eq:del-Pezzo-dynamic-alpha-degree-6-pointgeneric-generic-boundm0}
\end{equation}
Intersecting $Q$ and $D$:
\begin{equation}
3=Q\cdot D \geq m_0+m_1.
\label{eq:del-Pezzo-dynamic-alpha-degree-6-pointgeneric-generic-boundm0m1}
\end{equation}
This gives us:
\begin{equation}
\lambda\beta(2m_0)-2\beta < 4\omega_1\beta-2\beta\leq 1
\label{eq:del-Pezzo-dynamic-alpha-degree-6-pointgeneric-generic-bound-blowup2}
\end{equation}
where the last inequality follows from the last inequality in \eqref{eq:del-Pezzo-dynamic-alpha-degree-6-pointgeneric-bound-blowup3}. By Theorem \ref{thm:pseudo-inductive-blow-up} (iv) with $i=2$, the pair \eqref{eq:del-Pezzo-dynamic-alpha-degree-6-pointgeneric-badpair2} is not log canonical at $q_2=F_2\cap C^2$.

Since
$$\lambda\beta(m_0+m_1)-2\beta<4\omega_1\beta m_0-2\beta\leq 1,$$
by \eqref{eq:del-Pezzo-dynamic-alpha-degree-6-pointgeneric-bound-m0} and \eqref{eq:del-Pezzo-dynamic-alpha-degree-6-pointgeneric-bound-blowup3}, then Theorem \ref{thm:pseudo-inductive-blow-up} (i) with $i=3$ implies the pair
$$(S_3, (1-\beta)C^3 + \lambda\beta D^3 + (\lambda\beta(m_0+m_1) -2\beta)F^3_2 + (\lambda\beta(m_0+m_1+m_2)-3\beta)F_3)$$
is not log canonical at some $t_3 \in F_3$. Moreover, given that $m_1\leq m_0$, \eqref{eq:del-Pezzo-dynamic-alpha-degree-6-pointgeneric-generic-boundm0} and \eqref{eq:del-Pezzo-dynamic-alpha-degree-6-pointgeneric-generic-boundm0m1} give us
$$\lambda\beta(m_0+2m_1)-3\beta<5\omega_1\beta-3\beta\leq 1,$$
where the last inequality follows by case analysis: if $0<\beta\leq \frac{1}{2}$, then $5\omega_1\beta-3\beta=2\beta\leq 1$ and if $\frac{1}{2}\leq \beta\leq 1$, then $5\omega_1\beta-3\beta=\frac{5}{2}-3\beta\leq 1$.

Hence we may apply Lemma (iv) with $i=3$ to conclude that $t_3=F_3\cap C^3=q_3$. Moreover
$$\lambda\beta(m_0+m_1+m_2)-3\beta\leq \lambda\beta(m_0+2m_1)-3\beta\leq 1,$$
which we use to apply Theorem \ref{thm:pseudo-inductive-blow-up} (i) with $i=4$, to show that
\begin{equation}
(S_4, (1-\beta)C^4 + \lambda\beta D^4 +(\lambda\beta(m_0+m_1+m_2)-3\beta)F_3 + (\lambda\beta(m_0+m_1+m_2+m_3)-4\beta)F_4)
\label{eq:del-Pezzo-dynamic-alpha-degree-6-pointgeneric-generic-badpair4}
\end{equation}
is not log canonical at some $t_4 \in F_4$.
Since
\begin{align*}
&\lambda\beta(m_0+m_1+2m_2)-4\beta\leq 2\lambda\beta(m_0+m_1)-4\beta<6\omega_1\beta-4\beta\leq 1
\end{align*}
by \eqref{eq:del-Pezzo-dynamic-alpha-degree-6-pointgeneric-generic-boundm0m1}. The last inequality follows from case analysis. If $0<\beta\leq \frac{1}{2}$, then $6\omega_1\beta-4\beta=6\beta-4\beta=2\beta\leq 1$. If $\frac{1}{2}\leq \beta\leq 1$, then $6\omega_1\beta-4\beta\leq 3-4\beta\leq 3-2=1$. Hence $t_4=C^4\cap F_4=q_4$ and the pair  above is not log canonical only at $q_4=F_4\cap C^4$, which implies the claim in case (3), finishing the proof of this lemma.
\end{proof}

\begin{lem}
\label{lem:del-Pezzo-dynamic-alpha-degree-6-pointline}
If $q_0$ belongs precisely to one $(-1)$-curve and $\lambda<\omega_1$, then
$$(S_4, (1-\beta)C^4 + \lambda\beta D^4 + (\lambda\beta(m_0+m_1+m_2+m_3)-4\beta)F_4)$$
is not log canonical only at $q_4=C^4\cap F_4$.

If $q_0$ belongs precisely to one $(-1)$-curve and $\lambda<\omega_2$, then
$$(S_3, (1-\beta)C^3 + \lambda\beta D^3 + (\lambda\beta(m_0+m_1+m_2)-3\beta)F_3)$$
is not log canonical only at $q_3=C^3\cap F_3$.
\end{lem}
\begin{proof}
We will follow the notation at the beginning of this section. Recall $\omega_2\leq \omega_1$. By Lemma \ref{lem:del-Pezzo-good-model} we may choose a model $\pi\colon S \ra \bbP^2$ for $S$ such that $q_0\in E_1$ and $q_0$ lies in no other line. Let $L_q$ be the strict transform of the unique line of $\bbP^2$ through $p_1$ such that $q_0\in L_q$, i.e.
$$L_q\in \vert \pioplane{1}-E_1\vert.$$
Since $E_1$ is the only line containing $q_0$, then $L_q$ is irreducible. Otherwise there would be a line $L\subset S$ with class $L\sim\pioplane{1}-2E_1$ such that $E_1+L\sim L_q$, which is not possible by Lemma \ref{lem:del-Pezzo-lines9-d}. Observe that $(L_q)^2=0$ and $-K_S\cdot L_q=2$.

Let $S_1$ be the surface resulting from the blow-up of $q_0$. Let $F_1$ be the exceptional divisor. Let $q_1=F_1\cap C^1$.  Let $C_q$ be the strict transform of the unique conic in $\bbP^2$ through $p_1,p_2,p_3$ such that $q_0\in C_q$ and $q_1\in C_q^1\subset S_1$. Therefore
$$C_q\in \vert \pioplane{2}-E_1-E_2-E_3\vert.$$
Observe that $(C_q)^2=1$ and $-K_S\cdot C_q=3$.

\textbf{Case (A):} Suppose $C_q$ is reducible.
Then it decomposes in a line and a (possibly reducible) conic. One of these curves $M\subset\Supp(C_q)$ has $q_0\in M$ and $q_1\in M^1$ where $M^1\subset S_1$ is its strict transform in $S_1$. Since $q_1\in C^1$ by construction and $C^1\cdot E_1^1=C\cdot E_1-1=0$, then $q_1\not\in E^1_1$ and we conclude $M\neq E_1$. But $E_1$ is the only line through $q_0$, so $M$ is an irreducible conic, and in fact $M=L_q$, since there is no other conic through $q_0$. Observe that by construction we are in the situation in which $(L_q\cdot C)\vert_{q_0}=2$. Hence $\lambda<\omega_2$. In particular $q_1\in L_q^1$ but $q_2\not\in L_q^2$, since $q_2\in C^2$ and $C^2\cdot L_q^2=C\cdot L_q-2=0$. Let $L$ be the unique line such that $L_q+L\sim C_q$. Then it is clear that $L=L_{23}\sim\pioplane{1}-E_2-E_3$. Let $B=2L_q+L_{23}+E_1\sim-K_S$. By Lemma \ref{lem:appendix-delPezzo-deg6-dynamic-lct4}, the pair
$$(S, (1-\beta)C + \lambda\beta B)$$
is log canonical. Hence, by Lemma \ref{lem:log-convexity} we may assume there is an irreducible curve $Z\subseteq\Supp(B)$ such that $Z\not\subseteq\Supp(D)$. Notice that $E_1\subset \Supp(D)$, since otherwise we get the following contradiction
$$1\geq 1-\beta + \lambda\beta=((1-\beta)C + \lambda\beta D)\cdot E_1\geq \mult_{q_0}((1-\beta)C + \lambda\beta D))>1$$
where we use Lemma \ref{lem:adjunction} (i). Therefore either
\begin{itemize}
	\item[(A1)]$L_q\not\subseteq \Supp(D)$, or 
	\item[(A2)]$L_{23}\not\subseteq\Supp(D)$ and $L_q\subseteq\Supp(D)$.
\end{itemize}

\textbf{Case (A1):} $L_q\not\subseteq \Supp(D)$.

We may write $D=aE_1+\Omega$ where $L_q, E_1, \not\subseteq \Supp(\Omega)$. Let $x_i=\mult_{q_i}\Omega^i$. Then $m_0=a+x_0$. Since $q_1\in C^1$ but $E_1\cdot C=1$, then $E_1^1\cdot C^1=0$ so $q_1\not\in E_1$. Hence $m_1=x_1$. 

We bound the multiplicity of $D$:
\begin{align}
1&=D\cdot E_1 \geq -a+x_0 \label{eq:del-Pezzo-dynamic-alpha-degree-6-pointline-0-bound1}\\
2&=D\cdot L_q \geq a+x_0+x_1= m_0+m_1 \label{eq:del-Pezzo-dynamic-alpha-degree-6-pointline-0-bound2}
\end{align}
In particular
\begin{equation}
\lambda\beta m_0-\beta\leq \lambda\beta (m_0+m_1)-\beta\leq 2\lambda\beta -\beta \leq 2\omega_2\beta -\beta \leq 2\omega_1\beta-\beta\leq 1-\beta \leq 1.
\label{eq:del-Pezzo-dynamic-alpha-degree-6-pointline-0-m0}
\end{equation}
where we use \eqref{eq:del-Pezzo-dynamic-alpha-degree-6-pointline-0-bound2}.

Hence, by Theorem \ref{thm:pseudo-inductive-blow-up} (i) with $i=1$, we have that the pair
$$(S_1, (1-\beta)C^1 +\lambda\beta D^1 + (\lambda\beta m_0-\beta)F_1)$$
is not log canonical at some $t_1\in F_1$. Observe by inspection in \eqref{eq:del-Pezzo-dynamic-alpha-degree-6-pointline-0-m0} that $\lambda\beta(m_0+m_1)-\beta\leq 1$, which is a necessary condition for Theorem \ref{thm:pseudo-inductive-blow-up} (ii)-(iv) whenever $i\geq 2$. Therefore, by Theorem \ref{thm:pseudo-inductive-blow-up} (ii) when $i=1$ we conclude that $t_1=F_1\cap C^1=q_1$. But then part (iii) of Theorem \ref{thm:pseudo-inductive-blow-up} when $i=2$ gives that
\begin{equation}
(S_2, (1-\beta) C^1 + \lambda\beta D^1 + (\lambda\beta m_0-\beta)F^2_1 +(\lambda\beta(m_0+m_1)-2\beta)F_2)
\label{eq:del-Pezzo-dynamic-alpha-degree-6-pointline-0-badpair2}
\end{equation}
is not log canonical only at $q_2=F_2\cap C^2$.

From inequality \eqref{eq:del-Pezzo-dynamic-alpha-degree-6-pointline-0-m0} we deduce
$$\lambda\beta(m_0+m_1)-2\beta\leq \lambda\beta(m_0+m_1)-\beta\leq 1$$
which is the condition in Theorem \ref{thm:pseudo-inductive-blow-up} (i) when $i=3$. Therefore the pair
$$(S_3, (1-\beta)C^3 + \lambda\beta D^3 + (\lambda\beta(m_0+m_1) -2\beta)F^3_2 + (\lambda\beta(m_0+m_1+m_2)-3\beta)F_3)$$
is not log canonical at some $t_3\in F_3$.

Using \eqref{eq:del-Pezzo-dynamic-alpha-degree-6-pointline-0-bound2} we obtain
$$\lambda\beta (m_0+2m_1)-3\beta<4\omega_2\beta-3\beta\leq 4\omega_1\beta-3\beta\leq 1.$$
The last inequality follows from case analysis: $0<\beta\leq \frac{1}{2}$ we have $4\omega_1\beta-3\beta=4\beta-3\beta=\beta<1$. For $\frac{1}{2}\leq\beta\leq 1$, $4\omega_1\beta-3\beta=2-3\beta<1.$
Hence, by Theorem \ref{thm:pseudo-inductive-blow-up} (iv) with $i=3$ we conclude $t_3=F_3\cap C^3=q_3$, which implies the claim in case (A1).

\textbf{Case (A2):} $L_q\subseteq \Supp(D)$ and $L_{23}\not\subseteq \Supp(D)$.

We may write $D=aE_1+bL_q+\Omega$ where $L_q, E_1, \not\subseteq \Supp(\Omega)$. Let $x_i=\mult_{q_i}\Omega^i$. Then $m_0=a+b+x_0$. 

Observe that $B=L_q + L_{12}+L_{13}+2E_1\sim -K_S$. By Lemma \ref{lem:appendix-delPezzo-deg6-dynamic-lct5}, the pair
$$(S,(1-\beta)C + \lambda\beta B)$$
is log canonical for $\lambda<\omega_1$. Therefore, using Lemma \ref{lem:log-convexity}, we may assume there is one irreducible component $F\subseteq\Supp(B)$ such that $F\not\subseteq\Supp(D)$. Since $E_1\subset\Supp(B)$ and $L_p\subset\Supp(B)$, we have that either $F=L_{12}$ or $F=L_{13}$. In any case $1=F\cdot D\geq a$.

We bound the multiplicities of $D$:
\begin{align}
1&=D\cdot F \geq a \label{eq:del-Pezzo-dynamic-alpha-degree-6-pointline-0b-bound0}\\
1&=D\cdot E_1 \geq -a+b+x_0 \label{eq:del-Pezzo-dynamic-alpha-degree-6-pointline-0b-bound1}\\
1&=D\cdot L_{23} \geq b \label{eq:del-Pezzo-dynamic-alpha-degree-6-pointline-0b-bound3}
\end{align}

Using \eqref{eq:del-Pezzo-dynamic-alpha-degree-6-pointline-0b-bound0} and \eqref{eq:del-Pezzo-dynamic-alpha-degree-6-pointline-0b-bound1} it is immediate to obtain:
\begin{equation}
\lambda\beta m_0 -\beta=\lambda\beta (a+b+x_0+x_1)-\beta \leq \omega_2\beta (2+1)-\beta \leq 3\omega_1\beta-\beta\leq 1.
\label{eq:del-Pezzo-dynamic-alpha-degree-6-pointline-0b-m0}
\end{equation}
Indeed, if $0<\beta\leq \frac{1}{2}$, then $3\omega_1\beta-\beta=2\beta\leq 1$, whereas if $\frac{1}{2}\leq\beta\leq 1$, then $3\omega_1\beta-\beta\leq \frac{3}{2}-\beta\leq 1$.

By Theorem \ref{thm:pseudo-inductive-blow-up} (i) with $i=1$, we have that the pair
$$(S_1, (1-\beta)C^1 +\lambda\beta (aE_1^1+ bL_q^1 +\Omega^1) + (\lambda\beta m_0-\beta)F_1)$$
is not log canonical at some $t_1\in F_1$ and is log canonical in codimension $1$.

Since $L_q\cdot E_1=C\cdot E_1=1$, then $L_q^1\cap E_1=\emptyset$ and $C^1\cap E_1=\emptyset$. Suppose $t_1\not\in (E_1^1\cup L_q^1)$. Then $t_1\not\in C^1$, and the pair
$$(S_1, \lambda\beta (\Omega^1) + (\lambda\beta m_0-\beta)F_1)$$
is not log canonical at $t_1$. Applying Lemma \ref{lem:adjunction} (i) we obtain
$$1<\mult_{t_1}(\lambda\beta\Omega^1 + (\lambda\beta m_0-\beta)F_1)=\lambda\beta x_1+\lambda\beta(a+b+x_0)-\beta\leq 3\lambda\beta-\beta< 3\omega_1\beta-\beta\leq 1,$$
which is a contradiction, where the second inequality follows from \eqref{eq:del-Pezzo-dynamic-alpha-degree-6-pointline-0b-bound0}
and \eqref{eq:del-Pezzo-dynamic-alpha-degree-6-pointline-0b-bound1}. The last inequality follows by case analysis: if $0<\beta\leq \frac{1}{2}$, then $3\omega_1\beta-\beta=2\beta\leq 1$ and if $\frac{1}{2}\leq\beta\leq 1$, then $3\omega_1\beta-\beta=\frac{3}{2}-\beta\leq 1$. Therefore $t_1\in ((C^1\cap L_{q}^1)\cup(E_1^1))\cap F_1$. If $t_1=E_1^1\cap F_1$, then $t_1\neq C^1\cap L_{q}^1$ and the pair
$$(S_1, \lambda\beta (aE_1^1+\Omega^1) + (\lambda\beta m_0-\beta)F_1)$$
is not log canonical at $t_1$. Applying Lemma \ref{lem:adjunction} (iii) with $E_1^1$, we obtain
$$1<E_1^1\cdot (\lambda\beta\Omega^1 + (\lambda\beta m_0-\beta)F_1)=\lambda\beta(1+2a)-\beta\leq 3\lambda\beta-\beta\leq 1,$$
applying \eqref{eq:del-Pezzo-dynamic-alpha-degree-6-pointline-0b-bound0}. Therefore the pair
$$(S_1, (1-\beta)C^1 +\lambda\beta (bL_q^1 +\Omega^1) + (\lambda\beta m_0-\beta)F_1)$$
is not log canonical at $t_1=C^1\cap L_{q}^1\cap F_1$.

Since $t_1\in C^1$ but $E_1\cdot C=1$, then $E_1^1\cdot C^1=0$ so $t_1\not\in E_1$. Hence $m_1=b+x_1$, since $t_1\in L_q^1$, given that $C^1\cdot L_q^1=1$. We conclude $m_i=x_i$ for $i\geq 2$. Moreover
\begin{equation}
2=D\cdot L_q \geq a+x_0+x_1
\label{eq:del-Pezzo-dynamic-alpha-degree-6-pointline-0b-bound2}
\end{equation}

Applying Lemma \ref{lem:log-pullback-preserves-lc} the pair
\begin{equation}
(S_2, (1-\beta)C^2 +\lambda\beta (bL_q^2 +\Omega^2) + (\lambda\beta m_0-\beta)F_1^2+(\lambda\beta (m_0+m_1)-2\beta)F_2)
\label{eq:del-Pezzo-dynamic-alpha-degree-6-pointline-0b-badpair2}
\end{equation}
is not log canonical at some (possibly all) $t_2\in F_2$. By \eqref{eq:del-Pezzo-dynamic-alpha-degree-6-pointline-0b-bound3} and \eqref{eq:del-Pezzo-dynamic-alpha-degree-6-pointline-0b-bound2} inequality 
\begin{equation}
\lambda\beta(m_0+m_1)-2\beta=\lambda\beta(a+2b+x_0+x_1)-2\beta\leq 4\lambda\beta-2\beta<4\omega_i\beta-2\beta\leq 4\omega_1\beta-2\beta\leq 1
\label{}
\end{equation}
holds. The last inequality follows by case analysis. If $0<\beta\leq \frac{1}{2}$, then $4\omega_1\beta-2\beta=2\beta\leq 1$. If $\frac{1}{2}\leq\beta\leq 1$, then $4\omega_1\beta-2\beta\leq 2-2\beta\leq 1$.

Therefore, by Lemma \ref{lem:adjunction} (ii), the pair \eqref{eq:del-Pezzo-dynamic-alpha-degree-6-pointline-0b-badpair2} is log canonical in codimension $1$, and as a result $t_2$ is isolated. Observe that $L_q^2\cdot C^2=(F_1^2)\cdot C^2=L_q^2 \cdot C^2 = 0$. Therefore we have essentially $4$ distinct situations. Either
\begin{itemize}
	\item[(i)] the point $t_2\not\in F_2\cap (L_{q}^2\cup F_1^2\cup C^2)$, or
	\item[(ii)] the point $t_2=F_2\cap L_{q}^2$, or
	\item[(iii)] the point $t_2=F_2\cap F_1^2$, or
	\item[(iv)] the point $t_2=F_2\cap C^2=q_2$.
\end{itemize}
\textbf{Case (i):} the pair
$$(S_2, \lambda\beta (\Omega^2) + (\lambda\beta (m_0+m_1)-2\beta)F_2)$$
is not log canonical. But then, by Lemma \ref{lem:adjunction} (i), and adding \eqref{eq:del-Pezzo-dynamic-alpha-degree-6-pointline-0b-bound0}, \eqref{eq:del-Pezzo-dynamic-alpha-degree-6-pointline-0b-bound1}, \eqref{eq:del-Pezzo-dynamic-alpha-degree-6-pointline-0b-bound2}, \eqref{eq:del-Pezzo-dynamic-alpha-degree-6-pointline-0b-bound3} we obtain a contradiction:
\begin{align}
1&<\mult_{t_2}(\lambda\beta (\Omega^2) + (\lambda\beta (m_0+m_1)-2\beta)F_2)=\lambda\beta(m_0+m_1+x_2)-2\beta \label{eq:del-Pezzo-dynamic-alpha-degree-6-pointline-0b-bound-m0m1m2-trick}\\
&=\lambda\beta(a+2b+x_0+x_1+x_2)-2\beta\leq \lambda\beta(a+2b+2x_0+x_1)-2\beta\leq \nonumber\\
&\leq5\lambda\beta-2\beta<5\omega_i\beta-2\beta\leq 5\omega_2\beta-2\beta\leq 1.\nonumber
\end{align}
The last inequality follows by case analysis on $\beta$. If $0<\beta\leq \frac{1}{3}$, then $5\omega_1\beta-2\beta=3\beta\leq 1$. If $\frac{1}{3}\leq\beta\leq \frac{3}{4}$, then $5\omega_1\beta-2\beta=1$.  If $\frac{3}{4}\leq\beta\leq 1$, then $5\omega_1\beta-2\beta=\frac{5}{2}-2\beta\leq 1$. 

\textbf{Case (ii):} $t_2=F_2\cap L_{q}^2$. Then the pair
$$(S_2, \lambda\beta (bL_q^2 +\Omega^2) +(\lambda\beta (m_0+m_1)-2\beta)F_2)$$
is not log canonical at $t_2$. It follows from \eqref{eq:del-Pezzo-dynamic-alpha-degree-6-pointline-0b-bound0} and \eqref{eq:del-Pezzo-dynamic-alpha-degree-6-pointline-0b-bound1} that $b+x_2\leq b+x_0\leq 2$. Since $\lambda\beta< \frac{1}{2}$, applying Lemma \ref{lem:adjunction} (iii) to our pair, we obtain a contradiction:
$$1<\lambda\beta(bL_q^2+\Omega^2)\cdot F_2=\lambda\beta (b+x_2)\leq 2\omega_1\beta\leq 1.$$

\textbf{Case (iii):} $t_2=F_2\cap F_1^2$. Then the pair
$$(S_2, \lambda\beta (\Omega^2) + (\lambda\beta m_0-\beta)F_1^2+(\lambda\beta (m_0+m_1)-2\beta)F_2)$$
is not log canonical at $t_2$. Applying Lemma \ref{lem:adjunction} (iii) to this pair with $F_2$, then
\begin{align*}
1<&\lambda\beta (\Omega^2) +(\lambda\beta m_0-\beta)F_1^2)\cdot F_2\\
\leq&\lambda\beta(a+b+x_0+x_1)-\beta<\\
\leq&3\omega_1\beta-\beta\leq 1,
\end{align*}
using \eqref{eq:del-Pezzo-dynamic-alpha-degree-6-pointline-0b-bound2} and \eqref{eq:del-Pezzo-dynamic-alpha-degree-6-pointline-0b-bound3}. The last inequality is easy to see case by case: if $0<\beta\leq\frac{1}{2}$, then $3\omega_1\beta-\beta=2\beta\leq 1$ and if $\frac{1}{2}\leq\beta\leq1$, then $3\omega_1\beta-\beta=\frac{3}{2}-\beta\leq 1$.

Therefore the pair
$$(S_2, (1-\beta)C^2 +\lambda\beta (\Omega^2) + (\lambda\beta (m_0+m_1)-2\beta)F_2)$$
is not log canonical at $t_2=F_2 \cap C^2=q_2$ (which is precisely case (iv) above). By Lemma \ref{lem:log-pullback-preserves-lc}, the pair
$$(S_3, (1-\beta)C^3 +\lambda\beta (\Omega^3) + (\lambda\beta (m_0+m_1)-2\beta)F_2^3+(\lambda\beta (m_0+m_1+m_2)-3\beta)F_3)$$
is not log canonical at some $t_3 \in F_3$. If $\lambda\beta(m_0+m_1+m_2)-3\beta\leq 1$, then such $t_3$ is isolated. This is indeed the case:
$$\lambda\beta (m_0+m_1+m_2)-3\beta=\lambda\beta(m_0+m_1+x_2)-3\beta\leq \lambda\beta(m_0+m_1+x_2)-2\beta\leq 1$$
as seen in \eqref{eq:del-Pezzo-dynamic-alpha-degree-6-pointline-0b-bound-m0m1m2-trick}. Observe that since $F_2\cdot C^2=1$, then $F_2^3\cap C^3=\emptyset$. The point $t_3$ can be in essentially $3$ different positions: $t_3\not\in (F_2^3\cup C^3)$, $t_3=F_2^3\cap F_3$ or $t_3=C^3\cap F_3=q_3$.

If $t_3\not\in (F_2^3\cup C^3)$ then the pair 
$$(S_3, \lambda\beta (\Omega^3) +(\lambda\beta (m_0+m_1+m_2)-3\beta)F_3)$$
is not log canonical at $t_3$. By Lemma \ref{lem:adjunction} (iii) applied with $F_3$ we obtain a contradiction:
$$1<\lambda\beta\Omega\cdot F_3<\omega_1\beta x_3\leq \frac{1}{2}(a+x_0+x_1)\leq 1,$$
by \eqref{eq:del-Pezzo-dynamic-alpha-degree-6-pointline-0b-bound2}.

If $t_3= (F_2^3\cup F_3)$ then the pair 
$$(S_3, \lambda\beta (\Omega^3) + (\lambda\beta (m_0+m_1)-2\beta)F_2^3+(\lambda\beta (m_0+m_1+m_2)-3\beta)F_3)$$
is not log canonical at $t_3$. By Lemma \ref{lem:adjunction} (iii) applied with $F^3_2$ we obtain a contradiction:
\begin{align*}
1<&F^3_2\cdot(\lambda\beta (\Omega^3) +(\lambda\beta (m_0+m_1+m_2)-3\beta)F_3)\leq\\
\leq& \lambda\beta(x_1-x_2+m_0+m_1+m_2)-3\beta\\
=&\lambda\beta(a+2b+x_0+2x_1)-3\beta\\
<& 5\omega_2\beta-3\beta\leq 5\omega_2\beta-2\beta\leq 1,
\end{align*}
where we add \eqref{eq:del-Pezzo-dynamic-alpha-degree-6-pointline-0b-bound0}, \eqref{eq:del-Pezzo-dynamic-alpha-degree-6-pointline-0b-bound1}, \eqref{eq:del-Pezzo-dynamic-alpha-degree-6-pointline-0b-bound2} and \eqref{eq:del-Pezzo-dynamic-alpha-degree-6-pointline-0b-bound3}. The last inequality follows from \eqref{eq:del-Pezzo-dynamic-alpha-degree-6-pointline-0b-bound-m0m1m2-trick}.

Hence $t_3=C^3\cap F_3=q_3$, but that implies the pair
$$(S_3, (1-\beta)C^3 +\lambda\beta (\Omega^3) + (\lambda\beta (m_0+m_1+m_2)-3\beta)F_3)$$
is not log canonical only at $q_3$, finishing the proof in case (A2).

\textbf{Case (B):} We may assume the curve $C_q$ is irreducible, where $C_q\sim \pioplane{2}-E_1-E_2-E_3$ with $q_0\in C_q$ and $q_1\in C_q^1$. Recall $\lambda<\omega_i\leq \omega_1$ and $L_q\sim\pioplane{1}-E_1$ is irreducible with $q_0\in E_1$.

Notice that 
$$B:=E_1+L_q+C_q\sim -K_S.$$
By Lemma \ref{lem:appendix-delPezzo-deg6-dynamic-lct3} the pair
$$(S, (1-\beta)C + \lambda\beta B)$$
is log canonical. Hence, by Lemma \ref{lem:log-convexity} we may assume there is an irreducible curve $Z\subseteq\Supp(B)$ such that $Z\not\subseteq\Supp(D)$. Notice that $E_1\subset \Supp(D)$, since otherwise we get the following contradiction
$$1\geq 1-\beta + \lambda\beta=((1-\beta)C + \lambda\beta D)\cdot E_1\geq\mult_{q_0}((1-\beta)C + \lambda\beta D))>1$$
where we use Lemma \ref{lem:adjunction} (i). Therefore, by Lemma \ref{lem:log-convexity}, either
\begin{itemize}
	\item[(B1)]$L_q\not\subseteq \Supp(D)$, or 
	\item[(B2)]$C_q\not\subseteq\Supp(D)$,
\end{itemize}
since $E_1\subseteq\Supp(D)$.

\textbf{Case (B1):} $L_q\not\subseteq \Supp(D)$.

We may write $D=aE_1+bC_q+\Omega$ where $L_q, E_1, C_q\not\subseteq \Supp(\Omega)$. Let $x_0=\mult_{q_0}\Omega$. Then $m_0=a+b+x_0$. We bound the multiplicity of $D$:
\begin{align}
1&=D\cdot E_1 \geq -a+b+x_0 \label{eq:del-Pezzo-dynamic-alpha-degree-6-pointline-1-bound1}\\
2&=D\cdot L_q \geq a+b+x_0=m_0\label{eq:del-Pezzo-dynamic-alpha-degree-6-pointline-1-bound3}
\end{align}

In particular
\begin{equation}
\lambda\beta m_0 \leq \omega_1\beta m_0\leq 1.
\label{eq:del-Pezzo-dynamic-alpha-degree-6-pointline-1-m0}
\end{equation}
This is a necessary hypothesis for Theorem \ref{thm:pseudo-inductive-blow-up} (ii)-(iv) for $i\geq 2$ and we will assume it from now onwards.

Hence, by Theorem \ref{thm:pseudo-inductive-blow-up} (i) with $i=1$, we have that the pair
$$(S_1, (1-\beta)C^1 +\lambda\beta D^1 + (\lambda\beta m_0-\beta)F_1)$$
is not log canonical at some $t_1\in F_1$. Observe \eqref{eq:del-Pezzo-dynamic-alpha-degree-6-pointline-1-m0} is a necessary condition for Theorem \ref{thm:pseudo-inductive-blow-up} (ii)-(iv) whenever $i\geq 2$. Also, by Theorem \ref{thm:pseudo-inductive-blow-up} (ii) when $i=1$ we conclude that $t_1=F_1\cap C^1=q_1$. But then part (i) of Theorem \ref{thm:pseudo-inductive-blow-up} when $i=2$ gives that
\begin{equation}
(S_2, (1-\beta) C^1 + \lambda\beta D^1 + (\lambda\beta m_0-\beta)F^2_1 +(\lambda\beta(m_0+m_1)-2\beta)F_2)
\label{eq:del-Pezzo-dynamic-alpha-degree-6-pointline-badpair2}
\end{equation}
is not log canonical at some $t_2\in F_2$.

If $0<\beta\leq \frac{1}{2}$, then $4\omega_1\beta-2\beta=4\beta-2\beta=2\beta\leq 1$. If $\frac{1}{2}\leq \beta\leq 1$ then $4\omega_1\beta-2\beta=2-2\beta\leq 1$. Therefore we have proven
\begin{equation}
2\lambda\beta m_0-2\beta\leq 4\lambda\beta-2\beta\leq 4\omega_1\beta-2\beta\leq 1,
\label{eq:del-Pezzo-dynamic-alpha-degree-6-pointline-1-bound4}
\end{equation}
by means of \eqref{eq:del-Pezzo-dynamic-alpha-degree-6-pointline-1-bound3}. 

Notice that by \eqref{eq:del-Pezzo-dynamic-alpha-degree-6-pointline-1-bound3} 
$$\lambda\beta(2m_0)-2\beta\leq 4\lambda\beta-2\beta< 4\omega_1\beta-2\beta\leq 1.$$
Indeed, if $0<\beta\leq \frac{1}{2}$, then $4\omega_1\beta-2\beta=2\beta\leq 1$ while if $\frac{1}{2}\leq \beta\leq 1$, then $4\omega_1\beta-2\beta=2-2\beta\leq 1$.
Using this, Theorem \ref{thm:pseudo-inductive-blow-up} (iv) with $i=2$ implies $t_2 = F_2\cap C^2=q_2$.

Let $x_i=\mult_{q_i}\Omega^i$. Since $q_1\in C^1$ but $E_1\cdot C=1$, then $E_1^1\cdot C^1=0$ so $q_1\not\in E_1$. Hence $m_1=b+x_1$. 
Observe that since $q_1\in C_q^1\cap C^1$, then $3=C_q\cdot C\geq(C_q\cdot C)\vert_{q_0}\geq 2$. If $(C_q\cdot C)\vert_{q_0}=2$, then $C^3\cdot (C_q)^2=0$ so $q_2\not \in C_q^2$ and then $m_2=x_2$. Otherwise $q_2\in C^2_q$, $C^2_q \cdot C^2 =1$ and $m_2=b+x_2$. In both cases $m_i=x_i$ for $i\geq 3$. 
\begin{equation}
3=D\cdot C_q\geq a+b+x_0+x_1=m_0+x_1.
\label{eq:del-Pezzo-dynamic-alpha-degree-6-pointline-1-bound2}
\end{equation}

Since
$$\lambda\beta(m_0+m_1)-2\beta\leq \lambda\beta(2m_0)-2\beta\leq 1$$
we can apply Theorem \ref{thm:pseudo-inductive-blow-up} (i) when $i=3$ and conclude that the pair
$$(S_3, (1-\beta)C^3 + \lambda\beta D^3 + (\lambda\beta(m_0+m_1) -2\beta)F^3_2 + (\lambda\beta(m_0+m_1+m_2)-3\beta)F_3)$$
is not log canonical at some $t_3\in F_3$.

Observe that using \eqref{eq:del-Pezzo-dynamic-alpha-degree-6-pointline-1-bound1} and \eqref{eq:del-Pezzo-dynamic-alpha-degree-6-pointline-1-bound3} the inequality
$$m_0+2m_1\leq a+3b+x_0+x_1\leq a+3b+3x_0=(2a+2b+2x_0)+(-a+b+x_0)\leq 4+1=5,$$
holds. Therefore
$$\lambda\beta(m_0+2m_1)-3\beta< 5\omega_1\beta-3\beta\leq 1$$
since for $0<\beta\leq\frac{1}{2}$, we have $5\omega_1\beta-3\beta=	2\beta\leq 1$, and for $\frac{1}{2}\leq \beta\leq 1$ we have $5\omega_1\beta-3\beta\leq \frac{5}{2}-3\beta\leq 1$.
Hence, by Theorem \ref{thm:pseudo-inductive-blow-up} (iii) with $i=3$ we conclude $t_3=F_3\cap C^3=q_3$. Moreover, since
$$\lambda\beta(m_0+m_1+m_2)-3\beta\leq \lambda\beta (m_0+2m_1)-3\beta\leq 1$$
we can apply Theorem \ref{thm:pseudo-inductive-blow-up} (i) with $i=4$, and show that
\begin{equation}
(S_4, (1-\beta)C^4 + \lambda\beta D^4 + (\lambda\beta(m_0+m_1+m_2)-3\beta)F_3^4 + (\lambda\beta(m_0+m_1+m_2+m_3)-4\beta)F_4)
\label{eq:del-Pezzo-dynamic-alpha-degree-6-pointline-badpair4}
\end{equation}
is not log canonical at some $t_4\in F_4$.

Suppose $(C_q\cdot C)\vert_{q_0}=2$. Then $m_2=x_2$. and using \eqref{eq:del-Pezzo-dynamic-alpha-degree-6-pointline-1-bound2} and \eqref{eq:del-Pezzo-dynamic-alpha-degree-6-pointline-1-bound3} we obtain
$$m_0+m_1+2m_2\leq 2b+a+x_0+x_1+2x_2\leq 3+2x_2+a-a\leq 5+x_2-a\leq 6.$$
Therefore
$$\lambda\beta(m_0+m_1+2m_2)-4\beta\leq 6\lambda\beta-4\beta<6\omega_1\beta-4\beta\leq 1$$
since for $0<\beta\leq \frac{1}{2}$ $6\omega_1\beta-4\beta=2\beta\leq 1$ and for $\frac{1}{2}\leq \beta\leq 1$ we have $6\omega_1\beta-4\beta=3-4\beta\leq 1$. Now we can use Theorem \ref{thm:pseudo-inductive-blow-up} (iii) with $i=4$ to prove that pair \eqref{eq:del-Pezzo-dynamic-alpha-degree-6-pointline-badpair4} is not log canonical only at $t_4=F_4\cap C^4=q_4$, which implies the claim in case (B1) when $(C_q\cdot C)\vert_{q_0}=2$.

Now suppose $(C_q\cdot C)\vert_{q_0}=3$. Then
\begin{equation}
\lambda\beta(m_0+m_1+m_2+m_3)-4\beta\leq 1.
\label{eq:del-Pezzo-dynamic-alpha-degree-6-pointline-crazycase}
\end{equation}
Indeed, notice that
\begin{align}
m_0+m_1+m_2+m_3	&=a+3b+x_0+x_1+x_2+x_3\leq  a+3b+x_0+x_1+2x_2\label{eq:del-Pezzo-dynamic-alpha-degree-6-pointline-crazycase2}\\
								&\leq3+2b+2x_1-a+a\nonumber\\
								&\leq 3+1+a+b+x_0\nonumber\\
								&\leq 6,\nonumber
\end{align}
by \eqref{eq:del-Pezzo-dynamic-alpha-degree-6-pointline-1-bound2}, \eqref{eq:del-Pezzo-dynamic-alpha-degree-6-pointline-1-bound1} and \eqref{eq:del-Pezzo-dynamic-alpha-degree-6-pointline-1-bound3} in each step. Now, if $0<\beta\leq \frac{1}{2}$, then $6\omega_1\beta-4\beta=2\beta\leq1$ and $\frac{1}{2}\leq \beta\leq 1$, then $6\omega_1\beta-4\beta=3-4\beta\leq 1$. Inequality \eqref{eq:del-Pezzo-dynamic-alpha-degree-6-pointline-crazycase} is proven. In particular \eqref{eq:del-Pezzo-dynamic-alpha-degree-6-pointline-badpair4} is log canonical in codimension $1$, i.e. $t_4$ is isolated.

Since $F_4$ has empty intersection with $C_q$ and $E_1$, if $t_4\neq C^4\cap F_4=q_4$, then the pair
$$(S_4, \lambda\beta \Omega^4 + (\lambda\beta(m_0+m_1+m_2)-3\beta)F_3^4 + (\lambda\beta(m_0+m_1+m_2+m_3)-4\beta)F_4)$$
is not log canonical at $q_4\in F_4$. Applying Lemma \ref{lem:adjunction} (iii) to this pair with $F_3^4$ we obtain a contradiction:
\begin{align*}
1&<(\lambda\beta\Omega^4+\lambda\beta(m_0+m_1+m_2+m_3)-4\beta)F_4)\cdot F_3^4\\
&=\lambda\beta(x_2-x_3+a+3b+x_0+x_1+x_2+x_3)-4\beta\\
&\leq6\omega_1\beta-4\beta\leq 1,
\end{align*}
by \eqref{eq:del-Pezzo-dynamic-alpha-degree-6-pointline-crazycase2} and \eqref{eq:del-Pezzo-dynamic-alpha-degree-6-pointline-crazycase}, which gives a contradiction. Therefore $t_4=C^4\cap F_4=q_4$ and the Lemma is proven in case (B1).

\textbf{Case (B2):} $C_q\not\subseteq \Supp(D)$.

We may write $D=aE_1+bL_q+\Omega$ where $L_q, E_1, C_q\not\subseteq \Supp(\Omega)$. Let $x_0=\mult_{q_0}\Omega$. 

We bound the multiplicities of $D$:
\begin{align}
1&=D\cdot E_1 \geq -a+b+x_0 \label{eq:del-Pezzo-dynamic-alpha-degree-6-pointline-1-bound6}\\
2&=D\cdot L_q \geq a+x_0.\label{eq:del-Pezzo-dynamic-alpha-degree-6-pointline-1-bound8}
\end{align}
Notice that $3=D\cdot C_q\geq m_0$ so
\begin{equation}
\lambda\beta m_0-\beta<3\omega_1\beta -\beta \leq 1.
\label{eq:del-Pezzo-dynamic-alpha-degree-6-pointline-2-m0}
\end{equation}
For the last inequality, observe that if $0<\beta\leq \frac{1}{2}$, then $3\omega_1\beta-\beta=2\beta\leq 1$, while if $\frac{1}{2}\leq \beta\leq 1$, then $3\omega_1\beta-\beta=\frac{3}{2}-\beta\leq 1$.

Hence, by Theorem \ref{thm:pseudo-inductive-blow-up} (i) with $i=1$, we have that the pair
$$(S_1, (1-\beta)C^1 +\lambda\beta D^1 + (\lambda\beta m_0-\beta)F_1)$$
is not log canonical at some $t_1\in F_1$.

Since $L_q\cdot E_1=1$, then $L_q^1\cap E_1^1=\emptyset$. Suppose $t_1\neq C^1\cap F_1=q_1$. If $t_1\neq L_q^1\cap F_1$, then the pair
$$(S_1,\lambda\beta(aE_1^1+\Omega^1)+(\lambda\beta m_0 -\beta)F_1)$$
is not log canonical at $t_1$. By Lemma \ref{lem:adjunction} (iii) and \eqref{eq:del-Pezzo-dynamic-alpha-degree-6-pointline-1-bound8}
$$1<F_1\cdot (\lambda\beta(aE_1+\Omega^1)=\lambda\beta(a+x_0)<2\omega_1\beta\leq 1,$$
which is absurd.

If $t_1=L_q^1\cap F_1, t_1\neq q_1$, the pair
$$(S_1,\lambda\beta(bL_q^1+\Omega^1)+(\lambda\beta m_0 -\beta)F_1)$$
is not log canonical at $t_1$. Then, by Lemma \ref{lem:adjunction} (iii) applied with $F_1$ we get
$$1<\lambda\beta F_1(bL_q^1+\Omega^1)<\omega_1\beta(b+x_0)\leq 2\omega_1\beta\leq 1,$$
which is absurd. We have used $b+x_0\leq 2$, which follows from adding \eqref{eq:del-Pezzo-dynamic-alpha-degree-6-pointline-1-bound6} to $3=D\cdot C_1\geq a+b+x_0$. Therefore $t_1=q_1=C^1\cap F_1$.

Let $x_i=\mult_{q_i}\Omega^i$. Then $m_0=a+b+x_0$. Since $E_1\cdot C=E_1\cdot L_q = E_1\cdot C_q=L_q\cdot C_q=1$, then $E_1^1\cdot C^1=C^1\cdot L_q^1=C_q^1\cdot L_q^1=0$ so $q_1\not\in E_1$, $q_1\not\in L_q$. Hence $m_i=x_i$ for $i\geq 1$. Then
\begin{equation}
3=D\cdot C_q\geq a+b+x_0+x_1=m_0+m_1.
\label{eq:del-Pezzo-dynamic-alpha-degree-6-pointline-1-bound7}
\end{equation}
Observe that \eqref{eq:del-Pezzo-dynamic-alpha-degree-6-pointline-1-bound7} and inspection in \eqref{eq:del-Pezzo-dynamic-alpha-degree-6-pointline-2-m0} give that $\lambda\beta(m_0+m_1)-\beta\leq 1$, which is a necessary condition for Theorem \ref{thm:pseudo-inductive-blow-up} (ii)-(iv) whenever $i\geq 2$. We will assume it from now on.

Observe that
\begin{equation}
\lambda\beta(m_0+m_1)-\beta\leq 3\omega_1\beta-\beta\leq 1
\label{eq:del-Pezzo-dynamic-alpha-degree-6-pointline-2-tiredofthis}
\end{equation}
by \eqref{eq:del-Pezzo-dynamic-alpha-degree-6-pointline-1-bound7} and \eqref{eq:del-Pezzo-dynamic-alpha-degree-6-pointline-2-m0}. Now we apply 
Theorem \ref{thm:pseudo-inductive-blow-up} (iii) when $i=2$ to get that
\begin{equation}
(S_2, (1-\beta) C^1 + \lambda\beta D^1 + (\lambda\beta m_0-\beta)F^2_1 +(\lambda\beta(m_0+m_1)-2\beta)F_2)
\label{eq:del-Pezzo-dynamic-alpha-degree-6-pointline-2-badpair2}
\end{equation}
is not log canonical only at $q_2=F_2\cap C^2$.

From inequality \eqref{eq:del-Pezzo-dynamic-alpha-degree-6-pointline-2-tiredofthis} we deduce
$$\lambda\beta(m_0+m_1)-2\beta\leq \lambda\beta(m_0+m_1)-\beta\leq 1$$
which is the condition in Theorem \ref{thm:pseudo-inductive-blow-up} (i) when $i=3$. Therefore the pair
$$(S_3, (1-\beta)C^3 + \lambda\beta D^3 + (\lambda\beta(m_0+m_1) -2\beta)F^3_2 + (\lambda\beta(m_0+m_1+m_2)-3\beta)F_3)$$
is not log canonical at some $t_3\in F_3$.

Using \eqref{eq:del-Pezzo-dynamic-alpha-degree-6-pointline-1-bound7} and \eqref{eq:del-Pezzo-dynamic-alpha-degree-6-pointline-1-bound8} we obtain
$$\lambda\beta(m_0+2m_1)-3\beta\leq \lambda\beta (3 + 2-a)-3\beta\leq 5\omega_1\beta-3\beta\leq1.$$
To prove the last inequality observe that for $0<\beta\leq \frac{1}{2}$ we have $5\omega_1\beta-3\beta=5\beta-3\beta\leq 1$. For $\frac{1}{2}\leq\beta\leq 1$, $5\omega_1\beta-3\beta=\frac{5}{2}-3\beta\leq 1.$
Hence, by Theorem \ref{thm:pseudo-inductive-blow-up} (iv) with $i=3$ we conclude $t_3=F_3\cap C^3=q_3$. Moreover, since
$$\lambda\beta(m_0+m_1+m_2)-3\beta\leq \lambda\beta (m_0+2m_1)-3\beta\leq 1$$
we can apply Theorem \ref{thm:pseudo-inductive-blow-up} (i) with $i=4$, and show that
\begin{equation}
(S_4, (1-\beta)C^4 + \lambda\beta D^4 + (\lambda\beta(m_0+m_1+m_2)-3\beta)F_3^4 + (\lambda\beta(m_0+m_1+m_2+m_3)-4\beta)F_4)
\label{eq:del-Pezzo-dynamic-alpha-degree-6-pointline-2-badpair4}
\end{equation}
is not log canonical at some $t_4\in F_4$.

We use \eqref{eq:del-Pezzo-dynamic-alpha-degree-6-pointline-1-bound6}, \eqref{eq:del-Pezzo-dynamic-alpha-degree-6-pointline-1-bound7} and \eqref{eq:del-Pezzo-dynamic-alpha-degree-6-pointline-1-bound8} to show
$$\lambda\beta(m_0+m_1+2m_2)-4\beta\leq \lambda\beta(3+2x_0)-4\beta\leq \lambda\beta(3-a+b+x_0+a+x_0)-4\beta\leq 6\lambda\beta-4\beta\leq 6\omega_1\beta-4\beta\leq1.$$
To verify the last inequality observe that for $0<\beta\leq \frac{1}{2}$ we have $6\omega_1\beta-4\beta=6\beta-4\beta\leq 1$. For $\frac{1}{2}\leq\beta\leq 1$, $6\omega_1\beta-4\beta=\frac{6}{2}-4\beta\leq 1.$

Now we can use Theorem \ref{thm:pseudo-inductive-blow-up} (iv) with $i=4$ to prove that the pair \eqref{eq:del-Pezzo-dynamic-alpha-degree-6-pointline-2-badpair4} is not log canonical only at $t_4=F_4\cap C^4=q_4$, which implies the claim in case (B2).
\end{proof}

Lemmas \ref{lem:del-Pezzo-dynamic-alpha-degree-6-pointgeneric} and \ref{lem:del-Pezzo-dynamic-alpha-degree-6-pointline} cover all possible $q_0\in S$, proving Claim \ref{clm:del-Pezzo-dynamic-alpha-degree-6-proof} and finishing the proof of Theorem \ref{thm:del-Pezzo-dynamic-alpha-degree-6}.

\section{Del Pezzo surface of degree $5$}
The method for this surface is different than in higher degrees and more similar to the proof of Theorem \ref{thm:del-Pezzo-glct-charp} when $K_S^2=4$. Therefore first we need to find and classify low degree curves in $S$ in order to construct certain $\bbQ$-divisors with good properties that we will use in the proof of Theorem \ref{thm:del-Pezzo-dynamic-alpha-degree-5}.
\subsection{Curves of low degree and models of $S$}
\label{sec:delPezzo-5-curves}
Let $\pi\colon S\ra \bbP^2$ be the blow-up at $p_1,\ldots,p_4\in \bbP^2$ in general position. Let $E_1,\ldots,E_4$ be the exceptional divisors. Recall that
$-K_S\sim \pioplane{3}-\sum^4_{i=1}E_i$ and $E_i^2=-1$.
\begin{table}[!ht]%
\begin{center}
\begin{tabular}{|c|c|c|c|c|c|}
\hline
Linear system $\calL\calS$																																						&$\deg C$	&$C^2$	&Fix $p$	&Fix $q$	&$C'$\\
\hline \hline
$\vert E_{i}\vert$																																										&$1$			&$-1$		&N					&N			&$E_i$\\
\hline
$\calL_{ij}= \left \vert \pioplane{1} -E_i-E_j\right\vert$																						&$1$			&$-1$		&N					&N			&$L_{ij}$\\
\hline
$\calB_{i}= \left \vert \pioplane{1} -E_i\right\vert$																									&$2$			&$0$		&Y					&N			&$B_i$\\
\hline
$\displaystyle{\calA = \left \vert \pioplane{2} -\sum^4_{j=1}E_j \right\vert}$												&$2$			&$0$		&Y					&N			&$A$\\
\hline
$\calR= \left \vert \pioplane{1}\right\vert$																													&$3$			&$1$		&Y					&Y			&$R$\\
\hline
$\calR_{i}= \left \vert \pioplane{2} -\sum^4_{\substack{j=1\\ j\neq i}}E_j\right\vert$								&$3$			&$1$		&Y					&Y			&$R_i$\\
\hline
\end{tabular}
\end{center}
\caption{Catalogue of curves of low degree of the del Pezzo surface of degree $5$.}
\label{tab:del-Pezzo-5-lowdegree-curves}
\end{table}

Observe Table \ref{tab:del-Pezzo-5-lowdegree-curves}. In the first column we have defined certain complete linear systems $\calL\calS$ in $S$. Let $C\in \calL\calS$ be any divisor. Its numerical properties ($C^2, \deg(C)$) are the same for any divisor in a given $\calL\calS$ and are easy to compute. We list them in the second and third columns of Table \ref{tab:del-Pezzo-5-lowdegree-curves}. Note that, by the genus formula, $p_a(C)=0$ in all cases. 

If $\deg C=2$, then by Proposition \ref{prop:rational-surfaces-sections-rational-curves}, $h^0(\calL\calS)\geq 2$. Take $\calL\calS'\subset \calL\calS$ to be the sublinear system fixing a given point $p\in S$. Then $h^0(\calL\calS')\geq 1$ and we can find a curve $C'\in\calL\calS$ with $p\in C'$. The notation for each particular $C'$ is in the last column of the table. 

When the curve $C'$ is irreducible, we can see it as the strict transform of an irreducible curve in $\bbP^2$ via the model $\pi$. For instance $L_{ij}$ is the strict transform of the unique line through $p_i$ and $p_j$. $B_i$ is the strict transform of a line passing through $p_i$ and $p$ and $A$ is the strict transform of a conic through all $p_j$.

If $\deg C=3$, then let $\sigma\colon \widetilde S \ra S$ be the blow-up of a point $p\in S$ with exceptional divisor $E$. Let $\calL\calS'=\{D\in \calL\calS \setsep p\in \Supp(D)\}$ and let $\widetilde{\calL\calS'}=\vert\sigma^*(\calL\calS')-E\vert$. By Proposition \ref{prop:rational-surfaces-sections-rational-curves}
$$h^0(\widetilde{\calL\calS'})=h^0(\calL\calS')=h^0(\calL\calS)-1\geq 2,$$
so we can choose $B\in \widetilde{\calL\calS'}$ an effective divisor passing through $q$. If $E\not\subset\Supp(B)$, then let $C'=\sigma_*(B)$ and $B=\widetilde C' \sim \sigma^*(C')-E$ where $B\cdot E=1$.

Conversely, if $E\subset \Supp(B)$, let $B=A+nE$ where $E\not\subset\Supp(A)$, $n\geq1$ and $A$ is effective. Then $\widetilde C'=A=B-nE\sim\sigma^*(C')-(n+1)E$ for $C'=\sigma_*(B)=\sigma_*(A)$ and $C'$ is singular at $p$. $C'$ is reducible, since otherwise $p_a(\widetilde C')<p_a(C')=0$, which is impossible. If $\calL \calS = \cal R$, then denote $C'=R$. If $\calL \calS = \calR_i$, then denote $C'=R_i$. 

When the curve $C'$ is irreducible, we can see it as the strict transform of an irreducible curve in $\bbP^2$ via the model $\pi$. For instance $L$ is the strict transform of the unique line through $\pi(p)$ such that its strict transform in $\widetilde S$ passes through $q$. $R_i$ is the strict transform of a conic passing through $\pi(p)$ and all $p_j$  for $j\neq i$ and such that its strict transform in $\widetilde S$ contains $q$.

In order to understand the geometry of $S$ we need to understand which are its curves of low degree and how they intersect each other. We have seen some of these curves. We want to show that all the lines in $S$ are the ones described in Table \ref{tab:del-Pezzo-5-lowdegree-curves}. Furthermore, we will show that the conics described in the table are all the conics in $S$ passing through a given point $p$. Moreover, since there is more than one model $S\ra \bbP^2$ to characterise $S$ as a blow-up of $\bbP^2$ in $4$ points, we need to show that we can choose a model adequate to our needs.

\begin{lem}
\label{lem:del-Pezzo-deg5-good-curves-smooth}
Let $S$ be a non-singular del Pezzo surface of degree $5$ and $C'$ a curve as in Table \ref{tab:del-Pezzo-5-lowdegree-curves}. Suppose $C'$ is irreducible. Then $C'$ is smooth.
\end{lem}
\begin{proof}
Suppose $C'\neq E_i$. Then $\pi(C')$ is an irreducible curve of degree $1$ or $2$ in $\bbP^2$. Therefore $\pi(C')$ is smooth. Since $S$ is just the blow-up of smooth points of $\bbP^2$, if $\pi(C')$ is a smooth curve of $\bbP^2$, then its strict transform $C'$ is a smooth curve in $S$.

If $C'=E_i$, then $C'$ is a line, which is smooth by Lemma \ref{lem:del-Pezzo-lines-numerical}.
\end{proof}

\begin{lem}
\label{lem:del-Pezzo-deg5-lines-list}
The $10$ lines in Table \ref{tab:delPezzo-4-lowdegree} are all the lines in $S$. The intersection of these lines are:
\begin{align*}
&E_i\cdot E_j = -\delta_{ij}, \qquad L_{ij}\cdot E_i = L_{ij}\cdot E_j = 1,\\
&{L_{ij}\cdot L_{kl}=\left\{ \begin{array}{rl}
																-1 &\text{if } i=k \text{ and } j=l \\
																0 & \text{if precisely two subindices are the same} \\
																1 & \text{if none of the subindices are the same.} \\
														\end{array} \right.}
\end{align*}
\end{lem}
\begin{proof}
See Lemma \ref{lem:del-Pezzo-lines9-d}.
\end{proof}

\begin{lem}
\label{lem:del-Pezzo-deg5-model-line-choice}
Given a line $L\subset S$, we can choose a model $\gamma\colon S\ra \bbP^2$ such that $L=E_1$. If $p=L_1\cap L_2$, the intersection of two lines, we can choose $\gamma$ such that $L_1=E_1,\ L_2=L_{12}$.
\end{lem}
\begin{proof}
We construct $\gamma\colon S \ra \bbP^2$ by contracting $4$ disjoint exceptional curves $F_i$ (i.e. $F_i\cdot F_j=0$ if $i\neq j$, $-K_S\cdot F_i=1$ and $F^2_i=-1$ for all $i$). Let $F_1=L$. 
\begin{itemize}
	\item[(i)] If $F_1=E_1$, take $F_i=E_i$, for $i=2,3,4$.
	\item[(ii)] If $F_1=L_{12}$, take $F_2=L_{13}$, $F_3=L_{23}$ and  $F_4=E_4$.
\end{itemize}
Obvious relabelling exhausts all $10$ lines in Lemma \ref{lem:del-Pezzo-deg5-lines-list}. By Castelnuovo contractibility criterion \cite[V.5.7]{HartshorneAG} we can contract each $F_i$, leaving every other point intact. Th	e image of $\gamma$ is $\bbP^2$, because the relative minimal model of $S$, once $4$ exceptional curves are contracted, is unique. For the second part we can assume already $L_1=E_1$ and run this lemma again. Then $L_2=L_{1j}$, since $L_1\cdot L_2=1$ and by relabelling the lines $F_i$, we may assume that $L_2=L_{12}$.
\end{proof}

\begin{lem}
\label{lem:del-Pezzo-deg5-conics}
If $C$ is a conic in $S$ passing through $p$, then $C=A$ or $C=B_i$, with $\pi(C)$ either a conic through all marked points but $p_i$ or a line through $p$ and $p_i$, respectively.
\end{lem}
\begin{proof}
By Lemma \ref{lem:del-Pezzo-all-conics-rational}, $C$ has arithmetic genus $p_a(C)=0$. By the genus formula, this implies $C^2=0$. Let $\bar C =\pi_*(C)\subset\bbP^2$, $\bar C\sim \oplaned$ for some $d\geq 1$. Hence $C\sim \pioplaned-\sum a_i E_i$ for $a_i\geq 0$. This gives
\begin{equation}
0=C^2=d^2-\sum a_i^2,\qquad \qquad 2=(-K_S)\cdot C=3d-\sum a_i.
\label{eq:del-Pezzo-deg5-conics-proof}
\end{equation}
Given that $a_i$ are non-negative integers, we have $\sum a_i^2\geq \sum a_i$. Hence
$$0=d^2-\sum a_i^2\leq d^2-\sum a_i = d^2-3d+2=(d-1)(d-2),$$
so $d=1$ or $d=2$. The only possibilities for $a_i$ for \eqref{eq:del-Pezzo-deg5-conics-proof} to hold are $\sum a_i=2$ when $d=1$ and $\sum a_i=4$, when $d=2$. All these possibilities are precisely the ones classified in Table \ref{tab:delPezzo-4-lowdegree}, proving the Lemma.
\end{proof}

\begin{lem}
\label{lem:del-Pezzo-deg5-model-conic-choice}
There are no irreducible conics passing through pseudo-Eckardt points. Given an irreducible conic $C$ in $S$ and a point $p\in C$, we can choose a model $\gamma\colon S\ra \bbP^2$ such that $C=B_1$.
\end{lem}
\begin{proof}
First we prove the last assertion by case analysis on the position of $p\in C$.

If $C$ is reducible, then $C=L_1+L_2$ the union of two lines intersecting at a point $r=L_1\cap L_2$. By Lemma \ref{lem:del-Pezzo-deg5-model-line-choice} we can choose a model such that $C=E_1+L_{12}=B_1$.

If $C$ is irreducible we distinguish two cases: $p\in L$, a line or $p$ is in no line. If $p\in L$ a line, we may assume by Lemma \ref{lem:del-Pezzo-deg5-model-line-choice} that $L=E_1$. Then $C\neq B_j$ for $j\geq2$, since otherwise:
\begin{equation}
0=B_{j}\cdot E_1 =C \cdot E_1 \geq \mult_p(C)\cdot \mult_p(E_1)=1,
\label{eq:del-Pezzo-deg5-model-conic-choice-proof}
\end{equation}
a contradiction.

If $C=A$, take $F_i$ and $\gamma:S\ra \bbP^2$ as in the proof of Lemma \ref{lem:del-Pezzo-deg5-model-line-choice}, case (ii). Because $C$ is irreducible, $\overline C=\gamma(A)\sim\oplaned$ by the genus formula on $\bbP^2$. Moreover:
$$A\sim\gamma^*(\oplaned)-\sum_{i=1}^4 (F_i\cdot A) F_i= \gamma^*(\oplaned)-F_4,$$
and $2=A \cdot (-K_S)=3d -1,$
so $d=1$.
Therefore under the new blow-up $C$ is $B_4$. By obvious relabelling of the $F_j$ we can consider $C=B_1$.

We prove the first assertion. Suppose that $p=E_1\cap L_{12}$ but $C$ is irreducible. By Lemma \ref{lem:del-Pezzo-deg5-conics} and \eqref{eq:del-Pezzo-deg5-model-conic-choice-proof}, $C=A$ or $C=B_1$. However if $L_{12}\not\subseteq \Supp(C)$, then
$$0=B_1\cdot L_{12}\geq\mult_{p}B_1 \cdot \mult_{p}L_{12}\geq 1$$
$$0=B_1\cdot A\geq\mult_{p}B_1 \cdot \mult_{p}A\geq 1$$
which is absurd.
\end{proof}

The following two lemmas are needed in the proof of Theorem \ref{thm:del-Pezzo-dynamic-alpha-degree-5}. We give a joint proof after the statements.
\begin{lem}
\label{lem:del-Pezzo-deg5-aux-divisors-G}
Let $S$ be a non-singular del Pezzo surface of degree $5$. Let $C\sim -K_S$ be a non-singular curve and $p\in C$ be a point which belongs to at most one line. Let $0<\beta\leq 1$. Then there is an effective $\bbQ$-divisor $G=\sum g_i G_i\simq -K_S$ satisfying the following:
\begin{itemize}
	\item[(i)] The pair $(S,(1-\beta)C + \omega\beta G)$ is log canonical where $\omega=\min\{1, \frac{1}{2\beta}\}$.
	\item[(ii)] The point $p\in G_i$ for all $G_i$.
	\item[(iii)] Each $G_i$ is irreducible and smooth and has $\deg G_i\leq 2$.
\end{itemize}
\end{lem}

\begin{lem}
\label{lem:del-Pezzo-deg5-aux-divisors-H}
Let $S$ be a non-singular del Pezzo surface of degree $5$. Let $p\in S$, $q\in E\subset \widetilde S\stackrel{\sigma}{\lra} S$, the exceptional curve $E$ in the blow-up $\sigma$ of $S$ at $p$. Let $C\sim -K_S$ be a non-singular curve and $p\in C$. Assume that $p$ belongs to at most one $(-1)-$curve and that $q\not\in \widetilde L$, the strict transform of any $(-1)$-curve $L\subset S$. Moreover, if $p\in L$, a $(-1)$-curve, then $q\not\in \widetilde B$, the strict transform in $\widetilde S$ of any irreducible conic $B\subset S$. Let $0<\beta\leq 1$. Then there is an effective $\bbQ$-divisor $H=\sum h_i H_i\simq -K_S$ satisfying the following:
\begin{itemize}
	\item[(i)] The pair $(S,(1-\beta)C + \omega\beta H)$ is log canonical where $\omega=\min\{1, \frac{1}{2\beta}\}$.
	\item[(ii)] The point $p\in H_i$ for all $H_i$.
	\item[(iii)] Each $H_i$ is irreducible and smooth and has $\deg H_i\leq 3$.
	\item[(iv)] the point $q\in \widetilde H_i$, the strict transform of $H_i$ via $\sigma,\ \forall H_i$ such that $\deg H_i>1$.
\end{itemize}
\end{lem}

\begin{proof}[Proof of lemmas \ref{lem:del-Pezzo-deg5-aux-divisors-G} and \ref{lem:del-Pezzo-deg5-aux-divisors-H} ]
We will construct the $\bbQ$-divisors $G$ and $H$ explicitly, by case analysis on the position of $p\in S$ and $q\in E$. We will use curves from Table \ref{tab:del-Pezzo-5-lowdegree-curves}. These were constructed depending on $p$ and $q$ and were possibly reducible. The degree condition of each $G_i$ and $H_i$ in (iii) will be clear from the construction we give, as well as condition (ii) on $p\in G, H$ and condition (iv), for $H$. We will check log canonicity (condition (i)). The most involved part of the proof will consist on showing that the curves chosen in each particular case are irreducible (condition (iii)). Smoothnes follows from Lemma \ref{lem:del-Pezzo-deg5-good-curves-smooth}.

\textbf{Case 1:} The point $p$ is not in any line.

In particular $p\not\in E_i$ for all $i$. Let
$$G=\frac{1}{2}\sum_{i=1}^4 B_i +\frac{1}{2} A\simq-K_S.$$

The $\bbQ$-divisor $G$ only contains conics in its support. These conics are irreducible, since otherwise they would be the union of two lines, with one of them passing through $p$, contradicting the assumption.

Observe that $B_i\cdot B_j=1=A\cdot B_i=1$ for all $1\leq i < j \leq 4$. Furthermore $C\cdot B_i = C\cdot A=2$. Therefore $C$ can be tangent to at most one conic at $p$and in that case with multiplicity $2$. Assume that is the case, since it is when the discrepancy is the worst. the minimal log resolution $f\colon \widetilde S \ra S$ of $(S, (1-\beta) C +\omega\beta G)$ consists of two blow-ups. The log pullback is
$$f^*(K_S + (1-\beta) C +\omega\beta G) \simq-K_{\widetilde S}+ (1-\beta) \widetilde C + \omega\beta\widetilde G + (\frac{5}{2}\omega\beta-\beta) F_1 + (\frac{6\omega\beta}{2}-2\beta)F_2.$$
If $0<\beta\leq \frac{1}{2}$, then
$$\frac{5}{2}\omega\beta-\beta= \frac{5}{2}\beta-\beta = \frac{3}{2}\beta\leq \frac{3}{4}<1,$$
and
$$\frac{6\omega\beta}{2}-2\beta=3\beta-2\beta=\beta\leq \frac{1}{2}<1.$$
If $\frac{1}{2}\leq \beta\leq 1$, then
$$\frac{5}{2}\omega\beta-\beta=\frac{5}{4}-\beta\leq \frac{3}{4}<1,$$
and
$$\frac{6}{2}\omega\beta-2\beta=\frac{3}{2}-2\beta\leq \frac{1}{2}<1.$$
Therefore $(S, (1-\beta)C + \omega\beta G)$ is log canonical.

\textbf{Case 1a:} The point $q$ does not lie in the strict transform of any conic of $S$.

Let $$H=\frac{1}{3} \sum_{i=1}^4 R_i +\frac{1}{3} R\simq-K_S.$$ This $\bbQ$-divisor contains only cubics in its support. Let $H_i$ be any of these cubics. If $H_i$ was reducible, then $H_i=\sum H_i'+H_i''$ with $H_i'$ being a line and $H_i''$ being a conic. Since $p\in H_i$ and $q\in \widetilde H_i$, then either $p\in H_i'$ or $p\in H_i''$ and $q\in \widetilde H_i''$, contradicting the hypotheses 1 and 1a, respectively. Therefore $H_i$ is irreducible for all $i$.

Observe that $R_i\cdot R_j=2$ for $1\leq i< j\leq 4$. Moreover $R\cdot R_i=2$. Since $C\cdot R = C\cdot R_i=3$ for all $i$, the \emph{worst} situation regarding computing the discrepancy of the pair $(S, (1-\beta)C + \omega\beta H)$ arises when $(R\cdot R_i)\vert_{p}=2$, $(C\cdot R_i)\vert_p=2$ for all $1\leq i\leq 4$, $(C\cdot R)\vert_{p}=3$ and $(R_i\cdot R_j)\vert_p=2$ for all $1\leq i < j\leq 4$, since $C$ cannot intersect with multiplicity $3$ locally at $p$ two curves which intersect each other at $p$ with multiplicity $2$.

In this case the minimal log resolution $f\colon \widetilde S \ra S$ of $(S, (1-\beta) C +\omega\beta H)$ consists of $3$ consecutive blow-ups. The log pullback is
\begin{align*}
f^*(-K_S + (1-\beta) C + \omega\beta H)&\simq -K_{\widetilde S} + (1-\beta) \widetilde C +\omega\beta \widetilde H + \\
&(\frac{5}{3}\omega\beta-\beta) F_1 + (\frac{10}{3}\omega\beta-2\beta) F_2 + (\frac{11}{3}\omega\beta-3\beta)F_3.
\end{align*}
If $0<\beta\leq \frac{1}{2}$, then
$$\frac{5}{3}\omega\beta-\beta=\frac{2}{3}\beta<1, \ \frac{10}{3}\omega\beta-2\beta=\frac{4}{3}\beta\leq \frac{2}{3}<1 \text{ and } \frac{11}{3}\omega\beta-3\beta=\frac{2}{3}\beta<1.$$
If $\frac{1}{2}\leq\beta\leq 1$, then
$$\frac{5}{3}\omega\beta-\beta=\frac{5}{6}-\beta<1, \ \frac{10}{3}\omega\beta-2\beta=\frac{5}{3}-2\beta\leq \frac{2}{3}<1 \text{ and } \frac{11}{3}\omega\beta-3\beta=\frac{11}{6}-3\beta\leq\frac{1}{3}<1.$$
Therefore the pair $(S, (1-\beta) C + \omega\beta H)$ is log canonical.

\textbf{Case 1b:} The point $q\in \widetilde Q$ for $Q$ a conic of $S$. By Lemma \ref{lem:del-Pezzo-deg5-model-conic-choice} we may assume $Q=B_1$.

Let $$H=R_1+B_1\simq-K_S.$$
$B_1$ is irreducible since $p\in B_1$ and $p$ lies in no line. If $R_1$ was reducible, then $R_1=Q+L$, the union of a conic $Q$ and a line $L$ with $p\in Q$, $q\in \widetilde Q$. Since all conics intersect each other normally (see Table \ref{tab:del-Pezzo-5-lowdegree-curves}), then $Q=B_1$. But then $L\sim R_1-B_1\sim\pioplane{1}-E_1-E_2-E_3-E_4$, which is not one of the lines in $S$, as classified in Lemma \ref{lem:del-Pezzo-deg5-lines-list}.

Observe that $R_1\cdot B_1=2$. Therefore $R_1$ and $B_1$ are, at worst, tangent at $p$ with multiplicity $2$. Moreover $C\cdot B_1=2$. Since $C\cdot R_1=3$, then the \emph{worst} situation regarding computing the discrepancy of the pair $(S, (1-\beta)C + \omega\beta H)$ arises when $(C\cdot R_1)\vert_{p}=3$, $(C\cdot B_1)\vert_{p}=2$ and $(B_1\cdot R_1)\vert_p=2$. In that case the minimal log resolution $f\colon \widetilde S \ra S$ of $(S, (1-\beta) C +\omega\beta H$ consists of $3$ consecutive blow-ups. The log pullback is
$$f^*(-K_S + (1-\beta) C + \omega\beta H)\simq -K_{\widetilde S} + (1-\beta) \widetilde C +\omega\beta \widetilde H + (2\omega\beta-\beta) F_1 + (4\omega\beta-2\beta) F_2 + (5\omega\beta-3\beta)F_3.$$
If $0<\beta\leq \frac{1}{2}$, then
$$2\omega\beta-\beta=\beta<1, \ 4\omega\beta-2\beta=2\beta\leq 1 \text{ and } 5\omega\beta-3\beta=2\beta\leq 1.$$
If $\frac{1}{2}\leq\beta\leq 1$, then
$$2\omega\beta-\beta=1-\beta<1, \ 4\omega\beta-2\beta=2-2\beta\leq 1 \text{ and } 5\omega\beta-3\beta=\frac{5}{2}-3\beta\leq 1.$$
Therefore the pair $(S, (1-\beta) C + \omega\beta H)$ is log canonical.

\textbf{Case 2:} The point $p$ is in exactly one line. By Lemma \ref{lem:del-Pezzo-deg5-model-line-choice} we may assume that $p\in E_1$. Let
$$G=A+B_1+E_1\sim -K_S.$$
If $A$ or $B_1$ were reducible, then there would be lines with rational classes
\begin{align*}
&A-E_1\sim\pioplane{2}-2E_1-E_2-E_3-E_4,\ \text{or}\\
&B_1-E_1\sim\pioplane{1}-2E_1,
\end{align*}
which is not possible by Lemma \ref{lem:del-Pezzo-deg5-lines-list}. Since
$$A\cdot E_1=B_1\cdot E_1=B_1\cdot A = C\cdot E_1=1 \quad \text{ and }\quad A\cdot C = B_1\cdot C=2,$$
the worst situation regarding the computation of the discrepancies takes place when $(A\cdot C)\vert_p=2$ and $(B_1\cdot C)\vert_p=1$. Let $f\colon\widetilde S \ra S$ be the minimal log resolution of $(S, (1-\beta) C + \omega\beta G)$. It consists of $2$ blow-ups over $p$ with exceptional divisors $F_1,F_2$. Therefore the log pullback is 
$$f^*(K_S +(1-\beta)C + \omega \beta G) =K_{\widetilde S} + (1-\beta)\widetilde C + \omega\beta \widetilde G + (3\omega\beta-\beta)F_1 + (4\omega\beta-2\beta)F_2.$$
If $0<\beta\leq \frac{1}{2}$, then $3\omega\beta-\beta =2\beta\leq 1$ and $4\omega\beta-2\beta=2\beta\leq 1$. If $\frac{1}{2}\leq\beta\leq 1$, then $3\omega\beta-\beta\leq \frac{3}{2}-\beta\leq 1$ and $4\omega\beta-2\beta\leq 2-2\beta\leq 1$.

Therefore the pair $(S, (1-\beta) C + \omega\beta G)$ is log canonical.

\textbf{Case 2a:} The point $q\not\in \widetilde Q$ for $Q$ a line or a conic in $S$. In particular $q\not\in \widetilde E_1$. Let
$$H=\frac{1}{2}E_1+\frac{1}{2}\sum_{i=2}^4 R_i\simq-K_S.$$
By assumption $E_1$ is the only line which contains $p$. By Lemma \ref{lem:del-Pezzo-deg5-conics} there is a finite number of conics containing $p$.  Therefore $R_2,R_3,R_4$ are all irreducible, since each of them has degree $3$, their strict transform in $S$ contains $q$ and there is no conic or line whose strict transform passes through $q$.

Observe that $R_i\cdot R_j=2$ for $i\neq j$, $E_1\cdot R_i=1$ for $i\geq 2$ and $C\cdot R_i=3,\ \forall i$. Therefore in order to resolve $(S, (1-\beta)C + \omega\beta G)$, the worst situation arises when $(C \cdot R_2)\vert_p=3$ and $(R_i\cdot R_j)\vert_p=2,\ \forall i\neq j$. Let $f\colon\widetilde S \ra S$ be the minimal log resolution. It consists of $3$ blow-ups over $p$ with exceptional divisors $F_1,F_2,F_3$ and after the second blow-up the strict transforms of $R_i$ are disjoint. Therefore the log pullback is
$$f^*(K_S +(1-\beta)C + \omega \beta G) =K_{\widetilde S} + (1-\beta)\widetilde C + \omega\beta \widetilde G + (2\omega\beta-\beta)F_1 + (\frac{7}{2}\omega\beta-2\beta)F_2 + (4\omega\beta-3\beta)F_3.$$
Observe that $2\omega\beta-\beta\leq 1-\beta\leq 1$. If $0<\beta\leq \frac{1}{2}$, then $\frac{7}{2}\omega\beta-2\beta \leq \frac{3}{2}\beta\leq \frac{3}{4}\leq 1$ and $4\omega\beta-3\beta\leq \beta\leq 1$. If $\frac{1}{2}\leq\beta\leq 1$, then $\frac{7}{2}\omega\beta-2\beta\leq \frac{7}{4}-2\beta\leq \frac{3}{4}\leq 1$ and $4\omega\beta-3\beta\leq 2-3\beta\leq 1$.

Therefore the pair
$$(S, (1-\beta)C+\omega\beta H)$$
is log canonical.

\end{proof}

\subsection{Computation of the dynamic $\alpha$--invariant}
\begin{lem}
\label{lem:del-Pezzo-dynamic-alpha-degree-5-upbound}
Let $S$ be a non-singular del Pezzo surface of degree $5$ and $C\in\vert-K_S\vert$ be a smooth curve. Then
	\begin{equation}
			\alpha(S,(1-\beta)C)\leq\omega:=
					\begin{dcases}
							1 												&\text{ for } 0<\beta\leq \frac{1}{2},\\
							\frac{1}{2\beta}					&\text{ for } \frac{1}{2}\leq \beta\leq 1.
					\end{dcases}
		\label{eq:del-Pezzo-dynamic-alpha-degree-5-upbound}
	\end{equation}
\end{lem}
\begin{proof}
If $C$ contains a pseudo-Eckardt point $p$, by Lemma \ref{lem:del-Pezzo-deg5-model-line-choice} we can choose a model $\pi\colon S \ra \bbP^2$ such that $p=E_1\cap L_{12}.$

Notice that
$$D:=2E_1+L_{12}+L_{13}+L_{14}\sim-K_S$$
is a divisor with simple normal crossings. Let $f\colon \widetilde S \ra S $ be the blow-up of $p=E_1\cap L$ with exceptional divisor $E$. Then $f$ is a log resolution of $(S,(1-\beta)C+\lambda\beta D)$, since $C\cdot E_1=C\cdot L =C\cdot L'=C\cdot E_2=1$. Its log pullback is
$$f^*(K_S+(1-\beta)C+\lambda\beta D)\sim-K_{\widetilde S}+(1-\beta)\widetilde C +\lambda\beta \widetilde D  + (3\lambda\beta-\beta)E$$
and therefore
\begin{align*}
\alpha(S,(1-\beta)C)&\leq \min\{\lct(S,(1-\beta)C,\beta C),\lct(S,(1-\beta)C,\beta D)\}\\
&\leq \min\{1,\frac{1}{2\beta},\frac{1+\beta}{3\beta}\}=\min\{1,\frac{1}{2\beta}\}=\omega.
\end{align*}
If $C$ contains no pseudo-Eckardt points, the pair
$$(S,(1-\beta)C+\lambda\beta D)$$
has simple normal crossings and we obtain
\begin{align*}
\alpha(S,(1-\beta)C)&\leq \min\{\lct(S,(1-\beta)C,\beta C),\lct(S,(1-\beta)C,\beta D)\}\leq \min\{1,\frac{1}{2\beta}\}=\omega.
\end{align*}
\end{proof}

\begin{thm}
\label{thm:del-Pezzo-dynamic-alpha-degree-5}
Let $S$ be a non-singular del Pezzo surface of degree $5$ and $C\in\vert-K_S\vert$ be a smooth curve. Then
	\begin{equation}
			\alpha(S,(1-\beta)C)=\omega:=
					\begin{dcases}
							1 												&\text{ for } 0<\beta\leq \frac{1}{2},\\
							\frac{1}{2\beta}					&\text{ for } \frac{1}{2}\leq \beta\leq 1.
					\end{dcases}
		\label{eq:del-Pezzo-dynamic-alpha-degree-5}
	\end{equation}
\end{thm}
\begin{proof}
By Lemma \ref{lem:del-Pezzo-dynamic-alpha-degree-5-upbound} we have that $\alpha(S,(1-\beta)C)\leq \omega$. We proceed by \emph{reductio ad absurdum}. Suppose that $\alpha(S, (1-\beta) C)< \omega$. Then, there is an effective $\bbQ$-divisor $D\simq-K_S$ such that the pair
\begin{equation}
(S, (1-\beta)C + \lambda\beta D)
\label{eq:del-Pezzo-dynamic-alpha-degree-5-proof-badpair}
\end{equation}
is not log canonical for some $\lambda<\omega$ at some point $p\in S$. Observe that
\begin{equation}
\lambda\beta<\omega\beta\leq \frac{1}{2}=\glct(S), \quad \forall 0<\beta\leq 1
\label{eq:del-Pezzo-dynamic-alpha-degree-5-proof-basicinequality}
\end{equation}
by Theorem \ref{thm:del-Pezzo-dynamic-alpha-degree-5}. By Lemma \ref{lem:pairs-fixed-boundary-lcs}, the pair \eqref{eq:del-Pezzo-dynamic-alpha-degree-5-proof-badpair} is log canonical in codimension $1$ and $p\in C$. Observe that by Lemma \ref{lem:adjunction} (i), we have that
$$(1-\beta)+\lambda\beta\mult_pD>1,$$
which implies
\begin{equation}
\mult_pD>\frac{1}{\lambda}>1.
\label{eq:del-Pezzo-dynamic-alpha-degree-5-proof-boundbelow}
\end{equation}

\textbf{Step 1: We show the point $p$ is not a pseudo-Eckardt point.}

Suppose for contradiction that $p$ is a pseudo-Eckardt point. By Lemma \ref{lem:del-Pezzo-deg5-model-line-choice} we may choose $\pi\colon S \ra \bbP^2$ such that $p=E_1\cap L_{12}$. For $L=E_1,L_{12}$ we have that $L\subseteq\Supp(D)$ since otherwise
$$1=L\cdot D \geq \mult_pD>1,$$
by \eqref{eq:del-Pezzo-dynamic-alpha-degree-5-proof-boundbelow}. Hence we may write $D=aE_1+bL_{12}+\Omega$ where $a,b>0$ and $E_1,L_{12}\not\subseteq\Supp(\Omega)$. We claim that the pair
\begin{equation}
(S, (1-\beta)C + \lambda\beta(2E_1+L_{12}+L_{13}+L_{14}))
\label{eq:del-pezzo-dynamic-alpha-degree-5-proof-lc-logpair}
\end{equation}
is log canonical. Indeed, if $\lambda\beta\leq \frac{1}{2}$ it is log canonical in codimension $1$ and the only point in which the pair does not have simple normal crossings is $p$. However, blowing up $p$ with exceptional divisor $E$ it is easy to see that the discrepancy along $E$ is $a(E)=3\lambda\beta-\beta<3\omega\beta-\beta\leq 1$. Indeed if $0<\beta\leq \frac{1}{2}$, then $3\omega\beta-\beta=2\beta\leq 1$ and if $\frac{1}{2}\leq \beta\leq 1$, then $3\omega\beta-\beta=\frac{3}{2}-\beta\leq 1$. Since $2E_1+L_{12}+L_{13}+L_{14}\sim-K_S$, by Lemma \ref{lem:log-convexity} we may assume that for $L=L_{13}$ or $L=L_{14}$, $L\not\subseteq\Supp(D)$. Therefore $1=D\cdot L \geq a$. Similary the pair
$$(S, (1-\beta)C + \lambda\beta(E_1+2L_{12}+E_2+L_{34}))$$
is log canonical. Given that $E_1+2L_{12}+E_2+L_{34}\sim-K_S$, by Lemma \ref{lem:log-convexity}, we may assume that for $L=L_{34}$ or $L=E_2$, $L\not\subset\Supp(D)$. Therefore $1=D\cdot L\geq b$. We conclude
\begin{equation}
a+b\leq 2.
\label{eq:del-Pezzo-dynamic-alpha-degree-5-proof-linecoefficients}
\end{equation}
Now observe that
$$1=E_1\cdot D \geq -a+b+\mult_p\Omega,$$
$$1=L_{12}\cdot D \geq a-b+\mult_p\Omega,$$
and adding these two equations it follows that $\mult_p\Omega\leq 1$. Therefore
$$\mult_p((1-\beta)C + \lambda\beta\Omega)\leq 1-\beta+\lambda\beta<1-\beta+\beta\leq 1.$$
The hypotheses of Theorem \ref{thm:inequality-Cheltsov} are satisfied. Therefore one of the following holds:
$$2(1-\lambda\beta a)<L_{12}\cdot ((1-\beta)C +\lambda\beta\Omega)=1-\beta+\lambda\beta(1-a+b)$$
$$2(1-\lambda\beta b)<E_1\cdot ((1-\beta)C +\lambda\beta\Omega)=1-\beta+\lambda\beta(1+a-b).$$
Since the roles of $a$ and $b$ are symmetric, it is enough to disprove the latter equation to obtain a contradiction. Indeed, the last inequality implies
$$1<\lambda\beta(1+a+b)-\beta<3\omega\beta-\beta$$
by \eqref{eq:del-Pezzo-dynamic-alpha-degree-5-proof-linecoefficients}. We show that this is impossible. Indeed, if $0<\beta\leq \frac{1}{2}$, then $3\omega\beta-\beta=2\beta\leq 1$ and if $\frac{1}{2}\leq \beta\leq 1$, then $3\omega\beta-\beta=\frac{3}{2}-\beta\leq 1$.

\textbf{Step 2: The blow-up setting.}

Let $\sigma\colon \widetilde S \ra S$ be the blow-up of $p$ with exceptional divisor $E$. By Lemma \ref{lem:log-pullback-preserves-lc} the pair
\begin{equation}
(\widetilde S, (1-\beta)\widetilde C + \omega \beta \widetilde D + (\lambda \beta \mult_p D -\beta)E)
\label{eq:del-Pezzo-dynamic-alpha-degree-5-proof-badpair-blow-up}
\end{equation}
is not log canonical at some $q\in E$. By Lemma \ref{lem:del-Pezzo-deg5-aux-divisors-G} there is an effective $\bbQ$-divisor $G=\sum a_i G_i\simq-K_S$ with all $G_i$ irreducible and satisfying $\deg G_i\leq 2$ and $p\in G_i$. Moreover the pair
$$(S,(1-\beta) C + \lambda\beta G)$$
is log canonical at $p$. Then by Lemma \ref{lem:log-convexity}, we may assume there is an irreducible curve $G_i\subseteq \Supp(G)$ such that $G_i\not\subseteq \Supp(D)$. Then
\begin{equation}
2\geq D\cdot G_j\geq\mult_p D.
\label{eq:del-Pezzo-dynamic-alpha-degree-5-proof-boundm}
\end{equation}
This implies
\begin{equation}
\lambda \beta \mult_p D -\beta<2\omega\beta-\beta\leq 1.
\label{eq:del-Pezzo-dynamic-alpha-degree-5-proof-firstbound}
\end{equation}
Indeed, if $0<\beta\leq \frac{1}{2}$, then
$$3\omega\beta-\beta=2\beta\leq 1.$$
If $\frac{1}{2}\leq\beta\leq 1$, then
$$3\omega\beta-\beta=\frac{3}{2}-\beta\leq 1.$$
Therefore the pair \eqref{eq:del-Pezzo-dynamic-alpha-degree-5-proof-badpair-blow-up} is log canonical along $E$, and therefore log canonical in codimension $1$, i.e. the point $q$ is an isolated centre of non-log canonical singularities.

In fact, $q\in \widetilde C$, since otherwise the pair
$$(\widetilde S, \lambda \beta \widetilde D + (\lambda\beta\mult_pD-\beta)E)$$
is not log canonical at $q\in E$, but this implies
$$1<\lambda\beta \widetilde D \cdot E =\lambda\beta\mult_pD\leq 1$$
by Lemma \ref{lem:adjunction} (iii) and \eqref{eq:del-Pezzo-dynamic-alpha-degree-5-proof-boundm}, which is absurd.

Suppose $p\in L$ a line in $S$. Since $C\cdot L=\deg L =1$, then $\widetilde C \cdot \widetilde L=0$, so $\widetilde C \cap \widetilde L =\emptyset$ and as a result $q$ does not belong to the strict transform of $\widetilde L$.

To ease the readibility, we leave the proof of the following claim for later.
\begin{clm}
\label{clm:del-Pezzo-deg5-awkwardcase}
If $p$ belongs to a $(-1)$-curve, then $q$ does not belong to the strict transform of a conic.
\end{clm}

Now we are on the hypothesis of Lemma \ref{lem:del-Pezzo-deg5-aux-divisors-H} and we conclude that there is an effective $\bbQ$-divisor $H=\sum h_i H_i\simq-K_S$ such that for all $H_i$ the point $p\in H_i$ and $\deg H_i\leq 3$. Moreover $q\in \widetilde H_i$ for all $H_i$ with $\deg H_i>1$ and the pair
$$(S, (1-\beta) C + \lambda\beta H)$$
is log canonical. Then, by Lemma \ref{lem:log-convexity} we may assume there is $H_i\not\subseteq\Supp(D)$. Observe that if $\deg H_i=1$ then
$$1= H_i \cdot D\geq \mult_p D>1$$
by \eqref{eq:del-Pezzo-dynamic-alpha-degree-5-proof-boundbelow}, which is absurd.

We use $H_i$ to bound $\mult_q\widetilde D$:
$$3-\mult_p D\geq \widetilde H_i \widetilde D \geq \mult_q \widetilde D,$$
which gives
\begin{equation}
3\geq\mult_p D + \mult_q \widetilde D.
\label{eq:del-Pezzo-dynamic-alpha-degree-5-proof-boundm-above}
\end{equation}
We claim that
\begin{equation}
\mult_q(\lambda\beta \widetilde D + (\lambda\beta\mult_p D -\beta) E)<\omega\beta(\mult_p D + \mult_q \widetilde D) -\beta\leq 1.
\label{eq:del-Pezzo-dynamic-alpha-degree-5-proof-boundpair-above}
\end{equation}
Indeed, if $1\geq \beta\geq \frac{1}{2}$, then \eqref{eq:del-Pezzo-dynamic-alpha-degree-5-proof-boundm-above} and \eqref{eq:del-Pezzo-dynamic-alpha-degree-5-proof-basicinequality} give
\begin{equation}
\omega\beta(\mult_q \widetilde D + \mult_p D)-\beta\leq \frac{3}{2}-\beta \leq 1.
\label{eq:del-Pezzo-dynamic-alpha-degree-5-proof-funnylemma1}
\end{equation}
If $0<\beta\leq\frac{1}{2}$, then
$$\omega\beta(\mult_q+\mult_p D)-\beta\leq 3\beta-\beta=2\beta\leq 1.$$

We will apply Theorem \ref{thm:inequality-local-blowup-bound} to the pair \eqref{eq:del-Pezzo-dynamic-alpha-degree-5-proof-badpair-blow-up} with $n=2$. First we use \eqref{eq:del-Pezzo-dynamic-alpha-degree-5-proof-boundm-above} and claim
\begin{equation}
\mult_q(\lambda\beta \widetilde D + (\lambda\beta\mult_p D -\beta)E)< \omega\beta(\mult_p D + \mult_q \widetilde D) -\beta\leq 3\omega\beta-\beta\leq \frac{1}{2}+\beta.
\label{eq:del-Pezzo-dynamic-alpha-degree-5-proof-funnylemma2}
\end{equation}
Indeed, if $0<\beta\leq \frac{1}{2}$, then
$$3\omega\beta-\beta=2\beta = \beta+\beta \leq \frac{1}{2}+\beta.$$
If $\frac{1}{2}\leq \beta\leq 1$, then
$$3\omega\beta-\beta = \frac{3}{2}-\beta = \frac{1}{2} + (1-\beta)\leq \frac{1}{2}+2\beta-\beta=\frac{1}{2}+\beta.$$
Inequalities \eqref{eq:del-Pezzo-dynamic-alpha-degree-5-proof-funnylemma1} and \eqref{eq:del-Pezzo-dynamic-alpha-degree-5-proof-funnylemma2} are the hypotheses of Theorem \ref{thm:inequality-local-blowup-bound} when $n=2$. The Lemma gives
$$\widetilde C \cdot (\lambda\beta \widetilde D + (\lambda\beta\mult_p D -\beta)E)>1+2\beta.$$
We claim
$$\widetilde C \cdot (\lambda\beta \widetilde D + (\lambda \beta \mult_pD -\beta)E)\leq 1+2\beta,$$
which will give a contradiction, finishing the proof.

Observe that
\begin{align*}
&\widetilde C\cdot (\lambda\beta \widetilde D + (\lambda\beta\mult_p D -\beta)E)\\
<&\omega\beta(C\cdot D -\mult_p D +\mult_p D)-\beta\\
=&5\omega\beta-\beta.
\end{align*}
If $0<\beta\leq \frac{1}{2}$, then
$$5\omega\beta-\beta=4\beta=2\beta+2\beta\leq 1+2\beta.$$
If $\frac{1}{2}\leq \beta\leq 1$, then
$$5\omega\beta-\beta = \frac{5}{2}-\beta=1+\frac{3}{2}-\beta\leq 1+3\beta-\beta=1+2\beta.$$
\end{proof}

\begin{proof}[Proof of Claim \ref{clm:del-Pezzo-deg5-awkwardcase}]
Suppose for contradiction that $q_0:=p\in L$, where $L$ is a line. By Lemma \ref{lem:del-Pezzo-deg5-model-line-choice} we may assume that $L=E_1$. Let $q_1:=q\in E=:F_1$, the exceptional curve of $S_1:=\widetilde S$, the blow-up of $q_0$. Suppose for contradiction that $q_1\in \widetilde Z$ for $\widetilde Z$, the strict transform of a conic $Z\subset S$ with $q_0\in Z$. Then, by Lemma \ref{lem:del-Pezzo-deg5-model-conic-choice}, we may assume that $Z=B_1$. The pair
$$(S_1, (1-\beta)C^1 + \lambda\beta D^1 + \lambda\beta(\mult_pD-\beta)F_1)$$
is not log canonical at some point $q_1=B^1\cap C^1\cap F_1$ where for any $\bbQ$-divisor $A$ we denote by $A^1$ its strict transform in $S_1$. The curve $B_1\not\Supp(D)$, since otherwise, by Lemma \ref{lem:adjunction} (i), we obtain a contradiction:
\begin{align*}
 1&<\mult_p((1-\beta)C^1 + \lambda\beta D^1 + (\lambda\beta \mult_p D-\beta)F_1)\\
	&\leq B^1\cdot ((1-\beta)C^1 + \lambda\beta D^1 + (\lambda\beta\mult_p D -\beta)F_1)\\
	&=(1-\beta) + \lambda\beta(2-\mult_p D) + (\lambda\beta \mult_p D -\beta)\\
	&=1+2\lambda\beta-2\beta<1+2\omega\beta-2\beta\leq 1.
\end{align*}
Indeed, if $0<\beta\leq \frac{1}{2}$, then $1+2\omega\beta-2\beta=1$ and if $\frac{1}{2}\leq\beta\leq 1$, then $1+2\omega\beta -2\beta\leq 1+1-2\beta\leq 1$.

Therefore we may write $D=aE_1+bB_1+\Omega$ where $E_1, B_1\not\Supp(\Omega)$ and $a>0$, $b\geq 0$. Let $S_0=S$, $C_0=C$, $D_0=D$ and $q_0=q$ as in Theorem \ref{thm:pseudo-inductive-blow-up}. Let $i\geq 1$ and let $f_i\colon S_i\ra S_{i-1}$ be the blow-up of the point $q_{i-1}=C^i\cap F_{i-1}$ with exceptional curve $F_i$. Let $A^{i-1}$ or $A$ be any $\bbQ$-divisor in $S_{i-1}$. We will denote its strict transform in $S^i$ by $A^i$. Let $m_i=\mult_{q_i} D_i$. Recall that the assumption for contradiction is that the pair
\begin{equation}
(S_1, (1-\beta)C^1 + \lambda\beta D^1 + (\lambda\beta m_0-\beta)F_1)
\label{eq:del-Pezzo-dynamic-alpha-degree-5-proof-pair1}
\end{equation}
is not log canonical at $q_1=B_1^1 \cap C^1\cap F_1$ and is log canonical near $q_1$. Let $x+0=\mult_{q_i}\Omega^i$. Then $m_0=a+b+x_0$ and $m_1=b+x_1$. Recall from \eqref{eq:del-Pezzo-dynamic-alpha-degree-5-proof-firstbound} that $\lambda\beta m_0\leq 1$. This is one of the hypothesis of Theorem \ref{thm:pseudo-inductive-blow-up} and we will assume it from now onwards. Since the pair is log canonical near $q_1$, but not at $q_1$ by Lemma \ref{lem:log-pullback-preserves-lc}, the pair
$$(S_2, (1-\beta)C^2 + \lambda\beta D^2 + (\lambda\beta m_0 -\beta)F_1^2 + (\lambda\beta(m_0+m_1)-2\beta)F_2)$$
is not log canonical at some (possibly all) $t_2\in F_2$. We bound the multiplicities of $D$:
\begin{equation}
1=E_1\cdot D \geq -a+b+x_0.
\label{eq:del-Pezzo-dynamic-alpha-degree-5-proof-bound1}
\end{equation}
Let $A$ be the only conic in
$$\calA = \vert \pioplane{2} -E_1-E_2-E_3-E-4\vert$$
with $q_0\in A$, as constructed in section \ref{sec:delPezzo-5-curves}. The curve $A$ is irreducible since otherwise there would be a line in $S$ with class
$$A-E_1 \sim\pioplane{2}-2E_1-E_2-E_3-E_4$$
which is impossible by Lemma \ref{lem:del-Pezzo-deg5-lines-list}. The pair
$$(S, (1-\beta) C + \lambda\beta (A + B_1+ E_1))$$
is log canonical by Lemma \ref{lem:appendix-delPezzo-deg5-dynamic-lct}.

Therefore by Lemma \ref{lem:log-convexity}, we may assume that $A\not\subseteq\Supp(D)$ and
\begin{equation}
2=A\cdot D \geq a+ b+ x_0 = m_0.
\label{eq:del-Pezzo-dynamic-alpha-degree-5-proof-bound3}
\end{equation}
This implies
$$2\lambda\beta m_0-2\beta\leq 4\lambda\beta-2\beta<4\omega\beta-2\beta\leq 1.$$
Indeed, if $0<\beta\leq \frac{1}{2}$, then $4\omega\beta-2\beta=2\beta\leq 1$, while if $\frac{1}{2}\leq \beta\leq 1$, then $4\omega\beta-2\beta=2-2\beta\leq 1$.

Now, by Theorem \ref{thm:pseudo-inductive-blow-up} (iv) with $i=2$, we may conclude that $t_2=F_2\cap C^2=q_2$. Since
$$\lambda\beta(m_0+m_1)-2\beta\leq 2\lambda\beta m_0 -\beta\leq 1,$$
by Theorem \ref{thm:pseudo-inductive-blow-up} (i) with $i=3$, we conclude that the pair
$$(S_3, (1-\beta)C^3 + (\lambda\beta(m_0+m_1)-2\beta)F_2^3 + (\lambda\beta(m_0+m_1+m_2)-3\beta)F_3)$$
is not log canonical at some $t_3\in F_3$. Since $(C\cdot B_1)_{q_0}=2$, $C\cdot E_1=1$, then $m_1= b+x_1$ and $m_i=x_i$ for $i\geq 2$. From \eqref{eq:del-Pezzo-dynamic-alpha-degree-5-proof-bound1} and \eqref{eq:del-Pezzo-dynamic-alpha-degree-5-proof-bound3} we obtain
\begin{align*}
			&\lambda\beta(m_0+2m_1)-3\beta\\
<			&\omega\beta(a+3b+x_0+2x_1)-3\beta\\
\leq	&\omega\beta(a+3b+3x_0)-3\beta\\
\leq	&5\omega\beta-3\beta\leq 1.
\end{align*}
Indeed, if $0<\beta\leq \frac{1}{2}$, then $5\omega\beta-3\beta=2\beta\leq 1$ while if $\frac{1}{2}\leq\beta\leq 1$, then $5\omega\beta-3\beta=\frac{5}{2}-3\beta\leq 1$.

Therefore by Theorem \ref{thm:pseudo-inductive-blow-up} (iv) with $i=3$, we conclude that $t_3=C^3\cap F_3=q_3$. Therefore the pair
$$(S_3, (1-\beta)C^3 + \lambda\beta D^3 +(\lambda\beta(m_0+m_1+m_2)-3\beta)F_3)$$
Now, by Lemma \ref{lem:adjunction} (iii) applied with $C^3$, we obtain a contradiction:
\begin{align*}
 1&<C^3\cdot (\lambda\beta D^3 + (\lambda\beta(m_0+m_1+m_2)-3\beta)F_3)\\
	&<5\omega\beta-3\beta \leq 1.
\end{align*}
\end{proof}

\section{Smooth intersection of two quadrics}
In this section we will use notation and curves from Section \eqref{sec:delPezzo-4-curves}.
\begin{thm}
\label{thm:del-Pezzo-dynamic-alpha-degree-4}
Let $S$ be a smooth del Pezzo surface of degree $4$ and let $C$ be a smooth elliptic curve in $S$, $C\simq-K_S$. The dynamic $\alpha$-invariant $\alpha(S,(1-\beta)C)$ is as follows:
\begin{itemize}
	\item[(1)] Suppose $C$ does not contain any pseudo-Eckardt point nor any point $p$ with two irreducible conics $A$ and $B$ satisfying $A+B\sim-K_S$ and $A\cap B=\{p\}$ with $(A\cdot C)\vert_p=(B\cdot C)\vert_p=2$. Then
		\begin{equation}
				\alpha(S,(1-\beta)C)=\omega_1:=
						\begin{dcases}
								1 									&\text{ for } 0<\beta\leq \frac{2}{3},\\
								\frac{2}{3\beta}		&\text{ for } \frac{2}{3}\leq \beta\leq 1.
						\end{dcases}
			\label{eq:del-Pezzo-dynamic-alpha-degree-4-very-general}
		\end{equation}
	\item[(2)] Suppose $C$ contains no pseudo-Eckardt points but there is a point $p \in C$ such that there are two irreducible conics $A$ and $B$ satisfying $A+B\sim-K_S$ and $A\cap B=\{p\}$ with $(A\cdot C)\vert_p=(B\cdot C)\vert_p=2$. Then
		\begin{equation}
				\alpha(S,(1-\beta)C)=\omega_2:=
						\begin{dcases}
									1 									&\text{ for } 0<\beta\leq \frac{1}{2},\\
									\frac{1+2\beta}{4\beta}		&\text{ for } \frac{1}{2}\leq \beta\leq \frac{5}{6},\\
									\frac{2}{3\beta}		&\text{ for } \frac{5}{6}\leq \beta\leq 1.
						\end{dcases}
			\label{eq:del-Pezzo-dynamic-alpha-degree-4-specific1}
		\end{equation}
	\item[(3)] Suppose $C$ contains a pseudo-Eckardt point. Then
		\begin{equation}
				\alpha(S,(1-\beta)C)=\omega_3:=
						\begin{dcases}
								1 									&\text{ for } 0<\beta\leq \frac{1}{2},\\
								\frac{1+\beta}{3\beta}		&\text{ for } \frac{1}{2}\leq \beta\leq 1.
						\end{dcases}
			\label{eq:del-Pezzo-dynamic-alpha-degree-4-specific2}
		\end{equation}
\end{itemize}
\end{thm}
\begin{proof}[Proof of Theorem \ref{thm:del-Pezzo-dynamic-alpha-degree-4}]
We choose an effective $\bbQ$-divisor $D_i\simq -K_S$ for each case in the statement:
\begin{itemize}
	\item[(1)] Let $L_1$ and $L_2$ be two lines intersecting each other at a pseudo-Eckardt point $p$ and $Q$ be the unique conic satisfying $D_1=L_1+L_2+Q\sim -K_S$ and such that $p\in Q$. It is clear that $Q$ exists as a divisor. We will show we can choose $Q$ to be effective. Observe that $Q^2=0$ and $K\cdot Q=-2$. Therefore, by the genus formula, $p_a(Q)=0$. By Proposition \ref{prop:rational-surfaces-sections-rational-curves}, $h^0(S,\calO_S(Q))\geq 2$. Hence we can choose $Q$ to be effective and such that $p\in Q$.
	\item[(2)] Let $D_2=A+B\sim -K_S$ if $\beta\leq \frac{5}{6}$. If $\beta>\frac{5}{6}$ let $D_2=D_1$.
	\item[(3)] Let $D_3=L_1+L_2+Q\sim-K_S$ be the two lines intersecting each other at the pseudo-Eckardt point $p\in C$ and $Q$ be the unique conic satisfying $D_1=L_1+L_2+Q\sim -K_S$ and such that $p\in Q$, as above.
\end{itemize}
By lemmas \ref{lem:appendix-delPezzo-deg4-dynamic-lct1}, \ref{lem:appendix-delPezzo-deg4-dynamic-lct2} and \ref{lem:appendix-delPezzo-deg4-dynamic-lct3}, $\alpha(S,(1-\beta)C)\leq \lct(S,(1-\beta)C, D_i)=\omega_i$ holds for $\omega_i$ as in the statement of the Theorem.
We will need the following arithmetic observation.
\begin{clm}
The following inequality holds
\begin{equation}
				\omega_3\beta\leq\omega_2\beta\leq\omega_1\beta=
						\begin{dcases}
								\frac{1}{2} 									&\text{ for } 0<\beta\leq \frac{1}{2},\\
								\frac{2}{3}										&\text{ for } \frac{1}{2}\leq \beta\leq 1,
						\end{dcases}\leq \frac{2}{3}=\glct(S).
			\label{eq:del-Pezzo-dynamic-alpha-degree-4-proof-basic-inequality}
		\end{equation}
Moreover $\omega_3\leq\omega_2\leq\omega_1$.
\end{clm}
\begin{proof}
The statement is trivial for $0<\beta\leq \frac{1}{2}$. Suppose $\beta\geq \frac{1}{2}$. Then it is straight forward to check that $\frac{1+\beta}{3\beta}\leq \frac{1+2\beta}{4\beta}$. If $\beta\geq \frac{2}{3}$, then $\frac{1+2\beta}{4\beta}\leq \frac{2}{3\beta}$ also holds. Therefore $\omega_3\leq\omega_2\leq\omega_1$ holds. The claim follows by Theorem \ref{thm:del-Pezzo-glct-charp}.
\end{proof}

Suppose for contradiction that $\exists D\simq-K_S$ an effective $\bbQ$-divisor and $0< \lambda<\omega_i$ depending on $\beta$ and $\omega_i$ such that the pair $(S,(1-\beta)C+\lambda\beta D)$ is not log canonical at some $p\in S$. By Lemma \ref{lem:pairs-fixed-boundary-lcs}, this pair is log canonical near $p$ and $p\in C$.

Since $(S,C)$ is log canonical, by Lemma \ref{lem:log-convexity}, we may assume that $C\not\subseteq\Supp(D)$. We need to show that the pair
\begin{equation}
\label{eq:del-Pezzo-dynamic-alpha-degree-4-proof-bad-pair}
(S,(1-\beta)C+\lambda\beta D)
\end{equation}
is log canonical.

\textbf{Step 1: We show $p=C\cap L$ where $L$ is some line.}

Suppose there is no line $L$ such that $p\in L$. Since $\lambda<\omega\leq 1$, by Lemma \ref{lem:disc-monotonous}, the pair $(S, (1-\beta)C + \beta D)$ is not log canonical. Then by Lemma \ref{lem:del-Pezzo-cat-degree-4-generic}, the surface $S$ satisfies the Cat Property at $p$ and therefore $\exists T\in \vert-K_S\vert$ with $p\in T$ such that $T\subset\Supp((1-\beta)C+\beta D)$. By Lemma \ref{lem:convexity} we may assume $(S,T)$ is not log canonical and therefore $T\subset \Supp(D)$ (i.e. $T$ is a cat). The cats of $S$ at $p$ are classified in Lemma \ref{lem:del-Pezzo-cat-degree-4-generic-list}.

If $T$ is irreducible, then the pair $(S,(1-\beta)C + \lambda \beta T)$ is log canonical by Lemma \ref{lem:appendix-delPezzo-deg4-dynamic-lct4}.

If $T=A+B$, the union of two conics intersecting only at $p$, then by lemma \ref{lem:appendix-delPezzo-deg4-dynamic-lct2} the pair $(S,(1-\beta)C + \lambda \beta T)$ is log canonical.

In any case we may assume, using Lemma \ref{lem:log-convexity} that $T\not\subset\Supp(D)$. But then $\Supp(D)$ contains no cats, contradicting Lemma \ref{lem:Cat-property-not-lc-has-cats}. Hence, $p=L\cap C$ for some line $L$.

\textbf{Step 2: We show $L$ is the only line through $p$ (proof of case (3)).}

Suppose $p=L\cap M$ is a pseudo-Eckardt point, i.e. $L$ and $M$ are lines intersecting at $p$. Notice that since $C\cdot L =-K_S\cdot L =1$, the curves $C$ and $L$ intersect with simple normal crossings. By hypothesis, in cases (1) and (2) in the statement of the Theorem the point $p\not\in L'$ for any other line $L'$. Therefore we are in case (3) in the statement of the Theorem and $\lambda<\omega_3$.

The line $M\subseteq \Supp(D)$, since otherwise
$$1=M\cdot D\geq\mult_p D>\frac{\beta}{\lambda\beta}>1$$
where we apply Lemma \ref{lem:adjunction} (i) to the pair \eqref{eq:del-Pezzo-dynamic-alpha-degree-4-proof-bad-pair}. In the same fashion $L\subset\Supp(D)$. We write
$$D=a_1 L + a_2 M + \Omega$$
where $1\geq\lambda\beta a_i>0$ since the pair \eqref{eq:del-Pezzo-dynamic-alpha-degree-4-proof-bad-pair} is log canonical in codimension $1$. Let $Q\sim -K_S-L-M$ be a divisor. Since $L\cdot M=1$, then $Q\cdot -K_S=2$ and $Q^2=0$. Therefore by Lemma \ref{lem:del-Pezzo-all-conics-rational}, $p_a(Q)=0$ and by Proposition \ref{prop:rational-surfaces-sections-rational-curves}, we may take $Q$ to be effective and such that $p\in Q$. The conic $Q$ is irreducible, since otherwise there would be a third line $N$ through $p$, which would be an Eckardt point. This is impossible by Lemma \ref{lem:del-Pezzo-lines-through-a-point}. 

The pair $(S, (1-\beta)C + \omega_3\beta (L+M+Q))$ is log canonical by Lemma \ref{lem:appendix-delPezzo-deg4-dynamic-lct3}. Therefore by Lemma \ref{lem:log-convexity} we may assume $Q\not\subseteq\Supp(D)$ and 
\begin{equation}
2=D\cdot Q \geq a+b+ \mult_p\Omega\geq a+b.
\label{eq:del-Pezzo-dynamic-alpha-degree-4-proof-bound-pEckardt}
\end{equation}
Note that $1=D\cdot L\geq -a+b+\mult_p\Omega$ and $1=D\cdot M\geq a-b+\mult_p\Omega$. Adding these two inequalities it follows that $\mult_p\Omega\leq 1$. As a result
$$\mult_p\Omega((1-\beta)C + \lambda\beta\Omega)\leq 1-\beta+\beta\mult_p\Omega\leq 1-\beta+\beta\leq 1.$$

Finally, since $L\cdot M=1$ the pair
$$(S,(1-\beta)C + \lambda\beta(a_1L+a_2M+\Omega))$$
satisfies the hypotheses of Theorem \ref{thm:inequality-Cheltsov}, so 
\begin{equation}
\left(\left(1-\beta\right)C+\lambda\beta\Omega\right)\cdot L>2(1-\lambda\beta a_2)\quad \text{ or } \quad\left(\left(1-\beta\right)C+\lambda\beta\Omega\right)\cdot M>2(1-\lambda\beta a_1).
\label{eq:del-Pezzo-dynamic-alpha-degree-4-proof-bound-pEckardt-contradiction}
\end{equation}
The first inequality implies
\begin{align*}
2(1-\lambda\beta a_2)<\left(\left(1-\beta\right)C+\lambda\beta\Omega\right)\cdot L&=\\
(1-\beta)+\lambda\beta(-K_S-a_1L-a_2M)\cdot L&=\\
(1-\beta)+\lambda\beta(1+a_1-a_2).
\end{align*}
Rearranging this inequality we obtain
$$3\omega_3\beta>3\lambda\beta\geq\lambda\beta(1+a_1+a_2)>1+\beta,$$
by \eqref{eq:del-Pezzo-dynamic-alpha-degree-4-proof-bound-pEckardt}
However, this implies an absurdity:
$$\frac{(1+\beta)}{3\beta}<\omega_3=\min\{1, \frac{(1+\beta)}{3\beta}\}\leq \frac{(1+\beta)}{3\beta}.$$
Since the roles of $L$ and $M$ are symmetric in Theorem \ref{thm:inequality-Cheltsov} we also deduce a contradiction from the second inequality in \eqref{eq:del-Pezzo-dynamic-alpha-degree-4-proof-bound-pEckardt-contradiction}. We conclude that $L$ is the only line with $p\in L$.

\textbf{Step 3: Show there are no irreducible conics $A+B\sim-K_S$ such that $A\cap B=\{p\}$.}

It is easy to obtain a contradiction:
$$1=L\cdot -K_S = L\cdot (A+B)\geq (L\cdot A)\vert_p + (L\cdot B)\vert_p=1+1=2.$$

Therefore, we may assume that $\lambda<\omega_1$ and $(S,(1-\beta)C + \lambda\beta(aL+\Omega)$ is not log canonical at $p=L\cap C$ such that $L$ is the only line which contains $p$. The Theorem is proven thanks to the following Lemma \ref{lem:del-Pezzo-dynamic-alpha-degree-4-line}.
\end{proof}

\begin{lem}
\label{lem:del-Pezzo-dynamic-alpha-degree-4-line}
Let $S$ be a non-singular del Pezzo surface of degree $4$. Let $D\simq-K_S$ be an effective $\bbQ$-divisor. Let $C\in\vert-K_S\vert$ be a smooth curve and $L$ be a line in $S$. Let $q_0=L\cap C$ and assume $L$ is the only line which contains $q_0$. Let $\lambda<\omega_1=\min\{1,\frac{1}{2\beta}\}$ where $0<\beta\leq 1$. If the pair
$$(S, (1-\beta)C + \lambda\beta D)$$
is log canonical in codimension $1$, then it is also log canonical at $q_0$.
\end{lem}
\begin{proof}
\textbf{Step 1: Setting in $S$.}

Suppose $(S, (1-\beta)C + \lambda\beta D)$ is not log canonical at $q_0$. By Lemma \ref{lem:del-Pezzo-deg4-model-line-choice} we may assume $L=E_1$. Observe that $E_1\subset\Supp(D)$, since otherwise we obtain a contradiction by Lemma \ref{lem:adjunction} (i):
$$1=E_1\cdot D \geq \mult_{q_0}D>\frac{1-(1-\beta)}{\lambda\beta}=\frac{1}{\lambda}>1.$$
Write $D=aE_1+\Omega$ where $a>0$ and $E_1\not\subseteq\Supp(\Omega)$.

Let $S_0=S$, $q_0=q$, $\Omega_0=\Omega$, $E_1^0=E_1$ and $x_0=\mult_{q_0}\Omega^0$. For $i\geq1$, let $S_i\ra S_{i-1}$ be the blow-up of a point $t_{i-1}$ with exceptional curve $F_i$ where $t_{i-1}\in F_{i-1}$ for $i\geq 2$ and $q_0=E_1\cap C$ as above. Given any $\bbQ$-divisor $A$ or $A_{i-1}$ in $S_{i-1}$, denote its strict transform in $S_i$ by $A^i$. Let $x_i=\mult_{t_i}\Omega^i$. Let $q_1=C^1\cap F_1$. If $t_{i-1}=q_{i-1}$ for $i\geq 2$, let $q_i=C^i\cap F_i$. Step 2 will take care of the case $t_1\neq q_1$ and Steps 3 and 4 will analyse the case $t_1=q_1$.

Given a point $p\in S$, we constructed several curves of low degree which contained $p$ in Section \ref{sec:delPezzo-4-curves}. In particular we found all conics through a given point. Let $p=q_0$. Table \ref{tab:delPezzo-4-lowdegree} includes all the conics through $q_0$, its rational class and numerical properties. In particular the conics $B_1\sim \pioplane{1}-E_1$ and $A_i\sim\pioplane{2}-\sum_{\substack{j=1\\j\neq i}}^5 E_j$ with $2\leq i\leq 5$ are irreducible. Indeed, if they were reducible then they would split as $B_1=E_1+L$ or $A_i=E_i+L$ for some line $L$ not passing through $q_0$, since $L$ is the only line which contains $q_0$. However, the rational class of $L$ must be
$$L\sim B_1-E_1\sim \pioplane{2}-2E_1 \text{ or } L\sim A_i -E_1 \sim \pioplane{2}-2E_1$$
and by Lemma \ref{lem:del-Pezzo-deg4-lines-list} and Table \ref{tab:delPezzo-4-lowdegree} there is no line with such rational class.

Moreover $B_1$ and $A_i$ with $2\leq i\leq 5$ are all irreducible conics containing $q_0$. Indeed, by Lemma \ref{lem:del-Pezzo-deg4-conics}, every other conic must have class $B_i\sim \pioplane{1}-E_i$ for $2\leq i\leq 5$ or $A_1\sim\pioplane{2}-\sum_{i=2}^5$ (see Table \ref{tab:delPezzo-4-lowdegree}). However all these conics split in the union of two lines, one of them being $E_1$ and $q_0\in E_1$:
$$B_i=E_i+L_{1i} \text{ for } i\geq 2 \text{, and } A_1=C_0+E_1.$$

We define the $\bbQ$-divisor
$$G=\frac{1}{3}(B_1+\sum_{i=2}^5 A_i)+\frac{2}{3}E_1\simq-K_S$$
whose support consists of conics passing through $q_0$. By Lemma \ref{lem:appendix-delPezzo-deg4-dynamic-lct5} the pair $(S, (1-\beta)C + \lambda\beta G)$ is log canonical. Hence, by Lemma \ref{lem:log-convexity} we may assume there is some curve $M\subseteq\Supp(G)$ such that $M\not\subseteq\Supp(D)$. Since $\deg M= 2$, by Lemma \ref{lem:del-Pezzo-deg4-model-conic-choice} we may assume that $M=B_1$. Hence
\begin{align}
2&=D\cdot B_1\geq a+x_0 \label{eq:del-Pezzo-dynamic-alpha-degree-4-line-bound1}\\
1&=D\cdot E_1\geq -a+x_0 \label{eq:del-Pezzo-dynamic-alpha-degree-4-line-bound2}\\
\frac{3}{2}&\geq x_0 \label{eq:del-Pezzo-dynamic-alpha-degree-4-line-bound3}
\end{align}
where \eqref{eq:del-Pezzo-dynamic-alpha-degree-4-line-bound3} follows from adding \eqref{eq:del-Pezzo-dynamic-alpha-degree-4-line-bound1} and \eqref{eq:del-Pezzo-dynamic-alpha-degree-4-line-bound2}.

Now let
$$B=\frac{1}{2}C_0+\frac{1}{2}(L_{12}+L_{13}+L_{14}+L_{15})+\frac{3}{2}E_1\simq-K_S.$$
By Lemma \ref{lem:appendix-delPezzo-deg4-dynamic-lct7} the pair $(S, (1-\beta)C + \lambda\beta B)$ is log canonical. Therefore we may assume that for some irreducible component $L\subseteq\Supp(B)$, we have $L\not\subseteq\Supp(D)$ by Lemma \ref{lem:log-convexity}. Observe that all possible $L\neq E_1$ are lines and $L\cdot E_1=1$. Hence
\begin{equation}
1=D\cdot L \geq a.
\label{eq:del-Pezzo-dynamic-alpha-degree-4-line-bound4}
\end{equation}

By \eqref{eq:del-Pezzo-dynamic-alpha-degree-4-line-bound2}
\begin{equation}
\lambda\beta(a+x_0)-\beta\leq \lambda\beta D \cdot B_1\leq 2\lambda\beta-\beta\leq 2\omega_1\beta-\beta\leq 1
\label{eq:del-Pezzo-dynamic-alpha-degree-4-line-bound0}
\end{equation}
holds. Indeed, if $0<\beta\leq \frac{2}{3}$, then $2\omega_1\beta-\beta=\beta\leq \frac{2}{3}<1$ and if $\frac{2}{3}\leq\beta\leq 1$, then $2\omega_1\beta-\beta=\frac{4}{3}\beta\leq \frac{2}{3}<1$.

By  Lemma \ref{lem:log-pullback-preserves-lc}, and \eqref{eq:del-Pezzo-dynamic-alpha-degree-4-line-bound0} the pair
\begin{equation}
(S_1, (1-\beta)C^1+\lambda\beta D^1 + (\lambda\beta (a+x_0)-\beta)F_1)
\label{eq:del-Pezzo-dynamic-alpha-degree-4-line-pair1}
\end{equation}
is not log canonical at some $t_1\in E_1$ but it is log canonical near $t_1$. Observe that $C^1\cdot E_1=\emptyset$, since $C\cdot E_1=1$. If $t_1\not\in (C^1\cup E^1_1)\cap F_1$, then the pair
$$(S, \lambda\beta\Omega^1 + (\lambda\beta (a+x_0)-\beta)F_1)$$
is not log canonical at $t_1$, but Lemma \ref{lem:adjunction} (iii) and \eqref{eq:del-Pezzo-dynamic-alpha-degree-4-line-bound3} give a contradiction:
$$1<\lambda\beta\Omega^1 \cdot F_1<\omega_1\beta x_0 \leq \frac{2}{3}\cdot \frac{3}{2}\leq 1.$$

\textbf{Step 2: we show that the point $t_1\neq E_1^1\cap F_1$.}
Suppose for contradiction that$t_1=E^1_1\cap F_1$. Then the pair
\begin{equation}
(S_1, (1-\beta)C^1+\lambda\beta D^1 + (\lambda\beta (a+x_0)-\beta)F_1)
\label{eq:del-Pezzo-dynamic-alpha-degree-4-line-pair1-A}
\end{equation}
is not log canonical at $t_1$. 
In section \eqref{sec:delPezzo-4-curves} we constructed the cubic $Q_1\sim \pioplane{3}-2E_1-\sum^5_{i=2} E_i$ as the curve passing through a point $p=q_0$ such that its strict transform $Q_1^1$ contains $t_1$. We claim the curve $Q_1$ is irreducible. If this was not the case, then $Q_1=L+M$, the union of a line and a conic with one of them containing $t_1$. But $t_1\in E_1^1$, $A_i\cdot E_1=1$ for $i\geq 2$ and $B_1\cdot E_1=1$, then $t_1\not\in B_1^1\cup A_i^1$ for $i\geq 2$, which are all conics passing through $q_0$ by Lemma \ref{lem:del-Pezzo-deg4-conics} and the fact that $A_1$ is reducible. Therefore $L=E_1$ and 
$$M\sim Q_1-E_1\sim\pioplane{3}-3E_1-\sum^5_{i=2} E_i.$$
However there is no such conic in $S$ by Lemma \ref{lem:del-Pezzo-deg4-conics} and Table \ref{tab:delPezzo-4-lowdegree}. 

By Lemma \ref{lem:appendix-delPezzo-deg4-dynamic-lct6} the pair $$(S, (1-\beta)C + \lambda\beta(Q_1+E_1))$$ is log canonical. Since $Q_1+E_1\sim-K_S$, by Lemma \ref{lem:log-convexity}, we may assume that $Q_1\not\subseteq\Supp(D)$. Hence
\begin{equation}
3=Q_1\cdot D\geq 2a+x_0+x_1.
\label{eq:del-Pezzo-dynamic-alpha-degree-4-line-bound-E1-5}
\end{equation}
By Lemma \ref{lem:log-pullback-preserves-lc}, since \eqref{eq:del-Pezzo-dynamic-alpha-degree-4-line-pair1-A} is not log canonical, then the pair
\begin{equation}
(S_2, \lambda\beta(aE_1^2+\Omega^2) + (\lambda\beta(a+x_0)-\beta)F_1^2 + (\lambda\beta(2a+x_0+x_1)-\beta-1)F_2)
\label{eq:del-Pezzo-dynamic-alpha-degree-4-line-pair2-A}
\end{equation}
is not log canonical at some isolated $t_2\in F_2$. Indeed, the coefficient of $F_2$ satisfies
$$\lambda\beta(2a+x_0+x_1)-2\beta\leq 3\lambda\beta-2\beta<3\omega_1\beta-2\beta\leq 1$$
by \eqref{eq:del-Pezzo-dynamic-alpha-degree-4-line-bound-E1-5} and case analysis on $\beta$: for $0<\beta\leq \frac{2}{3}$, $3\omega_1\beta-2\beta=\beta\leq \frac{2}{3}$ and for $\frac{2}{3}\leq \beta\leq 1$, $3\omega_1\beta-2\beta=2-2\beta\leq \frac{2}{3}$.

If $t_2\not\in (F^2_1\cup E_1^2)\cap F_2$, then the pair
$$(S_2, \lambda\beta\Omega^2 + (\lambda\beta(2a+x_0+x_1)-2\beta)F_2)$$
is not log canonical at $t_2$ and Lemma \ref{lem:adjunction} (iii) gives a contradiction, using \eqref{eq:del-Pezzo-dynamic-alpha-degree-4-line-bound3}:
$$1<\lambda\beta\Omega^2\cdot F_2 = \lambda\beta x_1 \leq \lambda\beta x_0<\omega_1\beta x_0\leq 1.$$

If $t_2=E_1^2\cap F_2$, then the pair
$$(S_2, \lambda\beta(aE_1^2+\Omega^2)+(\lambda\beta(2a+x_0+x_1)-\beta-1)F_2)$$
is not log canonical. By Lemma \ref{lem:adjunction} (iii) with $E_1^2$, we obtain
\begin{align*}
1<&E_1^2\cdot (\lambda\beta\Omega^2+(\lambda\beta(2a+x_0+x_1)-\beta-1)F_2)\\
=&\lambda\beta(1+a-x_1-x_0+2a+x_0+x_1)-\beta -1\\
=&\lambda\beta(3a+1)-\beta-1<\omega_1\beta(3a+1)-\beta-1.
\end{align*}
We claim that $\omega_1\beta(3a+1)-\beta-1\leq 1$, which gives a contradiction ruling out the case when $t_2=E^2_1\cap F_2$. Indeed, if $0<\beta\leq \frac{2}{3}$, then
$$\omega_1\beta(3a+1)-\beta-1=3a\beta-1\leq 3\beta\leq 1$$
by \eqref{eq:del-Pezzo-dynamic-alpha-degree-4-line-bound4}. If $\frac{2}{3}\leq\beta\leq 1$, then
$$\omega_1\beta(3a+1)-\beta-1=2a-\beta-\frac{1}{3}\leq 2-1=1$$
again by \eqref{eq:del-Pezzo-dynamic-alpha-degree-4-line-bound4}. Therefore $t_2=F_1^2\cap F_2$. But then, the pair
$$(S_2, \lambda\beta\Omega^2 + (\lambda\beta(a+x_0)-\beta)F^2_1 + (\lambda\beta(2a+x_0+x_1)-\beta-1)F_2)$$
is not log canonical at $t_2$. By Lemma \ref{lem:adjunction} (iii) applied to $F^2_1$, we obtain
\begin{align*}
1<&(\lambda\beta\Omega^2 + (\lambda\beta(2a+x_0+x_1)-\beta-1)F_2)\cdot F^2_1\\
=&\lambda\beta(x_0-x_1+2a+x_0+x_1)-\beta-1\\
<&\omega_1\beta(2a+2x_0)-\beta-1 \leq 4\omega_1\beta-\beta-1,
\end{align*}
by \eqref{eq:del-Pezzo-dynamic-alpha-degree-4-line-bound1}. We claim that $4\omega_1-\beta\leq 2$ which contradicts the above statement. Indeed, if $0<\beta\leq \frac{2}{3}$, then $4\omega_1\beta-\beta=3\beta\leq 2$. If $\frac{2}{3}\leq\beta\leq 1$, then $4\omega_1\beta-\beta=\frac{8}{3}-\beta\leq 2$.

Therefore the pair \eqref{eq:del-Pezzo-dynamic-alpha-degree-4-line-pair2-A} is log canonical, which is impossible unless the pair \eqref{eq:del-Pezzo-dynamic-alpha-degree-4-line-pair1} is log canonical at $E^1_1\cap F_1$. Therefore the pair \eqref{eq:del-Pezzo-dynamic-alpha-degree-4-line-pair1} is not log canonical at $t_1=C^1\cap F_1=q_1$. We may now apply Lemma \ref{lem:log-pullback-preserves-lc} to conclude that the pair
\begin{equation}
(S_2, (1-\beta)C^2 + \lambda\beta\Omega^2 + (\lambda\beta(a_0+x_0)-\beta)F^1_1+(\lambda\beta(a+x_0+x_1)-2\beta)
\label{eq:del-Pezzo-dynamic-alpha-degree-4-line-pair2}
\end{equation}
is not log canonical at some (possibly all) $t_2\in F_2$.

\textbf{Step 3: We show that $q_1\not\in Q^1$ for any irreducible conic $Q\subset S$ such that $q_0\in Q$.}
Suppose for contradiction that $q_1\in Q^1$ for some irreducible conic $Q\subset S$. By Lemma \ref{lem:del-Pezzo-deg4-conics}, we may assume that $Q=B_1$ or $Q=A_i$. Recall from Step 1 that we chose a model such that $B_1\not\subset\Supp(D)$. Therefore we need to analyse each case separately.

\textbf{Step 3a: We show $q_1\not\in B_1^1$.}
Suppose for contractiction that $q_1=B_1^1\cap C^1\cap F_1$. This means $(B_1\cdot C)\vert_{q_0}=2$, i.e. $C$ and $B_1$ are tangent at $q_0$. Since $B_1\not\subset\Supp(\Omega)$, then
\begin{equation}
2-a=\Omega\cdot B_1\geq x_0+x_1
\label{eq:del-Pezzo-dynamic-alpha-degree-4-line-bound-B1-5}
\end{equation}
holds. In particular
$$\lambda\beta(a+x_0+x_1)-\beta<2\omega_1\beta-\beta\leq 1.$$
Indeed if $0<\beta\leq \frac{2}{3}$, then $2\omega_1\beta-\beta=\beta\leq \frac{2}{3}<1$ and if $\frac{2}{3}\leq \beta\leq 1$, then $2\omega_1\beta-\beta=\frac{4}{3}-\beta\leq \frac{2}{3}\leq 1.$

This is a necessary condition to apply Theorem \ref{thm:pseudo-inductive-blow-up} (ii)-(iv) when $i\geq 2$, since $\mult_{q_i} D^i=x_i$ for $i\geq 1$ and $\mult_q D=a+x_0$, so $\lambda\beta(\mult_{q_0}D+\mult_{q_1}D^1)-\beta\leq 1$. We will assume this condition from now onwards. Therefore by Theorem \ref{thm:pseudo-inductive-blow-up} (iii) with $i=2$ we conclude that $t_2=C^2\cap F_2$. Moreover
$$\lambda\beta(\mult_{q_0} D +\mult_{q_1}D^1)-2\beta \leq 1-\beta\leq 1$$
holds. Therefore the pair
$$(S_3, (1-\beta)C^3 + \lambda\beta\Omega^3 + (\lambda\beta(a+x_0+x_1)-2\beta)F_2^3 + (\lambda\beta(a+x_0+x_1+x_2)-3\beta)F_3)$$
is not log canonical at some $t_3\in F_3$ by Theorem \ref{thm:pseudo-inductive-blow-up} (i). By \eqref{eq:del-Pezzo-dynamic-alpha-degree-4-line-bound3} and \eqref{eq:del-Pezzo-dynamic-alpha-degree-4-line-bound-B1-5}
\begin{equation}
\lambda\beta(a+x_0+x_1+x_2)-2\beta<\frac{7}{2}\omega_1\beta-2\beta\leq 1.
\label{eq:del-Pezzo-dynamic-alpha-degree-4-line-aux1}
\end{equation}
Indeed, if $0<\beta\leq \frac{2}{3}$, then $\frac{7}{2}\omega_1\beta-2\beta=\frac{3}{2}\beta\leq 1$. If $\frac{2}{3}\leq \beta\leq 1$, then $\frac{7}{2}\omega_1\beta-2\beta=\frac{7}{3}-2\beta\leq 1$.

Now we can apply Theorem \ref{thm:pseudo-inductive-blow-up} (iii) with $i=3$ and conclude that $t_3=C^3\cap F_3=q_3$. Since $F_2^3\cdot C^3=0$, then $q_3\not\in F_2^3$ and the pair
$$(S_3, (1-\beta)C^3 + \lambda\beta\Omega^3 + (\lambda\beta(a+x_0+x_1+x_2)-3\beta) F_3)$$
is not log canonical at $q_3=F_3\cap C^3$. By Lemma \ref{lem:adjunction} (iii) with $C^3$, we obtain a contradiction:
\begin{align*}
1<&C^3\cdot (\lambda\beta\Omega^3+\lambda\beta(a+x_0+x_1+x_2)-3\beta F_3)=\\
<&\omega_1\beta(4-a-x_0-x_1-x_2+a+x_0+x_1+x_2)-3\beta\\
=&4\omega_1\beta-3\beta\leq 1.
\end{align*}
Indeed, if $0<\beta\leq \frac{2}{3}$, then $4\omega_1-3\beta=3\beta<1$ and if $\frac{2}{3}\leq \beta\leq 1$, then $4\omega_1\beta-3\beta=\frac{8}{3}-3\beta\leq \frac{2}{3}<1.$

\textbf{Step 3b: We show $q_1\not\in A_i^1$ for $i\geq 2$.}
We proceed by \emph{reductio ad absurdum}. Suppose $q_1\in A_i^1$. This means $(A_i\cdot C)\vert_{q_0}=2$, i.e. $C$ and $A_i$ are tangent at $q_0$. Moreover, since $B_1\cdot A_i=1$ for $i\geq 2$ and $A_i\cdot A_j=1$ for $i\neq j$, there is no other conic in $S$ such that its strict transform contains $q_1$. Without loss of generality assume $q_1=A_5^1\cap F_1$.

Write $D=aE_1+bA_2+\Gamma$ where $a>0$, $b\geq 0$, $\Omega=bA_5 + \Gamma$, and $E_1,B_1,A_5\not\subseteq\Gamma$. Let $y_i=\mult_{q_i}\Gamma^i$. Then $m_0=a+b+y_0$, $m_1=b+y_1$ and $m_i=y_i$ for $i\geq 2$, since $C^1\cdot A_1^5=1$. We note
\begin{equation}
2-a\geq A_5\cdot \Gamma\geq y_0+y_1.
\label{eq:del-Pezzo-dynamic-alpha-degree-4-line-bound-A5-5}
\end{equation}

The pair \eqref{eq:del-Pezzo-dynamic-alpha-degree-4-line-pair2}, at which some $t_2\in F_2$ is not log canonical,  may be rewritten as
\begin{equation}
(S_2, (1-\beta)C^2 + \lambda\beta(bA^2_5+\Gamma^2)+(\lambda\beta(a+b+y_0)-\beta)F_1^1+(\lambda\beta(a+2b+y_0+y_1)-2\beta)F_2).
\label{eq:del-Pezzo-dynamic-alpha-degree-4-line-pair2-A5}
\end{equation}

Recall the curves constructed in section \ref{sec:delPezzo-4-curves} and Table \ref{tab:delPezzo-4-lowdegree} using points $p=q_0\in S_0$ and $q=q_1\in F_1\subset S_1$. We use these curves to define the $\bbQ$-divisor
$$H=\frac{3}{5}A_5+\frac{1}{5}(R_{125}+R_{135}+R_{145})+\frac{1}{5}Q_5+\frac{2}{5}E_1\simq-K_S.$$
We claim that all the curves used to define $H$ are irreducible. It is enough to check it for the cubics $R_{125},R_{135},R_{145}$ and $Q_5$. If any such cubic $M$ was reducible, then it would split as $M=A_5+L$, where $L$ is a line, since $q_1\in M^1$ and the only line or cubic such that its strict transform contains $q_1$ is $A_5$. However by Lemma \ref{lem:del-Pezzo-deg4-lines-list} and Table \ref{tab:delPezzo-4-lowdegree} there is no line $L$ with any of the following expected rational classes:
\begin{align*}
&R_{125}-A_5\sim E_3+E_4-E_5,
&R_{135}-A_5\sim E_2+E_4-E_5,\\
&R_{125}-A_5\sim E_2+E_3-E_5,
&Q_5-A_5\sim\pioplane{1}-2E_5.
\end{align*}

By Lemma \ref{lem:appendix-delPezzo-deg4-dynamic-lct8}, the pair $(S, (1-\beta)C + \lambda\beta H)$ is log canonical. Therefore by Lemma \ref{lem:log-convexity} we may assume there is some irreducible component $H_i\subset\Supp(H)$ such that $H_i\not\subseteq\Supp(D)$. Since $E_1\subseteq\Supp(D)$, then either $H_i=A_5$ or $H_i$ is one of the irreducible cubic curves in $\Supp(H)$. Note that all these cubics have very similar numerical properties. In particular $H_i\cdot A_5=2$ and $H_i\cdot E_1=1$. If $A_5\not\subset\Supp(D)$, then $b=0$ and
$$a+2b+y_0+y_1\leq 2 \leq \frac{5}{2}<3$$
by \eqref{eq:del-Pezzo-dynamic-alpha-degree-4-line-bound-A5-5}. If $H_i$ is a cubic as above, then
$$3-a-2b=H_i\cdot \Gamma\geq y_0+y_1.$$
Therefore, we have proved that if $q_1\in A_5$, then
\begin{equation}
a+2b+y_0+y_1\leq 3.
\label{eq:del-Pezzo-dynamic-alpha-degree-4-line-bound-A5-6}
\end{equation}
If $0<\beta\leq \frac{2}{3}$, then $3\omega_1\beta-2\beta=\beta\leq \frac{2}{3}<1$. If $\frac{2}{3}\leq \beta 1$, then $3\omega_1\beta-2\beta=2-2\beta\leq \frac{2}{3}<1$. Therefore, using \eqref{eq:del-Pezzo-dynamic-alpha-degree-4-line-bound-A5-6} we have proven
$$\lambda\omega\beta(a+2b+y_0+y_1)-2\beta\leq 1$$
and the pair \eqref{eq:del-Pezzo-dynamic-alpha-degree-4-line-pair2-A5} is log canonical in codimension $1$. We analyse the position of $t_2\in F_2$. For this, notice that
$$A_5^1\dot C^1 = F_1\cdot A_5^1=F_1\cdot C^1=1$$
which gives the following possibilities for $t_2$:
\begin{itemize}
	\item[(i)] The point $t_2\in F_2\setminus (A_5^2\cup H_i^2\cup F_1^2\cup C^2)$ and
		$$(S_2, \lambda\beta\Gamma^2 + (\lambda\beta(a+2b+y_0+y_1)-2\beta)F_2)$$
		is not log canonical at $t_2$.
	\item[(ii)] The point $t_2=F_2\cup A_5^2$ and
		$$(S_2, \lambda\beta(b A_5^2 + \Gamma^2) + (\lambda\beta(a+2b+y_0+y_1)-2\beta)F_2)$$
		is not log canonical at $t_2$.
	\item[(iii)] The point $t_2=F_2\cup F_1^2$ and
		$$(S_2, \lambda\beta\Gamma^2 + (\lambda\beta(a+b+y_0)-2\beta)F_1^2+ (\lambda\beta(a+2b+y_0+y_1)-2\beta)F_2)$$
		is not log canonical at $t_2$.
	\item[(ii)] The point $t_2=q_2=F_2\cup C^2$ and
		$$(S_2, (1-\beta)C^2 + \lambda\beta\Gamma^2 + (\lambda\beta(a+2b+y_0+y_1)-2\beta)F_2)$$
		is not log canonical at $t_2$.
\end{itemize}
For (i) and (ii) as above we may apply Lemma \ref{lem:adjunction} with $F_2$ and obtain a contradiction using \eqref{eq:del-Pezzo-dynamic-alpha-degree-4-line-bound3}:
$$1<\lambda\beta(bA_5^2+\Gamma^2)\cdot F_2=\lambda\beta x_1<\frac{3}{2}\omega_1\beta \leq 1.$$
In case (iii) we apply Lemma \ref{lem:adjunction} (iii) with $F_1^2$ and obtain a contradiction via \eqref{eq:del-Pezzo-dynamic-alpha-degree-4-line-bound1} and \eqref{eq:del-Pezzo-dynamic-alpha-degree-4-line-bound2}:
\begin{align*}
1<&F_1^2\cdot (\lambda\beta \Gamma + (\lambda\beta(a+2b+y_0+y_1)-2\beta)F_2)=\\
=&\lambda\beta(y_0+y_1+a+2b+y_0+y_1)-2\beta
<&\omega_1\beta(a+2b+2y_0)-2\beta\leq \frac{7}{2}\omega_1\beta-2\beta\leq 1,
\end{align*}
where for the last inequality we proceed as for \eqref{eq:del-Pezzo-dynamic-alpha-degree-4-line-aux1}.

Therefore only case (iv) is possible. By Lemma \ref{lem:log-pullback-preserves-lc}, the pair
$$(S_3, (1-\beta)C^3+\lambda\beta\Gamma^3 + (\lambda\beta(a+2b+y_0+y_1)-2\beta)F_2^3)$$
is not log canonical at some $t_3\in F_3$ and log canonical near $t_3$, which follows from
$$\lambda\beta(a+2b+y_0+y_1+y_2)-3\beta<\frac{9}{2}\omega_1\beta-3\beta\leq 1,$$
where we apply \eqref{eq:del-Pezzo-dynamic-alpha-degree-4-line-bound3} and \eqref{eq:del-Pezzo-dynamic-alpha-degree-4-line-bound-A5-6}. The last inequality follows from case analysis: if $0<\beta\leq \frac{2}{3}$, then $\frac{9}{2}\omega_1\beta-3\beta=\frac{3}{2}\beta\leq 1$. If $\frac{2}{3}\leq \beta\leq 1$, then $\frac{9}{2}\omega_1\beta-3\beta=3-3\beta\leq 1$.

We rule out possible positions for $t_3\in F_3$. If $t_3\not\in (F_2^3\cup C^3)\cap F_3$, then
$$(S_3, \lambda\beta\Gamma^3 + (\lambda\beta(a+2b+y_0+y_1+y_2)-3\beta)F_3)$$
is not log canonical and Lemma \ref{lem:adjunction} (iii) with $F_3$ gives
$$1<\lambda\beta\Gamma^3\cdot F_3=\lambda\beta y_3<\omega_1\beta x_0\leq \frac{3}{2}\omega_1\beta\leq 1$$
by \eqref{eq:del-Pezzo-dynamic-alpha-degree-4-line-bound3}.

If $t_3=F_2^3\cap F_3$, then $t_3\not\in C^3$, but Lemma \ref{lem:adjunction} (iii) with $F_2^3$ gives a contradiction
\begin{align*}
1<&F_2^3\cap(\lambda\beta\Gamma^3 + (\lambda\beta(a+2b+y_0+y1+y_2)-3\beta)F_3)\\
=&\omega_1\beta(a+2b+2y_1+y_0)-3\beta\leq \frac{9}{2}\omega_1\beta-3\beta\leq 1
\end{align*}
using \eqref{eq:del-Pezzo-dynamic-alpha-degree-4-line-bound3} and \eqref{eq:del-Pezzo-dynamic-alpha-degree-4-line-bound-A5-6}.

Therefore $t_3=F_3\cap C^3=q_3$. Lemma \ref{lem:adjunction} (iii) applied with $C^3$ gives
\begin{align*}
1<&C^3\cdot (\lambda\beta\Gamma^3 + \lambda\beta(a+2b+y_0+y_1+y_2)-3\beta)F_3)=\\
=&4\lambda\beta-3\beta<4\omega_1\beta-3\beta\leq 1.
\end{align*}
which is impossible, hence the initial assumption of Step 3b, alas $q_1\in A_i^1$ for some $i\geq 2$, is not correct.

\textbf{Step 4: We finish the proof.}
From steps 2 and 3 the pair
$$(S_1, (1-\beta) C^1 + \lambda\beta\Omega^1 + (\lambda\beta(a+x_0)-\beta)F_1)$$
is not log canonical at $q_1=F_1\cap C^1$ and there is no conic or line $Z$ in $S$ such that $q_1\in Z^1$. Recall that $\Omega\simq-K_S-aE_1$ and $x_i=\mult_{q_i}\Omega^i$.

Let $Q_1\sim\pioplane{3}-2E_1-E_2-E_3-E_4-E_5$ be the cubic curve with $q_1\in Q_1^1$. By Lemma \ref{lem:appendix-delPezzo-deg4-dynamic-lct6} the pair $(S, (1-\beta)C + \lambda\beta (E_1+Q_1))$ is log canonical. Since there are no lines or conics whose strict transform in $S_1$ contains $q_1$, the cubic $Q_1$ is irreducible. Therefore by Lemma \ref{lem:log-convexity} we may assume that $Q_1\not\subseteq \Supp(D)$. Hence
\begin{equation}
3=Q_1\cdot D \geq a+x_0+x_1.
\label{eq:del-Pezzo-dynamic-alpha-degree-4-line-bound5}
\end{equation}
By Lemma \ref{lem:log-pullback-preserves-lc}, the pair
$$(S_2, (1-\beta)C^2 + \lambda\beta\Omega^2 + (\lambda\beta(a+x_0)-\beta)F^2_1 + (\lambda\beta(a+x_0+x_1)-2\beta)F_2)$$
is not log canonical at some $t_2\in F_2$. Since \eqref{eq:del-Pezzo-dynamic-alpha-degree-4-line-bound5} holds the pair is log canonical near $t_2$. Indeed, observe that
$$\lambda\beta(a+x_0+x_1)-2\beta<3\omega_1\beta-2\beta\leq 1$$
where the last inequality follows by case analysis for $\beta$: if $0<\beta\leq \frac{2}{3}$, then $3\omega_1\beta-2\beta=\beta\leq \frac{2}{3}<1$ and if $\frac{2}{3}\leq\beta\leq 1$, then $3\omega_1\beta-2\beta=2-2\beta\leq \frac{2}{3}<1$.

If $t_2\neq C^2\cap F_2=q_2$ and $t_2\neq F_1^2\cap F_2$, then we apply Lemma \ref{lem:adjunction} (iii) with $F_2$ and obtain a contradiction:
$$1<(F_2\cdot \lambda\beta\Omega^2)\vert_{t_2}=\lambda\beta x_1<\omega_1\beta x_1\leq 1$$
by \eqref{eq:del-Pezzo-dynamic-alpha-degree-4-line-bound3}.

If $t_2=F^2_1\cap F_2$, then Lemma \ref{lem:adjunction} (iii) with $F_1^2$ also gives a contradiction:
\begin{align}
1<&(F_1^2\cdot (\lambda\beta\Omega^2 + (\lambda\beta(a+x_0+x_1)-2\beta)F_2))\vert_{t_2}\nonumber \\
<&\omega_1\beta(a+2x_0)-2\beta \leq\frac{7}{2}\omega_1\beta-2\beta\leq 1 \label{eq:del-Pezzo-dynamic-alpha-degree-4-line-easy-inequality}
\end{align}
by \eqref{eq:del-Pezzo-dynamic-alpha-degree-4-line-bound1} and \eqref{eq:del-Pezzo-dynamic-alpha-degree-4-line-bound3}. Hence $t_2=C^2\cap F_2=q_2$.

Applying Lemma \ref{lem:log-pullback-preserves-lc}, the pair
$$(S_3, (1-\beta)C^3+\lambda\beta \Omega^3 + (\lambda\beta(a+x_0+x_1)-2\beta)F_2^3+(\lambda\beta(a+x_0+x_1+x_2)-3\beta)F_3)$$
is not log canonical at some $t_3\in F_3$ and log canonical near $t_3$. Indeed
$$\lambda\beta(a+x_0+x_1+x_2)-3\beta\leq \frac{7}{2}\omega_1\beta-3\beta< \frac{7}{2}\omega_1\beta-2\beta\leq 1$$
by by \eqref{eq:del-Pezzo-dynamic-alpha-degree-4-line-bound5},\eqref{eq:del-Pezzo-dynamic-alpha-degree-4-line-bound3} and \eqref{eq:del-Pezzo-dynamic-alpha-degree-4-line-easy-inequality}. 

If $t_3\neq C^3\cap F_3=q_3$ and $t_3\neq F_2^3\cap F_3$, then we apply Lemma \ref{lem:adjunction} (iii) with $F_3$ and obtain a contradiction:
$$1<(F_3\cdot (\lambda\beta\Omega^3))\vert_{t_3}=\lambda\beta x_3<\omega_1\beta x_3\leq 1$$
by \eqref{eq:del-Pezzo-dynamic-alpha-degree-4-line-bound3}.

If $t_3=F^3_2\cap F_3$, then Lemma \ref{lem:adjunction} (iii) with $F_2^3$ also gives a contradiction:
\begin{align*}
1<&(F_2^3\cdot (\lambda\beta\Omega^3 + (\lambda\beta(a+x_0+x_1+x_2)-3\beta)F_3))\vert{t_3}\\
<&\omega_1\beta(a+x_0+2x_1)-3\beta \leq\frac{9}{2}\omega_1\beta-3\beta\leq 1
\end{align*}
by \eqref{eq:del-Pezzo-dynamic-alpha-degree-4-line-bound5} and \eqref{eq:del-Pezzo-dynamic-alpha-degree-4-line-bound3}. The last inequality follows by case analysis. Indeed, if $0<\beta\leq \frac{2}{3}$, then $\frac{9}{2}\omega_1\beta-3\beta=\frac{3}{2}\beta\leq 1$ and if $\frac{2}{3}\leq \beta\leq 1$, then $\frac{9}{2}\omega_1\beta-3\beta=3-3\beta\leq 1$. Hence $t_3=C^3\cap F_3=q_3$ and
$$(S_3, (1-\beta)C^3+\lambda\beta\Omega^3 + (\lambda\beta(a+x_0+x_1+x_2)-3\beta)F_3)$$
is not log canonical at $q_3$. This is impossible by Lemma \ref{lem:adjunction} (iii) with $C^3$, finishing the proof:
\begin{align*}
1&<C^3\cdot (\lambda\beta\Omega^3 + (\lambda\beta(a+x_0+x_1+x_2)-3\beta)F_3)<\\
 &\omega_1\beta(4-a-x_0-x_1-x_2+a+x_0+x_1+x_2)-3\beta<4\omega_1\beta-3\beta\leq 1.
\end{align*}
The last inequality follows by case analysis. If $0<\beta\leq \frac{2}{3}$, then $\frac{9}{2}\omega_1\beta-3\beta=\frac{3}{2}\beta\leq 1$. If $\frac{2}{3}\leq\beta\leq 1$, then $\frac{9}{2}\omega_1\beta-3\beta=3-3\beta\leq 1$.
\end{proof}

\section{Smooth cubic surface}
\label{sec:dynamic-alpha-degree3}
\begin{thm}
\label{thm:del-Pezzo-dynamic-alpha-degree-3}
Let $S$ be a smooth del Pezzo surface of degree $3$ and let $C$ be a smooth elliptic curve in $S$, $C\simq-K_S$.
Suppose $S$ contains at least an Eckardt point.
\begin{itemize}
	\item[(i)] If for some Eckardt point $p$, we have $p\in C$ then
		\begin{equation}
				\alpha(S,(1-\beta)C)=\omega_5=
						\begin{dcases}
								1 									&\text{ for } 0<\beta\leq \frac{1}{2},\\
								\frac{1+\beta}{3\beta}		&\text{ for } \frac{1}{2}\leq \beta\leq 1.
						\end{dcases}
			\label{eq:del-Pezzo-dynamic-alpha-degree-3-Eckardt-special}
		\end{equation}
	\item[(ii)] If $C$ contains no Eckardt points, then
		\begin{equation}
				\alpha(S,(1-\beta)C)=\omega_4=
						\begin{dcases}
								1 									&\text{ for } 0<\beta\leq \frac{2}{3},\\
								\frac{2}{3\beta}		&\text{ for } \frac{2}{3}\leq \beta\leq 1.
						\end{dcases}
			\label{eq:del-Pezzo-dynamic-alpha-degree-3-Eckardt-general}
		\end{equation}
\end{itemize}

Suppose $S$ contains no Eckardt points.
\begin{itemize}
	\item[(i)] If $p\in C$ for $p\in B\in \vert-K_S\vert$, where $B=L+M$, a line and an irreducible conic intersecting only at $p$, then
		\begin{equation}
				\alpha(S,(1-\beta)C)=\omega_3=
						\begin{dcases}
								1 									&\text{ for } 0<\beta\leq \frac{2}{3},\\
								\frac{2+\beta}{4\beta}		&\text{ for } \frac{2}{3}\leq \beta\leq 1.
						\end{dcases}
			\label{eq:del-Pezzo-dynamic-alpha-degree-3-general-special}
		\end{equation}
	\item[(ii)] If for all $B\in \vert-K_S\vert$ where $B=L+M$, a line and an irreducible conic intersecting only at $p_B$, we have that $p_B\not\in C$ but there is an irreducible $T\in \vert-K_S\vert$ with a cuspidal point $p_T\in T$ such that $p_T\in C$ and $(T\cdot C)\vert_{p_T},=3$ then
		\begin{equation}
				\alpha(S,(1-\beta)C)=\omega_2=
						\begin{dcases}
								1 									&\text{ for } 0<\beta\leq \frac{2}{3},\\
								\frac{2+3\beta}{6\beta}		&\text{ for } \frac{2}{3}\leq \beta\leq \frac{5}{6}.\\
								\frac{3}{4\beta}		&\text{ for } \frac{5}{6}\leq \beta\leq 1.
						\end{dcases}
			\label{eq:del-Pezzo-dynamic-alpha-degree-3-general-semispecial}
		\end{equation}
	\item[(iii)] In any other case
		\begin{equation}
				\alpha(S,(1-\beta)C)=\omega_1=
						\begin{dcases}
								1 									&\text{ for } 0<\beta\leq \frac{3}{4},\\
								\frac{3}{4\beta}		&\text{ for } \frac{3}{4}\leq \beta\leq 1.
						\end{dcases}
			\label{eq:del-Pezzo-dynamic-alpha-degree-3-general-general}
		\end{equation}

\end{itemize}

\end{thm}

\begin{proof}
Recall that $S$ satisfies the Cat Property (Theorem \ref{thm:del-Pezzo-cat-degree-3}). 

Suppose $S$ contains no Eckardt points. Then, by Lemma \ref{lem:del-Pezzo-cat-list-degree-3}, all cats of $S$ are either rational curves $T\in \vert-K_S\vert$ with a cuspidal singularity at $p_T\in T$ and $B=L+M$, the union of a line and a conic intersecting only at a point $p_B=L\cap M$. 

If $p_B\not\in C$ for all $B$ as above and for all $T$ as above such that $p_T\in C$ we have $(C\cdot T)_{p_T}=2$, then
$$\alpha(S, (1-\beta)C)=\min\{1,\frac{5}{6\beta}, \frac{3}{4\beta}, \frac{3+2\beta}{6\beta}\}=\min\{1, \frac{3}{4\beta}\}=\omega_1$$
by Corollary \ref{cor:cat-property-easy-dynamic-alpha} and lemmas \ref{lem:appendix-local-cusp-snc}, \ref{lem:appendix-local-cusp-normal} and \ref{lem:appendix-local-tacnode-snc}.

If $p_B\not\in C$ for all $B$ as above but for some $T$ as above such that $p_T\in C$ we have $(C\cdot T)\vert_{p_T}=3$, then
$$\alpha(S, (1-\beta)C)=\min\{1,\frac{5}{6\beta}, \frac{3}{4\beta}, \frac{3+2\beta}{6\beta}, \frac{2+3\beta}{6\beta}\}=\min\{1, \frac{3}{4\beta}, \frac{3+2\beta}{6\beta}\}=\omega_2$$
by Corollary \ref{cor:cat-property-easy-dynamic-alpha} and lemmas \ref{lem:appendix-local-cusp-snc}, \ref{lem:appendix-local-cusp-normal} and \ref{lem:appendix-local-tacnode-snc}.

If $p_B\in C$ for some cat $B$ as above, observe that $(C\cdot L)=1$ and that $C$ and $M$ cannot be tangent at $p$, since then $C$ and $L$ would be tangent. Hence $(C\cdot M)\vert_p=1$. Then
$$\alpha(S, (1-\beta)C)=\min\{1,\frac{5}{6\beta}, \frac{3}{4\beta}, \frac{3+2\beta}{6\beta}, \frac{2+3\beta}{6\beta}, \frac{2+\beta}{4\beta}\}=\min\{1, \frac{2+\beta}{4\beta}\}=\omega_3$$
by Corollary \ref{cor:cat-property-easy-dynamic-alpha} and lemmas \ref{lem:appendix-local-cusp-snc}, \ref{lem:appendix-local-cusp-normal}, \ref{lem:appendix-local-tacnode-snc} and \ref{lem:appendix-local-tacnode-normal}.

Suppose $S$ contains at least an Eckardt point. Then, by Lemma \ref{lem:del-Pezzo-cat-list-degree-3}, all cats of $S$ are rational curves $T\in \vert-K_S\vert$ with a cuspidal singularity at $p_T\in T$ and $B=L+M$, the union of a line and a conic intersecting only at a point $p_B$ and $A=L_1+L_2+L_3$, the union of $3$ lines intersecting at an Eckardt point $p_A$. 

If $p_A\not\in C$ for all $A$ as above, then
$$\alpha(S, (1-\beta)C)=\min\{1,\frac{5}{6\beta}, \frac{3}{4\beta}, \frac{3+2\beta}{6\beta}, \frac{2+\beta}{4\beta}, \frac{2+3\beta}{6\beta}, \frac{2}{3\beta}\}=\min\{1, \frac{2}{3\beta}\}=\omega_4$$
by Corollary \ref{cor:cat-property-easy-dynamic-alpha} and lemmas \ref{lem:appendix-local-cusp-snc}, \ref{lem:appendix-local-cusp-normal}, \ref{lem:appendix-local-tacnode-snc}, \ref{lem:appendix-local-tacnode-normal} and \ref{lem:appendix-local-3lines}.

If $p_A\not\in C$ for $p_A$ an Eckardt point of any $A$ as above, then
$$\alpha(S, (1-\beta)C)=\min\{1,\frac{5}{6\beta}, \frac{3}{4\beta}, \frac{3+2\beta}{6\beta}, \frac{2+\beta}{4\beta}, \frac{2}{3\beta}, \frac{2+3\beta}{6\beta}, \frac{1+\beta}{3\beta}\}=\min\{1, \frac{1+\beta}{3\beta}\}=\omega_5$$
by Corollary \ref{cor:cat-property-easy-dynamic-alpha} and lemmas \ref{lem:appendix-local-cusp-snc}, \ref{lem:appendix-local-cusp-normal}, \ref{lem:appendix-local-tacnode-snc}, \ref{lem:appendix-local-tacnode-normal}, \ref{lem:appendix-local-3lines} and \ref{lem:appendix-local-4lines}.
\end{proof}

\section{Del Pezzo surface of degree $2$}
\label{sec:dynamic-alpha-degree2}

\begin{thm}
\label{thm:del-Pezzo-dynamic-alpha-degree-2}
Let $S$ be a smooth del Pezzo surface of degree $2$ and let $C$ be a smooth elliptic curve in $S$, $C\simq-K_S$.

Suppose $\vert -K_S\vert$ contains at least one curve with a tacnodal singularity.
\begin{itemize}
	\item[(i)] If $p\in C$ for $p\in B\in \vert-K_S\vert$, where $B$ has a tacnodal singularity at $p$, then
		\begin{equation}
				\alpha(S,(1-\beta)C)=\omega_4=
						\begin{dcases}
								1 									&\text{ for } 0<\beta\leq \frac{2}{3},\\
								\frac{2+\beta}{4\beta}		&\text{ for } \frac{2}{3}\leq \beta\leq 1.
						\end{dcases}
			\label{eq:del-Pezzo-dynamic-alpha-degree-2-tacnode-special}
		\end{equation}
	\item[(ii)] In any other case
		\begin{equation}
				\alpha(S,(1-\beta)C)=\omega_3=
						\begin{dcases}
								1 									&\text{ for } 0<\beta\leq \frac{3}{4},\\
								\frac{3}{4\beta}		&\text{ for } \frac{3}{4}\leq \beta\leq 1.
						\end{dcases}
			\label{eq:del-Pezzo-dynamic-alpha-degree-2-tacnode-general}
		\end{equation}
\end{itemize}
Suppose $\vert -K_S\vert$ contains no curves with a tacnodal singularity.
\begin{itemize}
	\item[(i)] If $p\in C$ for $p\in B\in \vert-K_S\vert$, where $B$ is irreducible and has a cuspidal singularity at $p$, then
		\begin{equation}
				\alpha(S,(1-\beta)C)=\omega_2=
						\begin{dcases}
								1 									&\text{ for } 0<\beta\leq \frac{3}{4},\\
								\frac{3+2\beta}{6\beta}		&\text{ for } \frac{3}{4}\leq \beta\leq 1.
						\end{dcases}
			\label{eq:del-Pezzo-dynamic-alpha-degree-2-general-special}
		\end{equation}
	\item[(ii)] In any other case
		\begin{equation}
				\alpha(S,(1-\beta)C)=\omega_1=
						\begin{dcases}
								1 									&\text{ for } 0<\beta\leq \frac{5}{6},\\
								\frac{5}{6\beta}		&\text{ for } \frac{5}{6}\leq \beta\leq 1.
						\end{dcases}
			\label{eq:del-Pezzo-dynamic-alpha-degree-2-general-general}
		\end{equation}
\end{itemize}
\end{thm}
\begin{proof}
Recall that $S$ satisfies the Cat Property (Lemma \ref{lem:del-Pezzo-cat-degree-2}). 

Suppose $\vert-K_S\vert$ contains no tacnodal curves. Then, by Lemma \ref{lem:del-Pezzo-cat-list-degree-2}, all cats of $S$ are rational curves $T\in \vert-K_S\vert$ with a cuspidal singularity at $p_T\in T$. 

If $p_T\not\in C$ for all cats $T$, then
$$\alpha(S, (1-\beta)C)=\min\{1,\frac{5}{6\beta}\}=\omega_1$$
by Corollary \ref{cor:cat-property-easy-dynamic-alpha} and Lemma \ref{lem:appendix-local-cusp-snc}.

If $p_T\in C$ for some irreducible cat $T$ with a cuspidal singularity, observe that $(T\cdot C)\vert_{p_T}=2$, since $2\geq (T\cdot C)\vert_{p_T}\geq \mult_pT\geq 2$. Then
$$\alpha(S, (1-\beta)C)=\min\{1,\frac{5}{6\beta}, \frac{3+2\beta}{6\beta}\}=\min\{1,\frac{3+2\beta}{6\beta}\}=\omega_2$$
by Corollary \ref{cor:cat-property-easy-dynamic-alpha} and Lemma \ref{lem:appendix-local-cusp-normal}.

Suppose $\vert-K_S\vert$ contains at least a tacnodal curve. Then, by Lemma \ref{lem:del-Pezzo-cat-list-degree-2}, all cats of $S$ are rational curves $T\in \vert-K_S\vert$ with a cuspidal singularity at $p_T\in T$ and tacnodal curves $Q\in \vert-K_S\vert$ with a tacnodal singularity $p_Q\in Q$.

If $p_Q\not\in C$ for all tacnodal curves $Q$, then
$$\alpha(S, (1-\beta)C)=\min\{1,\frac{3}{4\beta}, \frac{3+2\beta}{6\beta}\}=\min\{1,\frac{3}{4\beta}\}=\omega_3$$
by Corollary \ref{cor:cat-property-easy-dynamic-alpha}, and lemmas \ref{lem:appendix-local-cusp-normal} and \ref{lem:appendix-local-tacnode-snc}.

If $p_Q\in C$ for some tacnodal curve $Q=L+M$, observe that $L$ and $M$ are lines. Therefore $L\cdot C=1$ and $M\cdot C=1$. Therefore
$$\alpha(S, (1-\beta)C)=\min\{1,\frac{2+\beta}{4\beta}, \frac{3+2\beta}{6\beta}\}=\min\{1,\frac{2+\beta}{4\beta}\}=\omega_4$$
by Corollary \ref{cor:cat-property-easy-dynamic-alpha}, and lemmas \ref{lem:appendix-local-cusp-normal} and \ref{lem:appendix-local-tacnode-normal}.
\end{proof}

\section{Del Pezzo surface of degree $1$}
\label{sec:dynamic-alpha-degree1}

\begin{thm}
\label{thm:del-Pezzo-dynamic-alpha-degree-1}
Let $S$ be a smooth del Pezzo surface of degree $1$ and let $C$ be a smooth elliptic curve in $S$, $C\simq-K_S$. The dynamic $\alpha$-invariant $\alpha(S,(1-\beta)C)$ is as follows:
\begin{itemize}
	\item[(i)] If $\vert -K_S\vert$ contains a curve with a cuspidal rational curve, then
		\begin{equation}
				\alpha(S,(1-\beta)C)=\omega_2=
						\begin{dcases}
								1 									&\text{ for } 0<\beta\leq \frac{5}{6},\\
								\frac{5}{6\beta}		&\text{ for } \frac{5}{6}\leq \beta\leq 1.
						\end{dcases}
			\label{eq:del-Pezzo-dynamic-alpha-degree-1-special}
		\end{equation}
	\item[(ii)] In any other case
		\begin{equation}
				\omega_1=\alpha(S,(1-\beta)C)=1.
			\label{eq:del-Pezzo-dynamic-alpha-degree-1-general}
		\end{equation}
\end{itemize}
\end{thm}
\begin{proof}
If $\vert-K_S\vert$ has no element with a cuspidal point, then $\lct(S, (1-\beta)C, \beta T)=1$ for all $T\in \vert-K_S\vert$, since $(S, T)$ is log canonical by Lemma \ref{lem:del-Pezzo-cat-list-degree-1}. Since $S$ satisfies the Cat Property (see Lemma \ref{lem:del-Pezzo-cat-degree-1}), Observation \ref{obs:cat-property-easy-glct} implies:
$$\alpha(S, (1-\beta)C )=\lct(S, (1-\beta)C, \beta T)=\omega_1.$$

If $\exists T\in \vert-K_S\vert$ with a cuspidal singular point, then
$$1=C\cdot T \geq \mult_pC \cdot \mult_p T =\mult_pT$$
for any $p\in C\cap T$. Therefore $T$ and $C$ intersect with simple normal crossings. By Lemma \ref{lem:appendix-local-cusp-snc} $\lct(S, (1-\beta)C, \beta T)=\frac{5}{6\beta}$. Since $S$ satisfies the Cat Property (see Lemma \ref{lem:del-Pezzo-cat-degree-1}), Observation \ref{obs:cat-property-easy-glct} and Lemma \ref{lem:del-Pezzo-cat-list-degree-1} imply
$$\alpha(S, (1-\beta)C )=\min\{1,\lct(S, (1-\beta)C, \beta T)\}=\omega_2.$$
\end{proof}

\appendix
\chapter{Local computations: log canonical thresholds of pairs.}
\label{app:lcts}
Many of the proofs in this thesis construct very specific effective divisors and $\bbQ$-divisors in order to apply Convexity (lemmas \ref{lem:convexity} and \ref{lem:log-convexity}). In order to apply those lemmas, the pair surface--divisor needs to be log canonical. The purpose of this Appendix is to show that all specific pairs constructed are log canonical. Unless otherwise stated, given a pair $(S, D)$ where $S$ is a surface and $D$ a $\bbQ$-divisor, we denote its minimal log resolution by $f\colon \widetilde S\ra S$ and the strict transform of $D$ in $\widetilde S$ by $\widetilde D$. If $f$ decomposes in $N$ blow-ups, the exceptional divisors of $f$ will be denoted by $F_1,\ldots,F_N$. Usually $F_1,\ldots,F_N$ form a chain, i.e. $F_i\cdot F_j=1$ if and only if $\vert i-j\vert=1$, and $F_i$ will be the exceptional curve obtained after blowing-up a point in $F_{i-1}$. We will not explicitly justify the precise number of blow-ups required to obtain the minimal log resolution, since we consider this should be obvious from the intersection matrix of the irreducible prime divisors in $\Supp(D)$ and it would just clutter the argument. The thorough reader is encouraged to draw a picture for each case. In some occasions the singularities of the pair considered are not defined precisely. In that case we will show that, in the presence of the  worst possible discrepancy, the pair considered is log canonical.

\section{Hirzebruch surface $\bbF_1$}
\begin{lem}
\label{lem:appendix-F1-pair1}
Let $S\cong\bbF_1$ be the blow-up of $\bbP^2$ at one point with exceptional curve $E\subset S$. Let $p\in S$ with $C\in \vert-K_S\vert=\vert\pioplane{3}-E\vert$ smooth and $p\in C$. Let $H\in \vert\pioplane{1}\vert$ such that $H$ is irreducible. Suppose $\exists L_p\in \vert\pioplane{1}-E\vert$ with $p\in L_p$ an irreducible curve such that $(L_p\cdot C)\vert_p=1$. Then the pair
\begin{equation}
(S, (1-\beta)C +\omega_2\beta(2H+L_p))
\label{eq:appendix-del-Pezzo-dynamic-alpha-degree-8-F1-proof-aux1}
\end{equation}
is log canonical at $p$ where
$$\omega_2:=
					\begin{dcases}
							1 												&\text{ for } 0<\beta\leq \frac{1}{4},\\
							\frac{1+\beta}{5\beta}		&\text{ for } \frac{1}{4}\leq \beta\leq \frac{2}{3},\\
							\frac{1}{3\beta}					&\text{ for } \frac{2}{3}\leq \beta\leq 1.\\
					\end{dcases}
$$
\end{lem}
\begin{proof}
Indeed the pair \eqref{eq:appendix-del-Pezzo-dynamic-alpha-degree-8-F1-proof-aux1} is log canonical in codimension $1$, since $\omega_2\beta\leq \frac{1}{3}$. The worst discrepancy takes place when $(H\cdot C)\vert_p=3$, so we assume this. The minimal log resolution $f\colon \hat{S} \ra S$ of \eqref{eq:appendix-del-Pezzo-dynamic-alpha-degree-8-F1-proof-aux1} consists of three blow-ups over $p$ with exceptional divisors $F_1,F_2, F_3$ and log pullback
\begin{align*}
f^*(K_S + (1-\beta)C +\omega_2\beta(2H+L_p))&\simq K_{\hat S} +(1-\beta) \hat{S}+\omega_2\beta(2\hat H+\hat L_p)\\
&+ (3\omega_2\beta-\beta)F_1 + (5\omega_2\beta-2\beta)F_2 + (7\omega_2\beta-3\beta)F_3.
\end{align*}
If $0<\beta\leq \frac{1}{4}$, then $3\omega_2\beta-\beta=2\beta\leq 1$, $5\omega_2\beta-2\beta=3\beta\leq 1$ and $7\omega_2\beta-3\beta=4\beta\leq 1$.

If $\frac{1}{4}\leq \beta\leq \frac{2}{3}$, then $3\omega_2\beta-\beta=\frac{3(1+\beta)}{5}-\beta=\frac{3}{5}-\frac{2}{5}\beta\leq 1$, $5\omega_2\beta-2\beta=(1+\beta)-2\beta\leq 1$ and $7\omega_2\beta-2\beta=\frac{7(1+\beta)}{5}-2\beta=\frac{7}{5}-\frac{8}{5}\beta\leq 1$.

Finally, if $\frac{2}{3}\leq \beta\leq 1$, then $3\omega_2\beta-\beta=1-\beta\leq 1$ and $5\omega_2\beta-2\beta=\frac{5}{3}-2\beta\leq 1$ and $7\omega_2\beta-3\beta=\frac{7}{3}-3\beta\leq 1.$
Therefore \eqref{eq:appendix-del-Pezzo-dynamic-alpha-degree-8-F1-proof-aux1} is log canonical. 
\end{proof}

\section{Del Pezzo surface of degree $7$}
Recall Notation \ref{nota:del-Pezzo-deg7}:
\begin{nota0}
Let $S$ be a del Pezzo surface of degree $7$. By Lemma \ref{lem:del-Pezzo-no-moduli}, the surface $S$ is unique up to isomorphism. By Lemma \ref{lem:del-Pezzo-uniquemodel} there is a unique morphism $\pi:S\ra \bbP^2$, up to isomorphism, that contracts two $(-1)$-curves $E_1,E_2$ to points $p_1,p_2$ in $\bbP^2$. By Lemma \ref{lem:del-Pezzo-lines9-d} there is a unique line $L\neq E_1,E_2$ with
$$L\sim\pioplane{1}-E_1-E_2,$$
corresponding to the strict transform of the unique line in $\bbP^2$ passing through $p_1$ and $p_2$.\ We have $L\cdot E_1=L\cdot E_2=1$. Let $C$ be a smooth curve, $C\in\vert-K_S\vert$. The curve $C$ intersects each of $E_1,E_2$ and $L$ at precisely one point. At most two of these points coincide.

Let $L_i$ be the unique curve
$$L_i\sim\pioplane{1}-E_i,$$
containing $r_i=E_i\cap C$, for $i=1,2$, which is precisely the strict transform of the unique line $\bar{L_i}$ in $\bbP^2$ tangent to $\bar{C}=\pi_*(C)$ at $p_i$. Let $r=L\cap C$ and such that $R$ be the unique curve passing through $r$ such that 
$$R\sim \pioplane{1}$$
and $R$ is tangent to $C$ at $r$.
\end{nota0}

\begin{lem}
\label{lem:appendix-delPezzo-deg7-dynamic-lct1}
Let $q\in C \subset S$ be a point which does not belong to any $(-1)$-curve. Let $H_i$, where $i=1,2$, be the strict transform of the line in $\bbP^2$ through $\pi(q)$ and $p_i$. Let $H$ be the strict transform of a line in $\bbP^2$ such that $H$ is tangent to $C$ at $q$. The pair
$$(S,(1-\beta) C + \omega_1\beta (H_1+H_2+H) ) $$
is log canonical where
\begin{equation*}
			\omega_1=
					\begin{dcases}
							1 												&\text{ for } 0<\beta\leq \frac{1}{3},\\
							\frac{1}{3\beta}					&\text{ for } \frac{1}{3}\leq \beta\leq 1.\\
					\end{dcases}
		\end{equation*} 
\end{lem}
\begin{proof}
Notice that $H_1\cdot H_2=H\cdot H_1=1$, $C\cdot H_1=C\cdot H_2=2$ and $3\geq (H\cdot C)\vert_q\geq 2$. Therefore, in terms of computing discrepancies, the worst situation arises when $(C\cdot H_1)\vert_q=1, (C\cdot H_2)\vert_q=1$ and $(C\cdot H)\vert_q=3$. The minimal log resolution $f\colon \widetilde S \ra S$ of the pair $(S,(1-\beta) C + \omega_1\beta (H_1+H_2+H) ) $ consists of $3$ blow-ups over $q$ with exceptional curves $F_1,F_2,F_3$. The log pullback is
\begin{align*}
f^*(K_S+(1-\beta)C &+ \omega_1\beta(H_1+H_2+H))\simq K_{\widetilde S}+(1-\beta)\widetilde C + \omega_1\beta(\widetilde H_1+\widetilde H_2+\widetilde H)\\
&+(3\omega_1\beta-\beta)F_1+(4\omega_1\beta-2\beta)F_2+(5\omega_1\beta-3\beta)F_3.
\end{align*} 
Observe that $3\omega_1\beta-\beta\leq 1-\beta<1$. If $0<\beta\leq \frac{1}{3}$, then $4\omega_1\beta-2\beta=2\beta\leq \frac{2}{3}<1$ and $5\omega_1\beta-3\beta=2\beta\leq \frac{2}{3}<1$. Finally if $\frac{1}{3}\leq\beta\leq 1$, then $4\omega_1\beta-2\beta=\frac{4}{3}-2\beta\leq \frac{2}{3}<1$ and $5\omega_1\beta-3\beta\leq\frac{5}{3}-3\beta\leq \frac{2}{3}<1$. We conclude that $(S,(1-\beta) C + \omega_1\beta (H_1+H_2+H) ) $ is log canonical.

\end{proof}

\begin{lem}
\label{lem:appendix-delPezzo-deg7-dynamic-lct2}
Let $q\in C \subset S$ be a point which does not belong to any $(-1)$-curve. Let $H_i$, where $i=1,2$, be the strict transform of the line in $\bbP^2$ through $\pi(q)$ and $p_i$. Assume $H_1$ is tangent to $C$ at $q$. The pair
$$(S,(1-\beta) C + \omega_1\beta(2H_1+H_2+E_1) ) $$
is log canonical where
\begin{equation*}
			\omega_1=
					\begin{dcases}
							1 												&\text{ for } 0<\beta\leq \frac{1}{3},\\
							\frac{1}{3\beta}					&\text{ for } \frac{1}{3}\leq \beta\leq 1.\\
					\end{dcases}
		\end{equation*}
\end{lem}
\begin{proof}

Notice that $H_1\cdot H_2=E_1\cdot H_1=1$, $H_2\cdot E_1=0$, $q\not\in E_1$, $E_1\cdot C=1$ and $C\cdot H_1=C\cdot H_2=2$. Therefore, in terms of computing discrepancies, the worst situation arises when $(C\cdot H_1)\vert_q=2$ and $(C\cdot H_2)\vert_q=1$. The minimal log resolution $f\colon \widetilde S \ra S$ of the pair $(S,(1-\beta) C + \omega_1\beta (2H_1+H_2+E_1) ) $ consists of $2$ blow-ups over $q$ with exceptional curves $F_1,F_2$. The log pullback is
\begin{align*}
f^*(K_S+(1-\beta)C &+ \omega_1\beta(2H_1+H_2+E_1))\simq K_{\widetilde S}+(1-\beta)\widetilde C + \omega_1\beta(2\widetilde H_1+\widetilde H_2+\widetilde E_1)\\
&+(3\omega_1\beta-\beta)F_1+(5\omega_1\beta-2\beta)F_2.
\end{align*} 
Observe that $3\omega_1\beta-\beta\leq 1-\beta<1$. If $0<\beta\leq \frac{1}{3}$, then $5\omega_1\beta-2\beta=3\beta\leq 1$. If $\frac{1}{3}\leq\beta\leq 1$, then $5\omega_1\beta-2\beta=\frac{5}{3}-2\beta\leq 1$. We conclude that $(S,(1-\beta) C + \omega_1\beta (2H_1+H_2+E_1) ) $ is log canonical.

\end{proof}

\begin{lem}
\label{lem:appendix-delPezzo-deg7-dynamic-lct3}
Suppose $q\in C \subset S$ is a point which does not belong to any $(-1)$-curve. Let $H\in \vert \pioplane{1}\vert$ be the unique curve such that $q\in H$ and $q_1\in H^1$, where $q_1\in F_1\subset S_1$, the blow-up of $q$ and $H^1$ is the strict transform of $H$. Suppose $H$ is irreducible. Then the pair
$$(S, (1-\beta)C + \omega_1\beta(2H+L))$$
is log canonical where
\begin{equation*}
			\omega_1=
					\begin{dcases}
							1 												&\text{ for } 0<\beta\leq \frac{1}{3},\\
							\frac{1}{3\beta}					&\text{ for } \frac{1}{3}\leq \beta\leq 1.\\
					\end{dcases}
		\end{equation*} 
\end{lem}
\begin{proof}
Notice that $H\cdot L=L\cdot C=1$ and $3\geq (H\cdot C)\vert_q\geq 2$. Therefore, in terms of computing discrepancies, the worst situation arises when $(C\cdot H)\vert_q=3$. The minimal log resolution $f\colon \widetilde S \ra S$ of the pair $(S,(1-\beta) C + \omega_1\beta (2H+L) ) $ consists of $3$ blow-ups over $q$ with exceptional curves $F_1,F_2,F_3$. The log pullback is
\begin{align*}
f^*(K_S+(1-\beta)C &+ \omega_1\beta(2H+L))\simq K_{\widetilde S}+(1-\beta)\widetilde C + \omega_1\beta(2\widetilde H+\widetilde L)\\
&+(2\omega_1\beta-\beta)F_1+(4\omega_1\beta-2\beta)F_2+(6\omega_1\beta-3\beta)F_3.
\end{align*} 
Observe that $2\omega_1\beta-\beta\leq 1-\beta<1$. If $0<\beta\leq \frac{1}{3}$, then $4\omega_1\beta-2\beta=2\beta\leq \frac{2}{3}<1$ and $6\omega_1\beta-3\beta=3\beta\leq 1$. Finally if $\frac{1}{3}\leq\beta\leq 1$, then $4\omega_1\beta-2\beta=\frac{4}{3}-2\beta\leq \frac{2}{3}<1$ and $6\omega_1\beta-3\beta\leq2-3\beta\leq 1$. We conclude that $(S,(1-\beta) C + \omega_1\beta (H_1+H_2+H) ) $ is log canonical.
\end{proof}

\begin{lem}
\label{lem:appendix-delPezzo-deg7-dynamic-lct5}
If $(R\cdot C)\vert_{r}=3$, then the pair
$$(S, (1-\beta)C + \omega_2\beta (L+2R))$$
is log canonical where 
		\begin{equation*}
			\omega_2=
					\begin{dcases}
							1 												&\text{ for } 0<\beta\leq \frac{1}{4},\\
							\frac{1+3\beta}{7\beta}		&\text{ for } \frac{1}{4}\leq \beta\leq \frac{4}{9},\\
							\frac{1}{3\beta}					&\text{ for } \frac{4}{9}\leq \beta\leq 1.\\
					\end{dcases}
					\end{equation*}

If $(R\cdot C)\vert_{r}=2$, then the pair
$$(S, (1-\beta)C + \omega_1\beta (L+2R))$$
is log canonical where 
		\begin{equation*}
			\omega_1=
					\begin{dcases}
							1 												&\text{ for } 0<\beta\leq \frac{1}{3},\\
							\frac{1}{3\beta}					&\text{ for } \frac{1}{3}\leq \beta\leq 1.\\
					\end{dcases}
		\end{equation*}

\end{lem}
\begin{proof}
Let $f\colon \widetilde S \ra S$ be the minimal log resolution of the pair $(S, (1-\beta)C + \omega_i\beta(L+2R))$. The exceptional locus of $f$ is a chain of exceptional curves with centre $r$.

If $(C\cdot R)\vert_r=2$, then $\Exc(f)=F_1\cup F_2$ and the log pullback is
\begin{align*}
f^*(K_S + (1-\beta)C &+ \omega_1\beta(L+2R)) \simq K_{\widetilde S} + (1-\beta)\widetilde C + \omega_2\beta (\widetilde L + 2\widetilde R)\\
&+ (3\omega_1\beta-\beta)F_1 + (5\omega_1\beta-2\beta)F_2.
\end{align*}

If $0<\beta\leq \frac{1}{3}$, then $3\omega_1\beta-\beta=2\beta\leq 1$ and $5\omega_1\beta-2\beta = 3\beta\leq 1$. If $\frac{1}{3}<\beta\leq 1$, then $3\omega_1\beta=1-\beta\leq 1$ and $5\omega_1\beta-2\beta\leq \frac{5}{3}-2\beta \leq 1$.

Therefore $(S, (1-\beta)C + \omega_1\beta (L+2R))$ when $(C\cdot R)\vert_r=2$.
If $(C\cdot R)\vert_r=3$, then $\Exc(f)=F_1\cup F_2\cup F_3$ and the log pullback is
\begin{align*}
f^*(K_S + (1-\beta)C + &\omega_2\beta(L+2R)) \simq K_{\widetilde S} + (1-\beta)\widetilde C + \omega_2\beta (\widetilde L + 2\widetilde R) \\
			&+ (3\omega_2\beta-\beta)F_1 + (5\omega_2\beta-2\beta)F_2 + (7\omega_2\beta-3\beta)F_3.
\end{align*}

If $0<\beta\leq \frac{1}{4}$, then $3\omega_2\beta-\beta=2\beta\leq 1$, $5\omega_2\beta-2\beta=3\beta\leq 1$ and $7\omega_2\beta-3\beta=4\beta\leq 1$. If $\frac{1}{4}\leq\beta\leq\frac{4}{9}$ then $3\omega_2\beta-\beta=\frac{3}{7}+ \frac{2}{7}\beta\leq 1$, $5\omega_2\beta-2\beta=\frac{5}{7}+\frac{\beta}{7}\leq \frac{45+4}{63}\leq 1$ and $7\omega_2\beta-3\beta=1+3\beta-3\beta=1$. If $\frac{4}{9}\leq\beta\leq 1$, then $3\omega_2\beta-\beta =1-\beta\leq 1$, $5\omega_2\beta-2\beta=\frac{5}{3}-2\beta\leq \frac{7}{9}\leq 1$ and $7\omega_2\beta-3\beta\leq \frac{7}{3}-3\beta\leq 1$.

Therefore $(S, (1-\beta)C + \omega_2\beta (L+2R))$ when $(C\cdot R)\vert_r=3$.

\end{proof}

\begin{lem}
\label{lem:appendix-delPezzo-deg7-dynamic-lct6}
If $(L_1\cdot C)\vert_{r_1}=2$, then the pair
$$(S, (1-\beta)C + \omega_3\beta(2L_1+L+2E_1))$$
is log canonical where
		\begin{equation*}
			\omega_3=
					\begin{dcases}
							1 												&\text{ for } 0<\beta\leq \frac{1}{4},\\
							\frac{1+2\beta}{6\beta}		&\text{ for } \frac{1}{4}\leq \beta\leq \frac{1}{2},\\
							\frac{1}{3\beta}					&\text{ for } \frac{1}{2}\leq \beta\leq 1.\\
					\end{dcases}
	\end{equation*}

If $(L_1\cdot C)\vert_{r_1}=1$, then the pair
$$(S, (1-\beta)C + \omega_1\beta(2L_1+L+2E_1))$$
is log canonical where 
		\begin{equation*}
			\omega_1=
					\begin{dcases}
							1 												&\text{ for } 0<\beta\leq \frac{1}{3},\\
							\frac{1}{3\beta}					&\text{ for } \frac{1}{3}\leq \beta\leq 1.\\
					\end{dcases}
		\end{equation*} 
\end{lem}
\begin{proof}
Let $f\colon \widetilde S \ra S$ be the minimal log resolution of the pair
$$(S, (1-\beta)C + \omega_i\beta(2L_1+L+2E_1)).$$
The exceptional locus of $f$ is a chain of exceptional curves with centre $r_1$.

If $(L_1\cdot C)\vert_{r_1}=2$, then $\Exc(f)=F_1\cup F_2$ and the log pullback is
\begin{align*}
f^*(K_S + (1-\beta)C &+ \omega_3\beta(2L_1+L+2E_1)) \simq K_{\widetilde S} + (1-\beta)\widetilde C + \omega_3\beta (2\widetilde L_1+\widetilde L+2\widetilde E_1)\\
& + (4\omega_3\beta-\beta)F_1 + (6\omega_3\beta-2\beta)F_2.
\end{align*}
If $(L_1\cdot C)\vert_{r_1}=1$, then $\Exc(f)=F_1$ and the log pullback is
\begin{align*}
f^*(K_S + (1-\beta)C &+ \omega_1\beta(2L_1+L+2E_1)) \simq K_{\widetilde S} + (1-\beta)\widetilde C \\
&+ \omega_1\beta (2\widetilde L_1+\widetilde L+2\widetilde E_1) + (4\omega_1\beta-\beta)F_1.
\end{align*}

If $0<\beta\leq \frac{1}{3}$, then $4\omega_1\beta-\beta=3\beta\leq 1$. If $\frac{1}{3}\leq\beta\leq 1$, then $4\omega_1\beta-\beta=\frac{4}{3}-\beta\leq 1$.

If $0<\beta\leq \frac{1}{4}$, then $4\omega_3\beta-\beta=3\beta\leq 1$ and $6\omega_3\beta-2\beta=4\beta\leq 1$. If $\frac{1}{4}\leq\beta\leq\frac{1}{2}$, then $4\omega_3\beta-\beta=\frac{2+4\beta}{3}-\beta=\frac{2}{3}+\frac{1}{3}\beta\leq 1$ and $6\omega_3\beta-2\beta=1+2\beta-2\beta=1$. If $\frac{1}{2}\leq\beta\leq1$, then $4\omega_3\beta-\beta=\frac{4}{3}-\beta \leq 1$ and $6\omega_3\beta-2\beta=2-2\beta\leq 1$.

Therefore $(S, (1-\beta)C + \omega_1\beta(2L_1+L+2E_1))$ is log canonical when $(L_1\cdot C)\vert_{r_1}=1$ and $(S, (1-\beta)C + \omega_3\beta(2L_1+L+2E_1))$ is log canonical when $(L_1\cdot C)\vert_{r_1}=2$.

\end{proof}

\section{Del Pezzo surface of degree $6$}
We recall the notation from section \ref{sec:dynamic-alpha-degree6}. Let $S$ be a non-singular del Pezzo surface of degree $6$. Given any model $\pi\colon S \ra \bbP^2$ we have exceptional curves $E_1,E_2,E_3\subset S$ mapping to points $p_1,p_2,p_3\in \bbP^2$, respectively. The other $3$ lines in $S$ (see Lemma \ref{lem:del-Pezzo-lines9-d}) correspond to strict transforms of lines in $\bbP^2$ through $p_i,p_j$. We will denote them by
$$L_{ij}\sim\pioplane{1}-E_i-E_j\text{ for } 1\leq i <j\leq 3.$$
For each $C$, we define the following functions with variable $\beta \in (0,1]$:
\begin{itemize}
	\item[(i)] If $C$ contains a pseudo-Eckardt point of $S$, then
	\begin{equation}
			\omega_3:=
					\begin{dcases}
							1 												&\text{ for } 0<\beta\leq \frac{1}{3},\\
							\frac{1+\beta}{4\beta}		&\text{ for } \frac{1}{3}\leq \beta\leq 1.
					\end{dcases}
		\label{eq:app-del-Pezzo-dynamic-alpha-degree-6-special}
	\end{equation}
	\item[(ii)]If $C$ contains no pseudo-Eckardt points but there is a model $\pi\colon S\ra \bbP^2$ such that through $p=C\cap E_1$ there is a smooth rational curve $L\sim\pioplane{1} -E_1$ satisfying $(C\cdot L )\vert_{p}=2$, then
				\begin{equation}
						\omega_2:=
								\begin{dcases}
										1 												&\text{ for } 0<\beta\leq \frac{1}{3},\\
										\frac{1+2\beta}{5\beta}					&\text{ for } \frac{1}{3}\leq \beta\leq \frac{3}{4},\\
										\frac{1}{2\beta}					&\text{ for } \frac{3}{4}\leq \beta\leq 1.
								\end{dcases}
					\label{eq:app-del-Pezzo-dynamic-alpha-degree-6-semispecial}
				\end{equation}
	
	\item[(iii)]If $C$ contains no pseudo-Eckardt points and for all models $\pi\colon S \ra \bbP^2$ the unique irreducible curve $L\in \vert\pioplane{1} -E_1\vert$ passing through $p=C\cap E_1$ has simple normal crossings with $C$ (i.e. $(C\cdot L)\vert_{p}=E_1$), then
				\begin{equation}
						\omega_1:=
								\begin{dcases}
										1 												&\text{ for } 0<\beta\leq \frac{1}{2},\\
										\frac{1}{2\beta}					&\text{ for } \frac{1}{2}\leq \beta\leq 1.
								\end{dcases}
					\label{eq:app-del-Pezzo-dynamic-alpha-degree-6-generic}
				\end{equation}
\end{itemize}	

\begin{lem}
\label{lem:appendix-delPezzo-deg6-dynamic-lct0}
Let $q_0\in C$ be a point not lying in any $(-1)$-curve of $S$. Let $L_2$ be the unique curve such that
\begin{equation*}
L_2\in \vert\pioplane{1}-E_2\vert \text{ with }q_0\in L_2
\end{equation*}
The pair
\begin{equation}
(S, (1-\beta)C + \omega_1\beta(L_{13}+2L_2+E_2))
\label{eq:appendix-delPezzo-deg6-dynamic-lct0}
\end{equation}
is log canonical. 
\end{lem}
\begin{proof}
Since $\omega_1\beta\leq \frac{1}{2}$, the pair \eqref{eq:appendix-delPezzo-deg6-dynamic-lct0} is log canonical in codimension $1$. Since $L_2\cdot L_{13}=L_2\cdot E_2=1$ and $L_13\cdot E_2=0$, the worse discrepancy takes place if $(L_2\cdot C)\vert_{q_0}=2$. The the minimal log resolution of \eqref{eq:appendix-delPezzo-deg6-dynamic-lct0}, $f \colon \widetilde S \ra S$ consists of two blow-ups over $q_0$ and the log pullback is
\begin{align*}
f^*(K_S+(1-\beta)C + \omega_1\beta(L_{13}+2L_2+E_2))&\simq K_{\widetilde S} + (1-\beta)\widetilde C + \omega_1\beta(\widetilde L_{13} + 2\widetilde L_2 + \widetilde E_2)\\
&+ (2\omega_1\beta-\beta)F_1 + (3\omega_1\beta-2\beta)F_2.
\end{align*}
Clearly $2\omega_1\beta-\beta\leq 2\omega_1\beta\leq 1$ and $3\omega_1\beta-\beta\leq 1$.
Indeed if $0<\beta\leq \frac{1}{2}$, then $3\omega_1\beta-2\beta=\beta\leq 1$ and if $\frac{1}{2}\leq\beta\leq 1$, then $3\omega_1\beta-2\beta\leq\frac{3}{2}-2\beta\leq 1$. Therefore \eqref{eq:appendix-delPezzo-deg6-dynamic-lct0} is log canonical. 
\end{proof}

\begin{lem}
\label{lem:appendix-delPezzo-deg6-dynamic-lct1}
Let $q_0\in C$ be a point not lying in any $(-1)$-curve of $S$. Let $L_i$ be the unique curve such that
\begin{equation}
L_i\in \vert\pioplane{1}-E_i\vert \text{ with }q_0\in L_i\text{ for } i=1,2,3.
\label{eq:appendix-delPezzo-deg6-dynamic-lct1-pointgeneric-cubics}
\end{equation}
The pair
\begin{equation}
(S,(1-\beta)C + \omega_1\beta (L_1+L_2+L_3))
\label{eq:appendix-delPezzo-deg6-dynamic-lct1}
\end{equation}
is log canonical. 
\end{lem}
\begin{proof}
Observe that $L_i\cdot L_j=1$ for $i\neq j$. Since $C\cdot L_i=2$ for all $i$, $C$ can be at most tangent to one $L_i$. Without loss of generality suppose that $(C\cdot L_1)\vert_{q_0}=2$. The minimal log resolution $f\colon \widetilde S \ra S$ of \eqref{eq:appendix-delPezzo-deg6-dynamic-lct1} consists of two blow-ups over $q_0$. Let $B=L_1+L_2+L_3$. The log pullback is
\begin{align*}
f^*(K_S+(1-\beta)C+\omega_1\beta B)&\simq-K_{\widetilde S} + (1-\beta)\widetilde C +\omega_1\beta\widetilde B \\
&+ (3\omega_1\beta-\beta)F_1 + (4\omega_1\beta-2\beta)F_2.
\end{align*}
If $0<\beta\leq \frac{1}{2}$, then $3\omega_1\beta-\beta=2\beta\leq 1$ and $4\omega_1\beta-2\beta=2\beta\leq 1$.

If $\frac{1}{2}\leq \beta\leq 1$, then $3\omega_1\beta-\beta=\frac{3}{2}-\beta\leq 1$ and $4\omega_1\beta-2\beta\leq 2 -2\beta\leq 1$.

Therefore \eqref{eq:appendix-delPezzo-deg6-dynamic-lct1} is log canonical.
\end{proof}

\begin{lem}
\label{lem:appendix-delPezzo-deg6-dynamic-lct2}
Let $q_0\in C$ be a point not lying in any $(-1)$-curve of $S$. Let $H\sim \pioplane{1}$ and $G\sim\pioplane{2}-E_1-E_2-E_3$ be irreducible curves passing through $q_0$.
The pair
\begin{equation}
(S,(1-\beta)C + \omega_1\beta (G+H))
\label{eq:appendix-delPezzo-deg6-dynamic-lct2}
\end{equation}
is log canonical. 
\end{lem}
\begin{proof}
Since $G\cdot H=2$ and $G\cdot C = H\cdot C =3$, the worst discrepancy for pair \eqref{eq:appendix-delPezzo-deg6-dynamic-lct2} happens when $(G\cdot H )\vert_{q_0}=(G\cdot C )\vert_{q_0}=2$ and $(H\cdot C )\vert_{q_0}=3$. The minimal log resolution $f\colon \widetilde S \ra S$ of \eqref{eq:appendix-delPezzo-deg6-dynamic-lct2} consists of $3$ blow-ups over $q_0$. The log pullback is
\begin{align*}
f^*(K_S+(1-\beta)C+\omega_1\beta (G+H))\simq&-K_{\widetilde S} + (1-\beta)\widetilde C +\omega_1\beta(\widetilde G+\widetilde H)+ (2\omega_1\beta-\beta)F_1\\
&+ (4\omega_1\beta-2\beta)F_2+(5\omega_1\beta-3\beta)F_3.
\end{align*}
If $0<\beta\leq \frac{1}{2}$, then $2\omega_1\beta-\beta=\beta\leq \frac{1}{2}\leq 1$, $4\omega_1\beta-2\beta=2\beta\leq 1$ and $5\omega_1\beta-3\beta=2\beta\leq 1$.

If $\frac{1}{2}\leq \beta\leq 1$, then $2\omega_1\beta-\beta=1-\beta\leq 1$, $4\omega_1\beta-2\beta\leq 2 -2\beta\leq 1$ and $5\omega_1\beta-3\beta=\frac{5}{2}-3\beta\leq 1$.

Therefore \eqref{eq:appendix-delPezzo-deg6-dynamic-lct2} is log canonical.

\end{proof}

\begin{lem}
\label{lem:appendix-delPezzo-deg6-dynamic-lct3}
Let $q_0=C\cap E_1$ and such that $q_0$ does not belong to any other $(-1)$-curve of $S$. Let $L_q\sim \pioplane{1}-E_1$ and $C_q\sim\pioplane{2}-E_1-E_2-E_3$ be irreducible curves passing through $q_0$. Let $B=E_1+L_q+C_q$. The pair
\begin{equation}
(S,(1-\beta)C + \omega_1\beta B)
\label{eq:app-del-Pezzo-dynamic-alpha-degree-6-auxpair3}
\end{equation}
is log canonical.
\end{lem}
\begin{proof}
Observe that $E_1\cdot L_q=C_q\cdot L_q=E_1\cdot C_q=C\cdot E_1=1$.

If $(L_q\cdot C)\vert_{q_0}=1$, let $\sigma\colon \widetilde S, \ra S$ be the minimal log resolution of the pair \eqref{eq:app-del-Pezzo-dynamic-alpha-degree-6-auxpair3}. The worst multiplicities arise when $(C_q\cdot C)\vert_{q_0}=2$. We consider just this case. Then $\sigma$ consists of $2$ blow-ups with exceptional divisors $F_1$ and $F_2$. The log pullback is
$$\sigma^*(K_S+ (1-\beta)C + \omega_1\beta B) = K_{\widetilde S} + (1-\beta) \widetilde C + \omega_1\beta \widetilde B + (3\omega_1\beta-\beta)F_1 + (4\omega_1\beta-2\beta)F_2.$$
If $0<\beta\leq\frac{1}{2}$, then $3\omega_1\beta-\beta=2\beta\leq 1$ and $4\omega_1\beta-2\beta=2\beta\leq 1$. If $\frac{1}{2}\leq\beta \leq 1$, then $3\omega_1\beta-\beta=\frac{3}{2}-\beta\leq 1$ and $4\omega_1\beta-2\beta=2-2\beta\leq 1$.

If $(L_q\cdot C)\vert_{q_0}=2$, let $\sigma\colon \widetilde S, \ra S$ be the minimal log resolution of the pair \eqref{eq:app-del-Pezzo-dynamic-alpha-degree-6-auxpair3}. Since $L_q$ is tangent to $C$ at $q_0$ but $L_q\cdot C_q=1$, then $(C_q\cdot C)\vert_{q_0}=1$. Therefore $\sigma$ consists of $2$ blow-ups with exceptional divisors $F_1$ and $F_2$. The log pullback is
$$\sigma^*(K_S+ (1-\beta)C + \omega_1\beta B) = K_{\widetilde S} + (1-\beta) \widetilde C + \omega_1\beta \widetilde B + (3\omega_1\beta-\beta)F_1 + (4\omega_1\beta-2\beta)F_2.$$
If $0<\beta\leq\frac{1}{2}$, then $3\omega_1\beta-\beta=2\beta\leq 1$ and $4\omega_1\beta-2\beta=2\beta\leq 1$. If $\frac{1}{2}\leq\beta \leq 1$, then $3\omega_1\beta-\beta=\frac{3}{2}-\beta\leq 1$ and $4\omega_1\beta-2\beta=2-2\beta\leq 1$.

\end{proof}

\begin{lem}
\label{lem:appendix-delPezzo-deg6-dynamic-lct4}
Let $L_q\sim \pioplane{1}-E_1$ be an irreducible curve passing through a point $q=C\cap E_1$, not lying in any other $(-1)$-curve. Suppose $(L_q\cdot C)\vert_q=2$. Let $B=2L_q+L_{23}+E_1$. The pair
\begin{equation}
(S, (1-\beta)C + \omega_2\beta B)
\label{eq:appendix-delPezzo-deg6-dynamic-lct4}
\end{equation}
is log canonical.
\end{lem}
\begin{proof}
Since $L_{23}\cdot L_q=L_q\cdot E_1=L_{23}\cdot C=E_1\cdot C=1$ and $L_{23}\cdot E_1=0$ then $\Supp(B)$ has simple normal crossings.

The minimal log resolution $f\colon \widetilde S \ra S$ of \eqref{eq:appendix-delPezzo-deg6-dynamic-lct5} consists of two blow-ups over $p$. The log pullback is
$$f^*(K_S+(1-\beta)C+\omega_1\beta B)\simq-K_{\widetilde S} + (1-\beta)\widetilde C +\omega_1\beta\widetilde B + (3\omega_1\beta-\beta)F_1+(5\omega_1\beta-2\beta)F_2.$$
If $0<\beta\leq\frac{1}{3}$, then $3\omega_1\beta-\beta=2\beta\leq 1$ and $5\omega_1\beta-2\beta=3\beta\leq 1$. If $\frac{1}{3}\leq\beta \leq \frac{3}{4}$, then $3\omega_1\beta-\beta=\frac{3}{5}+\frac{1}{5}\beta\leq \frac{3}{4}$ and $5\omega_1\beta-2\beta=1$. Finally, if $\frac{3}{4}\leq \beta\leq 1$, then $3\omega_1\beta-\beta=\frac{3}{2}-\beta=\frac{3}{4}<1$ and $5\omega_1\beta-2\beta=\frac{5}{2}-2\beta\leq 1$.

Therefore the pair \eqref{eq:appendix-delPezzo-deg6-dynamic-lct4} is log canonical.
\end{proof}
\begin{lem}
\label{lem:appendix-delPezzo-deg6-dynamic-lct5}
Let $L_q\sim \pioplane{1}-E_1$ be an irreducible curve passing through a point $q=C\cap E_1$, not lying in any other $(-1)$-curve. Let $B=L_q + L_{12}+L_{13}+2E_1$. The pair
\begin{equation}
(S, (1-\beta)C + \omega_1\beta B)
\label{eq:appendix-delPezzo-deg6-dynamic-lct5}
\end{equation}
is log canonical.
\end{lem}
\begin{proof}
Since $L_{1j}\cdot L_{1k}=0$ for $j\neq k$, $L_q\cdot L_{1j}=0$ for all $j$, $L_q\cdot E_1=1$ and $L_{1j}\cdot E_1=1$ for all $j$, then $\Supp(B)$ has simple normal crossings. 

Since $C\cdot L_{1j}=C\cdot E_1=1$ for all $j$ and $C\cdot L_q=2$, then $C$ intersects at most either $L_q$ and $E_1$ at the same point with $(C\cdot L_q)\vert_{q_0}\leq 2$ or $L_{1j}$ and $E_1$ at the same point. The worst discrepancy appears in the former case. Without loss of generality assume that $(C\cdot L_q)\vert_{q_0}=2$ and $C\cap L_q\cap E_1=q$. Then the minimal log resolution $f\colon \widetilde S \ra S$ of \eqref{eq:appendix-delPezzo-deg6-dynamic-lct5} consists of two blow-ups over $q$. The log pullback is
$$f^*(K_S+(1-\beta)C+\omega_1\beta B)\simq-K_{\widetilde S} + (1-\beta)\widetilde C +\omega_1\beta\widetilde B + (3\omega_1\beta-\beta)F_1+(4\omega_1\beta-2\beta)F_2.$$
If $0<\beta\leq\frac{1}{2}$, then $3\omega_1\beta-\beta=2\beta\leq 1$ and $4\omega_1\beta-2\beta=2\beta\leq 1$. If $\frac{1}{2}\leq\beta \leq 1$, then $3\omega_1\beta-\beta=\frac{3}{2}-\beta\leq 1$ and $4\omega_1\beta-2\beta=2-2\beta\leq 1$.

Since $\omega_1\beta\leq \frac{1}{2}$, then the pair \eqref{eq:appendix-delPezzo-deg6-dynamic-lct5} is log canonical.
\end{proof}

\section{Del Pezzo surface of degree $5$}
\begin{lem}
\label{lem:appendix-delPezzo-deg5-dynamic-lct}
Let $S$ be a non-singular del Pezzo surface, with $K_S^2=5$.
Let $A$ be an irreducible conic in
$$\calA = \vert \pioplane{2} -E_1-E_2-E_3-E-4\vert$$
with $q_0\in A$. Let $B_1\sim\pioplane{1}-E_1$, $C\sim-K_S$ be irreducible smooth curves with $(C\cdot B_1)\vert_{q_0}=2$. Assume $q_0\in E_1$. Let $\beta\in (0,1]$. The pair
$$(S, (1-\beta) C + \omega\beta (A + B_1+ E_1))$$
is log canonical where
	\begin{equation*}
				\omega=
							\begin{dcases}
									1 									&\text{ for } 0<\beta\leq \frac{1}{2},\\
									\frac{2}{3\beta}		&\text{ for } \frac{1}{2}\leq \beta\leq 1.
							\end{dcases}
		\end{equation*}
\end{lem}
\begin{proof}
Since
$$A\cdot E_1 = B_1\cdot E_1=A\cdot B_1=C\cdot E_1=1$$
and $(C\cdot B_1)\vert_{q_0}=2$, then $(C\cot A)\vert_{p}=1$. Therefore the minimal log resolution $\widetilde S \ra S$ consists of $2$ blow-ups over $q_0$ with exceptional divisors $F_1$, $F_2$. Let $D=A + B_1+ E_1$. The log pullback is
$$f^*(K_S+(1-\beta)C + \omega\beta D ) \simq K_{\widetilde S} + (1-\beta)\widetilde C + \omega\beta \widetilde D  + (3\omega\beta -\beta)F_1 + (4\omega\beta-2\beta)F_2.$$
If $0<\beta\leq \frac{1}{2}$, then $3\omega\beta-\beta=2\beta\leq 1$ and $4\omega\beta-2\beta=2\beta\leq 1$. If $\frac{1}{2}\leq \beta\leq 1$, then $3\omega\beta-\beta = \frac{3}{2}-\beta\leq 1$ and $4\omega\beta-2\beta=2-2\beta\leq 1$. Therefore the pair is log canonical.
\end{proof}

\section{Del Pezzo surface of degree $4$}
\begin{lem}
\label{lem:appendix-delPezzo-deg4-dynamic-lct1}
Let $S$ be a smooth del Pezzo surface, with $K_S^2=4$. Let $L_1+L_2+Q\sim-K_S$ be two lines and an irreducible conic, respectively, intersecting at a point $p$. Let $C$ be a smooth hyperplane section not passing through $p$. Then
	\begin{equation*}
				\lct(S,(1-\beta)C, \beta(L_1+L_2+Q))=
							\begin{dcases}
									1 									&\text{ for } 0<\beta\leq \frac{2}{3},\\
									\frac{2}{3\beta}		&\text{ for } \frac{2}{3}\leq \beta\leq 1.
							\end{dcases}
		\end{equation*}
\end{lem}
\begin{proof}
Let $\sigma:S_1 \ra S$ be the blow-up of $p$ with exceptional divisor $E_1$. From the log pullback
$$\sigma^*(K_S+(1-\beta)C+\lambda\beta(L_1+L_2+Q))\simq K_{S_1}+(1-\beta)\widetilde C+\lambda\beta(\widetilde L_1+\widetilde L_2+\widetilde Q)+(3\lambda\beta-1)E,$$
we conclude that
\begin{equation}
(S,(1-\beta)C+\lambda\beta(L_1+L_2+Q))
\label{eq:appendix-delPezzo-deg4-dynamic-lct1}
\end{equation} is log canonical at $p$ if and only if $\lambda\leq\min\{1,\frac{2}{3\beta}\}$, given that $L_1\cdot L_2=C\cdot L_i=1$ gives us simple normal crossings in the blow-up. Note that $C\cdot L_i=1$ so the components of $\sigma^{-1}(\{L_1\cup L_2 \cap Q \cap C\})$ have simple normal crossings unless $C\cap Q=\{p'\}$, consists of one point only. Suppose that is the case, and let $\sigma': S_2\ra S$ be the repeated blow-up to resolve $(S, C+Q)$. Let $E_1, E_2$ be the exceptional divisors. Since $p'\not\in L_i$, then the pair \eqref{eq:appendix-delPezzo-deg4-dynamic-lct1}  is log canonical at $p'$ if and only if $(S, (1-\beta)C + \lambda\beta Q)$ is. We have
$$(\sigma')^*(K_S+(1-\beta)C+\lambda\beta(L_1+L_2+Q))\simq K_{S_2}+(1-\beta)\widetilde C+\lambda\beta\widetilde Q+(\beta(\lambda-1))E_1 + (2\beta(\lambda-1))E_2,$$
so $(S, (1-\beta)C + \lambda\beta Q)$ is log canonical for $\lambda\leq 1$ and any $\beta\in [0,1]$. Since
$$\lct(S,(1-\beta),D)=\min_{p\in S}\{ \lct_p(S,(1-\beta),D)\},$$
the result follows.
\end{proof}

\begin{lem}
\label{lem:appendix-delPezzo-deg4-dynamic-lct2}
Let $S$ be a smooth del Pezzo surface, with $K_S^2=4$. Let $T=A+B\sim-K_S$ be the union of two irreducible conics intersecting only at a point $p\in C$, where $C$ is a smooth hyperplane section. Then $\lct(S,(1-\beta)C, \beta(T))=\lambda$ where $\lambda=1$ for $(C\cdot A)\vert_p=1$ and
		\begin{equation*}
				\lambda=
						\begin{dcases}
									1 									&\text{ for } 0<\beta\leq \frac{1}{2},\\
									\frac{1+2\beta}{4\beta}		&\text{ for } \frac{1}{2}\leq \beta\leq 1.
						\end{dcases}
		\end{equation*}
for $(C\cdot A)\vert_p=(C\cdot B)\vert_p=2$.
\end{lem}
\begin{proof}
Let $D=(1-\beta)C+\lambda\beta(A+B)$. Since $A$ and $B$ are conics in local coordinates $x,y$ around $p$ we can give them local equations
$$A\colon\{y-x^2=0\}\qquad B\colon\{y+x^2=0\}$$
to satisfy the hypothesis in the statement. If $(C\cdot A)\vert_p=1$, then $(C\cdot B)\vert_{p}=1$, since $A$ and $B$ are tangent to each other. Therefore the local equation for $C$ can be taken to be
$$C\colon\{x=0\}$$
when $(C\cdot A)\vert_{p}=1$ or
$$C\colon\{y=0\}$$
when $(C\cdot A)\vert_{p}=2$.
In both cases the minimal log resolution $f\colon\widetilde S \ra S$ of $(S,D=(1-\beta)C+\lambda\beta(T))$ consists of $2$ blow-ups. In the first case the discrepancies are
$$f^*(K_S+D)=K_{\widetilde S} + \widetilde D + \beta(2\lambda-\beta)E_1+(\beta(4\lambda-\beta)-1)E_2.$$
In the second case the discrepancies are
$$f^*(K_S+D)=K_{\widetilde S} + \widetilde D + \beta(2\lambda-\beta)E_1+(\beta(4\lambda-2))E_2.$$
In each case $\disc(S,D)=-1$.
\end{proof}

\begin{lem}
\label{lem:appendix-delPezzo-deg4-dynamic-lct3}
Let $S$ be a smooth del Pezzo surface, with $K_S^2=4$. Let $L_1,L_2,Q$ be two lines and an irreducible conic, respectively, intersecting at a point $p$ and such that $L_1+L_2+Q\sim -K_S$. Let $C$ be a smooth hyperplane section such that $p\in C$. Then
	\begin{equation*}
				\lct(S,(1-\beta)C, \beta(L_1+L_2+Q))=\lambda=
							\begin{dcases}
								1 									&\text{ for } 0<\beta\leq \frac{1}{2},\\
								\frac{1+\beta}{3\beta}		&\text{ for } \frac{1}{2}\leq \beta\leq 1.
						\end{dcases}
		\end{equation*}
\end{lem}
\begin{proof}
Let $D=(1-\beta)C + \lambda\beta (L_1+L_2+Q)$. There are two different situations to consider. Either $(C\cdot Q)\vert_p=1$ or $(C\cdot Q)\vert_p=2$. However, the log canonical threshold is the same, since the worst discrepancy is reached after the first blow-up. Indeed, suppose we are in the second case. The minimal log resolution $f\colon\widetilde S\ra S$ of $(S,D)$ consists of $2$ consecutive blow-ups. The log pullback is
$$f^*(K_S+(1-\beta)C+\lambda\beta(L_1+L_2+Q))\simq K_{\widetilde S}+(1-\beta)\widetilde C+\lambda\beta(\widetilde L_1+\widetilde L_2+\widetilde Q)+\beta(3\lambda-1)E_1+\beta(4\lambda-2)E_2.$$
Note that $\disc(S,D)=\disc(S,D,E_1)=-1$.
\end{proof}

\begin{lem}
\label{lem:appendix-delPezzo-deg4-dynamic-lct4}
Let $S$ be a smooth del Pezzo surface, with $K_S^2=4$. Let $T\sim-K_S$ be an irreducible curve with a cusp at some point $p\in C$, where $C$ is a smooth hyperplane section. Then the pair $(S,(1-\beta)C+\lambda\beta T)$ is log canonical where 
	\begin{equation*}
				\lambda<
							\begin{dcases}
									1 									&\text{ for } 0<\beta\leq \frac{2}{3},\\
									\frac{2}{3\beta}		&\text{ for } \frac{2}{3}\leq \beta\leq 1.
							\end{dcases}
		\end{equation*}
\end{lem}
\begin{proof}
Let $D=(1-\beta)C+\lambda\beta T$. From Lemma \ref{lem:del-Pezzo-sections-anticanonical} $h^0(S,\calO_S(-K_S))=5$. Therefore any curve passing through $p\in S$ with local coordinates $x,y$ has a local equation
\begin{equation}
0=ax+by+cxy+dx^2+y^2+ \text{ higher order terms},
\label{eq:appendix-delPezzo-deg4-dynamic-lct4}
\end{equation}
where only $3$ parameters in $\{a,b,c,d\}$ are free. Indeed the linear system $\calC\subset\vert-K_S\vert$ of curves passing through $p$ has projective dimension $5-1-1=3$. The local equation for $T$ is $y^2+x^3=0$ corresponding to all values in \eqref{eq:appendix-delPezzo-deg4-dynamic-lct4} vanishing. Since $C$ is smooth its local equation in \eqref{eq:appendix-delPezzo-deg4-dynamic-lct4} has values $a\neq 0$ or $b\neq 0$. After blowing up $p$ the strict transforms of $T$ and $C$ do not intersect unless the local equation of $C$ around $p$ is $y=0$. This would give a lower discrepancy, hence we suppose this. Let $f:\widetilde S \ra S$ be the minimal log resolution of $(S,D)$, which consists of $3$ blow-ups. 
The discrepancies are:
$$f^*(K_S+D)=K_{\widetilde S} + (1-\beta)\widetilde C + \lambda \beta \widetilde T+\beta(2\lambda-1)E_1+\beta(3\lambda-2)E_2+(\beta(6\lambda-3)-1)E_3,$$
and $\disc(S,D)\leq -1$ for $\lambda$ as in the hypothesis.
\end{proof}

\begin{lem}
\label{lem:appendix-delPezzo-deg4-dynamic-lct5}
Let $S$ be a smooth del Pezzo surface, with $K_S^2=4$. Let $E_1$ be a line and $q_0\in E_1$ a point such that no line other than $E_1$ contains it. Let $C\in\vert-K_S\vert$ a smooth hyperplane section with $q_0\in C$. Let
$$G=\frac{1}{3}B_1+\frac{1}{3}\sum_{i=2}^5 A_i+\frac{2}{3}E_1\simq-K_S$$
for irreducible curves as in Table \ref{tab:delPezzo-4-lowdegree}, such that $q_0\in B_1\cap A_i$ for all $i=2,\ldots, 5$.
Then the pair $(S,(1-\beta)C+\lambda\beta G)$ is log canonical where 
	\begin{equation*}
				\lambda<
							\begin{dcases}
									1 									&\text{ for } 0<\beta\leq \frac{2}{3},\\
									\frac{2}{3\beta}		&\text{ for } \frac{2}{3}\leq \beta\leq 1.
							\end{dcases}
		\end{equation*}
\end{lem}
\begin{proof}
Notice that $A_i\cdot A_j=1$ for $i\neq j$. Moreover $A_i\cdot B_1=A_i\cdot E_1=1$ for $i=2,\ldots,5$ and $B_1\cdot E_1=1$. Since $A_i\cdot C = B_1\cdot C =2$, then $C$ can be tangent to at most one conic $A_i$ or $B_1$. Without loss of generality assume that $(C\cdot B_1)\vert_{q_0}=2$ and $(C\cdot A_i)\vert_{q_0}=1$ for $i=2,\ldots,5$. The minimal log resolution of the pair $(S, (1-\beta)C + \lambda \beta G)$ consists of two blow-ups over $q_0$. The log pullback is
$$f^*(K_S+(1-\beta)C+\lambda\beta G)\simq K_{\widetilde S} + (1-\beta)\widetilde C +\lambda\beta \widetilde G +(\frac{7}{3}\lambda\beta-\beta)F_1+ (\frac{8}{3}\lambda\beta -2\beta)F_2.$$
If $0<\beta\leq\frac{2}{3}$, then $\frac{7}{3}\lambda\beta-\beta<\frac{4}{3}\beta\leq \frac{8}{9}\leq 1$, and $\frac{8}{3}\lambda\beta-2\beta<\frac{2}{3}\beta\leq 1$.

If $\frac{2}{3}\leq\beta\leq1$, then $\frac{7}{3}\lambda\beta-\beta<\frac{14}{9}-\beta\leq \frac{8}{9}\leq 1$, and $\frac{8}{3}\lambda\beta-2\beta<\frac{16}{9}-2\beta\leq\frac{4}{9}\leq 1$.

Therefore $(S, (1-\beta)C +\lambda\beta G)$ is log canonical.
\end{proof}

\begin{lem}
\label{lem:appendix-delPezzo-deg4-dynamic-lct6}
Let $S$ be a smooth del Pezzo surface, with $K_S^2=4$. Let $E_1$ be a line and $q_0\in E_1$ a point such that no line other than $E_1$ contains it. Let $C\in\vert-K_S\vert$ a smooth hyperplane section with $q_0\in C$. Let $Q_1\sim \pioplane{3}-2E_1-\sum^5_{i=2} E_i$ be a smooth irreducible cubic such that $q_0\in Q_1$. 
Then the pair $(S,(1-\beta)C+\omega_1\beta (Q_1+E_1))$ is log canonical where 
	\begin{equation*}
				\omega_1=
							\begin{dcases}
									1 									&\text{ for } 0<\beta\leq \frac{2}{3},\\
									\frac{2}{3\beta}		&\text{ for } \frac{2}{3}\leq \beta\leq 1.
							\end{dcases}
		\end{equation*}
\end{lem}
\begin{proof}
Note that $Q_1\cdot E_1=2$, $Q_1\cdot C =3$ and $E_1\cdot C =1$. Therefore if $(Q_1\cdot E_1)\vert_{q_0}=2$, then $(Q_1\cdot C)\vert_{q_0}=1$ and if $3\geq (Q_1\cdot C)\vert_{q_0}\geq 2$, then $(Q_1\cdot E_1)\vert_{q_0}=1$.

Suppose $(Q_1\cdot E_1)\vert_{q_0}=2$. Then $(Q_1\cdot E_1)\vert_{q_0}=1$ since $C\cdot E_1=1$ and the minimal log resolution $f\colon \widetilde S \ra S$ of $(S, (1-\beta)C +\omega_1\beta(Q_1+E_1))$ consists of two blow-ups over $q_0$. The log pullback is
\begin{align*}
f^*(K_S+(1-\beta)C + \omega_1\beta(Q_1+E_1))&\simq-K_{\widetilde S} + (1-\beta)\widetilde C + \omega_1\beta(\widetilde Q_1+\widetilde E_1)\\
&+(2\omega_1\beta-\beta)F_1+(4\omega_1\beta-\beta-1)F_2.
\end{align*}
If $0<\beta\leq \frac{2}{3}$ then $2\omega_1\beta-\beta=\beta\leq \frac{2}{3}<1$ and $4\omega_1\beta-\beta-1=3\beta-1\leq 1$. If $\frac{2}{3}\leq \beta\leq 1$ then $2\omega_1\beta-\beta=\frac{4}{3}-\beta\leq \frac{2}{3}<1$ and $4\omega_1\beta-\beta-1=\frac{5}{3}-\beta\leq 1$. Therefore the pair $(S, (1-\beta)C +\omega_1\beta(Q_1+E_1))$ is log canonical.

Suppose $(Q_1\cdot C)\vert_{q_0}=3$. Then the minimal log resolution $f\colon \widetilde S \ra S$ of the pair $(S, (1-\beta)C +\omega_1\beta(Q_1+E_1))$ consists of 	three blow-ups over $q_0$. The log pullback is
\begin{align*}
f^*(K_S+(1-\beta)C + \omega_1\beta(Q_1+E_1))&\simq-K_{\widetilde S} + (1-\beta)\widetilde C + \omega_1\beta(\widetilde Q_1+\widetilde E_1)+\\
&(2\omega_1\beta-\beta)F_1+(3\omega_1\beta-2\beta)F_2+(4\omega_1\beta-3\beta)F_3.
\end{align*}
If $0<\beta\leq \frac{2}{3}$, then $2\omega_1\beta-\beta=\beta\leq \frac{2}{3}\leq 1$, $3\omega_1\beta-2\beta=\beta\leq \frac{2}{3} \leq 1$ and $4\omega_1\beta-3\beta=\beta\leq \frac{2}{3}\leq 1$.

If $0<\beta\leq \frac{2}{3}$, then $2\omega_1\beta-\beta=\frac{4}{3}-\beta\leq \frac{2}{3}\leq 1$, $3\omega_1\beta-2\beta=2-2\beta\leq \frac{2}{3} \leq 1$ and $4\omega_1\beta-3\beta=\frac{8}{3}-3\beta\leq \frac{2}{3}\leq 1$.

Therefore the pair $(S, (1-\beta)C +\omega_1\beta(Q_1+E_1))$ is log canonical.
\end{proof}

\begin{lem}
\label{lem:appendix-delPezzo-deg4-dynamic-lct7}
Let $S$ be a smooth del Pezzo surface, with $K_S^2=4$. Let $E_1$ be a line and $q_0\in E_1$ a point such that no line other than $E_1$ contains it. Let $C\in\vert-K_S\vert$ a smooth hyperplane section with $q_0\in C$. Consider the lines of $S$ with the notation of Table \ref{tab:delPezzo-4-lowdegree}. Let
$$B=\frac{1}{2}C_0+\frac{1}{2}(L_{12}+L_{13}+L_{14}+L_{15})+\frac{3}{2}E_1\simq-K_S.$$

Then the pair $(S,(1-\beta)C+\omega_1\beta B)$ is log canonical where 
	\begin{equation*}
				\omega_1=
							\begin{dcases}
									1 									&\text{ for } 0<\beta\leq \frac{2}{3},\\
									\frac{2}{3\beta}		&\text{ for } \frac{2}{3}\leq \beta\leq 1.
							\end{dcases}
		\end{equation*}
\end{lem}
\begin{proof}
Observe that $C_0\cdot L_{1j}=L_{1j}\cdot L_{1i}=0$ for all $i\neq j$ and that
$$E_1\cdot L_{1j}=C_0\cdot E_1=C\cdot C_0 = C\cdot L_{1j}=C\cdot E_1=1.$$
Therefore the worst discrepancy is achieved when $C$ intersects the $E_1$ at the point $q_0=E_1\cap L$ for $L$ a line, $L\neq E_1$ and $L\subset\Supp(B)$. Without loss of generality assume $q_0=E_1\cap C_0\cap C$. Then the minimal log resolution $f\colon \widetilde S \ra S$ of $(S, (1-\beta)C + \omega_1\beta B)$ consists of the blow-up of $q_0$. The log pullback is
$$f^*(K_S+(1-\beta)C +\omega_1\beta B)\simq K_{\widetilde S} + (1-\beta)\widetilde C + \omega_1\beta \widetilde B + (2\omega_1\beta-\beta)F_1.$$
If $0<\beta\leq \frac{2}{3}$, then $2\omega_1\beta-\beta=\beta\leq \frac{2}{3}\leq 1$.

If $\frac{2}{3}\leq \beta\leq 1$, then $2\omega_1\beta-\beta=\frac{4}{3}-\beta\leq \frac{2}{3}\leq 1$.

Since $\omega_1\beta\leq \frac{2}{3}$ for all $\beta$, then the pair $(S, (1-\beta)C + \omega_1\beta B)$ is log canonical.
\end{proof}

\begin{lem}
\label{lem:appendix-delPezzo-deg4-dynamic-lct8}
Let $S$ be a smooth del Pezzo surface, with $K_S^2=4$. Let $E_1$ be a line and $q_0\in E_1$ a point such that no line other than $E_1$ contains it. Let $C\in\vert-K_S\vert$ a smooth hyperplane section with $q_0\in C$. Consider the curves in $S$ with the notation of Table \ref{tab:delPezzo-4-lowdegree}. Let
$$H=\frac{3}{5}A_5+\frac{1}{5}(R_{125}+R_{135}+R_{145})+\frac{1}{5}Q_5+\frac{2}{5}E_1\simq-K_S,$$
and assume all the curves used in the definition of $H$ are irreducible and smooth. Then the pair $(S,(1-\beta)C+\omega_1\beta H)$ is log canonical where 
	\begin{equation*}
				\omega_1=
							\begin{dcases}
									1 									&\text{ for } 0<\beta\leq \frac{2}{3},\\
									\frac{2}{3\beta}		&\text{ for } \frac{2}{3}\leq \beta\leq 1.
							\end{dcases}
		\end{equation*}
\end{lem}
\begin{proof}
Obseve that
$$C\cdot A_5=A_5\cdot R_{1i5}=Q_5\cdot A_5=R_{1i5}\cdot R_{1j5}=2$$
for $i\neq j$, $C\cdot Q_5=C\cdot R_{1i5}=3$, and
$$E_1\cdot A_5=E_1\cdot R_{1i5}=E_1\cdot Q_5=C\cdot E_1=1.$$
Without loss of generality we may assume that $(C\cdot Q_5)\vert_{q_0}=3$ and $(C\cdot R_{1i5})\vert_{q_0}=2$, for all $i$. Moreover:
$$(C\cdot A_5)\vert_{q_0}=(C\cdot R_{1i5})\vert_{q_0}=(R_{1i5}\cdot R_{1j5})\vert_{q_0}=(A_5\cdot R_{1i5})\vert_{q_0}=2$$
for all $i$ and $i\neq j$. Then the minimal log resolution $f\colon \widetilde S \ra S$ of $(S,(1-\beta)C+\omega_1\beta H)$ consists of $3$ blow-ups with exceptional divisors $F_1,F_2,F_3$. The log pullback is
\begin{align*}
f^*(K_S+(1-\beta)C+\omega_1\beta H)&\simq K_{\widetilde S} + (1-\beta)\widetilde C + \omega_1\beta H \\
&+ (\frac{9}{5}\omega_1\beta-\beta)F_1 + (\frac{16}{5}\omega_1\beta-2\beta)F_2 + (\frac{17}{5}\omega_1\beta-3\beta)F_3.\\
\end{align*}
If $0<\beta\leq \frac{2}{3}$, then $\frac{9}{5}\omega_1\beta-\beta=\frac{4}{5}\beta\leq 1$, $\frac{16}{5}\omega_1\beta-2\beta=\frac{6}{5}\beta\leq 1$ and $\frac{17}{5}\omega_1\beta-3\beta=\frac{2}{5}\beta\leq 1$.

If $\frac{1}{2}<\beta\leq 1$, then $\frac{9}{5}\omega_1\beta-\beta=\frac{6}{5}-\beta\leq 1$, $\frac{16}{5}\omega_1\beta-2\beta=\frac{32}{15}-2\beta\leq 1$ and $\frac{17}{5}-3\beta\leq 1$. Therefore the pair $(S,(1-\beta)C+\omega_1\beta H)$ is log canonical. 
\end{proof}

\section{Local log canonical thresholds}
The computations in this section are local in nature and they are used in sections \ref{sec:dynamic-alpha-degree1}, \ref{sec:dynamic-alpha-degree2} and \ref{sec:dynamic-alpha-degree3}.
\begin{lem}
\label{lem:appendix-local-cusp-snc}
Let $S$ be  a smooth surface, $T$ be a cuspidal curve with local equation $\{y^2-x^3=0\}$ around $p=(0,0)$ and $C$ be a smooth curve not passing through $p$ and intersecting $T$ normally. If $0<\beta\leq 1$, then
$$\lct(S, (1-\beta)C, \beta T)=\frac{5}{6\beta}.$$
\end{lem}
\begin{proof}
Let $f\colon \widetilde S \ra S$ be a log resolution of the pair $(S, (1-\beta)C + \lambda\beta T)$. The morphism $f$ consists of $3$ consecutive blow-ups over $p$ with exceptional divisors $F_1,F_2,F_3$. The log pullback is
$$f^*(K_S+(1-\beta)C +\lambda\beta T)\simq K_{\widetilde S} + (1-\beta)\widetilde C + \lambda\beta \widetilde T + (2\lambda\beta-1)F_1 + (3\lambda\beta-2)F_2 + (6\lambda\beta-4)F_3.$$
Then $\lct(S, (1-\beta)C, \beta T)=\min\{\frac{1}{\beta}, \frac{5}{6\beta}\}=\frac{5}{6\beta}$.
\end{proof}

\begin{lem}
\label{lem:appendix-local-tacnode-snc}
Let $S$ be  a smooth surface, $L$ and $M$ be two smooth curves intersecting at a tacnodal point $p$ (local equation $\{(y^2-x)y=0\}$ around $p=(0,0)$) and $C$ be a smooth curve not passing through $p$ and intersecting $L$ and $M$ normally. If $0<\beta\leq 1$, then
$$\lct(S, (1-\beta)C, \beta (L+C))=\frac{3}{4\beta}.$$
\end{lem}
\begin{proof}
Let $f\colon \widetilde S \ra S$ be a log resolution of the pair $(S, (1-\beta)C + \lambda\beta (L+C))$. The morphism $f$ consists of $2$ consecutive blow-ups over $p$with exceptional curves $F_1,F_2$. The log pullback is
$$f^*(K_S+(1-\beta)C +\lambda\beta (L+C))\simq K_{\widetilde S} + (1-\beta)\widetilde C + \lambda\beta (\widetilde L + \widetilde M) + (2\lambda\beta-1)F_1 + (4\lambda\beta-2)F_2.$$
Then
$$\lct(S, (1-\beta)C, \beta (L+C))=\min\{\frac{1}{\beta}, \frac{3}{4\beta}\}=\frac{3}{4\beta}.$$
\end{proof}

\begin{lem}
\label{lem:appendix-local-cusp-normal}
Let $S$ be  a smooth surface, $T$ be a cuspidal curve with local equation $\{y^2-x^3=0\}$ around $p=(0,0)$ and $C$ be a smooth curve passing through $p$ such that $T$ and $C$ intersect with simple normal crossings away from $p$. Let $0<\beta\leq 1$. If $(C\cdot T)\vert_p=3$, then
$$\lct(S, (1-\beta)C, \beta T)=\frac{2+3\beta}{6\beta}.$$
If $(C\cdot T)\vert_p=2$, then
$$\lct(S, (1-\beta)C, \beta T)=\frac{3+2\beta}{6\beta}.$$
\end{lem}
\begin{proof}
Let $f\colon \widetilde S \ra S$ be a log resolution of the pair $(S, (1-\beta)C + \lambda\beta T)$. The morphism $f$ consists of $3$ consecutive blow-ups over $p$ with exceptional curves $F_1,F_2,F_3$.

Suppose $(C\cdot T)\vert_p=2$. Then the log pullback is
$$f^*(K_S+(1-\beta)C +\lambda\beta T)\simq K_{\widetilde S} + (1-\beta)\widetilde C + \lambda\beta \widetilde T + (2\lambda\beta-\beta)F_1 + (3\lambda\beta-\beta-1)F_2 + (6\lambda\beta-2\beta-2)F_3,$$
and we have that $$\lct(S, (1-\beta)C, \beta T)=\min\{\frac{1+\beta}{2\beta}, \frac{2+\beta}{3\beta}, \frac{3+2\beta}{6\beta}\}=\frac{3+2\beta}{6\beta}.$$

Suppose $(C\cdot T)\vert_p=3$. Then the log pullback is
$$f^*(K_S+(1-\beta)C +\lambda\beta T)\simq K_{\widetilde S} + (1-\beta)\widetilde C + \lambda\beta \widetilde T + (2\lambda\beta-\beta)F_1 + (3\lambda\beta-2\beta)F_2 + (6\lambda\beta-3\beta-1)F_3,$$
and we have that $$\lct(S, (1-\beta)C, \beta T)=\min\{\frac{1+\beta}{2\beta}, \frac{1+2\beta}{3\beta}, \frac{2+3\beta}{6\beta}\}=\frac{2+3\beta}{6\beta}.$$

\end{proof}

\begin{lem}
\label{lem:appendix-local-tacnode-normal}
Let $S$ be  a smooth surface, $L$ and $M$ be two smooth curves intersecting at a tacnode (local equation $\{(y^2-x)y=0\}$ around $p=(0,0)$) and $C$ be a smooth curve intersecting both $L$ and $M$ normally at $p$ and such that $\Supp(L\cup M\cup C)$ has simple normal crossings away from $p$. If $0<\beta\leq 1$, then
$$\lct(S, (1-\beta)C, \beta (L+C))=\frac{2+\beta}{4\beta}.$$
\end{lem}
\begin{proof}
Let $f\colon \widetilde S \ra S$ be a log resolution of the pair $(S, (1-\beta)C + \lambda\beta (L+C))$. The morphism $f$ consists of $2$ consecutive blow-ups over $p$. The log pullback is
$$f^*(K_S+(1-\beta)C +\lambda\beta (L+C))\simq K_{\widetilde S} + (1-\beta)\widetilde C + \lambda\beta (\widetilde L + \widetilde M) + (2\lambda\beta-\beta)F_1 + (4\lambda\beta-\beta-1)F_2.$$
Then
$$\lct(S, (1-\beta)C, \beta (L+C))=\min\{\frac{1+\beta}{2\beta}, \frac{2+\beta}{4\beta}\}=\frac{2+\beta}{4\beta}.$$
\end{proof}

\begin{lem}
\label{lem:appendix-local-4lines}
Let $S$ be  a smooth surface, $L_1, L_2,L_3, C$ four smooth curves intersecting each other normally at $p\in S$ and with simple normal crossings away from $p$. If $0<\beta\leq 1$, then
$$\lct(S, (1-\beta)C, \beta (L_1+L_2+L_3))=\frac{1+\beta}{3\beta}.$$
\end{lem}
\begin{proof}
Let $f\colon \widetilde S \ra S$ be a log resolution of the pair $(S, (1-\beta)C + \lambda\beta (L_1+L_2+L_3))$. The morphism $f$ consists of the blow-up of $p$. The log pullback is
$$f^*(K_S+(1-\beta)C +\lambda\beta (L_1+L_2+L_3))\simq K_{\widetilde S} + (1-\beta)\widetilde C + \lambda\beta (\widetilde L_1 + \widetilde L_2+\widetilde L_3) + (3\lambda\beta-\beta)F_1.$$
Then
$$\lct(S, (1-\beta)C, \beta (L+C))=\frac{1+\beta}{3\beta}.$$
\end{proof}

\begin{lem}
\label{lem:appendix-local-3lines}
Let $S$ be  a smooth surface, $L_1, L_2,L_3$ four smooth curves intersecting each other normally at $p\in S$ and $C$ a smooth curve intersecting $L_1,L_2,L_3$ at different points and normally. If $0<\beta\leq 1$, then
$$\lct(S, (1-\beta)C, \beta (L_1+L_2+L_3))=\frac{2}{3\beta}.$$
\end{lem}
\begin{proof}
Let $f\colon \widetilde S \ra S$ be a log resolution of the pair $(S, (1-\beta)C + \lambda\beta (L_1+L_2+L_3))$. The morphism $f$ consists of the blow-up of $p$. The log pullback is
$$f^*(K_S+(1-\beta)C +\lambda\beta (L_1+L_2+L_3))\simq K_{\widetilde S} + (1-\beta)\widetilde C + \lambda\beta (\widetilde L_1 + \widetilde L_2+\widetilde L_3) + (3\lambda\beta-1)F_1.$$
Then
$$\lct(S, (1-\beta)C, \beta (L+C))=\frac{2}{3\beta}.$$
\end{proof}

\chapter{Summary of results and numerical data}
\label{app:tables}
Let $S$ be a non-singular complex del Pezzo surface and $C\in \vert-K_S\vert$ be a smooth curve. Let $\beta\in (0,1]$. In this appendix, for the convenience of the reader, we summarise in tables the list of alpha--invariants $\alpha(S, (1-\beta)C)$. For each value, the last column indicates the interval for $\beta$ in which Theorem \ref{thm:Jeffres-Mazzeo-Rubinstein} guarantees the existence of a K\"ahler-Einstein metric with edge singularities (KEE) along $C$ of angle $2\pi\beta$. For simplicity we do not include all the notation when describing $C$. We refer the reader to each particular Theorem for the details.
\begin{table}[!ht]%
\centering
\begin{tabular}{|c|c|c|}
\hline
Characterisation of $C$ 						&$\alpha(S, (1-\beta)C)$			&$\alpha(S, (1-\beta)C>\frac{2}{3}$ 	\\
\hline
Any smooth curve.										&$\begin{dcases}
																				1 												&\text{ for } 0<\beta\leq \frac{1}{6},\\
																				\frac{1+3\beta}{9\beta}		&\text{ for } \frac{1}{6}\leq \beta\leq \frac{2}{3},\\
																				\frac{1}{3\beta}		&\text{ for } \frac{2}{3}\leq \beta\leq 1.
																		\end{dcases}$											
																																	&$(0,\frac{1}{3})$																																				\\
\hline
\end{tabular}
\caption{Projective Plane $\bbP^2$ (Theorem \ref{thm:del-Pezzo-dynamic-alpha-degree-9}).}
\label{tab:alpha-deg9}
\end{table}

\begin{table}[!h]%
\centering
\begin{tabular}{|c|c|c|}
\hline
Characterisation of $C$ 						&$\alpha(S, (1-\beta)C)$			&$\alpha(S, (1-\beta)C>\frac{2}{3}$ 	\\
\hline
Any smooth curve.										&$\begin{dcases}
																				1 												&\text{ for } 0<\beta\leq \frac{1}{4},\\
																				\frac{1+2\beta}{6\beta}		&\text{ for } \frac{1}{4}\leq \beta\leq1.
																		\end{dcases}$							
																																	&$(0,\frac{1}{2})$																																				\\
\hline
\end{tabular}
\caption{$\bbP^1\times\bbP^1$ (Theorem \ref{thm:del-Pezzo-dynamic-alpha-degree-8-F0}).}
\label{tab:alpha-deg8-F0}
\end{table}

\begin{table}[!ht]%
\centering
\begin{tabular}{|c|c|c|}
\hline
Characterisation of $C$ 						&$\alpha(S, (1-\beta)C)$			&$\alpha(S, (1-\beta)C>\frac{2}{3}$ 	\\
\hline
$(F\cdot C)\vert_r=2$
where $r=E\cap C$.									&$\omega_1=\begin{dcases}
																									1 												&\text{ for } 0<\beta\leq \frac{1}{6},\\
																									\frac{1+2\beta}{8\beta}		&\text{ for } \frac{1}{6}\leq \beta\leq \frac{5}{6},\\
																									\frac{1}{3\beta}					&\text{ for } \frac{5}{6}\leq \beta\leq 1.\\
																							\end{dcases}$
																																	&$(0,\frac{3}{10})$																																				\\
\hline
Otherwise.													&$\omega_2=\begin{dcases}
																				1 												&\text{ for } 0<\beta\leq \frac{1}{4},\\
																				\frac{1+\beta}{5\beta}		&\text{ for } \frac{1}{4}\leq \beta\leq \frac{2}{3},\\
																				\frac{1}{3\beta}					&\text{ for } \frac{2}{3}\leq \beta\leq 1.\\
																		\end{dcases}$
																																	&$(0,\frac{3}{7})$																		\\
\hline
\end{tabular}
\caption{Hirzebruch surface $\bbF_1$ (Theorem \ref{thm:del-Pezzo-dynamic-alpha-degree-8-F1}).}
\label{tab:alpha-deg8-F1}
\end{table}

\begin{table}[!ht]%
\centering
\begin{tabular}{|c|c|c|}
\hline
Characterisation of $C$ 						&$\alpha(S, (1-\beta)C)$			&$\alpha(S, (1-\beta)C>\frac{2}{3}$ 	\\
\hline
\bigcell{c}{$C$ has a \\
pseudo-Eckardt point.}							&$\omega_4=\begin{dcases}
																				1 												&\text{ for } 0<\beta\leq \frac{1}{4},\\
																				\frac{1+\beta}{5\beta}		&\text{ for } \frac{1}{4}\leq \beta\leq \frac{2}{3},\\
																				\frac{1}{3\beta}					&\text{ for } \frac{2}{3}\leq \beta\leq 1.\\
																		\end{dcases}$
																																	&$(0,\frac{3}{10})$																																				\\
\hline
\bigcell{c}{
No pseudo-Eckardt point in $C$\\
and $(C\cdot L_i)\vert_{r_i}=2$\\
for some $i=1,2$.}
																		&$\omega_3=\begin{dcases}
																				1 												&\text{ for } 0<\beta\leq \frac{1}{4},\\
																				\frac{1+2\beta}{6\beta}		&\text{ for } \frac{1}{4}\leq \beta\leq \frac{1}{2},\\
																				\frac{1}{3\beta}					&\text{ for } \frac{1}{2}\leq \beta\leq 1.\\
																		\end{dcases}$
																																	&$(0,\frac{1}{2})$																		\\
\hline
\bigcell{c}{
No pseudo-Eckardt point in $C$,\\
$(C\cdot L_i)\vert_{r_i}=1$\\
for $i=1,2$, and\\
$(R\cdot C)\vert_r=3$.}
																		&$\omega_2=\begin{dcases}
																				1 												&\text{ for } 0<\beta\leq \frac{1}{4},\\
																				\frac{1+3\beta}{7\beta}		&\text{ for } \frac{1}{4}\leq \beta\leq \frac{4}{9},\\
																				\frac{1}{3\beta}					&\text{ for } \frac{4}{9}\leq \beta\leq 1.\\
																		\end{dcases}$
																																	&$(0,\frac{1}{2})$																		\\
\hline
Otherwise.													&$\omega_1=\begin{dcases}
																				1 												&\text{ for } 0<\beta\leq \frac{1}{3},\\
																				\frac{1}{3\beta}					&\text{ for } \frac{1}{3}\leq \beta\leq 1.\\
																		\end{dcases}$
																																	&$(0,\frac{1}{2})$																		\\
\hline

\end{tabular}
\caption{Del Pezzo surface of degree 7 (Theorem \ref{thm:del-Pezzo-dynamic-alpha-degree-7}).}
\label{tab:alpha-deg7}
\end{table}

\begin{table}[!ht]%
\centering
\begin{tabular}{|c|c|c|}
\hline
Characterisation of $C$ 						&$\alpha(S, (1-\beta)C)$			&$\alpha(S, (1-\beta)C>\frac{2}{3}$ 	\\
\hline
\bigcell{c}{$C$ has a \\
pseudo-Eckardt point.}							&$\omega_3=\begin{dcases}
																				1 												&\text{ for } 0<\beta\leq \frac{1}{3},\\
																				\frac{1+\beta}{4\beta}		&\text{ for } \frac{1}{3}\leq \beta\leq 1.
																		\end{dcases}$
																																	&$(0,\frac{3}{5})$																																				\\
\hline
\bigcell{c}{
No pseudo-Eckardt point in $C$\\
and $(C\cdot L)\vert_{L\cap E_1}=2$\\
for some model of $S$.}
																		&$\omega_2=\begin{dcases}
																				1 												&\text{ for } 0<\beta\leq \frac{1}{3},\\
																				\frac{1+2\beta}{5\beta}					&\text{ for } \frac{1}{3}\leq \beta\leq \frac{3}{4},\\
																				\frac{1}{2\beta}					&\text{ for } \frac{3}{4}\leq \beta\leq 1.
																		\end{dcases}$
																																	&$(0,\frac{3}{4})$																		\\
\hline
Otherwise.													&$\omega_1=
																		\begin{dcases}
																				1 												&\text{ for } 0<\beta\leq \frac{1}{2},\\
																				\frac{1}{2\beta}					&\text{ for } \frac{1}{2}\leq \beta\leq 1.
																		\end{dcases}$
																																	&$(0,\frac{3}{4})$																		\\
\hline

\end{tabular}
\caption{Del Pezzo surface of degree 6 (Theorem \ref{thm:del-Pezzo-dynamic-alpha-degree-6}).}
\label{tab:alpha-deg6}
\end{table}

\begin{table}[!ht]%
\centering
\begin{tabular}{|c|c|c|}
\hline
Characterisation of $C$ 						&$\alpha(S, (1-\beta)C)$			&$\alpha(S, (1-\beta)C>\frac{2}{3}$ 	\\
\hline
Any smooth curve.										&$\begin{dcases}
																				1 												&\text{ for } 0<\beta\leq \frac{1}{2},\\
																				\frac{1}{2\beta}					&\text{ for } \frac{1}{2}\leq \beta\leq 1.
																		\end{dcases}$
																																	&$(0,\frac{3}{4})$																																				\\
\hline
\end{tabular}
\caption{Del Pezzo surface of degree 5 (Theorem \ref{thm:del-Pezzo-dynamic-alpha-degree-5}).}
\label{tab:alpha-deg5}
\end{table}

\begin{table}[!ht]%
\centering
\begin{tabular}{|c|c|c|}
\hline
Characterisation of $C$ 						&$\alpha(S, (1-\beta)C)$			&$\alpha(S, (1-\beta)C>\frac{2}{3}$ 	\\
\hline
\bigcell{c}{$C$ has a \\
pseudo-Eckardt point.}							&$\omega_3=\begin{dcases}
																				1 									&\text{ for } 0<\beta\leq \frac{1}{2},\\
																				\frac{1+\beta}{3\beta}		&\text{ for } \frac{1}{2}\leq \beta\leq 1.
																		\end{dcases}$
																																	&$(0,1)$																																				\\
\hline
\bigcell{c}{
No pseudo-Eckardt point in $C$\\
but $(A\cdot C)\vert_p=(B\cdot C)\vert_p=2$\\
for $A+B\in \vert-K_S\vert$ and $A\cup b = p$.}
																		&$\omega_2=\begin{dcases}
																					1 									&\text{ for } 0<\beta\leq \frac{1}{2},\\
																					\frac{1+2\beta}{4\beta}		&\text{ for } \frac{1}{2}\leq \beta\leq \frac{5}{6},\\
																					\frac{2}{3\beta}		&\text{ for } \frac{5}{6}\leq \beta\leq 1.
																		\end{dcases}$
																																	&$(0,1)$																		\\
\hline
Otherwise.														&$\omega_1=\begin{dcases}
																				1 									&\text{ for } 0<\beta\leq \frac{2}{3},\\
																				\frac{2}{3\beta}		&\text{ for } \frac{2}{3}\leq \beta\leq 1.
																		\end{dcases}$
																																	&$(0,1)$																		\\
\hline

\end{tabular}
\caption{Del Pezzo surface of degree 4 (Theorem \ref{thm:del-Pezzo-dynamic-alpha-degree-4}).}
\label{tab:alpha-deg4}
\end{table}

\begin{table}[!ht]%
\centering
\begin{tabular}{|c|c|c|}
\hline
Characterisation of $C$ 						&$\alpha(S, (1-\beta)C)$			&$\alpha(S, (1-\beta)C>\frac{2}{3}$ 	\\
\hline
\bigcell{c}{$p_B\in C$ the tacnodal\\
point of some $B\in \vert-K_S\vert$.}							
																		&$\omega_3=\begin{dcases}
																				1 									&\text{ for } 0<\beta\leq \frac{2}{3},\\
																				\frac{2+\beta}{4\beta}		&\text{ for } \frac{2}{3}\leq \beta\leq 1.
																		\end{dcases}$
																																	&$(0,1)$																																				\\
\hline
\bigcell{c}{
$C$ contains no tacnodal points\\
but $p_T\in C$, the cuspidal point\\
of some $T\in \vert-K_S\vert$\\
such that $(C\cdot D)\vert_{p_T}=3$.}
																		&$\omega_2=\begin{dcases}
																				1 									&\text{ for } 0<\beta\leq \frac{2}{3},\\
																				\frac{2+3\beta}{6\beta}		&\text{ for } \frac{2}{3}\leq \beta\leq \frac{5}{6}.\\
																				\frac{3}{4\beta}		&\text{ for } \frac{5}{6}\leq \beta\leq 1.
																		\end{dcases}$
																																	&$(0,1]$																		\\
\hline
Otherwise.													&$\omega_1=\begin{dcases}
																				1 									&\text{ for } 0<\beta\leq \frac{3}{4},\\
																				\frac{3}{4\beta}		&\text{ for } \frac{3}{4}\leq \beta\leq 1.
																		\end{dcases}$
																																	&$(0,1]$																		\\
\hline

\end{tabular}
\caption{Smooth cubic surface with no Eckardt points (Theorem \ref{thm:del-Pezzo-dynamic-alpha-degree-3}).}
\label{tab:alpha-deg3a}
\end{table}

\begin{table}[!ht]%
\centering
\begin{tabular}{|c|c|c|}
\hline
Characterisation of $C$ 						&$\alpha(S, (1-\beta)C)$			&$\alpha(S, (1-\beta)C>\frac{2}{3}$ 	\\
\hline
\bigcell{c}{
$C$ contains some\\
Eckardt point.}
																		&$\omega_5=\begin{dcases}
																				1 									&\text{ for } 0<\beta\leq \frac{1}{2},\\
																				\frac{1+\beta}{3\beta}		&\text{ for } \frac{1}{2}\leq \beta\leq 1.
																		\end{dcases}$
																																	&$(0,1]$																		\\
\hline
Otherwise.													&$\omega_4=\begin{dcases}
																				1 									&\text{ for } 0<\beta\leq \frac{2}{3},\\
																				\frac{2}{3\beta}		&\text{ for } \frac{2}{3}\leq \beta\leq 1.
																		\end{dcases}$
																																	&$(0,1]$																		\\
\hline

\end{tabular}
\caption{Smooth cubic surface with some Eckardt point (Theorem \ref{thm:del-Pezzo-dynamic-alpha-degree-3}).}
\label{tab:alpha-deg3b}
\end{table}

\begin{table}[!ht]%
\centering
\begin{tabular}{|c|c|c|}
\hline
Characterisation of $C$ 						&$\alpha(S, (1-\beta)C)$			&$\alpha(S, (1-\beta)C>\frac{2}{3}$ 	\\
\hline
\bigcell{c}{
$p\in C$ for $p$ the singularity\\
of a cuspidal $B\in \vert-K_S\vert$}
																		&$\omega_2=\begin{dcases}
																					1 									&\text{ for } 0<\beta\leq \frac{3}{4},\\
																					\frac{3+2\beta}{6\beta}		&\text{ for } \frac{3}{4}\leq \beta\leq 1.
																			\end{dcases}$
																																	&$(0,1]$																		\\
\hline
Otherwise.													&$\omega_1=\begin{dcases}
																					1 									&\text{ for } 0<\beta\leq \frac{5}{6},\\
																					\frac{5}{6\beta}		&\text{ for } \frac{5}{6}\leq \beta\leq 1.
																			\end{dcases}$
																																	&$(0,1]$																		\\
\hline

\end{tabular}
\caption{Del Pezzo surface of degree $2$ without tacnodal singular curves in $\vert-K_S\vert$ (Theorem \ref{thm:del-Pezzo-dynamic-alpha-degree-2}).}
\label{tab:alpha-deg2a}
\end{table}

\begin{table}[!ht]%
\centering
\begin{tabular}{|c|c|c|}
\hline
Characterisation of $C$ 						&$\alpha(S, (1-\beta)C)$			&$\alpha(S, (1-\beta)C>\frac{2}{3}$ 	\\
\hline
\bigcell{c}{
$p\in C$ for $p$ the singularity\\
of a tacnodal $B\in \vert-K_S\vert$}
																		&$\omega_4=\begin{dcases}
																					1 									&\text{ for } 0<\beta\leq \frac{2}{3},\\
																					\frac{2+\beta}{4\beta}		&\text{ for } \frac{2}{3}\leq \beta\leq 1.
																			\end{dcases}$
																																	&$(0,1]$																		\\
\hline
Otherwise.													&$\omega_3=\begin{dcases}
																					1 									&\text{ for } 0<\beta\leq \frac{3}{4},\\
																					\frac{3}{4\beta}		&\text{ for } \frac{3}{4}\leq \beta\leq 1.
																			\end{dcases}$
																																	&$(0,1]$																		\\
\hline

\end{tabular}
\caption{Del Pezzo surface of degree $2$ without some tacnodal singular curve in $\vert-K_S\vert$ (Theorem \ref{thm:del-Pezzo-dynamic-alpha-degree-2}).}
\label{tab:alpha-deg2b}
\end{table}

\begin{table}[!ht]%
\centering
\begin{tabular}{|c|c|c|}
\hline
Characterisation of $C$ 						&$\alpha(S, (1-\beta)C)$			&$\alpha(S, (1-\beta)C>\frac{2}{3}$ 	\\
\hline
\bigcell{c}{
$\vert-K_S\vert$ contains a cuspidal\\
rational curve.}
																		&$\omega_2=\begin{dcases}
																					1 									&\text{ for } 0<\beta\leq \frac{5}{6},\\
																					\frac{5}{6}		&\text{ for } \frac{5}{6}\leq \beta\leq 1.
																			\end{dcases}$
																																	&$(0,1]$																		\\
\hline
Otherwise.													&$\omega_1=1$									&$(0,1]$																		\\
\hline

\end{tabular}
\caption{Del Pezzo surface of degree $1$ (Theorem \ref{thm:del-Pezzo-dynamic-alpha-degree-1}).}
\label{tab:alpha-deg1}
\end{table}

\bibliographystyle{hep}
\bibliography{Bibliography}

\begin{thebibliography}{{Don}10}

\bibitem[Amb06]{Ambro-Minimal-log-discrepancy}
F.~Ambro, \textsl{ The minimal log discrepancy},
\newblock Proc. Symp. Multiplier ideals and arc spaces, K. Watanabe (Ed.).
  \textbf{ 1550}, 121--130 (2006).

\bibitem[Aub76]{Aubin-CalabiProblem}
T.~Aubin, \textsl{ \'{E}quations du type {M}onge-{A}mp\`ere sur les
  vari\'et\'es k\"ahleriennes compactes},
\newblock C. R. Acad. Sci. Paris S\'er. A-B \textbf{ 283}(3), Aiii, A119--A121
  (1976).

\bibitem[Bea96]{Beau}
A.~Beauville,
\newblock \textsl{ Complex algebraic surfaces}, volume~34 of \textsl{ London
  Mathematical Society Student Texts},
\newblock Cambridge University Press, Cambridge, second edition, 1996,
\newblock Translated from the 1978 French original by R. Barlow, with
  assistance from N. I. Shepherd-Barron and M. Reid.

\bibitem[{Ber}10]{Berman-dynamic-alpha}
R.~J. {Berman}, \textsl{ {A thermodynamical formalism for Monge-Ampere
  equations, Moser-Trudinger inequalities and Kahler-Einstein metrics}},
\newblock ArXiv e-prints  (November 2010), {1011.3976}.

\bibitem[{Ber}12]{Berman-ke-implies-kstability}
R.~J. {Berman}, \textsl{ {K-polystability of Q-Fano varieties admitting
  Kahler-Einstein metrics}},
\newblock ArXiv e-prints  (May 2012), {1205.6214}.

\bibitem[CDS12a]{Chen-Donaldson-Sun-Kstability2}
X.~{Chen}, S.~{Donaldson} and S.~{Sun}, \textsl{ {{Kahler-Einstein metrics on
  Fano manifolds, II: limits with cone angle less than 2{$\pi$}}}},
\newblock ArXiv e-prints  (December 2012), {1212.4714}.

\bibitem[CDS12b]{Chen-Donaldson-Sun-Kstability}
X.-X. {Chen}, S.~{Donaldson} and S.~{Sun}, \textsl{ {Kahler-Einstein metrics
  and stability}},
\newblock ArXiv e-prints  (October 2012), {1210.7494}.

\bibitem[CDS12c]{Chen-Donaldson-Sun-Kstability1}
X.-X. {Chen}, S.~{Donaldson} and S.~{Sun}, \textsl{ {Kahler-Einstein metrics on
  Fano manifolds, I: approximation of metrics with cone singularities}},
\newblock ArXiv e-prints  (November 2012), {1211.4566}.

\bibitem[CDS13]{Chen-Donaldson-Sun-Kstability3}
X.~{Chen}, S.~{Donaldson} and S.~{Sun}, \textsl{ {{Kahler-Einstein metrics on
  Fano manifolds, III: limits as cone angle approaches 2{$\pi$} and completion
  of the main proof}}},
\newblock ArXiv e-prints  (February 2013), {1302.0282}.

\bibitem[Che]{Cheltsov-Trento-Proceedings}
I.~Cheltsov, \textsl{ Del Pezzo surfaces and local inequalities},
\newblock Proceedings of the conference Groups of automorphisms in birational
  and affine geometry. CIRM, Trento. To appear .

\bibitem[Che05]{Cheltsov-birationally-rigid-del-Pezzo-fibrations}
I.~Cheltsov, \textsl{ Birationally rigid del {P}ezzo fibrations},
\newblock Manuscripta Math. \textbf{ 116}(4), 385--396 (2005).

\bibitem[Che08]{CheltsovLCTdP}
I.~Cheltsov, \textsl{ Log canonical thresholds of del {P}ezzo surfaces},
\newblock Geom. Funct. Anal. \textbf{ 18}(4), 1118--1144 (2008).

\bibitem[CK10]{Cheltsov-Kosta}
I.~{Cheltsov} and D.~{Kosta}, \textsl{ {Computing {$\alpha$}--invariants of
  singular del Pezzo surfaces}},
\newblock ArXiv e-prints  (September 2010), {1010.0043}.

\bibitem[Cor07]{Corti-book-Flips-3folds-4folds}
A.~Corti, editor,
\newblock \textsl{ Flips for 3-folds and 4-folds}, volume~35 of \textsl{ Oxford
  Lecture Series in Mathematics and its Applications},
\newblock Oxford University Press, Oxford, 2007.

\bibitem[CP08]{CossartPiltant1}
V.~Cossart and O.~Piltant, \textsl{ Resolution of singularities of threefolds
  in positive characteristic. {I}. {R}eduction to local uniformization on
  {A}rtin-{S}chreier and purely inseparable coverings},
\newblock J. Algebra \textbf{ 320}(3), 1051--1082 (2008).

\bibitem[CP09]{CossartPiltant2}
V.~Cossart and O.~Piltant, \textsl{ Resolution of singularities of threefolds
  in positive characteristic. {II}},
\newblock J. Algebra \textbf{ 321}(7), 1836--1976 (2009).

\bibitem[CPR00]{Corti-Pukhlikov-Reid}
A.~Corti, A.~Pukhlikov and M.~Reid,
\newblock Fano {$3$}-fold hypersurfaces,
\newblock in \textsl{ Explicit birational geometry of 3-folds}, volume 281 of
  \textsl{ London Math. Soc. Lecture Note Ser.}, pages 175--258, Cambridge
  Univ. Press, Cambridge, 2000.

\bibitem[CPS10]{Cheltsov-Shramov-Park-WPS}
I.~Cheltsov, J.~Park and C.~Shramov, \textsl{ Exceptional del {P}ezzo
  hypersurfaces},
\newblock J. Geom. Anal. \textbf{ 20}(4), 787--816 (2010).

\bibitem[CPW13]{Cheltsov-Park-Won-Sancho}
I.~{Cheltsov}, J.~{Park} and J.~{Won}, \textsl{ {Affine cones over smooth cubic
  surfaces}},
\newblock ArXiv e-prints  (March 2013), {1303.2648}.

\bibitem[CS08]{CheltsovShramovLct3folds}
I.~A. Cheltsov and K.~A. Shramov, \textsl{ Log-canonical thresholds for
  nonsingular {F}ano threefolds},
\newblock Uspekhi Mat. Nauk \textbf{ 63}(5(383)), 73--180 (2008).

\bibitem[DK01]{Demailly-Kollar}
J.-P. Demailly and J.~Koll{\'a}r, \textsl{ Semi-continuity of complex
  singularity exponents and {K}\"ahler-{E}instein metrics on {F}ano orbifolds},
\newblock Ann. Sci. \'Ecole Norm. Sup. (4) \textbf{ 34}(4), 525--556 (2001).

\bibitem[Dol12]{DolgTopicsClass}
I.~V. Dolgachev,
\newblock \textsl{ Classical algebraic geometry},
\newblock Cambridge University Press, Cambridge, 2012,
\newblock A modern view.

\bibitem[{Don}10]{DonaldsonStabilityKEsurvey}
S.~{Donaldson}, \textsl{ {Stability, birational transformations and the
  Kahler-Einstein problem}},
\newblock ArXiv e-prints  (July 2010), {1007.4220}.

\bibitem[DPT80]{DemazureDelPezzo}
M.~Demazure, H.~C. Pinkham and B.~Teissier, editors,
\newblock \textsl{ S\'eminaire sur les {S}ingularit\'es des {S}urfaces}, volume
  777 of \textsl{ Lecture Notes in Mathematics},
\newblock Springer, Berlin, 1980,
\newblock Held at the Centre de Math{\'e}matiques de l'{\'E}cole Polytechnique,
  Palaiseau, 1976--1977.

\bibitem[DV34a]{DuVal1}
P.~Du~Val, \textsl{ On isolated singularities of surfaces which do not affect
  the conditions of adjunction (Part I.)},
\newblock Mathematical Proceedings of the Cambridge Philosophical Society
  \textbf{ 30}, 453--459 (9 1934).

\bibitem[DV34b]{DuVal2}
P.~Du~Val, \textsl{ On isolated singularities of surfaces which do not affect
  the conditions of adjunction (Part II.)},
\newblock Mathematical Proceedings of the Cambridge Philosophical Society
  \textbf{ 30}, 460--465 (9 1934).

\bibitem[DV34c]{DuVal3}
P.~Du~Val, \textsl{ On isolated singularities of surfaces which do not affect
  the conditions of adjunction (Part III.)},
\newblock Mathematical Proceedings of the Cambridge Philosophical Society
  \textbf{ 30}, 483--491 (9 1934).

\bibitem[Eck76]{Eckardt-points}
F.~E. Eckardt, \textsl{ Ueber diejenigen {F}l\"achen dritten {G}rades, auf
  denen sich drei gerade {L}inien in einem {P}unkte schneiden},
\newblock Math. Ann. \textbf{ 10}(2), 227--272 (1876).

\bibitem[Ful89]{FultonAlgebraicCurves}
W.~Fulton,
\newblock \textsl{ Algebraic curves},
\newblock Advanced Book Classics, Addison-Wesley Publishing Company Advanced
  Book Program, Redwood City, CA, 1989,
\newblock An introduction to algebraic geometry, Notes written with the
  collaboration of Richard Weiss, Reprint of 1969 original.

\bibitem[Har77]{HartshorneAG}
R.~Hartshorne,
\newblock \textsl{ Algebraic geometry},
\newblock Springer-Verlag, New York, 1977,
\newblock Graduate Texts in Mathematics, No. 52.

\bibitem[Hau03]{HauserSingResol}
H.~Hauser, \textsl{ The {H}ironaka theorem on resolution of singularities (or:
  {A} proof we always wanted to understand)},
\newblock Bull. Amer. Math. Soc. (N.S.) \textbf{ 40}(3), 323--403 (electronic)
  (2003).

\bibitem[Hir64]{HironakaSingResol}
H.~Hironaka, \textsl{ Resolution of singularities of an algebraic variety over
  a field of characteristic zero. {I}, {II}},
\newblock Ann. of Math. (2) 79 (1964), 109--203; ibid. (2) \textbf{ 79},
  205--326 (1964).

\bibitem[Igu77]{Igusa}
J.~Igusa,
\newblock On the first terms of certain asymptotic expansions,
\newblock in \textsl{ Complex analysis and algebraic geometry}, pages 357--368,
  Iwanami Shoten, Tokyo, 1977.

\bibitem[JMR11]{JeffresMazzeoRubinstein}
T.~D. {Jeffres}, R.~{Mazzeo} and Y.~A. {Rubinstein}, \textsl{
  {K}\"ahler-Einstein metrics with edge singularities},
\newblock ArXiv e-prints  (May 2011), {1105.5216}.

\bibitem[KM98]{KollarMori}
J.~Koll{\'a}r and S.~Mori,
\newblock \textsl{ Birational geometry of algebraic varieties}, volume 134 of
  \textsl{ Cambridge Tracts in Mathematics},
\newblock Cambridge University Press, Cambridge, 1998,
\newblock With the collaboration of C. H. Clemens and A. Corti, Translated from
  the 1998 Japanese original.

\bibitem[KM99]{Keel-McKernan-book}
S.~Keel and J.~McKernan, \textsl{ Rational curves on quasi-projective
  surfaces},
\newblock Mem. Amer. Math. Soc. \textbf{ 140}(669), viii+153 (1999).

\bibitem[Kol96]{Kollar-Rational-Curves-Algebraic-Varieties}
J.~Koll{\'a}r,
\newblock \textsl{ Rational curves on algebraic varieties}, volume~32 of
  \textsl{ Ergebnisse der Mathematik und ihrer Grenzgebiete. 3. Folge. A Series
  of Modern Surveys in Mathematics [Results in Mathematics and Related Areas.
  3rd Series. A Series of Modern Surveys in Mathematics]},
\newblock Springer-Verlag, Berlin, 1996.

\bibitem[Kol97]{Kollar-Singularities-of-pairs}
J.~Koll{\'a}r,
\newblock Singularities of pairs,
\newblock in \textsl{ Algebraic geometry---{S}anta {C}ruz 1995}, volume~62 of
  \textsl{ Proc. Sympos. Pure Math.}, pages 221--287, Amer. Math. Soc.,
  Providence, RI, 1997.

\bibitem[Kol07]{Kollar-Resolution-Singularities}
J.~Koll{\'a}r,
\newblock \textsl{ Lectures on resolution of singularities}, volume 166 of
  \textsl{ Annals of Mathematics Studies},
\newblock Princeton University Press, Princeton, NJ, 2007.

\bibitem[Kos09]{Kosta-thesis}
D.~Kosta,
\newblock \textsl{ Del Pezzo surfaces with Du Val singularities},
\newblock PhD thesis, University of Edinburgh, 2009.

\bibitem[KPZ11a]{Kishimoto-Prokhorov-Zaidenberg-3fold}
T.~{Kishimoto}, Y.~{Prokhorov} and M.~{Zaidenberg}, \textsl{ {Affine cones over
  Fano threefolds and additive group actions}},
\newblock ArXiv e-prints  (June 2011), {1106.1312}.

\bibitem[KPZ11b]{Kishimoto-Prokhorov-Zaidenberg1}
T.~Kishimoto, Y.~Prokhorov and M.~Zaidenberg,
\newblock Group actions on affine cones,
\newblock in \textsl{ Affine algebraic geometry}, volume~54 of \textsl{ CRM
  Proc. Lecture Notes}, pages 123--163, Amer. Math. Soc., Providence, RI, 2011.

\bibitem[KPZ12a]{Kishimoto-Prokhorov-Zaidenberg-general}
T.~{Kishimoto}, Y.~{Prokhorov} and M.~{Zaidenberg}, \textsl{
  {{$\mathbb{G}_a$}-actions on affine cones}},
\newblock ArXiv e-prints  (December 2012), {1212.4249}.

\bibitem[KPZ12b]{Kishimoto-Prokhorov-Zaidenberg2}
T.~{Kishimoto}, Y.~{Prokhorov} and M.~{Zaidenberg}, \textsl{ {Unipotent Group
  Actions on Del Pezzo Cones}},
\newblock ArXiv e-prints  (December 2012), {1212.4479}.

\bibitem[Laz04]{Lazarsfeld1}
R.~Lazarsfeld,
\newblock \textsl{ Positivity in algebraic geometry. {I}}, volume~48 of
  \textsl{ Ergebnisse der Mathematik und ihrer Grenzgebiete. 3. Folge. A Series
  of Modern Surveys in Mathematics [Results in Mathematics and Related Areas.
  3rd Series. A Series of Modern Surveys in Mathematics]},
\newblock Springer-Verlag, Berlin, 2004,
\newblock Classical setting: line bundles and linear series.

\bibitem[LX11]{Li-Xu-K-stability}
C.~{Li} and C.~{Xu}, \textsl{ {Special test configurations and K-stability of
  Fano varieties}},
\newblock ArXiv e-prints  (November 2011), {1111.5398}.

\bibitem[Man86]{ManinCubicForms}
Y.~I. Manin,
\newblock \textsl{ Cubic forms}, volume~4 of \textsl{ North-Holland
  Mathematical Library},
\newblock North-Holland Publishing Co., Amsterdam, second edition, 1986,
\newblock Algebra, geometry, arithmetic, Translated from the Russian by M.
  Hazewinkel.

\bibitem[{Mar}12]{JMGlctCharP}
J.~{Martinez-Garcia}, \textsl{ {Log canonical thresholds of Del Pezzo Surfaces
  in characteristic p}},
\newblock ArXiv e-prints  (March 2012), {1203.0995}.

\bibitem[Mor79]{Mori-Hartshorne-Conjecture}
S.~Mori, \textsl{ Projective manifolds with ample tangent bundles},
\newblock Ann. of Math. (2) \textbf{ 110}(3), 593--606 (1979).

\bibitem[Mor82]{Mori-Cone-Theorem}
S.~Mori, \textsl{ Threefolds whose canonical bundles are not numerically
  effective},
\newblock Ann. of Math. (2) \textbf{ 116}(1), 133--176 (1982).

\bibitem[Nad90]{Nadel-KE-metrics}
A.~M. Nadel, \textsl{ Multiplier ideal sheaves and {K}\"ahler-{E}instein
  metrics of positive scalar curvature},
\newblock Ann. of Math. (2) \textbf{ 132}(3), 549--596 (1990).

\bibitem[Oda13]{OdakaAnnals}
Y.~Odaka, \textsl{ The {GIT} stability of polarized varieties via discrepancy},
\newblock Ann. of Math. (2) \textbf{ 177}(2), 645--661 (2013).

\bibitem[OS10]{OdakaSanoAlphaInvariant}
Y.~{Odaka} and Y.~{Sano}, \textsl{ {Alpha invariant and K-stability of Q-Fano
  varieties}},
\newblock ArXiv e-prints  (November 2010), {1011.6131}.

\bibitem[PW10]{Park-Won-lctDuVal}
J.~Park and J.~Won, \textsl{ Log-canonical thresholds on del {P}ezzo surfaces
  of degrees {$\geq 2$}},
\newblock Nagoya Math. J. \textbf{ 200}, 1--26 (2010).

\bibitem[Rei87]{Reid-YPG}
M.~Reid,
\newblock Young person's guide to canonical singularities,
\newblock in \textsl{ Algebraic geometry, {B}owdoin, 1985 ({B}runswick,
  {M}aine, 1985)}, volume~46 of \textsl{ Proc. Sympos. Pure Math.}, pages
  345--414, Amer. Math. Soc., Providence, RI, 1987.

\bibitem[Tia87]{TianAlphaInvariant}
G.~Tian, \textsl{ On {K}\"ahler-{E}instein metrics on certain {K}\"ahler
  manifolds with {$C_1(M)>0$}},
\newblock Invent. Math. \textbf{ 89}(2), 225--246 (1987).

\bibitem[Tia90a]{Tian-del-Pezzo-2}
G.~Tian, \textsl{ On {C}alabi's conjecture for complex surfaces with positive
  first {C}hern class},
\newblock Invent. Math. \textbf{ 101}(1), 101--172 (1990).

\bibitem[Tia90b]{Tian-conjecture}
G.~Tian, \textsl{ On a set of polarized {K}\"ahler metrics on algebraic
  manifolds},
\newblock J. Differential Geom. \textbf{ 32}(1), 99--130 (1990).

\bibitem[Tia97]{TianKEimpliesAnalyticKstability}
G.~Tian, \textsl{ K\"ahler-{E}instein metrics with positive scalar curvature},
\newblock Invent. Math. \textbf{ 130}(1), 1--37 (1997).

\bibitem[{Tia}12]{Tian-Kstability-solution}
G.~{Tian}, \textsl{ {K-stability and {K}\"ahler-Einstein metrics}},
\newblock ArXiv e-prints  (November 2012), {1211.4669}.

\bibitem[Yau77]{Yau-Calabi-announcement}
S.~T. Yau, \textsl{ Calabi's conjecture and some new results in algebraic
  geometry},
\newblock Proc. Nat. Acad. Sci. U.S.A. \textbf{ 74}(5), 1798--1799 (1977).

\bibitem[Yau78]{Yau-Calabi-conjecture-paper}
S.~T. Yau, \textsl{ On the {R}icci curvature of a compact {K}\"ahler manifold
  and the complex {M}onge-{A}mp\`ere equation. {I}},
\newblock Comm. Pure Appl. Math. \textbf{ 31}(3), 339--411 (1978).

\bibitem[Yau96]{Yau-YTD-conjecture}
S.-T. Yau,
\newblock Review on {K}\"ahler-{E}instein metrics in algebraic geometry,
\newblock in \textsl{ Proceedings of the {H}irzebruch 65 {C}onference on
  {A}lgebraic {G}eometry ({R}amat {G}an, 1993)}, volume~9 of \textsl{ Israel
  Math. Conf. Proc.}, pages 433--443, Ramat Gan, 1996, Bar-Ilan Univ.

\end{thebibliography}

\end{document}